\documentclass[a4paper,11pt]{article}
\usepackage[utf8]{inputenc}
\usepackage[T1]{fontenc}
\usepackage{textcomp}
\usepackage{fullpage}
\usepackage{lmodern}
\usepackage{amsmath,amssymb,bm,bbm}
\usepackage{algorithm}
\usepackage{algorithmic}
\usepackage{float}
\usepackage{graphicx}
\usepackage{extarrows}
\usepackage{caption}
\usepackage{graphicx, subfig}
\usepackage{titling}
\usepackage{yhmath}
\usepackage{mathrsfs}
\usepackage{cite}
\usepackage{amsthm}
\usepackage{color}
\usepackage{slashed}
\usepackage{verbatim}
\usepackage[colorlinks=true,linkcolor=red,citecolor=blue]{hyperref}
\usepackage{amsfonts,euscript}
\usepackage{latexsym, multicol, fancybox}
\usepackage{epstopdf}
\usepackage{psfrag}
\usepackage{mathrsfs}
\usepackage{xcolor}
\usepackage{comment}
\usepackage{authblk}
\usepackage{indentfirst}
\DeclareFontFamily{U}{mathx}{\hyphenchar\font45}
\DeclareFontShape{U}{mathx}{m}{n}{
<5> <6> <7> <8> <9> <10>
<10.95> <12> <14.4> <17.28> <20.74> <24.88>
mathx10}{}
\DeclareSymbolFont{mathx}{U}{mathx}{m}{n}
\DeclareFontSubstitution{U}{mathx}{m}{n}
\DeclareMathAccent{\widecheck}{0}{mathx}{"71}
\allowdisplaybreaks[3]
\numberwithin{equation}{section}
\newtheorem{theorem}{Theorem}[section]
\newtheorem{lemma}[theorem]{Lemma}
\newtheorem{proposition}[theorem]{Proposition}
\newtheorem{corollary}[theorem]{Corollary}
\newtheorem{definition}[theorem]{Definition}
\newtheorem{remark}[theorem]{Remark}
\newtheorem{conjecture}[theorem]{Conjecture}
\newtheorem*{thmM0}{Theorem M0}
\newtheorem*{thmM1}{Theorem M1}
\newtheorem*{thmM2}{Theorem M2}
\newtheorem*{thmM3}{Theorem M3}
\newtheorem*{thmM4}{Theorem M4}

\DeclareMathOperator{\sRic}{Ric}
\def\upa{\wideparen{u}}
\def\amr{\mathring{\a}}
\def\aamr{\mathring{\aa}}
\def\amrc{\widecheck{\amr}}
\def\aamrc{\widecheck{\aamr}}
\def\as{\slashed{\a}}
\def\aas{\underline{\slashed{\a}}}
\def\asc{\widecheck{\as}}
\def\aasc{\widecheck{\aas}}
\def\dkbaco{\widecheck{\a^{(1)}}}
\def\dkbbco{\widecheck{\b^{(1)}}}
\def\dkbrhoco{\widecheck{\rho^{(1)}}}
\def\dkbsico{\widecheck{\si^{(1)}}}
\def\dkbaaco{\widecheck{\aa^{(1)}}}
\def\dkbbbco{\widecheck{\bb^{(1)}}}
\def\dkbaac{\widecheck{\dkb\aa}}
\def\dkbbbc{\widecheck{\dkb\bb}}
\def\dkbrhoc{\widecheck{\dkb\rho}}
\def\dkbsic{\widecheck{\dkb\si}}
\def\dkbbc{\widecheck{\dkb\b}}
\def\dkbac{\widecheck{\dkb\a}}
\def\psic{\widecheck{\psi}}
\def\Gaw{\Ga_w}
\def\okk{\mathring{\OO}}

\def\Omomc{\widecheck{\Om\om}}
\def\Omombc{\widecheck{\Om\omb}}
\def\mubc{\widecheck{\mub}}

\def\dko{\dkb^{\leq 1}}

\def\be#1{\begin{equation}\label{#1}}
\def\eeq{\end{equation}}

\def\Jc{\widecheck{J}}
\def\Jbc{\widecheck{\Jb}}
\def\Lc{\widecheck{L}}

\def\sfr{\mathfrak{s}}
\def\sk{\mathfrak{s}}

\def\OOO{(\OO_{(0)})}

\def\Jb{\underline{J}}

\def\mum{[\mu]}
\def\mubm{[\mub]}
\def\mumc{\widecheck{\mum}}
\def\mubmc{\widecheck{\mubm}}
\def\Gag{\Ga_g}
\def\Gab{\Ga_b}
\def\Gaa{\Ga_a}
\def\kk{\mathcal{K}}
\def\a{{\alpha}}

\def\b{{\beta}}
\def\Ric{\mathbf{Ric}}
\def\be{{\beta}}
\def\ga{\gamma}
\def\Ga{\Gamma}
\def\de{\delta}
\def\De{\Delta}
\def\ep{\epsilon}

\def\la{\lambda}

\def\si{\sigma}
\def\Si{\Sigma}
\def\sm{\setminus}
\def\om{\omega}
\def\Om{\Omega}

\def\th{\theta}

\def\ze{\zeta}

\def\nab{\nabla}

\def\pr{{\partial}}
\def\les{\lesssim}

\def\c{\cdot}
\def\bg{\g}

\def\Fb{{\underline{F}}}
\def\AA{{\mathcal A}}

\def\MM{{\mathcal M}}

\def\UU{{\mathcal U}}

\def\EE{{\mathcal E}}

\def\OO{{\mathcal O}}

\def\KK{{\mathcal K}}

\def\RR{{\mathcal R}}

\def\AA{{\mathcal A}}

\def\D{{\bf D}}

\def\K{{\bf K}}
\def\M{{\bf M}}

\def\O{{\bf O}}

\def\R{{\bf R}}

\def\K{{\bf K}}

\def\g{{\bf g}}

\def\f12{{\frac 1 2}}
\def\fl{{\mathfrak{L}}}

\def\lot{\mbox{l.o.t.}}

\def\Lb{{\underline{L}}}

\def\mub{{\underline{\mu}}}

\def\trch{\tr\chi}
\def\chih{{\widehat \chi}}
\def\chib{{\underline \chi}}
\def\chibh{{\underline{\chih}}}

\def\etab{{\underline \eta}}
\def\omb{{\underline{\om}}}
\def\bb{{\underline{\b}}}
\def\aa{\protect\underline{\a}}
\def\xib{{\underline \xi}}

\def\Xib{\underline{\Xi}}

\def\ub{{\underline{u}}}

\def\cuv{{C_u^V}}
\def\Cb{{\underline{C}}}
\def\ucuv{{\Cb_{\ub}^V}}
\def\bGa{\Ga}

\DeclareMathOperator{\tr}{tr}
\DeclareMathOperator{\grad}{grad}
\DeclareMathOperator{\sdiv}{div}
\DeclareMathOperator{\curl}{curl}
\def\bfO{\mathbf{O}}

\def\trchb{\tr\chib}
\def\Div{\mathbf{Div}}

\def\hot{\widehat{\otimes}}

\def\fb{\protect\underline{f}}
\def\err{\mbox{Err}}
\def\ov{\overline}
\DeclareMathOperator{\osc}{Osc}
\def\RRb{{\underline{\RR}}}

\newcommand{\hch}{{\widehat{\chi}}}
\newcommand{\hchb}{\widehat{\underline{\chi}}}

\def\f12{\frac 1 2}

\def\Rk{\mathfrak{R}}

\def\sk{\mathfrak{s}}
\def\dkb{ \, \mathfrak{d}     \mkern-9mu /}
\def\dk{\mathfrak{d}}

\def\OOb{{\underline{\OO}}}
\DeclareFontFamily{U}{mathx}{\hyphenchar\font45}
\DeclareFontShape{U}{mathx}{m}{n}{
      <5> <6> <7> <8> <9> <10>
      <10.95> <12> <14.4> <17.28> <20.74> <24.88>
      mathx10
      }{}
\DeclareSymbolFont{mathx}{U}{mathx}{m}{n}
\DeclareFontSubstitution{U}{mathx}{m}{n}
\DeclareMathAccent{\widecheck}{0}{mathx}{"71}
\def\Vpa{\wideparen{V}}
\def\gac{\widecheck{\ga}}

\def\ombc{\widecheck{\omb}}
\def\Gac{{\widecheck{\Ga}}}

\def\Rc{\widecheck R}

\def\ac{\widecheck{\a}}
\def\rhoc{\widecheck{\rho}}

\def\bc{\widecheck{\beta}}
\def\bu{\underline{b}}
\def\bcc{{\widecheck{\bu}}}
\def\bbc{\widecheck{\bb}}
\def\aac{\widecheck{\aa}}

\def\Omc{\widecheck{\Omega}}

\def\omc{\widecheck{\omega}}
\def\ombc{\underline{\widecheck{\omega}}}

\def\trchc{\widecheck{\tr\chi}}
\def\trchbc{\widecheck{\tr\chib}}

\def\etac{\widecheck{\eta}}
\def\etabc{\widecheck{\etab}}
\def\muc{\widecheck{\mu}}

\def\zec{\widecheck{\ze}}
\def\inc{\widecheck{\in}}
\def\sic{\widecheck{\si}}
\def\hchc{\widecheck{\hch}}
\def\hchbc{\widecheck{\hchb}}
\def\dd{\mathfrak{d}}
\def\Ub{\underline{U}}

\def\vkab{\underline{\varkappa}}
\begin{document}
\title{Kerr stability in external regions}
\author{Dawei Shen\footnote{Email adress: dawei.shen@polytechnique.edu \par \indent\hspace{0.26cm} Laboratoire Jacques-Louis Lions, Sorbonne Universit\'e, 75252 Paris, France}}
\maketitle
\begin{abstract}
In 2003, Klainerman and Nicol\`o \cite{Kl-Ni} proved the stability of Minkowski in the case of the exterior of an outgoing null cone. Relying on the method used in \cite{Kl-Ni}, Caciotta and Nicol\`o \cite{Caciotta} proved the stability of Kerr spacetime in \emph{external regions}, i.e. outside an outgoing null cone far away from the Kerr \emph{event horizon}. In this paper, we give a new proof of \cite{Caciotta}. Compared to \cite{Caciotta}, we reduce the number of derivatives needed in the proof, simplify the treatment of the last slice, and provide a unified treatment of the decay of initial data which contains in particular the initial data considered by Klainerman and Szeftel in \cite{KS:main}. Also, concerning the treatment of curvature estimates, similar to \cite{ShenMink}, we replace the vectorfield method used in \cite{Kl-Ni,Caciotta} by $r^p$--\emph{weighted estimates} introduced by Dafermos and Rodnianski in \cite{Da}.

{\centering\subsubsection*{\small Keywords}}
\small\noindent Double null foliation, Geodesic foliation, Kerr stability, Linearization procedure, Peeling properties, $r^p$--weighted estimates
\end{abstract}
\tableofcontents
\section{Introduction}
\subsection{Einstein vacuum equations and the Cauchy problem}
A Lorentzian $4$--manifold $(\MM,\g)$ is called a vacuum spacetime if it solves the Einstein vacuum equations:
\begin{equation}\label{EVE}
    \Ric(\g)=0\quad \mbox{ in }\MM,
\end{equation}
where $\Ric$ denotes the Ricci tensor of the Lorentzian metric $\g$. The Einstein vacuum
equations are invariant under diffeomorphisms, and therefore one considers equivalence classes of solutions. Expressed in general coordinates, \eqref{EVE} is a nonlinear geometric coupled system of second order partial differential equations for $\g$. In suitable coordinates, for example so-called wave coordinates, it can be shown that \eqref{EVE} is hyperbolic and hence admits an initial value formulation. \\ \\
The corresponding initial data for the Einstein vacuum equations is given by specifying a
triplet $(\Si,g,k)$ where $(\Si,g)$ is a Riemannian $3$--dimension manifold and $k$ is a symmetric $2$--tensor on $\Si$ satisfying the constraint equations:
\begin{align}
    \begin{split}\label{constraintk}
        R&=|k|^2-(\tr k)^2,\\
        D^j k_{ij}&=D_i(\tr k),
    \end{split}
\end{align}
where $R$ denotes the scalar curvature of $g$, $D$ denotes the Levi-Civita connection of $g$ and
\begin{align*}
    |k|^2:=g^{ad}g^{bc}k_{ab}k_{cd},\qquad\quad \tr k:=g^{ij}k_{ij}.
\end{align*}
In the future development $(\MM,\g)$ of such initial data $(\Si,g,k)$, $\Si\subset \MM$ is a spacelike hypersurface with induced metric $g$ and second fundamental form $k$.
\subsubsection{Local existence theorem}
The seminal well-posedness results for the Cauchy problem obtained in \cite{choquet,choquetgeroch} ensure that for any smooth Cauchy data, there exists a unique smooth maximal globally hyperbolic development $(\MM,\g)$ solution of Einstein equations \eqref{EVE} such that $\Si\subset \MM$ and $g$, $k$ are respectively the first and second fundamental forms of $\Si$ in $\MM$. \\ \\
The simplest example of a vacuum spacetime is Minkowski spacetime:
\begin{equation*}
    \MM=\mathbb{R}^4,\qquad \g=-dt^2+(dx^1)^2 +(dx^2)^2+(dx^3)^2,
\end{equation*}
for which Cauchy data are given by
\begin{equation*}
    \Si=\mathbb{R}^3,\qquad g=(dx^1)^2+(dx^2)^2+(dx^3)^2,\qquad k=0.
\end{equation*}
\subsubsection{Special solutions}
Minkowski spacetime is the simplest solution of \eqref{EVE}. We recall two other relevant solutions of \eqref{EVE}:
\begin{itemize}
    \item Schwarzschild spacetime\\
    \eqref{EVE} admits a remarkable family of explicit, stationary, solutions given by the one parameter family of Schwarzschild solutions of mass $M>0$
    \begin{equation}
    \g_M=-\left(1-\frac{2M}{r}\right)dt^2+\left(1-\frac{2M}{r}\right)^{-1}dr^2+r^2d\si_{\mathbb{S}^2}.
    \end{equation}
Though the metric seems singular at $r = 2M$, it turns out that one can glue together two regions $r>2M$ and $r<2M$ of the Schwarzschild metric to obtain a metric which is smooth along $\mathcal{H}=\{r=2M\}$, called the Schwarzschild \emph{event horizon}. The portion of $r<2M$ to the future of the hypersurface $t=0$ is a \emph{black hole} whose future boundary $r=0$ is singular. The similar region to the past of $t=0$ is called a \emph{white hole}. The region $r>2M$, called the \emph{domain of outer communication} (DOC), is free of singularities. The Penrose diagram of Schwarzschild space is recalled in Figure \ref{SchwarzchildPenrose}.
\begin{figure}
    \centering
    \includegraphics[width=0.7\textwidth]{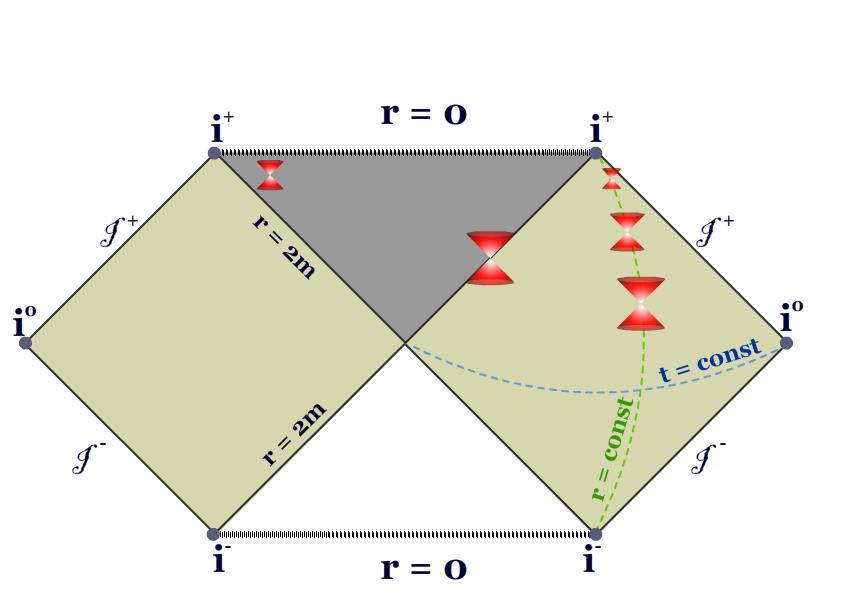}
    \caption{Complete Penrose diagram of Schwarzschild spacetime}
    \label{SchwarzchildPenrose}
\end{figure}
\item Kerr spacetime\\
The Schwarzschild family is included in a larger two parameter family of solutions $K(a,M)$ discovered by Kerr in \cite{Kerr}. A given Kerr spacetime, with $0\leq|a|\leq M$ has a well defined domain of outer communication $r>r_+:=M+(M^2-a^2)^\f12$. In Boyer-Lindquist coordinates, well adapted to $r>r_+$, the Kerr metric has the form:
\begin{align}\label{BL}
 \g_{a,M}=-\frac{(\De-a^2\sin^2\th)}{q^2}dt^2-\frac{4aMr}{q^2}\sin^2\th dt d\varphi+\frac{q^2}{\De}dr^2+q^2d\th^2+\frac{\Si^2}{q^2}\sin^2\th d\varphi^2,
\end{align}
with
\begin{align*}
    q^2=r^2+a^2\cos 2\th,\quad\De= r^2+a^2-2Mr,\quad\Si^2=(r^2+a^2)^2-a^2(\sin\th)^2\De.
\end{align*}
See Figure \ref{KerrPenrose} for the Penrose diagram of Kerr spacetime.
\begin{figure}
    \centering
    \includegraphics[width=1\textwidth]{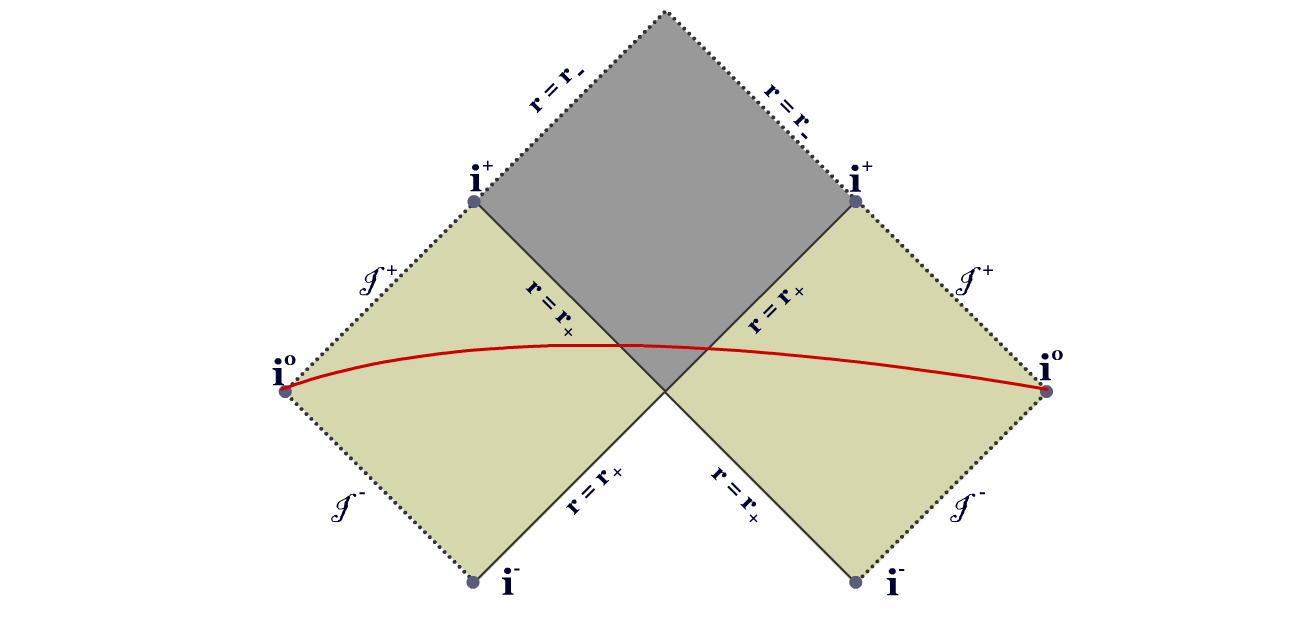}
    \caption{Penrose diagram of Kerr for $0<|a|<M$. The surface $r=r_+$, the larger root of $\De=0$, is the event horizon of the black hole, $r>r_+$ the domain of outer communication, $\mathscr{I}^+$ is the future null infinity, corresponding to $r =+\infty$.}
    \label{KerrPenrose}
\end{figure}
\end{itemize}
\subsection{Stability of Kerr conjecture}
\begin{conjecture}[Stability of Kerr conjecture]\label{kerrconj}
Vacuum initial data sets, sufficiently close to Kerr initial data, have a maximal development with complete future null infinity\footnote{This means, roughly, that observers which are  far away from the black hole  may live forever.} and with domain of outer communication which approaches (globally) a nearby Kerr solution.
\end{conjecture}
The stability of Kerr has been proved in the case of small angular momentum, $|a|\ll M$, in the sequence of works by Klainerman-Szeftel \cite{KS:Kerr1}, \cite{KS:Kerr2}, \cite{KS:main}, Giorgi-Klainerman-Szeftel \cite{GKS} and the author \cite{Shen}, but the conjecture is open in general. For a complete introduction to this problem, we refer the reader to \cite{KS:Survey} and to the introduction in \cite{KS:main} and \cite{GKS}.\\ \\
In order to attack the conjecture, it is convenient to divide the domain of outer communication to the future of an initial spacelike hypersurface $\Si_0$ into three causal regions as in Figure \ref{figmain}:
\begin{figure}
  \centering
  \includegraphics[width=0.9\textwidth]{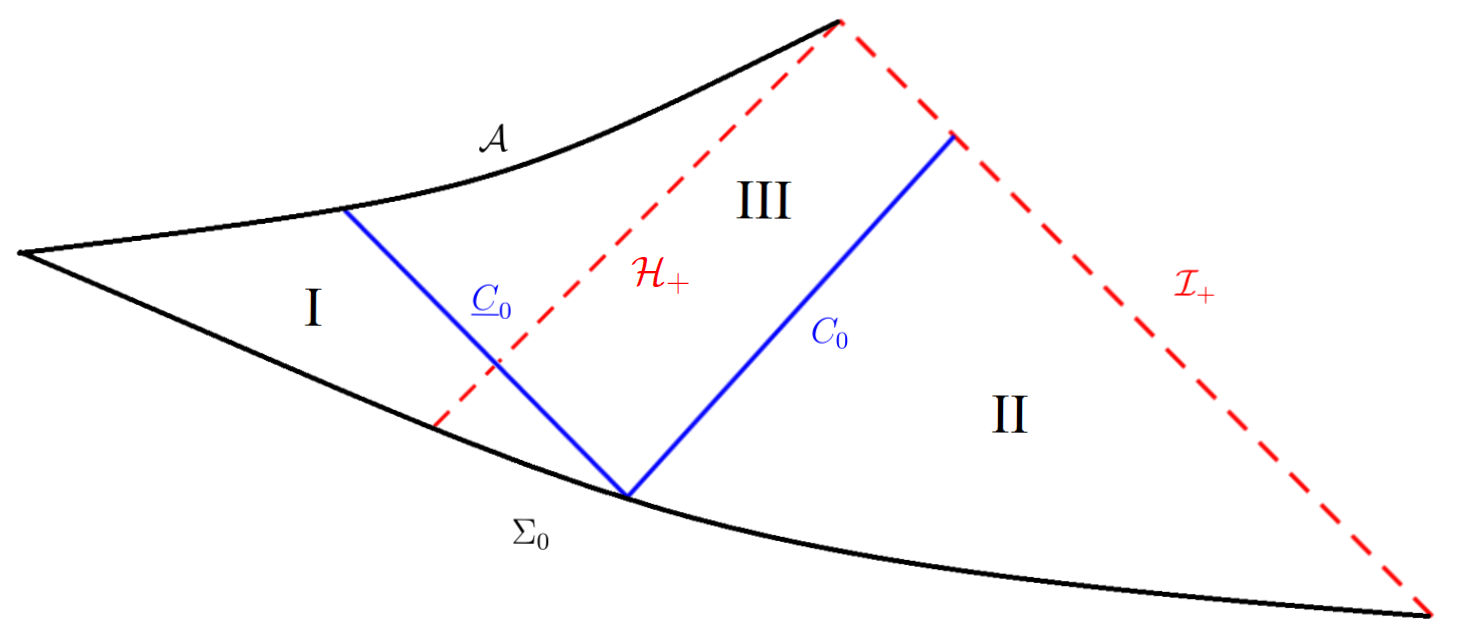}
  \caption{The Penrose diagram of a perturbation of Kerr}\label{figmain}
\end{figure}
\begin{itemize}
    \item Region I lies to the future of $\Si_0$ and to the past of an ingoing null cone $\Cb_0$. The stability of Kerr in spacetime region I is a direct consequence of the local existence results in \cite{choquet,choquetgeroch}.
    \item Region II lies to the future of $\Si_0$ and to the past of an outgoing null cone $C_0$ sufficiently far away from the event horizon. See Section \ref{kerrexternal} for a discussion of Kerr stability in this region. The goal of this paper is to provide a new proof of the stability of Kerr in Region II.
    \item Region III lies to the future of the ingoing null cone $\Cb_0$ and of the outgoing null cone $C_0$. The stability of Kerr in region III for small angular momentum $|a|\ll M$ was proved recently in the sequence of works \cite{KS:Kerr1}, \cite{KS:Kerr2}, \cite{KS:main}, \cite{GKS} and \cite{Shen}.
\end{itemize}
As mentioned above, the goal of this paper is to provide a new proof of the stability of Kerr in Region II. In the next section, we review the corresponding problem for Minkowksi.
\subsection{Minkowski stability in \texorpdfstring{\cite{Kl-Ni}}{}}\label{resultkn}
In 1993, Christodoulou and Klainerman \cite{ch-kl} proved the stability of Minkowski for the Einstein-vacuum equations, a milestone in the domain of mathematical general relativity. In 2003, Klainerman and Nicol\`o \cite{Kl-Ni} gave a second proof of this result in the exterior of an outgoing cone. Moreover, Klainerman and Nicol\`o \cite{knpeeling} showed that under stronger asymptotic decay and regularity properties than those used in \cite{ch-kl,Kl-Ni}, asymptotically flat initial data sets lead to solutions of the Einstein vacuum equations which have strong peeling properties. For more references on the stability of Minkowski for Einstein vacuum equations or Einstein equations coupled to matter, see for example the introduction of \cite{ShenMink}.\\ \\
In this section, we present the main result of \cite{Kl-Ni}, which is the motivation for the analog results in Kerr recalled in Section \ref{kerrexternal} and for our main result stated in Section \ref{simpleversion}. First, we recall the definition of a \emph{maximal hypersurface}, which plays an important role in the statements of the main theorems in \cite{Kl-Ni}.
\begin{definition}\label{def6.1}
An initial data $(\Si,g,k)$ is posed on a maximal hypersurface if it satisfies
\begin{equation}
    \tr k=0.
\end{equation}
In this case, we say that $(\Si,g,k)$ is a maximal initial data set, and the constraint equations \eqref{constraintk} reduce to
\begin{equation}
    R=|k|^2,\qquad \sdiv k=0,\qquad \tr k=0.
\end{equation}
\end{definition}We introduce the notion of \emph{$s$--asymptotically flat initial data}.
\begin{definition}\label{def6.3}
Let $s$ be a real number with $s>3$. We say that a data set $(\Sigma_0,g,k)$ is $s$--asymptotically flat if there exists a coordinates system $(x^1,x^2,x^3)$ defined outside a sufficiently large compact set such that\footnote{The notation $f=o_l(r^{-m})$ means $\pr^\a f=o(r^{-m-|\a|})$, $|\a|\leq l$.}
\begin{align}
    \begin{split}\label{old1.3}
        {g}&=\left(1-\frac{2M}{r}\right)^{-1} dr^2+ r^2 d\si_{\mathbb{S}^2}+o_4(r^{-\frac{s-1}{2}}),\\
        {k}&=o_3(r^{-\frac{s+1}{2}}).
    \end{split}
\end{align}
\end{definition}
We also introduce the following functional associated to any asymptotically flat initial data set:
\begin{equation}
    J_0(\Sigma_0,g,k):=\sup_{\Sigma_0}\Big((d_0^2+1)^3 |\sRic|^2 \Big)+\int_{\Sigma_0}\sum_{l=0}^3(d_0^2+1)^{l+1}|D^l k|^2 +\int_{\Sigma_0}\sum_{l=0}^1 (d_0^2+1)^{l+3}|D^l B|^2,\label{6.4}
\end{equation}
where $d_0$ is the geodesic distance from a fixed point $O\in\Sigma_0$, and $B_{ij}:=(\curl\widehat{\overline{R}})_{ij}$ is the \emph{Bach tensor}, $\widehat{\overline{R}}$ is the traceless part of $\sRic$. Now, we can state the main theorem of \cite{Kl-Ni}.
\begin{theorem}[Klainerman-Nicol\`o \cite{Kl-Ni}]\label{knmain}
Consider an initial data set $(\Sigma_0,g,k)$, $4$--asymptotically flat in the sense of Definition \ref{def6.3} and maximal, and assume that $J_0(\Sigma_0,g,k)$ is bounded. Then, given a sufficiently large compact set $K\subset\Sigma_0$ such that $\Sigma_0 \setminus K$ is diffeomorphic to $\mathbb{R}^3\setminus \overline{B}_1$, and assuming that $J_0(\Si_0,g,k)$ is small enough, there exists a unique development $(\M,\bg)$ with the following properties:
\begin{enumerate}
    \item $(\M,\bg)$ can be foliated by a double null foliation $\{C_u\}$ and $\{\Cb_\ub\}$ whose outgoing leaves $C_u$ are complete.
    \item We have detailed control of all the quantities associated with the double null foliation of the spacetime, see Theorem 3.7.1 of \cite{Kl-Ni}.
\end{enumerate}
\end{theorem}
\begin{remark}\label{peelingcqg}
\cite{knpeeling} and \cite{ShenMink} extend Theorem \ref{knmain} to the case of $s$--asymptotically flat initial data with $s>7$ and $s>3$ respectively. See also \cite{Hintz} for a proof in the case $3<s<5$ using wave coordinates, constraint damping, spacetime compactification, $b$--Sobolev spaces and Nash-Moser iteration.
\end{remark}
\subsection{Previous works of Kerr stability in external regions}\label{kerrexternal}
In 2009, Nicol\`o \cite{Nicolopeeling} proved that the asymptotic behaviour of the Riemann components is in agreement with the \emph{Peeling theorem} if the corrections to the Kerr initial data decay sufficiently fast. In 2010, Caciotta and Nicol\`o \cite{Caciotta} proved the nonlinear stability of the Kerr spacetime for any $|a|\leq M$ and appropriate class of initial data close to Kerr, in an external region where $r\geq R_0$ and $\frac{M}{R_0}\ll 1$. Their main results can be stated as follows.
\begin{theorem}[Caciotta-Nicol\`o \cite{Caciotta}]\label{CN10}
Assume that the initial data are given on $\Si_0$ outside of a ball centered on the origin of radius $R_0$, denoted $K:=B_{R_0}$. They are perturbations of the initial data of a Kerr spacetime with mass $M$ and angular momentum $a$ satisfying
\begin{align*}
    \frac{M}{R_0}\ll 1,\qquad |a|\leq M.
\end{align*}
Let $s>7$, and assume that the initial data satisfies for $q\geq 4$:
\begin{align}
\begin{split}\label{sperturbation}
    g_{ij}&=(g_{Kerr})_{ij}+o_{q+1}(r^{-\frac{s-1}{2}}),\\
    k_{ij}&=(k_{Kerr})_{ij}+o_q(r^{-\frac{s+1}{2}}).
\end{split}
\end{align}
Then, the initial data set $(\Si_0,g,k)$ has a unique development $(\MM,\g)$ defined outside $I^+(K)$, the domain of influence of $K$. Moreover, $\MM$ can be foliated by a canonical double null foliation $(u,\ub)$ whose outgoing leaves $C_u$ are complete and suitable norms of Ricci coefficients and curvature components are sufficiently small.
\end{theorem}
The goal of this paper is to provide a new proof of Theorem \ref{CN10}.
\subsection{Rough version of the main theorem}\label{simpleversion}
In this section, we state a simple version of our main theorem. For the precise statement, see Theorem \ref{maintheorem}.
\begin{theorem}[Main Theorem, first version]\label{firstmain}
Let $K=B_{R_0}\subset\Si_0$ a ball centered at the origin of radius $R_0$ and let $\g_{a,M}$ a Kerr metric with mass $M>0$ and angular momentum $a$ satisfying
\begin{align*}
    \frac{M}{R_0}\ll 1,\qquad |a|\leq M.
\end{align*}
Let $s>3$ and let an initial data set $(\Si_0,g,k)$ satisfying \eqref{sperturbation} with $q=2$. Let $\KK_{(0)}$\footnote{$\KK_{(0)}$ is defined in Section \ref{initiallayer}.} an initial layer region near $\Si_0\sm K$ and assume that the initial data in $\KK_{(0)}$ are perturbations from the initial data of $\g_{a,M}$ of size $\ep_0\ll \frac{M}{R_0}$.\footnote{See \eqref{l2condition} for the explicit initial data assumptions and see Section \ref{'norms} for the definitions of the norms on the initial data.} \\ \\
Then, the initial data set $(\Si_0,g,k)$ has a unique development $(\MM,\g)$ in its future domain of dependence with the following properties:
\begin{itemize}
    \item $\MM$ can be foliated by a double null foliation $(u,\ub)$ whose outgoing leaves $C_u$ are complete for all $u\leq u_0$;
    \item we have detailed control of all the quantities associated with the double null foliation of the spacetime, see Theorem \ref{maintheorem}.
\end{itemize}
\end{theorem}
\begin{remark}
    In the particular case of $s$--asymptotically flat initial data with $s>7$, we reobtain the results of \cite{Caciotta}.
\end{remark}
\begin{remark}
Let $C_{u_0}$ be the outgoing null cone located near $\pr K$ and choose $s=4+2\de_B$ with $\de_B>0$. Then, the estimates of Theorem \ref{firstmain} on $C_{u_0}$ are consistent with the initial data considered in \cite{KS} and \cite{KS:main}.
\end{remark}
\begin{remark}
    Theorem \ref{firstmain} is the analog for the Kerr spacetime of the result for Minkowski in \cite{ShenMink}.
\end{remark}
The proof of Theorem \ref{firstmain} has the same structure as in \cite{Kl-Ni,Caciotta} and is closely related to the one in \cite{ShenMink}. Below, we compare the proof of this paper and that of \cite{Caciotta}:
\begin{enumerate}
\item To estimate the norms of curvature components, \cite{Caciotta} uses the vectorfield method introduced in \cite{ch-kl,Kl-Ni}. Here, we estimate the curvature norms by the $r^p$--weighted estimate method introduced in \cite{Da}. This allows for a simple treatment of the curvature estimates.
\item \cite{Caciotta} uses fourth order derivatives of Ricci coefficients and third order derivatives of curvature components. Thanks to the use of $r^p$--weighted estimates, we only need first order derivatives of both Ricci coefficients and curvature components.
\item In order to control one more derivative for Ricci coefficients compared to curvature, \cite{Caciotta} relies on the canonical foliation on the last slice $\Cb_*$. Since we control the same number of derivatives for curvature and Ricci coefficients, instead of introducing the canonical foliation, we use a geodesic type foliation to simplify the estimates on the last slice $\Cb_*$. 
\end{enumerate}
\subsection{Structure of the paper}\label{structurepaper}
\begin{itemize}
    \item In Section \ref{preliminaries}, we recall the fundamental notions and basic equations used in this paper. We also introduce the linearization procedure and the linearized equations.
    \item In Section \ref{mainsection}, we state the precise version of the main theorem. We then state intermediate results, and rely on them to prove the main theorem. The rest of the paper focuses on the proof of these intermediary results. In Sections \ref{curvatureestimates}--\ref{initiallast}, we consider the case $s\in [4,6]$ and deal with the other cases in Appendices \ref{secc} and \ref{secd}.
    \item In Section \ref{curvatureestimates}, we apply $r^p$-weighted estimates to the linearized Bianchi equations to control curvature.
    \item In Section \ref{Ricciestimates}, we estimate the Ricci coefficients using the linearized null structure equations.
    \item In Section \ref{initiallast}, we estimate the curvature in the initial layer region $\kk_{(0)}$. Then, we estimate Ricci coefficients on the last slice $\Cb_*$. We also show how to extend the spacetime in the context of a bootstrap argument.
    \item In Appendix \ref{appa}, we provide the derivation of the linearized equations stated in Section \ref{seclinearequation}.
    \item In Appendix \ref{secc}, we prove Theorem \ref{firstmain} in the case $s\in(3,4)$.
    \item In Appendix \ref{secd}, we prove Theorem \ref{firstmain} in the case $s>6$ by applying $r^p$--weighted estimates to the linearized Teukolsky equation.
\end{itemize}
\section{Preliminaries}\label{preliminaries}
\subsection{Geometric set-up}\label{ssec7.1}
\subsubsection{Double null foliation}\label{doublenullsection}
We first introduce the geometric setup. The material in this section is classical, see for example Chapter 3 in \cite{Kl-Ni}.\\ \\ 
We denote $(\MM,\bg)$ a spacetime $\MM$ with the Lorentzian metric $\g$ and $\D$ its Levi-Civita connection. Let $u$ and $\ub$ be two optical functions on $\MM$, that is\footnote{{Here, $\grad$ denotes the spacetime gradient.}}
\begin{equation*}
    \g(\grad u,\grad u)=\g(\grad\ub,\grad\ub)=0,
\end{equation*}
The spacetime $\MM$ is foliated by the level sets of $u$ and $\ub$ respectively, and the functions $u,\ub$ are required to increase towards the future. We use $C_u$ to denote the outgoing null cones which are the level sets of $u$ and use $\Cb_\ub$ to denote the ingoing null cones which are the level sets of $\ub$. We denote
\begin{equation}
    S(u,\ub):=C_u\cap\Cb_\ub,
\end{equation}
which is a spacelike $2$--sphere. We define the area radius of $S(u,\ub)$:
\begin{equation}
    r(u,\ub):=\sqrt{\frac{|S(u,\ub)|}{4\pi}}.
\end{equation}
We introduce the vector fields $L$ and $\Lb$ by
\begin{equation*}
    L:=-\grad u,\qquad \Lb:= -\grad\ub.
\end{equation*}
We define {a} positive scalar function $\Om$ by:
\begin{equation*}
    \bg(L,\Lb)=-\frac{1}{2\Om^2},
\end{equation*}
called the lapse function. We then define the normalized null pair $(e_3,e_4)$ by
\begin{equation*}
   e_3=2\Om\Lb,\quad e_4=2\Om L.
\end{equation*}
{We also define the following vector fields:}
\begin{equation*}
    \underline{N}=\Om e_3, \qquad N=\Om e_4.
\end{equation*}
The flows generated by $\underline{N}$ and $N$ preserve the double null foliation. On a given two sphere $S(u,\ub)$, we choose a local frame $(e_1,e_2)$, and we call $(e_1,e_2,e_3,e_4)$ a null frame. As a convention, throughout the paper, we use capital Latin letters $A,B,C,...$ to denote an index from 1 to 2 and Greek letters $\alpha,\beta,\gamma,...$ to denote an index from 1 to 4, e.g. $e_A$ denotes either $e_1$ or $e_2$.\\ \\
The spacetime metric $\g$ induces a Riemannian metric $\ga$ on $S(u,\ub)$. Using $(u,\ub)$, we can introduce a local coordinates system $(u,\ub,x^A)$ on $\MM$ with $e_4(x^A)=0$. The metric $\g$ takes the form:
\begin{equation}\label{metricg}
\g=-2\Om^2(d\ub\otimes du+du\otimes d\ub)+\ga_{AB}(dx^A-\bu^Adu)\otimes (dx^B-\bu^Bdu),
\end{equation}
and we have $\underline{N}=\pr_u+\bu$ and $N=\pr_{\ub}$ where
\begin{equation*}
    \bu:=\bu^A \pr_{x^A}.
\end{equation*}
We recall the null decomposition of the Ricci coefficients and curvature components of the null frame $(e_1,e_2,e_3,e_4)$. We have:
\begin{align}
\begin{split}\label{defga}
\chib_{AB}&=\g(\D_A e_3, e_B),\qquad\quad \chi_{AB}=\g(\D_A e_4, e_B),\\
\xib_A&=\frac 1 2 \g(\D_3 e_3,e_A),\qquad\quad\,\,  \xi_A=\frac 1 2 \g(\D_4 e_4, e_A),\\
\omb&=\frac 1 4 \g(\D_3e_3 ,e_4),\qquad\quad \,\,\,\,\, \,\om=\frac 1 4 \g(\D_4 e_4, e_3), \\
\etab_A&=\frac 1 2 \g(\D_4 e_3, e_A),\qquad\quad\, \, \eta_A=\frac 1 2 \g(\D_3 e_4, e_A),\\
\ze_A&=\frac 1 2 \g(\D_{e_A}e_4, e_3),
\end{split}
\end{align}
and
\begin{align}
\begin{split}\label{defr}
\a_{AB} &=\R(e_A, e_4, e_B, e_4) , \qquad \,\,\,\aa_{AB} =\R(e_A, e_3, e_B, e_3), \\
\b_{A} &=\frac 1 2 \R(e_A, e_4, e_3, e_4), \qquad\,\,\; \bb_{A}=\frac 1 2 \R(e_A, e_3, e_3, e_4),\\
\rho&= \frac 1 4 \R(e_3, e_4, e_3, e_4), \qquad\,\;\;\,\,\,\;\si =\frac{1}{4}{^*\R}( e_3, e_4, e_3, e_4),
\end{split}
\end{align}
where $\R$ denotes the curvature tensor of $\g$ and $^*\R$ denotes the Hodge dual of $\R$. The null second fundamental forms $\chi, \chib$ are further decomposed in their traces $\trch$ and $\trchb$, and traceless parts $\hch$ and $\hchb$:
\begin{align*}
\trch&:=\de^{AB}\chi_{AB},\qquad\quad \,\hch_{AB}:=\chi_{AB}-\frac{1}{2}\de_{AB}\trch,\\
\trchb&:=\de^{AB}\chib_{AB},\qquad\quad \, \hchb_{AB}:=\chib_{AB}-\frac{1}{2}\de_{AB}\trchb.
\end{align*}
We define the horizontal covariant operator $\nab$ as follows:
\begin{equation*}
\nab_X Y:=\D_X Y-\frac{1}{2}\chib(X,Y)e_4-\frac{1}{2}\chi(X,Y)e_3.
\end{equation*}
We also define $\nab_4 X$ and $\nab_3 X$ to be the horizontal projections:
\begin{align*}
\nab_4 X&:=\D_4 X-\frac{1}{2} \g(X,\D_4e_3)e_4-\frac{1}{2} \g(X,\D_4e_4)e_3,\\
\nab_3 X&:=\D_3 X-\frac{1}{2} \g(X,\D_3e_3)e_3-\frac{1}{2} \g(X,\D_3e_4)e_4.
\end{align*}
A tensor field $\psi$ defined on $\MM$ is called tangent to $S$ if it is a priori defined on the spacetime $\MM$ and all the possible contractions of $\psi$ with either $e_3$ or $e_4$ are zero. We use $\nab_3\psi$ and $\nab_4 \psi$ to denote the projection to $S(u,\ub)$ of usual derivatives $\D_3\psi$ and $\D_4\psi$. As a direct consequence of \eqref{defga}, we have the Ricci formulae:
\begin{align}
\begin{split}\label{ricciformulas}
    \D_A e_B&=\nab_A e_B+\frac{1}{2}\chi_{AB} e_3+\frac{1}{2}\chib_{AB}e_4,\\
    \D_A e_3&=\chib_{AB}e_B+\ze_A e_3,\\
    \D_A e_4&=\chi_{AB}e_B-\ze_A e_4,\\
    \D_3 e_A&=\nab_3 e_A+\eta_A e_3+\xib_A e_4,\\
    \D_4 e_A&=\nab_4 e_A+\etab_A e_4+\xi_A e_4,\\
    \D_3 e_3&=-2\omb e_3+2\xib_B e_B,\\ 
    \D_3 e_4&=2\omb e_4+2\eta_B e_B,\\
    \D_4 e_4&=-2\om e_4+2\xi_B e_B,\\
    \D_4 e_3&=2\om e_3+2\etab_B e_B.
\end{split}
\end{align}
In addition to 
\begin{equation}
    \xi=\xib=0,
\end{equation}
the following identities hold for a double null foliation:
\begin{align}
\begin{split}
    \nabla\log\Omega&=\frac{1}{2}(\eta+\etab),\qquad\;\;\;\,\omega=-\frac{1}{2}\D_4(\log\Omega), \qquad\;\;\;\,\omb=-\frac{1}{2}\D_3(\log\Omega),\\ \label{6.6}
    \eta&=\ze+\nabla\log\Omega,\qquad \etab=-\ze+\nabla\log\Omega,
\end{split}
\end{align}
see for example (6) in \cite{kr}.
\subsubsection{Geometry of Kerr spacetime}\label{kerrgeometry}
In this section, we recall the basic notions of the geometry of Kerr spacetime used in the sequel. See Section 4 in \cite{pi} and Appendix A in \cite{Daluk} for more details. \\ \\
Let $\MM_{Kerr}$ be a $4$--dimensional manifold diffeomorphic to $\mathbb{R}^2\times \mathbb{S}^2$ which describe the Kerr black hole. The Kerr metric $\g_{a,M}$ with mass $M$ and angular momentum $a$ is a smooth metric on $\MM_{Kerr}$, which takes the following form in the Boyer-Lindquist local coordinates $(t,r,\th,\phi)$:
\begin{align}
    \begin{split}\label{Kerrmetric}
        \g_{a,M}=&-\left(1-\frac{2Mr}{\Si}\right)dt\otimes dt +\frac{\Si}{\De}dr\otimes dr+\Si\, d\th\otimes d\th ,\\
        &+R^2\sin^2\th\, d\phi\otimes d\phi-\frac{2Mar\sin^2\th}{\Si}(d\phi\otimes dt+dt\otimes d\phi),
    \end{split}
\end{align}
where
\begin{equation*}
    \Si:=r^2+a^2\cos^2\th,\qquad R^2:=r^2+a^2+\frac{2Ma^2r\sin^2\th}{\Si}, \qquad \De:=r^2+a^2-2Mr.
\end{equation*}
We denote
\begin{equation*}
    r_+:=M+(M^2-a^2)^\f12,
\end{equation*}
and recall that $r=r_+$ defines the event horizons of Kerr black hole. Note that \eqref{Kerrmetric} is valid in $r>r_+$.
\begin{remark}
Note that $(\th,\phi)$ are spherical coordinates on $S(t,r)$ and have a coordinate singularity at $\th=0,\pi$. We thus also introduce
\begin{equation*}
    x^1:=\sin\th\cos\phi,\qquad x^2:=\sin\th\sin\phi.
\end{equation*}
Then, $(x^1,x^2)$ are regular coordinates when $0\leq \th\leq \frac{\pi}{3}$ and $\frac{2\pi}{3}\leq \th\leq \pi$. Hence, $S(t,r)$ can be covered by three regular coordinates charts. In the sequel, we focus on the spherical coordinates chart $(\th,\phi)$.
\end{remark}
\begin{proposition}
    We can define two new functions $r_*=r(r,\th)$ and $\th_*=\th_*(r,\th)$ such that the Kerr metric takes the following form in $r>r_+$:
\begin{align}
\begin{split}\label{metricpi}
    \g_{a,M}=&\frac{\De}{R^2}\left(-dt\otimes dt + dr_*\otimes dr_* \right)+\frac{\ell^2}{R^2}d\th_* \otimes d\th_*\\
    &+R^2\sin^2\th \left(d\phi-\frac{2Mar}{\Si R^2}dt\right)\otimes \left(d\phi-\frac{2Mar}{\Si R^2}dt\right),
\end{split}
\end{align}
where
\begin{align*}
    \ell&:=-Ma\sqrt{\sin^2\th_*-\sin^2\th}\sqrt{(r^2+a^2)^2-a^2\sin^2\th_*\De},\\
    \De&:=r^2+a^2-2Mr,\\
    R^2&:=r^2+a^2+\frac{2Ma^2r\sin^2\th}{\Si},\\
    \Si&:=r^2+a^2\cos^2\th.
\end{align*}
\end{proposition}
\begin{proof}
    See Appendices A.1--A.3 in \cite{Daluk}.
\end{proof}
We can also write the Kerr metric in double null foliation. We define
\begin{align*}
    \ub_{Kerr}:=t+r_*,\qquad u_{Kerr}:=t-r_*,\qquad \phi_*:=\phi-h(r_*,\th_*),
\end{align*}
where
\begin{align*}
    \frac{d h}{dr_*}(r_*,\th_*)=\frac{2Mar}{\Si R^2}.
\end{align*}
We easily obtain in $r>r_+$\footnote{See (A.42) in \cite{Daluk} and notice the exchange of role between $\ub$ and $u$.}
\begin{align*}
    \g_{a,M}=&-2\Om_{Kerr}^2(d\ub_{Kerr}\otimes du_{Kerr}+du_{Kerr}\otimes d\ub_{Kerr})\\
    &+(\ga_{Kerr})_{AB}(dx_{Kerr}^A-\bu_{Kerr}^A\, du_{Kerr})\otimes (dx_{Kerr}^B-\bu_{Kerr}^B\, du_{Kerr}),
\end{align*}
where $(x_{Kerr}^1,x_{Kerr}^2)=(\th_*,\phi_*)$ or $(x_{Kerr}^1,x_{Kerr}^2)=(\sin\th_*\cos\phi_*,\sin\th_*\sin\phi_*)$.\\ \\ 
We also denote
\begin{align*}
    L_{Kerr}:=-\grad u_{Kerr},\qquad \Lb_{Kerr}:=-\grad \ub_{Kerr},\qquad \g(L_{Kerr},\Lb_{Kerr})=-\frac{1}{2\Om_{Kerr}^2},
\end{align*}
and
\begin{align*}
    (e_4)_{Kerr}:=2\Om_{Kerr}\,L_{Kerr},\qquad (e_3)_{Kerr}:=2\Om_{Kerr}\,\Lb_{Kerr}.
\end{align*}
The leaves of $\MM_{Kerr}$ are defined by
\begin{equation*}
    S_{Kerr}(u_{Kerr},\ub_{Kerr}):=(C_{Kerr})_{u_{Kerr}}\cap (\Cb_{Kerr})_{\ub_{Kerr}},
\end{equation*}
where $(C_{Kerr})_{u_{Kerr}}$ and $(\Cb_{Kerr})_{\ub_{Kerr}}$ respectively denote the outgoing and ingoing null cones of the $(u_{Kerr},\ub_{Kerr})$--foliation in $\MM_{Kerr}$. Let also $r_{Kerr}$ denote the area radius of the $S(u_{Kerr},\ub_{Kerr})$, i.e.
\begin{equation*}
    r_{Kerr}:=\sqrt{\frac{|S(u_{Kerr},\ub_{Kerr})|}{4\pi}}.
\end{equation*}
Note that $r_{Kerr}$ is a function of $r_*(r,\th)$.\\ \\
We define
\begin{equation*}
    \Si_{Kerr}:=\{p\in \MM_{Kerr}/\, u_{Kerr}+\ub_{Kerr}=0\},
\end{equation*}
and
\begin{equation*}
    w_{Kerr}:=\frac{\ub_{Kerr}-u_{Kerr}}{2}.
\end{equation*}
Let $\de_0>0$ be a fixed small constant. The initial layer region of $\MM_{Kerr}$ is defined by
\begin{equation}\label{de0df}
    (\KK_{(0)})_{Kerr}:=\big\{p\in \MM_{Kerr}/\,0\leq u_{Kerr}+\ub_{Kerr}\leq 2\de_0\big\}\setminus I^+(K_{Kerr}),
\end{equation}
where 
\begin{align}\label{R0constant}
    K_{Kerr}:=\{p\in\Si_{Kerr}/\, w_{Kerr}(p)\leq R_0\},
\end{align}
$I^+(K_{Kerr})$ denotes the future domain of influence of $K_{Kerr}$ and $R_0>0$ is a constant.
\subsubsection{Decay of quantities in Kerr spacetime}
In this section, we provide the decay of quantities associated to the double null foliation in $\MM_{Kerr}$ induced by $(u_{Kerr},\ub_{Kerr})$ introduced in Section \ref{kerrgeometry}.
\begin{definition}\label{bigOnotation}
For any $S_{Kerr}$--tangent tensor field $h_{Kerr}$ defined in $\MM_{Kerr}$, we denote
\begin{equation}
    h_{Kerr}=O_q^p,
\end{equation}
if for any integer $l\geq 0$
\begin{align*}
    |\nab^{\leq l}h_{Kerr}|\les\frac{M^p}{r_{Kerr}^{q+l}},
\end{align*}
where
\begin{equation*}
    \nab:=\left\{\nab_{\pr_{\ub_{Kerr}}},\, \nab_{\pr_{u_{Kerr}}},\, r_{Kerr}^{-1}\nab_{\pr_{x_{Kerr}^A}}\right\},
\end{equation*}
where the coordinates $(u_{Kerr},\ub_{Kerr},x_{Kerr}^A)$ and the area radius $r_{Kerr}$ of $S(u_{Kerr},\ub_{Kerr})$ have been introduced in Section \ref{kerrgeometry}.
\end{definition}
\begin{proposition}\label{decayGammaKerr}
The Ricci coefficients of the $(u_{Kerr},\ub_{Kerr})$ double null foliation satisfy in $\MM_{Kerr}$:
\begin{align*}
(\tr\chi)_{Kerr}-\frac{2}{r_{Kerr}}&=O_2^1,\qquad\;\,\;(\tr\chib)_{Kerr}+\frac{2}{r_{Kerr}}=O_2^1,\\
\om_{Kerr}&=O_2^1,\qquad\qquad \qquad\qquad\omb_{Kerr}=O_2^1,\\
\chih_{Kerr}&=O_3^2,\qquad \qquad\quad\;\;\qquad \;\,\chibh_{Kerr}=O_3^2,\\
(\nab\tr\chi)_{Kerr}&=O_4^2,\qquad \quad\qquad\; (\nab \tr\chib)_{Kerr}=O_4^2,\\
\eta_{Kerr}&=O_3^2,\qquad\qquad\quad\;\;\qquad\;\; \etab_{Kerr}=O_3^2,\\
(\nab\om)_{Kerr}&=O_4^2,\qquad\qquad\;\, \qquad(\nab\omb)_{Kerr}=O_4^2,\\
\bu_{Kerr}&=O_2^2,\qquad\quad\;\;\;\; \qquad\Om_{Kerr}-\f12=O_1^1.
\end{align*}
\end{proposition}
\begin{proof}
    See for example Section 2.2.3 in \cite{Caciotta} and Appendix A.4 in \cite{Daluk}.
\end{proof}
\begin{proposition}\label{decayRKerr}
The curvature components of the $(u_{Kerr},\ub_{Kerr})$ double null foliation satisfy in $\MM_{Kerr}$:
\begin{align*}
    \a_{Kerr} &=O_5^3, \qquad\qquad \,\aa_{Kerr} =O_5^3,\\
    \b_{Kerr} &=O_4^2, \qquad\qquad \,\bb_{Kerr} =O_4^2,\\
    \rho_{Kerr}&=O_3^1, \qquad\qquad\; \si_{Kerr} =O_4^2.
\end{align*}
\end{proposition}
\begin{proof}
    See Section 3.3.5 of \cite{Caciotta}.
\end{proof}
\subsubsection{Initial layer region}\label{initiallayer}
Let $\Si_0$ be a spacelike hypersurface, called the \emph{initial hypersurface}. Let $K\subset\Si_0$ be a compact subset such that $\Si_0\setminus K$ is diffeomorphic to $\mathbb{R}^3 \setminus \ov{B_1}$ where $\ov{B_1}$ is the unit closed ball in $\mathbb{R}^3$. We fix a radial foliation on the initial hypersurface $\Sigma_0 \setminus K$ by the level sets of a scalar function $w$, its leaves are denoted by
\begin{equation*}
    S_{(0)}(w_1)=\{p\in\Si_0/\, w(p)=w_1\},
\end{equation*}
where $w_1\geq 0$. We assume that
\begin{equation*}
    \pr K=\{p\in\Si_0/\,w(p)=R_0\},\qquad K=\{p\in\Si_0/\, w(p)\leq R_0\},
\end{equation*}
where $R_0>0$ is the constant appearing in \eqref{R0constant}.\\ \\
We now construct an \emph{initial layer region} $\kk_{(0)}$ near $\Si_0\sm K$ endowed with a double null foliation $(u_{(0)},\ub)$ as in Section \ref{doublenullsection} satisfying
\begin{align*}
    u_{(0)}(p)=-w(p),\qquad\ub(p)=w(p),\qquad\mbox{on}\;\;\Si_0\setminus K.
\end{align*}
We then define
\begin{equation}
    \kk_{(0)}:=\big\{p\big/\, 0\leq u_{(0)}(p)+\ub(p)\leq 2\de_0\big\},
\end{equation}
where $\de_0>0$ is the small constant appearing in \eqref{de0df}. The double null foliation $(u_{(0)},\ub)$ is called the \emph{initial layer foliation}. Its leaves are denoted by
\begin{equation}
    S_{(0)}(u_{(0)},\ub):=(C_{(0)})_{u_{(0)}}\cap \Cb_\ub,
\end{equation}
where $(C_{(0)})_{u_{(0)}}$ is the outgoing null cone emanating from $S_{(0)}(u_{(0)},-u_{(0)})$ and $\Cb_\ub$ is the ingoing null cone emanating from $S_{(0)}(-\ub,\ub)$. By construction, we have
\begin{equation}
    \Si_0\sm K=\big\{p\in \kk_{(0)}/\, u_{(0)}(p)+\ub(p)=0\big\},\qquad S_{(0)}(\ub)=S_{(0)}(-\ub,\ub).
\end{equation}
We extend the domain of the function $w$ to $\KK_{(0)}$ as follows:
\begin{align*}
    w(p):=\frac{\ub(p)-u_{(0)}(p)}{2},\qquad\quad \forall p\in\KK_{(0)}.
\end{align*}
Also, as in Section \ref{doublenullsection}, $\kk_{(0)}$ is covered by coordinates systems $(u_{(0)},\ub,x^A_{(0)})$ with $L_{(0)}(x^A_{(0)})=0$ and the metric $\g$ in $\kk_{(0)}$ has the form:
\begin{equation}
    \g=-2\Om_{(0)}^2(d\ub\otimes du_{(0)}+du_{(0)}\otimes d\ub)+(\ga_{(0)})_{AB}(dx^A_{(0)}-\bu_{(0)}^A du_{(0)})\otimes(dx^B_{(0)}-\bu_{(0)}^B du_{(0)}).
\end{equation}
Let $\Phi_{(0)}$ be the map defined on $\Si_0$ such that
\begin{align*}
    \Phi_{(0)}:\Si_0&\longrightarrow\Si_{Kerr},\\
    p&\longrightarrow \Phi_{(0)}(p)\big/\, u_{Kerr}(\Phi_{(0)}(p))=u(p),\;\ub_{Kerr}(\Phi_{(0)}(p))=\ub(p),\; x^A_{Kerr}(\Phi_{(0)}(p))=x^A_{(0)}(p).
\end{align*}
The map $\Phi_{(0)}$ being defined on $\Si_0$, we now extend it to $\kk_{(0)}$. We denote $L_{(0)}$ and $L_{Kerr}$ respectively the generator of $(C_{(0)})_{u_{(0)}}$ and $C_{Kerr}(u_{Kerr})$. For any $p\in \kk_{(0)}$, we denote $\phi^{L_{(0)}}_t(p)$ the flow of $L_{(0)}$ emanating from $p$ to the past. Then, we denote
\begin{equation*}
    p_{(0)}:=\phi^{L_{(0)}}_t(p)\cap\Si_0,\qquad q_{(0)}:=\Phi_{(0)}(p_{(0)}).
\end{equation*}
Next, we denote $\phi^{L_{Kerr}}_t(q_{(0)})$ the flow of $L_{Kerr}$ emanating from $q_{(0)}$ to the future. We define
\begin{equation*}
    q:=\phi^{L_{Kerr}}_t(q_{(0)})\cap S_{Kerr}\left(u_{Kerr}=u_{(0)}(p),\ub_{Kerr}=\ub(p)\right).
\end{equation*}
Finally, we define
\begin{equation}\label{Phi0}
    \Phi_{(0)}(p)=q.
\end{equation}
By construction, we have for any $p\in\KK_{(0)}$\footnote{Note that the identity $x^A_{Kerr}(\Phi_{(0)}(p))=x^A_{(0)}(p)$ follows from the fact that it holds on $\Si_0$ and from the identities $L_{(0)}(x^A_{(0)})=0$ and $L_{Kerr}(x^A_{Kerr})=0$.}
\begin{align*}
    \Phi_{(0)}(p)\in\kk_{Kerr},\quad u_{Kerr}(\Phi_{(0)}(p))=u(p),\quad \ub_{Kerr}(\Phi_{(0)}(p))=\ub(p),\quad x^A_{Kerr}(\Phi_{(0)}(p))=x^A_{(0)}(p).
\end{align*}
which implies that the map $\Phi_{(0)}$ is equivalent to an identification of $\KK_{(0)}$ and $\KK_{Kerr}$ via the values of the coordinates functions $(u_{(0)},\ub,x^1_{(0)},x^2_{(0)})$ and $(u_{Kerr},\ub_{Kerr},x^1_{Kerr},x^2_{Kerr})$.\footnote{See Section 4.3 in \cite{Daluk} for more explanations concerning this identification.}
\subsubsection{Bootstrap region}\label{bootregion}
For every value $\ub_*>R_0$, 
$$S_*:=S_{(0)}(-\ub_*,\ub_*)$$
denotes a leaf of the $(u_{(0)},\ub)$--foliation on $\Si_0\sm K$. Emanating from this sphere $S_*$, we construct an ingoing cone $\Cb_*$ called \emph{the last slice}. For an optical function $u$ defined on $\Cb_*$ such that $u=-\ub_*$ on $S_*$, the level sets of $u$ induce a foliation on $\Cb_*$, and its leaves are denoted by $S(u,\ub_*)$. We will take a particular choice of the coordinates systems $(u,x^A)$ on $\Cb_*$ in the context of a geodesic type foliation, see Definition \ref{geodesicfoliation}, with $x^A=x^A_{(0)}$ on $S_*$. \\ \\
We define on $\Cb_*$ a map $\Phi$\footnote{Notice that $\Phi=\Phi_{(0)}$ on $S_*$ since $u=u_{(0)}$ and $x^A=x^A_{(0)}$ on $S_*$.} by
\begin{align*}
    \Phi:\Cb_*&\longrightarrow (\Cb_*)_{Kerr},\\
    p&\longrightarrow \Phi(p)\big/\, u_{Kerr}(\Phi(p))=u(p),\; \ub_{Kerr}(\Phi(p))=\ub(p),\; x^A_{Kerr}(\Phi(p))=x^A(p),
\end{align*}
where
\begin{equation*}
    (\Cb_*)_{Kerr}=\big\{p\in \MM_{Kerr}/\,\ub_{Kerr}=\ub_*\big\}.
\end{equation*}\\
For every $u\geq -\ub_*$, we construct an outgoing null cone emanating from $S(u,\ub_*)$ on $\Cb_*$ to the past, denoted $C_u$. Then, $C_u\cap\Si_0$ is a sphere on $\Si_0$, which in general does not coincide with the leaves of the $(u_{(0)},\ub)$--foliation on $\Si_{0}$. We also define $\Cb_\ub$ the ingoing null cone emanating from $S_{(0)}(-\ub,\ub)$. Hence, we have constructed a double null foliation $(u,\ub)$ in $\kk$ and its leaves are given by:
\begin{equation}
    S(u,\ub):=C_u\cap \Cb_\ub.
\end{equation}
For every $(u,\ub)$, we define\footnote{For any set $A$, $J^\pm(A)$ denotes the causal future (and resp. causal past) of $A$.}
\begin{equation}
    V(u,\ub):=J^-(S(u,\ub))\cap J^+(\Si_0).
\end{equation}
Next, we denote
\begin{equation}
    \KK:=V(u_0,\ub_*),\qquad u_0=-R_0,
\end{equation}
called the \emph{bootstrap region}, see Figure \ref{figboot}.
\begin{figure}
  \centering
  \includegraphics[width=1\textwidth]{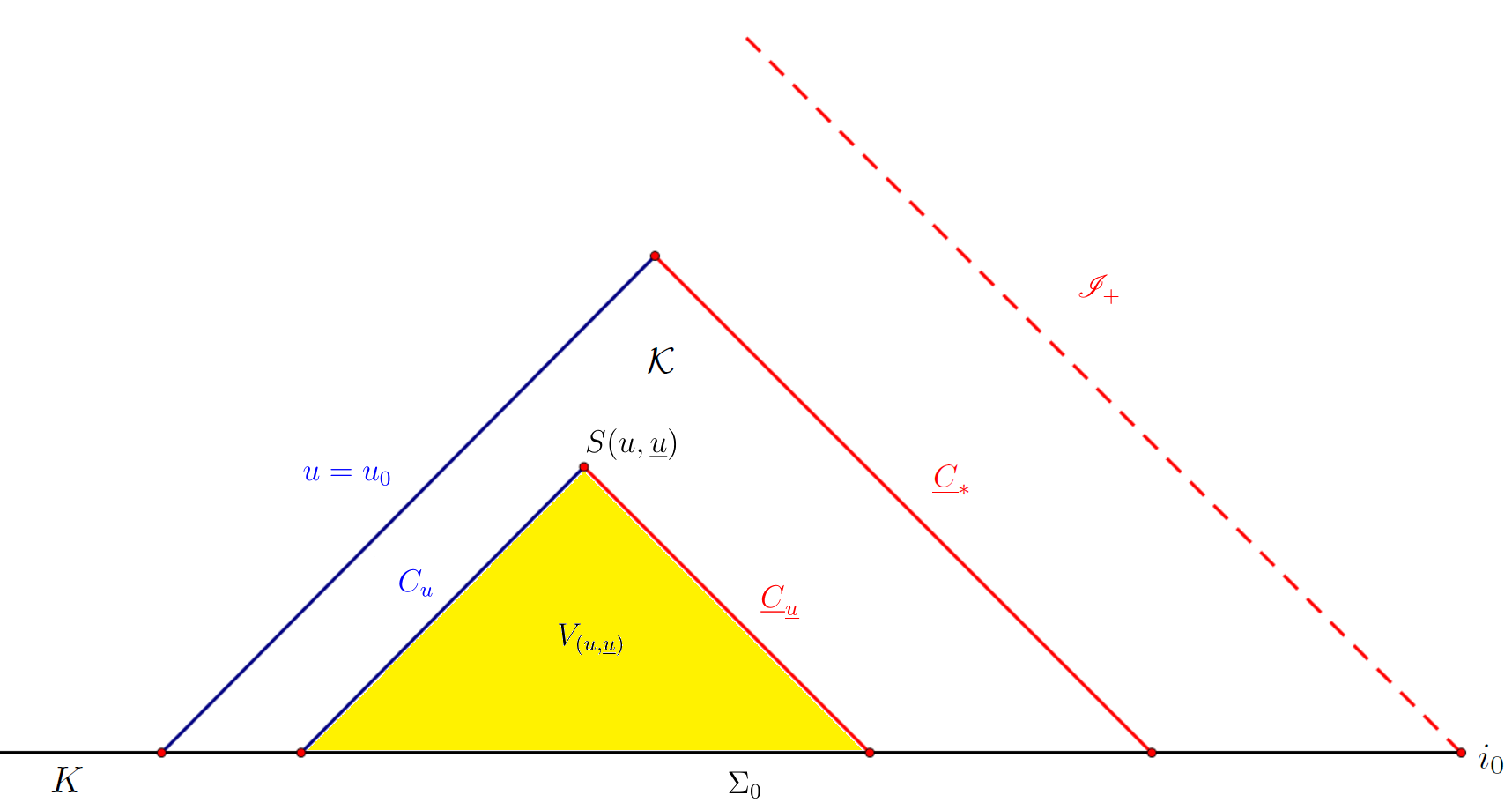}
  \caption{Bootstrap region $\KK$}\label{figboot}
\end{figure}
Remark that $\KK$ is covered by the double null foliation $(u,\ub)$, and that the boundary of $\KK$ consists of:
\begin{itemize}
\item a bounded subset of $\Si_0 \setminus K$;
\item a part of the outgoing null cone $C_{u_0}$;
\item a part of the last slice $\Cb_*$.
\end{itemize}
Recalling that the map $\Phi$ is defined on $\Cb_*$, we now extend it to $\kk$. As in Section \ref{doublenullsection}, we denote $L$ the generator of $C_u$ and extend $x^A$ to $\kk$ by
$$
L(x^A)=0,\qquad A=1,2.
$$
For any $p\in\kk$, we denote $\phi^L_t(p)$ the flow of $L$ emanating from $p$ to the future. Then, we denote
\begin{equation*}
    p_*:=\phi^L_t(p)\cap \Cb_*,\qquad q_*:=\Phi(p_*).
\end{equation*}
Next, we denote $\phi^{L_{Kerr}}_t(q_*)$ the flow of $L_{Kerr}$ emanating from $q_*$ to the past. We denote
\begin{equation*}
    q:=\phi^{L_{Kerr}}_t(q_*)\cap S_{Kerr}(u(p),\ub(p)).
\end{equation*}
Finally, we define
\begin{equation}\label{Phi}
    \Phi(p)=q.
\end{equation}
Thus, we have extended the domain of definition of $\Phi$ to $\kk$ and we have for any $p\in\kk$\footnote{Note that the identity $x^A_{Kerr}(\Phi(p))=x^A(p)$ follows from the fact that it holds on $\Cb_*$ and from the identities $L(x^A)=0$ and $L_{Kerr}(x^A_{Kerr})=0$.}
\begin{align}
\begin{split}\label{Phicondition}
\Phi(p)\in\MM_{Kerr},\quad u_{Kerr}(\Phi(p))=u(p),\quad \ub_{Kerr}(\Phi(p))=\ub(p),\quad x^A_{Kerr}(\Phi(p))=x^A(p).
\end{split}
\end{align}
In other words, the expression of the application $\Phi$ in the local coordinates charts $(u,\ub,x^1,x^2)$ and $(u_{Kerr},\ub_{Kerr},x^1_{Kerr},x^2_{Kerr})$ is the identity map, or equivalently
\begin{equation*}
    u=u_{Kerr}\circ\Phi,\qquad \ub=\ub_{Kerr}\circ\Phi,\qquad x^A=x^A_{Kerr}\circ\Phi,\qquad A=1,2.
\end{equation*}
\subsection{Integral formulae}\label{ssecave}
We define the $S$-average of scalar functions.
\begin{definition}\label{average}
Given any scalar function $f$, we denote its average by
\begin{equation*}
    \ov{f}:=\frac{1}{|S|}\int_{S}f\,d\ga.
\end{equation*}
\end{definition}
We recall the following integral formulae, which will be used repeatedly in this paper.
\begin{lemma}\label{dint}
Let $S(u,\ub)$ a leaf of a double null foliation $(u,\ub)$. For any scalar function $f$, the following equations hold:
\begin{align*}
    \Om e_4\left(\int_{S(u,\ub)} f d\ga\right) &= \int_{S(u,\ub)} \left(\Om e_4(f) + \Omega \tr \chi f \right) d\ga ,\\
    \Om e_3 \left(\int_{S(u,\ub)} f d\ga\right) &= \int_{S(u,\ub)} \left(\Om e_3(f)+ \Om\tr\chib f \right) d\ga .
\end{align*}
Taking $f=1$, we obtain
\begin{equation*}
    e_4(r)=\frac{\overline{\Omega\tr\chi}}{2\Omega}r,\qquad e_3(r)=\frac{\overline{\Omega\tr\chib}}{2\Omega}r,\label{e3e4r}
\end{equation*}
where we recall that $r$ is the \emph{area radius} defined by
\begin{equation*}
    r(u,\ub)=\sqrt{\frac{|S(u,\ub)|}{4\pi}}.
\end{equation*}
\end{lemma}
\begin{proof}
See Lemma 3.1.3 in \cite{Kl-Ni}.
\end{proof}
\subsection{Hodge systems}\label{ssec7.2}
\begin{definition}\label{tensorfields}
For tensor fields defined on a $2$--sphere $S$, we denote by $\sfr_0:=\sfr_0(S)$ the set of pairs of scalar functions, $\sfr_1:=\sfr_1(S)$ the set of $1$--forms and $\sfr_2:=\sfr_2(S)$ the set of symmetric traceless $2$--tensors.
\end{definition}
\begin{definition}\label{def7.2}
Given $\xi\in\sfr_1$, we define its Hodge dual
\begin{equation*}
    {^*\xi}_A := \in_{AB}\xi^B.
\end{equation*}
Clearly $^*\xi \in \sfr_1$ and
\begin{equation*}
    ^*(^*\xi)=-\xi.
\end{equation*}
Given $U \in \sfr_2$, we define its Hodge dual
\begin{equation*}
{^*U}_{AB}:= \in_{AC} {U^C}_B.
\end{equation*}
Observe that $^*U\in\sfr_2$ and
\begin{equation*}
     ^*(^*U)=-U.
\end{equation*}
\end{definition}
\begin{definition}
    Given $\xi,\eta\in\sfr_1$, we denote
\begin{align*}
    \xi\cdot \eta&:= \de^{AB}\xi_A \eta_B,\\
    \xi\wedge \eta&:= \in^{AB} \xi_A \eta_B,\\
    (\xi\hot \eta)_{AB}&:=\xi_A \eta_B +\xi_B \eta_A -\de_{AB}\xi\cdot \eta,
\end{align*}
where $\xi\hot\eta\in\sk_{2}$. Given $\xi\in \sfr_1$, $U\in \sfr_2$, we denote
\begin{align*}
    (\xi \cdot U)_A:= \de^{BC} \xi_B U_{AC}.
\end{align*}
Given $U,V\in \sfr_2$, we denote
\begin{align*}
    U\wedge V:=\in^{AB}U_{AC}V_{CB}.
\end{align*}
\end{definition}
\begin{definition}
    For a given $\xi\in\sfr_1$, we define the following differential operators:
    \begin{align*}
        \sdiv \xi&:= \de^{AB} \nab_A\xi_B,\\
        \curl \xi&:= \in^{AB} \nab_A \xi_B,\\
        (\nab\hot\xi)_{AB}&:=\nab_A \xi_B+\nab_B \xi_A-\de_{AB}(\sdiv\xi).
    \end{align*}
\end{definition}
\begin{definition}
    We define the Hodge type operators, as introduced in Section 2.2 in \cite{ch-kl}:
    \begin{itemize}
        \item $d_1$ takes $\sfr_1$ into $\sfr_0$ and is given by:
        \begin{equation*}
            d_1 \xi :=(\sdiv\xi,\curl \xi),
        \end{equation*}
        \item $d_2$ takes $\sfr_2$ into $\sfr_1$ and is given by:
        \begin{equation*}
            (d_2 U)_A := \nab^A U_{AB}, 
        \end{equation*}
        \item $d_1^*$ takes $\sfr_0$ into $\sfr_1$ and is given by:
        \begin{align*}
            d_1^*(f,f_*):=-\nab_A f +\in_{AB} \nab_B f_*,
        \end{align*}
        \item $d_2^*$ takes $\sfr_1$ into $\sfr_2$ and is given by:
        \begin{align*}
            d_2^* \xi := -\frac{1}{2} \nab \hot \xi.
        \end{align*}
    \end{itemize}
\end{definition}
Then, we have the following identities:
\begin{align}
    \begin{split}\label{dddd}
        d_1^*d_1&=-\De_1+\mathbf{K},\qquad \qquad d_1d_1^*=-\De_0,\\
        d_2^*d_2&=-\f12 \De_2+\mathbf{K},\qquad\quad\; d_2d_2^*=-\f12(\De_1+\mathbf{K}),
    \end{split}
\end{align}
where $\mathbf{K}$ denotes the Gauss curvature of $S$. See for example (2.2.2) in \cite{ch-kl} for a proof of \eqref{dddd}. 
\begin{definition}\label{dfdkb}
We define weighted angular derivatives $\dkb$ as follows:
\begin{align*}
    \dkb U &:= rd_2 U,\qquad \forall U\in \sk_2,\\
    \dkb \xi&:= rd_1 \xi,\qquad\,\,\, \forall \xi\in \sk_1,\\
    \dkb f&:= rd_1^* f,\qquad \,\,\forall f\in \sk_0.
\end{align*}
Also, we denote for any tensor $h\in\sk_k$, $k=0,1,2$,
\begin{equation*}
    h^{(0)}=h,\qquad h^{(1)}=(h,r\nab h).
\end{equation*}
\end{definition}
\subsection{Elliptic estimates for Hodge systems}
\begin{definition}\label{Lpnorms}
For a tensor field $f$ on a $2$--sphere $S$, we denote its $L^p$--norm by
\begin{equation}
    |f|_{L^p}:=\left(\int_S |f|^p \right)^\frac{1}{p},
\end{equation}
where $|f|^2:=f^{A_1...A_k}f_{A_1...A_k}$ for a $k$--tensor $f$ on $S$.
\end{definition}
\begin{proposition}[$L^p$ estimates for Hodge systems]\label{ellipticest}
Assume that $S$ is an arbitrary compact 2-surface satisfying $c_1\leq r^2K\leq c_2$ where $K$ denotes its Gauss curvature and $0<c_1\leq c_2<+\infty$ are two positive constants. Then the following statements hold for all\footnote{In practice, we only use Proposition \ref{ellipticest} with $p=2$ and $p=4$.} $p\in(1,+\infty)$:
\begin{enumerate}
\item Let $\phi\in\sfr_0$ be a solution of $\De\phi=f$. Then we have
\begin{align*}
    |\nabla^2 \phi|_{L^p}+r^{-1}|\nabla \phi|_{L^p}+r^{-2} |\phi-\overline{\phi}|_{L^p}\lesssim |f|_{L^p}.
\end{align*}
\item Let $\xi\in\sfr_1$ be a solution of $d_1\xi=(f,f_*)$. Then we have
\begin{align*}
    \int_S |\nab\xi|^p+r^{-p}|\xi|^p\lesssim \int_S |f|^p +|f_*|^p.
\end{align*}
\item Let $U\in\sfr_2$ be a solution of $d_2 U=f$. Then we have
\begin{align*}
    \int_S |\nab U|^p+r^{-p}|U|^p\lesssim \int_S|f|^p.
\end{align*}
\end{enumerate}
\end{proposition}
\begin{proof}
See Corollary 2.3.1.1 of \cite{ch-kl}.
\end{proof}
\subsection{Null structure equations}
We recall the standard null structure equations in a double null foliation.
\begin{proposition}\label{standardnull}
We have the following null structure equations in a double null foliation:
\begin{align*}
\nab_4 \eta =& -\chi\cdot(\eta-\etab)-\beta, \\
\nab_3\etab=& -\chib\cdot(\etab-\eta)+\bb, \\
\nab_4\hch +(\tr\chi)\hch =& -2\omega \hch -\alpha,\\
\nab_4\tr\chi +\frac{1}{2}(\tr\chi)^2 =&-|\hch|^2-2\omega \tr\chi,\\
\nab_3\hchb +(\tr\chib)\hchb=&-2\omb \hchb -\aa,\\
\nab_3\tr\chib +\frac{1}{2}(\tr\chib)^2=&-|\hchb|^2-2\omb\tr\chib,\\
\nab_4\hchb+\frac{1}{2}(\tr\chi)\hchb=&\nab\hot\etab+2\omega\hchb-\frac{1}{2}(\tr\chib)\hch+\etab\hot\etab,\\
\nab_3\hch+\frac{1}{2}(\tr\chib)\hch=&\nab\hot\eta+2\omb\hchb-\frac{1}{2}(\tr\chi)\hchb+\eta\hot\eta,\\
\nab_4\tr\chib+\frac{1}{2}(\tr\chi)\tr\chib=& 2\omega\tr\chib+2\rho-\hch\cdot\hchb+2\sdiv\etab+2|\etab|^2,\\
\nab_3\tr\chi+\frac{1}{2}(\tr\chib)\tr\chi=& 2\omb\tr\chi+2\rho-\hch\cdot\hchb+2\sdiv\eta+2|\eta|^2.
\end{align*}
We also have the Codazzi equations
\begin{align*}
\sdiv\hch=&\frac{1}{2}\nab\tr\chi-\ze\cdot \left(\hch-\f12 \tr\chi\right)-\b,\\
\sdiv\hchb=&\f12\nab\tr\chib+\ze\cdot\left(\hchb-\f12 \tr\chib\right)+\bb,
\end{align*}
the torsion equation
\begin{align*}
\curl \eta=& -\curl \etab=\sigma+\f12 \hchb\wedge \hch,
\end{align*}
and the Gauss equation
\begin{align*}
\K=-\frac{1}{4}\tr\chi\tr\chib+\frac{1}{2}\hchb\cdot\hch-\rho.
\end{align*}
Moreover, we have
\begin{align*}
\nab_4 \omb=& 2\omega\omb+\frac{3}{4}|\eta-\etab|^2-\frac{1}{4}(\eta-\etab)\cdot(\eta+\etab)-\frac{1}{8}|\eta+\etab|^2+\frac{1}{2}\rho,\\
\nab_3 \omega=& 2\omega\omb+\frac{3}{4}|\eta-\etab|^2+\frac{1}{4}(\eta-\etab)\cdot(\eta+\etab)-\frac{1}{8}|\eta+\etab|^2+\frac{1}{2}\rho.
\end{align*}
\end{proposition}
\begin{proof}
See (3.1)--(3.5) in \cite{kr}.
\end{proof}
\begin{definition}\label{defauxi}
We define the mass aspect functions as follows:
\begin{align}\label{defmassaspect}
    \mu:=-\sdiv\eta+\frac{1}{2}\hch\c\hchb-\rho,\qquad \mub:=-\sdiv\etab+\f12\hch\c\hchb-\rho.
\end{align}
We also define the following renormalized quantities which will be used in Sections \ref{Ricciestimates} and \ref{initiallast}:
\begin{align}
    \begin{split}\label{auxillaryquantities}
            \mum&:=\mu+\frac{1}{4}\trch\,\trchb,\qquad\qquad\mubm:=\mub+\frac{1}{4}\trch\,\trchb.
    \end{split}
\end{align}
\end{definition}
\begin{proposition}\label{cornull}
We have the following null structure equations:
\begin{align}
\begin{split}\label{cornulleq}
    \Om \nab_4(\Om \trch)+\f12(\Om\trch)^2&=-4\Om\trch (\Om\om)-\Om^2|\hch|^2,\\
    \Om \nab_3(\Om \trchb)+\f12(\Om\trchb)^2&=-4\Om\trchb (\Om\om)-\Om^2|\hchb|^2,\\
    \Om\nab_4(\Om\trchb)&=-2\Om^2[\mub]+2\Om^2|\etab|^2,\\
    \Om\nab_3(\Om\trch)&=-2\Om^2[\mu]+2\Om^2|\eta|^2.
\end{split}
\end{align}
We have the following equations for the renormalized mass aspect functions:
\begin{align}
\begin{split}\label{eqauxillary}
    \Om\nab_4\mum+\Om\trch\mum&=\Om [F]+\Om(\trch\,\rho-2\eta\c\b),\\
    [F]&:=\hch\c(\nab\hot\eta)+(\eta-\etab)\c(\nab\trch+\trch\,\ze)+\f12\trch(|\eta|^2-|\etab|^2)\\
    &-\f12\left(\trchb|\hch|^2+\trch(\hch\c\hchb)\right)+2\eta\c\hch\c\etab,\\ \\
    \Om\nab_3\mubm+\Om\trchb\mubm&=\Om [\Fb]+\Om(\trchb\,\rho+2\etab\c\bb),\\
    [\Fb]&:=\hchb\c(\nab\hot\etab)+(\etab-\eta)\c(\nab\trchb-\trchb\,\ze)+\f12\trchb(|\etab|^2-|\eta|^2)\\
    &-\f12\left(\trchb|\hchb|^2+\trchb(\hch\c\hchb)\right)+2\etab\c\hchb\c\eta.
\end{split}
\end{align}
\end{proposition}
\begin{proof}
Notice that \eqref{cornulleq} follows directly from Proposition \ref{standardnull} and \eqref{6.6}. See Lemma 4.3.2 in \cite{Kl-Ni} for the proof of \eqref{eqauxillary}.
\end{proof}
\subsection{Bianchi equations}
We recall the Bianchi equations in a double null foliation.
\begin{proposition}\label{standardbianchi}
The Bianchi equations take the following form in a double null foliation:
\begin{align*}
\nab_4\aa+\frac{1}{2}\tr\chi\,\aa&=-\nab\hat{\otimes}\bb+4\omega\aa-3(\hchb\rho-{^*\hchb}\sigma)+(\zeta-4\etab)\hat{\otimes}\bb,\\
\nab_3\bb+2\tr\chib\,\bb&=-\sdiv\aa-2\omb\bb+(2\ze-\eta)\cdot\aa,\\
\nab_4\bb+\tr\chi\,\bb&=-\nabla\rho+{^*\nabla\sigma}+2\omega\bb+2\hchb\cdot\beta-3(\etab\rho-{^*\etab}\si),\\
\nab_3\rho+\frac{3}{2}\tr\chib\,\rho&=-\sdiv\bb-\frac{1}{2}\hch\cdot\aa+\zeta\cdot\bb-2\eta\cdot\bb,\\
\nab_4\rho+\frac{3}{2}\tr\chi\,\rho&=\sdiv\beta-\frac{1}{2}\hchb\cdot\alpha+\zeta\cdot\b+2\etab\cdot\b,\\
\nab_3\si+\frac{3}{2}\tr\chib\,\si&=-\curl\bb+\frac{1}{2}\hch\cdot{^*\aa}-\zeta\cdot{^*\bb}+2\eta\cdot{^*\bb},\\
\nab_4\si+\frac{3}{2}\tr\chi\,\si&=-\curl\b+\frac{1}{2}\hchb\cdot{^*\alpha}-\zeta\cdot{^*\beta}-2\etab\cdot{^*\b},\\
\nab_3\beta+\tr\chib\,\b&=\nabla\rho+{^*\nabla}\si+2\omb\beta+2\hch\cdot\bb+3(\eta\rho+{^*\eta}\si),\\
\nab_4\beta+2\tr\chi\,\b&=\sdiv\alpha-2\omega\beta+(2\zeta+\etab)\a,\\
\nab_3\a+\f12\tr\chib\,\a&=\nab\hat{\otimes}\beta+4\omb\a-3(\hch\rho+{^*\hch}\sigma)+(\ze+4\eta)\hat{\otimes}\b.
\end{align*}
\end{proposition}
\begin{proof}
See Proposition 3.2.4 in \cite{Kl-Ni}.
\end{proof}
\subsection{Transport equations for metric components}
We introduce the following transport equations for metric components.
\begin{proposition}\label{metricequations}
The metric components satisfy the following equations:
\begin{align*}
    \nab_4(\ga_{AB})=&2\chi_{AB},\\
    \nab_4(\bu^A)=&-4\Om\ze^A,\\
    \nab_4(\in_{AB})=&\trch\in_{AB},\\
    \nab_3(\ga_{AB})=&2\chib_{AB}-\Om^{-1}\pr_A(\bu^C)\ga_{BC}-\Om^{-1}\pr_B(\bu^C)\ga_{AC},\\
    \nab_3(\in_{AB})=&\trchb\in_{AB}-\Om^{-1}\pr_A(\bu^C)\in_{BC}-\Om^{-1}\pr_B(\bu^C)\in_{BC}.
\end{align*}
\end{proposition}
\begin{proof}
We have from $[\pr_\ub,\pr_A]=0$ that
\begin{align*}
    0&=\nab_{\pr_\ub}(\ga)\left(\frac{\pr}{\pr x^A},\frac{\pr}{\pr x^B}\right)=\frac{\pr}{\pr\ub}\ga_{AB}-\ga\left(\D_{\pr_{\ub}}\pr_{A},\pr_B\right)-\ga\left(\pr_A,\D_{\pr_{\ub}}\pr_B\right)\\
    &=\Om\nab_4(\ga_{AB})-\Om\g(\D_{e_4}\pr_A,\pr_B)-\Om\g(\pr_A,\D_{e_4}\pr_B)\\
    &=\Om\nab_4(\ga_{AB})-2\Om\chi_{AB},
\end{align*}
which implies
\begin{equation*}
    \nab_4(\ga_{AB})=2\chi_{AB}
\end{equation*}
as stated in the first equation. For the proof of the second equation, see (3.1.62) in \cite{Kl-Ni}. \\ \\
We have from $[\pr_\ub,\pr_A]=0$ that
\begin{align*}
    0&=\nab_{\pr_{\ub}}(\in)\left(\frac{\pr}{\pr x^A},\frac{\pr}{\pr x^B}\right)=
    \frac{\pr}{\pr \ub}\in_{AB}-\in\left(\D_{\pr_{\ub}}\pr_{A},\pr_B\right)-\in\left(\pr_A,\D_{\pr_{\ub}}\pr_B\right)\\
    &=\frac{\pr}{\pr \ub}\in_{AB}-\in\left(\D_{\pr_{A}}\pr_{\ub},\pr_B\right)-\in\left(\pr_A,\D_{\pr_B}\pr_{\ub}\right)\\
    &=\frac{\pr}{\pr \ub}\in_{AB}-\Om\in({\chi_A}^C\pr_C,\pr_B)-\Om\in(\pr_{A},{\chi_{B}}^C\pr_C)\\
    &=\Om\nab_4(\in_{AB})-\Om{\chi_{A}}^C\in_{CB}-\Om{\chi_{B}}^C \in_{AC},
\end{align*}
which implies
\begin{align*}
    \nab_4(\in_{AB})={\chi_{A}}^C\in_{CB}+{\chi_{B}}^C \in_{AC}=\trch\in_{AB}.
\end{align*}
Next, we compute
\begin{align*}
    0&=\left(\nab_{\Om e_3}\ga\right)\left(\frac{\pr}{\pr x^A},\frac{\pr}{\pr x^B}\right)\\
    &=\Om\nab_3(\ga_{AB})-\ga\left(\nab_{\Om e_3}\pr_A,\pr_B\right)-\ga\left(\pr_A,\nab_{\Om e_3}\pr_B\right)\\
    &=\Om\nab_3(\ga_{AB})-\Om\g\left(\D_{\pr_A}e_3,\pr_B\right)-\Om\g\left(\pr_A,\D_{\pr_B}e_3\right)-\g([\Om e_3,\pr_A],\pr_B)-\g(\pr_A,[\Om e_3,\pr_B])\\
    &=\Om\nab_3(\ga_{AB})-2\Om\chib_{AB}-\g([\bu^B\pr_C,\pr_A],\pr_B)-\g(\pr_A,[\bu^C\pr_C,\pr_B])\\
    &=\Om\nab_3(\ga_{AB})-2\Om\chib_{AB}+\pr_A(\bu^C)\ga_{BC}+\pr_B(\bu^C)\ga_{AC},
\end{align*}
which implies
\begin{equation}\label{duga}
    \nab_3(\ga_{AB})=2\chib_{AB}-\Om^{-1}\pr_A(\bu^C)\ga_{BC}-\Om^{-1}\pr_B(\bu^C)\ga_{AC}.
\end{equation}
Similarly, we have
\begin{equation*}
    \nab_3(\in_{AB})=\trchb\in_{AB}-\Om^{-1}\pr_A(\bu^C)\in_{BC}-\Om^{-1}\pr_B(\bu^C)\in_{AC}.
\end{equation*}
This concludes the proof of Proposition \ref{metricequations}.
\end{proof}
\subsection{The linearized equations}\label{seclinearequation}
\subsubsection{Linearization procedure}
Let $\kk$ the bootstrap region defined in Section \ref{bootregion}. By construction, we have
\begin{equation*}
    L(x^A)=0.
\end{equation*}
Hence, the metric in $\kk$ takes the following form, see \eqref{metricg}:
\begin{equation}\label{gkerr}
    \g=-2\Omega^2(d\ub\otimes du+du\otimes d\ub)+\ga_{AB}(dx^A-\bu^Adu)\otimes(dx^B-\bu^Bdu).
\end{equation}
The inverse metric components are given by:
\begin{align}
\begin{split}\label{Gkerr}
    \g^{AB}&=\ga^{AB},\qquad\qquad\quad \g^{u\ub}=-\frac{1}{2\Om^2},\quad\qquad\; \g^{\ub A}=-\frac{\bu^A}{2\Om^2},\\
    \g^{uu}&=\g^{\ub\ub}=0,\qquad\;\quad \g^{uA}=0,\qquad\qquad\quad\, A,B=1,2.
\end{split}
\end{align}
Let $X$ be a tensor defined on $S:=S(u,\ub)$. We denote $X_{Kerr}$ the corresponding $S_{Kerr}$--tangent tensor in $\MM_{Kerr}$. Then, its pullback $\Phi^* X_{Kerr}$\footnote{Recall that the map $\Phi:\KK\to\MM_{Kerr}$ is defined in \eqref{Phi} and satisfies \eqref{Phicondition}.} is a $S$--tangent tensor in $\kk$ denoted by
\begin{equation}
X_K := \Phi^* X_{Kerr}.
\end{equation}
We now define the linearization of $X$ as follows
\begin{equation}
\widecheck{X}:= X-X_K.
\end{equation}
We also denote $\nab_K$, (resp. $\Ga_K$, $\sdiv_K$ and $\curl_K$) the Levi-Civita connection, (resp. Christoffel symbols, divergence operator and curl operator) of the metric $\g_K=\Phi^*(\g_{Kerr})$.\\ \\
The following lemma plays an important role in the linearization of the equations.
\begin{lemma}\label{checke4e3}
We have the following formulae:
\begin{align*}
    \widecheck{\Om e_4}&:=\Om e_4-\Om_K(e_4)_K=0,\\
    \widecheck{\Om e_3}&:=\Om e_3-\Om_K(e_3)_K=\bcc^A\pr_{x^A}.
\end{align*}
\end{lemma}
\begin{proof}
Notice that we have
\begin{align}
\begin{split}
    e_4(u)&=0,\qquad\;\,\, e_4(\ub)=\g(e_4,\grad \ub)=\g(2\Omega L,-\Lb)=\Omega^{-1}, \\
    e_3(\ub)&=0,\qquad\;\,\, e_3(u)=\g(e_3,\grad u)=\g(2\Omega \Lb,-L)=\Omega^{-1} ,\\
    e_4(x^A)&=0,\qquad e_3(x^A)=\Omega^{-1}\bu^A.
\end{split}
\end{align}
Hence, we obtain
\begin{equation}
\Om e_4=\Om e_4(u)\pr_u +\Om e_4(\ub)\pr_{\ub}+\Om e_4(x^A)\pr_{x^A}=\pr_\ub.
\end{equation}
Thus, we infer
\begin{align*}
\Phi_*(\Om e_4)=\Phi_*(\pr_\ub)=\pr_{\ub_{Kerr}}=\Om_{Kerr}(e_4)_{Kerr},
\end{align*}
which implies
\begin{equation*}
    \Om_K (e_4)_K=\Phi^*(\Om_{Kerr}(e_4)_{Kerr})=\Om e_4.
\end{equation*}
Hence, we deduce
\begin{equation}\label{diff4}
\widecheck{\Om e_4}=\Om e_4-\Om_K(e_4)_K=\pr_{\ub}-\pr_{\ub}=0.
\end{equation}
Similarly, we have
\begin{align*}
    \Om e_3=\pr_u +\bu^A\pr_{x^A},
\end{align*}
which implies
\begin{align*}
    \Phi_*(\Om e_3)=\pr_{u_{Kerr}}+(\Phi^{-1})^*(\underline{b}^A) \pr_{x^A_{Kerr}}.
\end{align*}
Hence, we obtain
\begin{align}
\begin{split}\label{diff3}
    \widecheck{\Om e_3}&=\Om e_3-\Om_K(e_3)_K=\Phi_*^{-1}\left(\Phi_*(\Om e_3)-\Om_{Kerr}(e_3)_{Kerr}\right)\\
    &=\Phi_*^{-1}\left(((\Phi^{-1})^*\bu^A-\bu_{Kerr}^A)\pr_{x^A_{Kerr}}\right)\\
    &=(\bu^A-\bu^A_K)\pr_{x^A}=\bcc^A\pr_{x^A}.
\end{split}
\end{align}
This concludes the proof of Lemma \ref{checke4e3}.
\end{proof}
\begin{proposition}\label{useful}
Denoting $\Ga$ and $\Ga_K$ respectively the Christoffel symbols of $\nab$ and $\nab_K$. For any $U\in\sk_2(S)$ and $\phi\in\sk_1(S)$ where $S=S(u,\ub)$, we have the following formulae:
\begin{align}
\begin{split}\label{diffnab}
    (\nab-\nab_K)_A\,U_{BC}=&-(\Ga-\Ga_K)_{AB}^D\,U_{CD}-(\Ga-\Ga_K)_{AC}^D\,U_{BD},\\
    (\nab-\nab_K)_A\,\phi_{B}=&-(\Ga-\Ga_K)_{AB}^C\,\phi_{C},\\
    (\nab-\nab_K)_{\pr_{\ub}}U_{AB}=&-(\bGa-\bGa_K)_{\ub A}^C U_{CB}-(\bGa-\bGa_K)_{\ub B}^C U_{CA},\\
    (\nab-\nab_K)_{\pr_{\ub}}\phi_A=&-(\bGa-\bGa_K)_{\ub A}^B \,\phi_B,\\
    (\nab-\nab_K)_{\pr_{u}}U_{AB}=&-(\bGa-\bGa_K)_{u A}^CU_{CB}-(\bGa-\bGa_K)_{u B}^CU_{CA},\\
    (\nab-\nab_K)_{\pr_{u}}\phi_A=&-(\bGa-\bGa_K)_{u A}^B \,\phi_B.
\end{split}
\end{align}
We also have:
\begin{align}
\begin{split}\label{diffdiv}
(\sdiv-\sdiv_K)\,U=&(\ga^{AB}-\ga^{AB}_K)\nab_AU_{B\bullet}-\ga_K^{AB}(\Ga-\Ga_K)_{AB}^D\,U_{\bullet D}\\
&-\ga_K^{AB}(\Ga-\Ga_K)_{A\bullet}^D\,U_{BD}, \\
    (\sdiv-\sdiv_K)\,\phi=&(\ga^{AB}-(\ga_K)^{AB})\nab_A\phi_{B}-(\ga_K)^{AB}(\Ga-\Ga_K)_{AB}^C\phi_{C}, \\
    (\curl-\curl_K)\,\phi=&(\in^{AB}-(\in_K)^{AB})\nab_A\phi_{B}-(\in_K)^{AB}(\Ga-\Ga_K)_{AB}^C\phi_{C},\\
    (\nab\hot\phi-\nab_K\hot\phi_K)_{AB}=&(\nab\hot\widecheck{\phi})_{AB}-2(\Ga-\Ga_K)_{AB}^C(\phi_K)_C\\
    +&\delta_{AB}[(\ga^{CD}-(\ga_K)^{CD})\nab_C(\phi_K)_D-(\ga_K)^{CD}(\Ga-\Ga_K)_{CD}^E(\phi_K)_E].
\end{split}
\end{align}
\end{proposition}
\begin{proof}
The proof follows from a direct computation. For example, we have
\begin{align*}
    (\nab-\nab_K)_A\,U_{BC}=&\left(\frac{\pr}{\pr x^A}(U_{BC})-\Ga_{AB}^D \,U_{CD}-\Ga_{AC}^D\,U_{BD}\right)\\
    &-\left(\frac{\pr}{\pr x^A}(U_{BC})-(\Ga_K)_{AB}^D\,U_{CD}-(\Ga_K)_{AC}^D\,U_{BD}\right)\\
    =&-(\Ga-\Ga_K)_{AB}^D\,U_{CD}-(\Ga-\Ga_K)_{AC}^D\,U_{BD},
\end{align*}
which implies the first equation in \eqref{diffnab}. The proof of the other equations in \eqref{diffnab} are similar. We also have
\begin{align*}
    (\sdiv-\sdiv_K)U_C&=\ga^{AB}\nab_AU_{BC}-\ga_K^{AB}(\nab_K)_AU_{BC}\\
    &=(\ga^{AB}-\ga^{AB}_K)\nab_AU_{BC}+\ga_K^{AB}(\nab_AU_{BC}-(\nab_K)_AU_{BC})\\
    &=(\ga^{AB}-\ga^{AB}_K)\nab_AU_{BC}-\ga_K^{AB}(\Ga-\Ga_K)_{AB}^D\,U_{CD}-\ga_K^{AB}(\Ga-\Ga_K)_{AC}^D\,U_{BD},
\end{align*}
which implies the first equation in \eqref{diffdiv}. The proof of the other equations in \eqref{diffdiv} are similar and left to the reader. This concludes the proof of Proposition \ref{useful}.
\end{proof}
\begin{lemma}\label{bGa}
We have the following identities:
\begin{align*}
    \bGa_{\ub A}^B=&\f12 \ga^{BC}\pr_{\ub}(\ga_{AC}),\\
    \bGa_{uA}^B=&\f12\ga^{BC}\left(\pr_u(\ga_{AC})-\frac{\pr}{\pr x^A}\left(\ga_{CD} \bu^D\right)+\frac{\pr}{\pr x^C}\left(\ga_{AD}\bu^D\right)\right)+\frac{\bu^B}{4\Om^2}\left(2\frac{\pr}{\pr x^A}(\Om^2)-\frac{\pr}{\pr \ub}(\ga_{AD}\bu^D)\right),\\
    \bGa_{BC}^A=&\f12\ga^{AD}\left(\frac{\pr \ga_{BD}}{\pr x^C}+\frac{\pr \ga_{CD}}{\pr x^B}-\frac{\pr \ga_{BC}}{\pr x^D}\right) +\frac{\bu^A}{4\Om^2}\frac{\pr\ga_{BC}}{\pr\ub}.
\end{align*}
\end{lemma}
\begin{proof}
We have from \eqref{gkerr} and \eqref{Gkerr}
\begin{align*}
    \bGa_{\ub A}^B=\f12\g^{BC}\left(\frac{\pr\g_{\ub C}}{\pr x^A}+\frac{\pr\g_{A C}}{\pr\ub}-\frac{\pr\g_{\ub A}}{\pr x^C}\right)+\f12\g^{B\ub}\left(\frac{\pr\g_{\ub\ub}}{\pr x^A}+\frac{\pr\g_{A \ub}}{\pr\ub}-\frac{\pr\g_{\ub A}}{\pr \ub}\right)=\f12\ga^{BC}\pr_{\ub}(\ga_{AC}),
\end{align*}
and similarly
\begin{align*}
    \bGa_{uA}^B=&\f12\g^{BC}\left(\frac{\pr\g_{uC}}{\pr x^A}+\frac{\pr\g_{AC}}{\pr u}-\frac{\pr\g_{uA}}{\pr x^C}\right)+\f12\g^{B\ub}\left(\frac{\pr\g_{u\ub}}{\pr x^A}+\frac{\pr\g_{A\ub}}{\pr u}-\frac{\pr\g_{uA}}{\pr \ub}\right)\\
    =&\f12\ga^{BC}\left(-\frac{\pr}{\pr x^A}\left(\ga_{CD} \bu^D\right)+\frac{\pr\ga_{AC}}{\pr u}+\frac{\pr}{\pr x^C}\left(\ga_{AD}\bu^D\right)\right)+\frac{\bu^B}{4\Om^2}\left(2\frac{\pr}{\pr x^A}(\Om^2)-\frac{\pr}{\pr\ub}(\ga_{AD}\bu^D)\right).
\end{align*}
Finally, we compute
\begin{align*}
    \bGa_{BC}^A=&\f12\ga^{AD}\left(\frac{\pr\ga_{BD}}{\pr x^C}+\frac{\pr \ga_{CD}}{\pr x^B}-\frac{\pr\ga_{BC}}{\pr x^D}\right) +\f12\g^{A\ub}\left(\frac{\pr\g_{B\ub}}{\pr x^C}+\frac{\pr\g_{C\ub}}{\pr x^B}-\frac{\pr\g_{BC}}{\pr\ub}\right)\\
    =&\f12\ga^{AD}\left(\frac{\pr\ga_{BD}}{\pr x^C}+\frac{\pr\ga_{CD}}{\pr x^B}-\frac{\pr\ga_{BC}}{\pr x^D}\right) +\frac{\bu^A}{4\Om^2}\frac{\pr\ga_{BC}}{\pr\ub}.
\end{align*}
This concludes the proof of Lemma \ref{bGa}.
\end{proof}
The following tensors play an important role in the linearization of the null structure equations and Bianchi equations.
\begin{definition}\label{dfGamma}
We define the following  tensors\footnote{Remark that the difference of two Christoffel symbols is a tensor.}
\begin{align}
    \begin{split}
        \Jbc_{A}^B:=(\Ga-\Ga_K)_{u A}^B,\qquad \Jc_A^B:=(\Ga-\Ga_K)_{\ub A}^B, \qquad \Lc_{AB}^C:=(\Ga-\Ga_K)_{AB}^C.
    \end{split}
\end{align}
\end{definition}
\subsubsection{Schematic notation \texorpdfstring{$\Gag$}{} and  \texorpdfstring{$\Gab$}{}}
We introduce the following schematic notations for metric components, Ricci coefficients and curvature components.
\begin{definition}\label{gammag}
We divide the linearized quantities into two parts:
\begin{align*}
    \Gag:=\Gamma_g^{(0)}&:=\left\{\etac,\,\etabc,\,\zec,\,\hchc,\,\trchc,\,\trchbc,\,\omc,\,r^{-1}\bcc,\,r^{-1}\Omc,\,r^{-1}\gac,\,r^{-1}\inc,\,r^{-2}\widecheck{r}\right\},\\
    \Gab:=\Gamma_b^{(0)}&:=\{\hchbc,\,\ombc\}.
\end{align*}
We also denote:
\begin{align*}
    \Gag^{(1)}&:=(r\nab)^{\leq 1}\Gag^{(0)}\cup\{\Jc,\,\Lc,\,r\bc,\,r\rhoc,\,r\sic\},\\
    \Gab^{(1)}&:=(r\nab)^{\leq 1}\Gab^{(0)}\cup\{\Jbc,\,r\bbc\}.
\end{align*}
\end{definition}
\begin{remark}
    Throughout this paper, $\Gag$ decays better than $\Gab$, see Lemma \ref{estGagba}.
\end{remark}
\subsubsection{Schematic notation \texorpdfstring{$\O^p_q$}{}}
\begin{definition}\label{KerrO}
For any $S$--tangent tensor field $X_K=\Phi^*(X_{Kerr})$, we denote
\begin{equation*}
    X_K=\O^p_q,
\end{equation*}
if for any integer $l\geq 0$ we have
\begin{equation*}
    |(\pr_K)^l X_K|\les\frac{M^p}{r_K^{q+l}},
\end{equation*}
where $r_K=\Phi^*(r_{Kerr})$ and
\begin{equation*}
    \pr_K :=\left\{(\nab_K)_{(e_3)_K},\,\nab_K,\,(\nab_K)_{(e_4)_K}\right\}.
\end{equation*}
\end{definition}
\begin{proposition}\label{decayGamma}
We have the following formulae:
\begin{align*}
\tr_K\chi_{K}-\frac{2}{r_K}&=\bfO_2^1,\qquad\qquad\,\tr_K\chib_{K}+\frac{2}{r_K}=\bfO_2^1,\\
\om_{K}&=\bfO_2^1,\qquad\qquad\qquad\;\,\qquad\omb_K=\bfO_2^1,\\
\hch_{K}&=\bfO_3^2,\qquad\qquad\qquad\qquad\;\, \hchb_{K}=\bfO_3^2,\\
\nab_K\tr_K\chi_K&=\bfO_4^2,\qquad\qquad\quad \nab_K\tr_K\chib_K=\bfO_4^2,\\
\eta_{K}&=\bfO_3^2,\qquad\qquad\qquad\qquad\,\,\,\etab_{K}=\bfO_3^2,\\
\bu_{K}&=\bfO_2^2,\qquad\qquad\qquad\;\, \Om_{K}-\f12=\bfO_1^1.
\end{align*}
We also have:
\begin{align*}
    \a_{K} &=\bfO_5^3, \qquad\qquad\qquad\aa_{K}=\bfO_5^3,\\
    \b_{K} &=\bfO_4^2, \qquad\qquad\qquad \bb_{K} =\bfO_4^2,\\
    \rho_{K}&=\bfO_3^1, \qquad\qquad\qquad \si_{K} =\bfO_4^2.
\end{align*}
\end{proposition}
\begin{proof}
    It follows directly from Propositions \ref{decayGammaKerr}, \ref{decayRKerr} and  Definition \ref{KerrO}.
\end{proof}
\begin{remark}\label{ignoreremark}
In the sequel, we consider the following conventions:
\begin{itemize}
    \item For a quantity $h$ satisfying the same or even better decay as $\Ga_{b}$, (resp. $\Gag$), we write
    \begin{equation*}
        h\in\Gab,\quad  (\mbox{resp. } h\in \Gag).
    \end{equation*}
    \item For a sum of schematic notations, we ignore the terms that decay better. For example, we write
    \begin{equation*}
        \Gag+\Gab=\Gab,\qquad \Gag\c\O^p_q+\Gab\c\O^p_q=\Gab\c\O^p_q,
    \end{equation*}
    since $\Gag$ decays better than $\Gab$ throughout this paper.
\end{itemize}
\end{remark}
\begin{lemma}\label{expressionGaR}
    We have the following formulae:
    \begin{align*}
        \trch-\frac{2}{r}&=\O^1_2+\Gag,\qquad\quad\;\,\trchb+\frac{2}{r}=\O^1_2+\Gab,\\
        \hch&=\O^2_3+\Gag,\qquad\qquad\quad\quad\,\, \hchb=\O^2_3+\Gab,\\
        \eta&=\O^2_3+\Gag,\qquad\qquad\qquad\;\,\etab=\bfO_3^2+\Gag,\\
        \om&=\O^1_2+\Gag,\qquad\qquad\qquad\;\,\omb=\O^1_2+\Gab,\\
        \bu&=\bfO_2^2+r\Gag,\qquad\qquad \Om-\frac{1}{2}=\bfO_2^1+r\Gag,\\
        \ga&=\bfO_0^0+r\Gag,\qquad\qquad\qquad\; \in=\bfO_0^0+r\Gag.
    \end{align*}
    We also have:
    \begin{align*}
        \a&=\O^3_5+\ac,\qquad \quad\qquad\quad \aa=\O^3_5+\aac,\\
        \b&=\O^2_4+\bc,\qquad\quad \qquad\quad \,\bb=\bfO_4^2+\bbc,\\
        \rho&=\bfO_3^1+\rhoc,\qquad \quad\quad\quad\;\quad \si=\bfO_4^2+\sic.
    \end{align*}
\end{lemma}
\begin{proof}
    It follows directly from Definitions \ref{gammag}, \ref{KerrO} and Proposition \ref{decayGamma}.
\end{proof}
The following lemma allows us to compute the derivatives of $\bfO_q^p$.
\begin{lemma}\label{newuseful}
    Let $X_K$ be a $S$--tangent tensor field satisfying
    \begin{equation*}
        X_K=\O^p_q.
    \end{equation*}
    Then, we have
    \begin{align*}
        \Om\nab_3 X_K&=\O^p_{q+1}+\bfO_q^p\c\Gag+\bfO_{q}^p\c\Jbc+\bfO_q^p\c\Lc,\\
        \Om\nab_4 X_K&=\O^p_{q+1}+\bfO_{q}^p\c\Jc,\\
        \nab X_K&=\O^p_{q+1}+\bfO_{q}^p\c\Lc.
    \end{align*}
\end{lemma}
\begin{proof}
    We have from Lemma \ref{checke4e3} and Proposition \ref{useful}
    \begin{align*}
        \Om\nab_3 X_K-\Om_K\nab_{(e_3)_K}X_K&=\left(\nab_{\Om e_3}-\nab_{\Om_K(e_3)_K}\right)X_K+\left(\nab_{\Om_K(e_3)_K}-(\nab_K)_{\Om_K(e_3)_K}\right)X_K\\
        &=\bcc^A\nab_{\pr_{x^A}}X_K+(\nab-\nab_K)_{\pr_u}X_K+\bu^A_K(\nab-\nab_K)_{\pr_{x^A}}X_K\\
        &=\bcc^A(\nab-\nab_K)_{\pr_{x^A}}X_K+\bcc^A(\nab_K)_{\pr_{x^A}}X_K+\Jbc\c\bfO_q^p+\Lc\c\bfO_q^p\\
        &=\Gag\c\bfO_q^p+\Jbc\c\bfO_q^p+\Lc\c\bfO_q^p,
    \end{align*}
    where we used Remark \ref{ignoreremark} to ignore the terms having better decay. Combining with Definition \ref{KerrO}, we obtain
    \begin{equation*}
        \Om\nab_3 X_K=\bfO_{q+1}^p+\Gag\c\bfO_q^p+\Jbc\c\bfO_q^p+\Lc\c\bfO_q^p.
    \end{equation*}
    Similarly, we have from Lemma \ref{checke4e3} and Proposition \ref{useful}
    \begin{align*}
        \Om\nab_4X_K-\Om_K\nab_{(e_4)_K}X_K&=\left(\nab_{\Om e_4}-\nab_{\Om_K(e_4)_K}\right)X_K+\left(\nab_{\Om_K(e_4)_K}-(\nab_K)_{\Om_K(e_4)_K}\right)X_K\\
        &=(\nab-\nab_K)_{\pr_\ub}X_K\\
        &=\Jc\c\bfO_q^p.
    \end{align*}
    Combining with Definition \ref{KerrO}, we deduce
    \begin{equation*}
        \Om\nab_4X_K=\bfO_{q+1}^p+\Jc\c\bfO_q^p.
    \end{equation*}
    Finally, we have from Proposition \ref{useful} and Definition \ref{KerrO}
    \begin{align*}
        \nab X_K=\nab_K X_K+(\nab-\nab_K)X_K=\bfO_{q+1}^p+\Lc\c\bfO_q^p. 
    \end{align*}
    This concludes the proof of Lemma \ref{newuseful}.
\end{proof}
\begin{corollary}\label{newuse}
    We have the following schematic computation:
    \begin{equation*}
        \nab(\bfO_q^p\c\Gag)=\bfO_{q+1}^p\c\Gag^{(1)},\qquad \nab(\bfO_q^p\c\Gab)=\bfO_{q+1}^p\c\Gab^{(1)}.
    \end{equation*}
\end{corollary}
\begin{proof}
    We have from Lemma \ref{newuseful}
    \begin{align*}
        \nab(\bfO_q^p\c\Ga_i)&=\bfO_q^p\c\nab\Ga_i+\nab\bfO_q^p\c\Ga_i\\
        &=\bfO_{q+1}^p\c\Ga_i^{(1)}+(\bfO_{q+1}^p+\bfO_q^p\c\Lc)\c\Ga_i\\
        &=\bfO_{q+1}^p\c\Ga_i^{(1)},
    \end{align*}
    for $i=g,b$. This concludes the proof of Corollary \ref{newuse}.
\end{proof}
\begin{remark}
    In the remainder of this paper, Lemma \ref{newuseful} and Corollary \ref{newuse} will be used frequently without explicit mention.
\end{remark}
\subsubsection{The linearized null structure equations}
In this section, we introduce the following linearized null structure equations, which play a fundamental role in Ricci coefficients estimates.
\begin{proposition}\label{nullstructure}
We have the following linearized null structure equations:
\begin{align*}
\nab_4\,\widecheck{\Om\trch}+\tr\chi\,\widecheck{\Om\trch}=&\bfO_1^0\c\widecheck{\Om\om}+\fl[\nab_4\,\trchc]+\err[\nab_4\,\trchc],\\
\fl[\nab_4\,\trchc]:=&\Gag\c\bfO_2^1,\\
\err[\nab_4\,\trchc]:=&\Gag\cdot\Gag,\\\\
\nab_3\,\widecheck{\Om\trchb}+\tr\chib\,\widecheck{\Om\trchb}=&\bfO_1^0\c\widecheck{\Om\omb}+\fl[\nab_3\,\trchc]+\err[\nab_3\,\trchc],\\
\fl[\nab_3\,\trchbc]:=&\Gag\c\bfO_2^1,\\
\err[\nab_3\,\trchbc]:=&\Gab\c\Gab,\\\\
\nab_4(\widecheck{\Om\trchb})=&-2\Om\mubmc+\fl[\nab_4\,\trchbc]+\err[\nab_4\,\trchbc],\\
\fl[\nab_4\,\trchbc]:=&\Gag\c\bfO_3^2,\\
\err[\nab_4\,\trchbc]:=&\Gag\c\Gag,\\\\
\nab_3(\widecheck{\Om\trch})=&-2\Om\mumc+\fl[\nab_3\,\trchc]+\err[\nab_3\,\trchc],\\
\fl[\nab_3\,\trchc]:=&\Gag\c\bfO_3^2,\\
\err[\nab_3\,\trchc]:=&\Gag\c\Gag.
\end{align*}
We also have the linearized Codazzi equations:
\begin{align*}
\sdiv(\hchc)=&-\bc+\bfO_0^0\c\nab\trchc+\bfO_1^0\c\zec+\fl[\sdiv(\widecheck{\hch})]+\err[\sdiv(\widecheck{\hch})],\\
\fl[\sdiv(\widecheck{\hch})]:=&\Gag\c\bfO_3^2+\bfO_3^2\c\Lc,\\
\err[\sdiv(\widecheck{\hch})]:=&\Gag\c\Gag,\\ \\
\sdiv(\widecheck{\hchb})=&\bbc+\bfO_0^0\c\nab\trchbc+\bfO_1^0\c\zec+\fl[\sdiv(\widecheck{\hchb})]+\err[\sdiv(\widecheck{\hchb})],\\
\fl[\sdiv(\widecheck{\hchb})]:=&\Gab\c\bfO_3^2+\bfO_3^2\c\Lc,\\
\err[\sdiv(\widecheck{\hchb})]:=&\Gab\c\Gag.
\end{align*}
The linearized mass aspect functions are given by:
\begin{align*}
    \sdiv\etac&=-\muc-\rhoc+\fl[\sdiv(\etac)]+\err[\sdiv(\etac)],\\
    \fl[\sdiv(\etac)]&=\Gab\c\bfO_3^2+\bfO_3^2\c\Lc,\\
    \err[\sdiv(\etac)]&=\Gag\c\Gab,\\ \\
    \sdiv\etabc&=-\mubc-\rhoc+\fl[\sdiv(\etabc)]+\err[\sdiv(\etabc)],\\
    \fl[\sdiv(\etabc)]&=\Gab\c\bfO_3^2+\bfO_3^2\c\Lc,\\
    \err[\sdiv(\etabc)]&=\Gag\c\Gab.
\end{align*}
The linearized torsion equations are given by:
\begin{align*}
\curl\etac=&\widecheck{\si}+\fl[\curl\etac]+\err[\curl\etac],\\
\fl[\curl\etac]:=&\Gab\c\bfO_3^2+\bfO_3^2\c\Lc,\\
\err[\curl\etac]:=&\Gab\c\Gag,\\ \\
\curl\etabc=&-\widecheck{\si}+\fl[\curl\etabc]+\err[\curl\etabc],\\
\fl[\curl\etabc]:=&\Gab\c\bfO_3^2+\bfO_3^2\c\Lc,\\
\err[\curl\etabc]:=&\Gab\c\Gag.
\end{align*}
Moreover, we have
\begin{align*}
\Om\nab_4(\widecheck{\Om\omb})=&\f12\widecheck{\Om^2\rho}+\fl[\nab_4\ombc]+\err[\nab_4\ombc],\\
\fl[\nab_4\ombc]:=&\Gag\c\bfO_2^1,\\
\err[\nab_4\ombc]:=&\Gag\c\Gag,\\ \\
\Om\nab_3(\widecheck{\Om\om})=&\f12\widecheck{\Om^2\rho}+\fl[\nab_3\omc]+\err[\nab_3\omc],\\
\fl[\nab_3\omc]:=&\Gag\c\bfO_2^1,\\
\err[\nab_3\omc]:=&\Gag\c\Gag.
\end{align*}
\end{proposition}
\begin{proof}
See Appendix \ref{pfnull}.
\end{proof}
\begin{remark}
    In the linearized equations of Proposition \ref{nullstructure}, as well as in the other linearized equations below, $\fl[...]$ contains linear terms which display more decay in $\frac{M}{r}$ and $\err[...]$ contains the nonlinear terms.
\end{remark}
\begin{proposition}\label{equationsmumub}
We have the following linearized equations:
\begin{align*}
\Om\nab_4\mumc+\Om\trch\mumc:=&\Om\trch\,\rhoc+\fl[\nab_4\mumc]+\err[\nab_4\mumc],\\
\fl[\nab_4\mumc]:=&\bfO_3^2\c (d_2^*\etac,\nab\trchc,\bc)+\bfO_3^1\c(\Gag,\Lc),\\
\err[\nab_4\mumc]:=&-2\hchc\c(d_2^*\etac)+(\etac-\etabc)\c\nab\trchc-2\widecheck{\Om\eta}\c\bc+r^{-1}\Gag\c\Lc+r^{-1}\Gag\c\Gab,\\ \\
\Om\nab_3\mubmc+\Om\trchb\mubmc:=&\Om\trchb\,\rhoc+\fl[\nab_3\mubmc]+\err[\nab_3\mubmc],\\
\fl[\nab_3\mubmc]:=&\bfO_3^2\c (d_2^*\etabc,\nab\trchbc,\bbc)+\bfO_3^1\c(\Gag,\Lc),\\
\err[\nab_3\mubmc]:=&-2\hchbc\c(d_2^*\etabc)+(\etabc-\etac)\c\nab\trchbc+2\widecheck{\Om\etab}\c\bbc+r^{-1}\Gab\c\Lc+r^{-1}\Gab\c\Gab.
\end{align*}
\end{proposition}
\begin{proof}
See Appendix \ref{equationsmu}.
\end{proof}
\subsubsection{The linearized Bianchi equations}
In this section, we introduce the linearized Bianchi equations, which play a fundamental role in curvature estimates.
\begin{proposition}\label{Bianchieq}
We have the following linearized Bianchi equations:
\begin{align*}
\nab_4\aac+\f12\tr\chi\,\aac
=&\nab\hot\bbc+\fl[\nab_4\aac]+\err[\nab_4\aac],\\
\fl[\nab_4\aac]:=&\Gab^{(1)}\c\bfO_3^1+\bfO_2^1\c\aac+\bfO_3^2\c(\bbc,\rhoc,\sic),\\
\err[\nab_4\aac]:=&\Gag\c(\aac,\bbc)+\Gab\c(\rhoc,\sic),\\ \\
\nab_{3}\bbc+2\tr\chib\,\bbc=&-\sdiv\aac+\fl[\nab_{3}\bbc]+\err[\nab_{3}\bbc],\\
\fl[\nab_{3}\bbc]:=&\Gab^{(1)}\c\bfO_4^2+\bfO_2^1\c\bbc+\bfO_3^2\c\aac,\\
\err[\nab_3\bbc]:=&\Gab\c\bbc+\Gag\c\aac,\\ \\
\nab_4\bbc+\trch\,\bbc=&-\nab\rhoc+{^*\nab}\sic+\fl[\nab_4\bbc]+\err[\nab_4\bbc],\\
\fl[\nab_4\bbc]:=&\Gag^{(1)}\c\bfO_3^1+\bfO_2^1\c\bbc+\bfO_3^2\c(\rhoc,\sic,\bc),\\
\err[\nab_4\bbc]:=&\Gag\c(\bbc,\rhoc,\sic)+\Gab\c\bc,\\ \\
\nab_3(\rhoc,\sic)+\frac{3}{2}\trchb\,(\rhoc,\sic)&=-d_1\bbc+\fl[\nab_3(\rhoc,\sic)]+\err[\nab_3(\rhoc,\sic)],\\
\fl[\nab_3(\rhoc,\sic)]:=&\Gag^{(1)}\c\bfO_3^1+\bfO_3^2\c(\aac,\bbc), \\
\err[\nab_3(\rhoc,\sic)]:=&\Gag\c(\aac,\bbc),\\ \\
\nab_4(\rhoc,-\sic)+\frac{3}{2}\trch(\rhoc,-\sic)
=&d_1\bc+\fl[\nab_4(\rhoc,-\sic)]+\err[\nab_4(\rhoc,-\sic)],\\
\fl[\nab_4(\rhoc,-\sic)]:=&\Gag^{(1)}\c\bfO_3^1+\bfO_3^2\c(\ac,\bc),\\
\err[\nab_4(\rhoc,-\sic)]:=&\Gab\c\ac+\Gag\c\bc,\\ \\
\nab_3\bc+\trchb\,\bc=&\nab\rhoc+{^*\nab\sic}+\fl[\nab_3\bc]+\err[\nab_3\bc],\\
\fl[\nab_3\bc]:=&\Gag^{(1)}\c\bfO_3^1+\bfO_2^1\c\bc+\bfO_3^2\c(\bbc,\rhoc,\sic),\\
\err[\nab_3\bc]:=&\Gab\c\bc+\Gag\c(\bbc,\rhoc,\sic),\\ \\
\nab_4\bc+2\trch\,\bc=&\sdiv\ac+\fl[\nab_4\bc]+\err[\nab_4\bc],\\
\fl[\nab_4\bc]:=&\Gag^{(1)}\c\bfO_4^2+\bfO_2^1\c(\bc,\ac),\\
\err[\nab_4\bc]:=&\Gag\c(\bc,\ac),\\ \\
\nab_3\ac+\f12\trchb\,\ac=&\nab\hot\bc+\fl[\nab_3\ac]+\err[\nab_3\ac],\\
\fl[\nab_3\ac]:=&\Gag^{(1)}\c\bfO_3^1+\bfO_2^1\c\ac+\bfO_3^2\c(\rhoc,\sic,\bc),\\
\err[\nab_3\ac]:=&\Gab\c\ac+\Gag\c\bc.
\end{align*}
\end{proposition}
\begin{proof}
See Appendix \ref{pfBianchi}.
\end{proof}
\begin{proposition}\label{Bianchieqdkb}
    We have the following linearized equations:\footnote{Recall that $\dkb$ is introduced in Definition \ref{dfdkb}.}
\begin{align*}
\nab_4\dkbaac+\f12\trchb\dkbaac=&-\nab\hot\dkbbbc+\fl[\nab_4\dkbaac]+\err[\nab_4\dkbaac],\\
\fl[\nab_4\dkbaac]=&\Gab^{(1)}\c\bfO_3^1+\bfO_2^1\c\aac^{(1)}+\bfO_3^2\c(\bbc^{(1)},\rhoc^{(1)},\sic^{(1)}),\\
\err[\nab_4\dkbaac]=&\Gag^{(1)}\c(\aac,\bbc)+\Gab^{(1)}\c(\rhoc,\sic)+\Gag\c(\aac^{(1)},\bbc^{(1)})+\Gab\c(\rhoc^{(1)},\sic^{(1)}),\\ \\
\nab_3\widecheck{\dkb\bb}+2\trchb\widecheck{\dkb\bb}=&-\sdiv\widecheck{\dkb\aa}+\fl[\nab_3\widecheck{\dkb\bb}]+\err[\nab_3\widecheck{\dkb\bb}],\\
\fl[\nab_3\widecheck{\dkb\bb}]:=&\Gab^{(1)}\c\bfO_4^2+\bfO_2^1\c\bbc^{(1)}+\bfO_3^2\c\aac^{(1)},\\
\err[\nab_3\widecheck{\dkb\bb}]:=&\Gab^{(1)}\c\bbc+\Gab\c\bbc^{(1)}+\Gag^{(1)}\c\aac+\Gag\c\aac^{(1)},\\\\
\nab_4\dkbbbc+\trch\dkbbbc:=&-\nab\dkbrhoc+{^*\nab}\dkbsic+\fl[\nab_4\dkbbbc]+\err[\nab_4\dkbbbc],\\
\fl[\nab_4\dkbbbc]:=&\Gag^{(1)}\c\bfO_3^1+\bfO_2^1\c\bbc^{(1)}+\bfO_3^2\c(\rhoc^{(1)},\sic^{(1)},\bc^{(1)}),\\
\err[\nab_4\dkbbbc]:=&\Gag^{(1)}\c(\bbc,\rhoc,\sic)+\Gag\c(\bbc^{(1)},\rhoc^{(1)},\sic^{(1)})+\Gab^{(1)}\c\bc+\Gab\c\bc^{(1)},\\ \\
\nab_3(\dkbrhoc,\dkbsic)+\frac{3}{2}\trchb(\dkbrhoc,\dkbsic)=&-d_1\dkb\bbc+\fl[\nab_3(\dkbrhoc,\dkbsic)]+\err[\nab_3(\dkbrhoc,\dkbsic)],\\
\fl[\nab_3(\dkbrhoc,\dkbsic)]:=&\Gab^{(1)}\c\bfO_3^1+\bfO_3^2\c(\aac^{(1)},\bbc^{(1)},\rhoc^{(1)},\sic^{(1)}),\\
\err[\nab_3(\dkbrhoc,\dkbsic)]:=&\Gag^{(1)}\c(\aac,\bbc)+\Gag\c(\aac^{(1)},\bbc^{(1)})+\Gab\c(\rhoc^{(1)},\sic^{(1)})+\Gab^{(1)}\c(\rhoc,\sic),\\\\
\nab_4(\dkbrhoc,-\dkbsic)+\frac{3}{2}\trch(\dkbrhoc,-\dkbsic)=&d_1\dkbbc+\fl[\nab_4(\dkbrhoc,-\dkbsic)]+\err[\nab_4(\dkbrhoc,-\dkbsic)],\\
\fl[\nab_4(\dkbrhoc,-\dkbsic)]:=&\Gag^{(1)}\c\bfO_3^1+\bfO_3^2\c(\ac^{(1)},\bc^{(1)},\rhoc^{(1)},\sic^{(1)}),\\
\err[\nab_4(\dkbrhoc,-\dkbsic)]:=&\Gab^{(1)}\c\ac+\Gab\c\ac^{(1)}+\Gag^{(1)}\c(\bc,\rhoc,\sic)+\Gag\c(\bc^{(1)},\rhoc^{(1)},\sic^{(1)}),\\ \\
\nab_3\dkb\bc+\trchb\,\dkb\bc=&\nab\dkb\rhoc+{^*\nab\dkb\sic}+\fl[\nab_3\dkbbc]+\err[\nab_3\dkbbc],\\
\fl[\nab_3\dkbbc]:=&\Gag^{(1)}\c\bfO_3^1+\bfO_2^1\c\bc^{(1)}+\bfO_3^2\c(\bbc^{(1)},\rhoc^{(1)},\sic^{(1)}),\\
\err[\nab_3\dkbbc]:=&\Gab^{(1)}\c\bc+\Gab\c\bc^{(1)}+\Gag^{(1)}\c(\bbc,\rhoc,\sic)+\Gag\c(\bbc^{(1)},\rhoc^{(1)},\sic^{(1)}),\\ \\
\nab_4\dkb\bc+2\trch\dkb\bc=&\sdiv\dkb\ac+\fl[\nab_4\dkbbc]+\err[\nab_4\dkbbc],\\
\fl[\nab_4\dkbbc]:=&\Gag^{(1)}\c\bfO_4^2+\bfO_2^1\c(\bc^{(1)},\ac^{(1)}),\\
\err[\nab_4\dkbbc]:=&\Gag^{(1)}\c(\bc,\ac)+\Gag\c(\bc^{(1)},\ac^{(1)}),\\ \\
\nab_3\dkb\ac+\f12\trchb\dkb\ac=&\nab\hot\dkb\bc+\fl[\nab_3\dkbac]+\err[\nab_3\dkbac],\\
\fl[\nab_3\dkbac]:=&\Gag^{(1)}\c\bfO_3^1+\bfO_2^1\c\ac^{(1)}+\bfO_3^2\c(\rhoc^{(1)},\sic^{(1)},\bc^{(1)}),\\
\err[\nab_3\dkbac]:=&\Gab^{(1)}\c\ac+\Gab\c\ac^{(1)}+\Gag^{(1)}\c\bc+\Gag\c\bc^{(1)}.
\end{align*}
\end{proposition}
\begin{remark}
    Notice that Proposition \ref{Bianchieqdkb} involves the quantities $\dkbac$, $\dkbbc$, $\dkbrhoc$, $\dkbsic$, $\dkbbbc$ and $\dkbaac$ instead of $\dkb\ac$, $\dkb\bc$, $\dkb\rhoc$, $\dkb\sic$, $\dkb\bbc$ and $\dkb\aac$. This allows to avoid a loss of a derivative of $\Jc$, $\Jbc$ and $\Lc$ in the linearization procedure. Indeed, we have for example
    \begin{equation*}
        \dkbbc=\dkb\bc+(\Jc,\Jbc,\Lc)\cdot \b_K,
    \end{equation*}
    and the norm $\OO_\Ga$ in Section \ref{onorms} only controls $(\Jc,\Jbc,\Lc)$ but not its first order derivatives.
\end{remark}
\begin{proof}
See Appendix \ref{pfBianchi}.
\end{proof}
\subsubsection{Linearized equations for the metric components}
\begin{proposition}\label{metriceqre}
We have the following transport equations:
\begin{align*}
        \nab_4(\gac_{AB})-\trch(\gac_{AB})=&\bfO_{-2}^0\c(\hchc,\trchc)+O(r)\Omc+r^3\Gag\c\Gag,\\
        \nab_4(\Omc)=&\bfO_0^0\c\widecheck{\Om\om}+\bfO_1^1\c\Gag+r\Gag\c\Gag,\\
        \nab_4(\bcc^A)=&\bfO_1^0\c\zec+\bfO_3^2\c\Gag+\Gag\c\Gag,\\
        \nab_4(\inc_{AB})-\trch(\inc_{AB})=&\bfO_{-2}^0\c\trchc+r^3\Gag\c\Gag.
\end{align*}
Moreover, we have the following transport equations in the case $\Omc=0$ and $\bcc=0$:\footnote{The conditions $\Omc=0$ and $\bcc=0$ hold on $\Cb_*$, see Definition \ref{geodesicfoliation}.}
\begin{align*}
    \nab_3(\gac_{AB})-\trchb(\gac_{AB})=&\bfO_{-2}^0\c(\widecheck{\trchb},\hchbc)+\bfO_3^2\c\gac_{AB},\\
    \nab_3(\inc_{AB})-\trchb(\inc_{AB})=&\bfO_{-2}^0\c\trchbc+\O^2_3\c\inc_{AB}.
\end{align*}
\end{proposition}
\begin{proof}
See Appendix \ref{pfmetric}.
\end{proof}
\subsection{Commutator identities}
We recall the following commutator identities for scalar functions, see Proposition 4.8.1 of \cite{Kl-Ni}.
\begin{proposition}\label{commkn}
For a scalar function $f$, we have
\begin{align*}
    [\Om\nab_4,\nabla]f&=-\Om\chi\cdot\nabla f,\\
    [\Om\nab_3,\nabla]f&=-\Om\chib\cdot\nabla f,\\
    [\nabla_4,\nabla_3]f&=2\omega\nabla_3 f -2\omb\nabla_4 f-4\zeta\cdot\nabla f.
\end{align*}
\end{proposition}
We also need the commutation identities for more general tensor fields. For this, we record the following commutation lemma.
\begin{lemma}\label{commck}
Let $U_{a_1...a_k}$ be an $S$-tangent $k$-covariant tensor on $(\MM,\g)$. Then
\begin{align*}
    [\nabla_4,\nabla_b]U_{a_1...a_k}&=-\chi_{BC}\nabla_c U_{a_1...a_k}+F_{4ba_1...a_k},\\
    F_{4ba_1...a_k}&:=(\zeta_b+\etab_b)\nabla_4U_{a_1...a_k}+\sum_{i=1}^k (\chi_{a_ib}\,\etab_c-\chi_{BC}\,\etab_{a_i}+\in_{a_ic}\,{^*\beta}_b)U_{a_1...c...a_k},\\
    [\nabla_3,\nabla_b]U_{a_1...a_k}&=-\chib_{BC}\nabla_c U_{a_1...a_k}+F_{3ba_1...a_k},\\
    F_{3ba_1...a_k}&:=(\eta_b-\zeta_b)\nabla_3U_{a_1...a_k}+\sum_{i=1}^k (\chib_{a_ib}\,\eta_c-\chib_{BC}\,\eta_{a_i}+\in_{a_ic}\,{^*\bb}_b)U_{a_1...c...a_k},\\
    [\nabla_3,\nabla_4]U_{a_1...a_k}&=F_{34a_1...a_k},\\
    F_{34a_1...a_k}&:=-2\omega\nabla_3 U+2\omb\nabla_4 U+4\zeta_b\nabla_b U_{a_1...a_k}+2\sum_{i=1}^k(\etab_{a_i}\,\eta_c-\etab_{a_i}\,\eta_c+\in_{a_ic}\,\sigma)U_{a_1...c...a_k}.
\end{align*}
\end{lemma}
\begin{proof}
It is a direct consequence of Lemma 7.3.3 in \cite{ch-kl}.
\end{proof}
\begin{corollary}\label{commutation}
We have the following schematic commutation identities:
\begin{align*}
[r\nabla,\Om\nab_4]&=\Gag\c r\nab+\bfO_3^2\c r\nab+\Gag^{(1)}+\bfO_3^2,\\
[r\nabla,\Om\nabla_3]&=\Gab\c r\nab+\bfO_3^2\c r\nab+\Gab^{(1)}+\bfO_3^2,\\
[\nab,\Om\nab_4]&=-\frac{1}{2}\Om\trch\nab+\Gag\cdot\nab+\bfO_3^2\c\nab+r^{-1}\Gag^{(1)}+\bfO_4^2,\\
[\nab,\Om\nab_3]&=-\frac{1}{2}\Om\trchb\nab+\Gab\cdot\nab+\bfO_3^2\c\nab+r^{-1}\Gab^{(1)}+\bfO_4^2,\\
[\nab_4,\nab_3]&=\Gag\c\nab_3+\Gab\c(\nab_4,\nab)+\bfO_2^1\c(\nab_3,\nab_4,\nab)+r^{-1}\Gag^{(1)}+\bfO_4^2,\\
[\Om{}\nab_4,\Om\nab_3]&=\Gab\c\nab+\bfO_3^2\c\nab+r^{-1}\Gag^{(1)}+\O_4^2.
\end{align*}
\end{corollary}
\begin{proof}
    It follows directly from Lemmas \ref{commck} and \ref{expressionGaR} together with the identities \eqref{6.6}.
\end{proof}
\subsection{Sobolev inequalities}
\begin{definition}\label{L2flux}
{For a scalar function $h$ on a null cone $C\in\{C_u,\Cb_\ub\}$, we define its integral on $C$ as follows:\footnote{The factor $\Om^2$ has been added since we have from Section \ref{doublenullsection} that $\pr_\ub=2\Om^2L$ and $\pr_u+\underline{b}^A\pr_A=2\Om^2\Lb$, where $L$ and $\Lb$ denote the generators of $C_u$ and $\Cb_\ub$ respectively.}
\begin{align*}
    \int_{C_u} h:=\int d\ub\int_{S(u,\ub)}\Om^2h,\qquad \int_{\Cb_\ub} h:=\int du\int_{S(u,\ub)}\Om^2h.
\end{align*}}
For a tensor field $h$ on a null cone $C$, we define its $L^2$--flux:
\begin{equation}
    \|h\|_{2,C}:= \left(\int_C |h|^2 \right)^\frac{1}{2}.
\end{equation}
\end{definition}
We recall the following Sobolev inequalities.
\begin{proposition}\label{sobolevkn}
Let $F$ be a tensor field tangent to $S:=S(u,\ub)$ at every point. We have the following estimates:
\begin{align*}
    |rF|_{4,S}&\les{|rF|_{4,S(u,\ub_0(u))}+}\|F\|_{2,C_u\cap V(u,\ub)}+\|r\nab F\|_{2,C_u\cap V(u,\ub)}+\|r\nab_4F\|_{2,C_u\cap V(u,\ub)},\\
    |r^\frac{1}{2}|u|^\frac{1}{2}F|_{4,S}&\les{|rF|_{4,S(u_0(\ub),\ub)}+}\|F\|_{2,\Cb_\ub\cap V(u,\ub)}+\|r\nab F\|_{2,\Cb_\ub\cap V(u,\ub)}+|u|\|\nab_3F\|_{2,\Cb_\ub\cap V(u,\ub)},
\end{align*}
{where $S(u,\ub_0(u))$ (resp. $S(u_0(\ub),\ub)$) denotes the unique leaf of $C_u$ (resp. $\Cb_\ub$), which located in the future of $\Si_0$ and touches $\Si_0\sm K$.}
\end{proposition}
\begin{proof}
See Corollary 3.2.1.1 in \cite{ch-kl} and Corollary 4.1.1 in \cite{Kl-Ni}.
\end{proof}
We also need the following standard Sobolev inequalities, see Lemma 4.1.3 in \cite{Kl-Ni}.
\begin{proposition}\label{standardsobolev}
Let $F$ be a tensor field tangent to the sphere $S:=S(u,\ub)$. Then, we have
\begin{equation*}
\sup_{S} r^{\frac{1}{2}}|F|\les\left(\int_S |F|^4 +|r\nabla F|^4\right)^{\frac{1}{4}}.
\end{equation*}
\end{proposition}
\section{Main theorem}\label{mainsection}
\subsection{Fundamental norms}\label{fundamentalnorms}
Our result holds for initial data sets satisfying \eqref{sperturbation} with $s>3$. We will focus on the case $s\in[4,6]$ in Sections \ref{mainsection}--\ref{initiallast}, and we postpone the necessary adaptations to the case $s\in(3,4)$ and $s>6$ to Appendices \ref{secc} and \ref{secd}. The norms in Sections \ref{Rnorms}--\ref{Ostar} are defined with respect to the foliation $(u,\ub)$ in the bootstrap region $\kk$ while the norms in Section \ref{'norms} are defined with respect to the foliation $(u_{(0)},\ub)$ in the initial data layer $\kk_{(0)}$.
\subsubsection{\texorpdfstring{$\RR$}{} norms (Curvature components)}\label{Rnorms}
We define
\begin{equation*}
\RR_0^{S}:=\RR_0^S[\bc],\qquad\RRb_0^{S}:=\RRb_0^{S}[\bc]+\RRb_0^{S}[\rhoc,\sic]+\RRb_0^{S}[\bbc],
\end{equation*}
where
\begin{align*}
    \RR_0^S[\bc]&:=\sup_\KK\sup_{p\in[2,4]}|r^{\frac{7}{2}-\frac{2}{p}}|u|^{\frac{s-4}{2}}\bc|_{p,S(u,\ub)},\\
    \RRb_0^S[\bc]&:=\sup_\KK\sup_{p\in[2,4]}|r^{\frac{s+2}{2}-\frac{2}{p}}|u|^{\f12}\bc|_{p,S(u,\ub)},\\
    \RRb_0^S[\rhoc,\sic]&:=\sup_\KK\sup_{p\in[2,4]}|r^{3-\frac{2}{p}}|u|^{\frac{s-3}{2}} (\rhoc,\sic) |_{p,S(u,\ub)},\\
    \RRb_0^S[\bbc]&:=\sup_\KK\sup_{p\in[2,4]}|r^{2-\frac{2}{p}}|u|^{\frac{s-1}{2}}\bbc |_{p,S(u,\ub)}.
\end{align*}
Then we define the norms of flux of curvature components. We denote
\begin{equation*}
    \RR:=\sum_{q=0}^1(\RR_{q}+\RRb_{q})+\RR_0^S+\RRb_0^S,
\end{equation*}
where
\begin{align*}
    \RR_{q}&:=\left(\RR_{q}[\ac]^2+\RR_{q}[\bc]^2 +\RR_{q}[(\rhoc,\sic)]^2 + \RR_{q}[\bbc]^2 \right)^{\f12} ,\\
        \RRb_{q}&:=\left(\RRb_{q}[\bc]^2+\RRb_{q}[(\rhoc,\sic)]^2+\RRb_{q}[\bbc]^2 +
        \RRb_{q}[\aac]^2\right)^{\f12},
\end{align*}
with
\begin{align*}
    \mathcal{R}_{q}[w]&:= \sup_{\mathcal{K}} \mathcal{R}_{q} [w](u,\ub), \\
    \underline{\mathcal{R}}_{q} [w]&:= \sup_{\mathcal{K}} \underline{\mathcal{R}}_{q} [w](u,\ub).
\end{align*}
We denote
\begin{equation}
   V:=V(u,\ub),\qquad C_u^V:=C_u\cap V,\qquad\Cb_\ub^V:=\Cb_\ub\cap V.
\end{equation}
It remains to define $\mathcal{R}_{q}[w](u,\ub)$ and $\underline{\mathcal{R}}_{q}[w](u,\ub)$:
\begin{align*}
    \RR_{q}[\ac](u,\ub)&:= \Vert r^{\frac{s}{2}} (r\nab)^q\ac\Vert_{2,\cuv},\\
    \RR_{q}[\bc](u,\ub)&:= \Vert r^{2}|u|^{\frac{s-4}{2}}(r\nab)^q\bc\Vert_{2,\cuv},\\
    \RR_{q}[(\rhoc,\sic)](u,\ub)&:=\Vert r|u|^{\frac{s-2}{2}} (r\nab)^q(\rhoc,\sic) \Vert_{2,\cuv},\\
    \RR_{q}[\bbc](u,\ub)&:= \Vert |u|^{\frac{s}{2}} (r\nab)^q\bbc \Vert_{2,\cuv} ,\\
    \RRb_{q}[\bc](u,\ub)&:= \Vert r^{\frac{s}{2}} (r\nab)^q\bc \Vert_{2, \ucuv},\\
    \RRb_{q}[(\rhoc,\sic)](u,\ub)&:=|u|^{\frac{s-4}{2}}\Vert {r^2}(r\nab)^q(\rhoc,\sic) \Vert_{2,\ucuv},\\
    \RRb_{q}[\bbc](u,\ub)&:=|u|^{\frac{s-2}{2}}\Vert {r}(r\nab)^q\bbc\Vert_{2,\ucuv},\\
    \RRb_{q}[\aac](u,\ub)&:= |u|^{\frac{s}{2}}\Vert(r\nab)^q\aac\Vert_{2,\ucuv}.
\end{align*}
\subsubsection{\texorpdfstring{$\OO$}{} norms (Ricci coefficients and metric components)} \label{onorms}
We denote for any linearized quantity:
\begin{equation}\label{defOOq}
    \OO_q(\widecheck{X}):=\sup_{\KK}\sup_{p\in [2,4]}\OO_q^{p,S}(\widecheck{X})(u,\ub).
\end{equation}
We first define
\begin{align*}
\OO_{[1]}:=\OO_1+\OO_0+\OOb_1+\OOb_0+\sum_{q=0}^2\OO_q(\Omc),
\end{align*}
where
\begin{align*}
\OO_q&:=\OO_q(\widecheck{\Om\trch})+\OO_q(\widecheck{\hch})+\OO_q(\etac)+\OO_q(\widecheck{\Om\omb}),\qquad\, q=0,1,\\
\OOb_q&:=\OO_q(\widecheck{\Om\trchb})+\OO_q(\widecheck{\chibh})+\OO_q(\etabc)+ \OO_q(\widecheck{\Om\om}),\qquad q=0,1.
\end{align*}
It remains to define $\OO_q^{p,S}(\widecheck{X})(u,\ub)$:
\begin{align*}
    \OO_q^{p,S}(\widecheck{\Om\trch})(u,\ub)&:=|r^{2+q-\frac{2}{p}}|u|^{\frac{s-3}{2}}\nab^q(\widecheck{\Om\trch})|_{p,S(u,\ub)} ,\\
    \OO_q^{p,S}(\widecheck{\Om\trchb})(u,\ub)&:=|r^{2+q-\frac{2}{p}}|u|^{\frac{s-3}{2}} \nab^q(\widecheck{\Om\trchb})|_{p,S(u,\ub)}, \\
  \OO_q^{p,S}(\hchc)(u,\ub)&:=|r^{2+q-\frac{2}{p}}|u|^{\frac{s-3}{2}}\nab^q \hchc|_{p,S(u,\ub)}, \\
     \OO_q^{p,S}(\hchbc)(u,\ub)&:=|r^{1+q-\frac{2}{p}}|u|^{\frac{s-1}{2}}\nabla^q\hchbc|_{p,S(u,\ub)}, \\
    \OO_q^{p,S}(\etac)(u,\ub)&:=|r^{2+q-\frac{2}{p}}|u|^{\frac{s-3}{2}}\nabla^q \etac|_{p,S(u,\ub)}, \\
    \OO_q^{p,S}(\etabc)(u,\ub)&:=|r^{2+q-\frac{2}{p}}|u|^{\frac{s-3}{2}}\nabla^q \etabc|_{p,S(u,\ub)}, \\
 \OO_q^{p,S}(\widecheck{\Om\om})(u,\ub)&:=|r^{2+q-\frac{2}{p}}|u|^\frac{s-3}{2}\nabla^q\widecheck{\Om\om}|_{p,S(u,\ub)}, \\
\OO_q^{p,S}(\widecheck{\Om\omb})(u,\ub)&:=|r^{1+q-\frac{2}{p}}|u|^\frac{s-1}{2}\nabla^q\widecheck{\Om\omb}|_{p,S(u,\ub)},\\
\OO_q^{p,S}(\Omc)(u,\ub)&:=|r^{1+q-\frac{2}{p}}|u|^{\frac{s-3}{2}}\nab^q\Omc|_{p,S(u,\ub)}.
\end{align*}
Next, we define the norms of metric components:
\begin{align*}
\OO_{\ga}:=&\sum_{q=0}^1\left(\OO_q(\gac)+\OO_q(\inc)+\OO_q(\bcc)\right),
\end{align*}
where
\begin{align*}
    \OO_q^{p,S}(\gac)(u,\ub)&:=|r^{-1+q-\frac{2}{p}}|u|^{\frac{s-3}{2}}\nab^q\gac_{AB}|_{p,S(u,\ub)},\\
\OO_q^{p,S}(\inc)(u,\ub)&:=|r^{-1+q-\frac{2}{p}}|u|^{\frac{s-3}{2}}\nab^q\inc_{AB}|_{p,S(u,\ub)},\\
\OO_q^{p,S}(\bcc)(u,\ub)&:=|r^{2+q-\frac{2}{p}}|u|^{\frac{s-3}{2}}\nab^q\,\bcc^A|_{p,S(u,\ub)}.
\end{align*}
We also define
\begin{align*}
\OO_\bGa:=\OO_0(\Jc)+\OO_0(\Jbc)+\OO_0(\Lc),
\end{align*}
where
\begin{align*}
\OO_0^{p,S}(\Lc)(u,\ub)&:=|r^{2-\frac{2}{p}}|u|^{\frac{s-3}{2}}\Lc|_{p,S(u,\ub)},\\
\OO_0^{p,S}(\Jc)(u,\ub)&:=|r^{2-\frac{2}{p}}|u|^{\frac{s-3}{2}}\Jc|_{p,S(u,\ub)},\\
\OO_0^{p,S}(\Jbc)(u,\ub)&:=|r^{1-\frac{2}{p}}|u|^{\frac{s-1}{2}}\Jbc|_{p,S(u,\ub)}.
\end{align*}
Finally, we denote
\begin{equation}
    \OO:= \OO_{[1]}+\OO_{\ga}+\OO_{\bGa}.
\end{equation}
\subsubsection{\texorpdfstring{$\okk(\Si_0)$}{} norms and \texorpdfstring{$\Rk_0$}{} norms (Initial data for the foliation of \texorpdfstring{$\kk$}{})}\label{initialnorms}
Notice that for every $\ub$, there exists a unique leaf $S(u_0(\ub),\ub)$ of $\Cb_\ub$, which located in the future of $\Si_0$ and touches $\Sigma_0\setminus K$. Moreover, we have
\begin{equation*}
    S(u_0(\ub),\ub)\subset \kk_{(0)}.
\end{equation*}
We define the following norms on the union of spheres $S(u_0(\ub),\ub)$:
\begin{align*}
    \okk_q(\Gac)&:=\sup_{p\in[2,4]}\sup_{\Si_0\setminus K}{\okk}_q^{p,S}(\Gac)(\ub),\quad q=0,1.
\end{align*}
We denote
\begin{align*}
    \okk(\Si_0):=\okk_{1}+\okk_0,
\end{align*}
where
\begin{align*}
\okk_q&:=\okk_q(\widecheck{\Om\trchb})+\okk_q(\Omomc)+\okk_q(\etabc),\qquad q=0,1.
\end{align*}
It remains to define ${\okk}_q^{p,S}(\Gac)(\ub)$:
\begin{align*}
    \okk_q^{p,S}(\widecheck{\Om\trchb})(\ub)&:=|{r}^{\frac{s+1}{2}+q-\frac{2}{p}}\nab^q(\widecheck{\Om\trchb})|_{p,S(u_0(\ub),\ub)},\\
    \okk_q^{p,S}(\etabc)(\ub)&:=|{r}^{\frac{s+1}{2}+q-\frac{2}{p}}\nab^q\etabc|_{p,S(u_0(\ub),\ub)},\\
    \okk_q^{p,S}(\Omomc)(\ub)&:=|r^{\frac{s+1}{2}+q-\frac{2}{p}}{\nab}^q(\Omomc)|_{p,S(u_0(\ub),\ub)}.
\end{align*}
We can extend the foliation $(u,\ub)$ to a neighborhood of $\Si_0\sm K$ in $J^-(\Si_0\sm K)$ such that it is well defined on $\Si_0\sm K$. The curvature flux on $\Si_0\sm K$ is defined by:
\begin{equation*}
    \Rk_0^2:=\int_{\Si_0\setminus K}\sum_{l=0}^1 r^{s}\left(|\dd^l\ac|^2+|\dd^l\bc|^2+|\dd^l(\rhoc,\sic)|^2+|\dd^l\bbc|^2+|\dd^l \aac|^2\right),
\end{equation*}
where
\begin{equation*}
    \dd:=(r\nabla,r\nabla_4,r\nabla_3).
\end{equation*}
\subsubsection{\texorpdfstring{$\OO^*(\Cb_*)$}{} norms (Ricci coefficients and metric components on \texorpdfstring{$\Cb_*$}{})}\label{Ostar}
For any linearized quantity $\Gac$ on $\Cb_*$, we define
\begin{equation*}
    \OO_q^{*}(\Gac):=\sup_{p\in [2,4]}\sup_{\Cb_*}\OO_q^{*p,S}(\Gac)(u).
\end{equation*}
We denote
\begin{equation*}
    \OO^{*}(\Cb_*):=\OO^{*}_1+\OO_0^*+\OO_\ga^*+\OO_\Ga^*+\OO^*_2(\trchbc)+\OO_1^*(\mubmc),
\end{equation*}
where
\begin{align*}
\OO^*_{\ga}:=&\sum_{q=0}^2\left(\OO_q^{*}(\gac)+\OO_q^{*}(\inc)\right),\\
\OO^*_\Ga:=&\sum_{q=0}^1\left(\OO_q^{*}(\Lc)+\OO_q^{*}(\Jbc)\right),\\
\OO^*_q:=&\OO_q^{*}(\trchc)+\OO_q^{*}(\hchc)+\OO_q^{*}(\zec)+\OO_q^{*}(\trchbc)+\OO_q^{*}(\hchbc),\qquad q=0,1.
\end{align*}
It remains to define $\OO_q^{*p,S}(\Gac)(u)$ for $q=0,1,2$:
\begin{align*}
        \OO_q^{*p,S}(\hchc)(u)&:=|r^{2+q-\frac{2}{p}}|u|^{\frac{s-3}{2}}\nab^q\hchc|_{p,S(u,\ub_*)},\\
        \OO_q^{*p,S}(\hchbc)(u)&:=|r^{1+q-\frac{2}{p}}|u|^{\frac{s-1}{2}}\nab^q\hchbc|_{p,S(u,\ub_*)},\\
        \OO_q^{*p,S}(\trchc)(u)&:=|r^{2+q-\frac{2}{p}}|u|^{\frac{s-3}{2}}\nab^q\trchc|_{p,S(u,\ub_*)},\\
        \OO_q^{*p,S}(\trchbc)(u)&:=|r^{2+q-\frac{2}{p}}|u|^{\frac{s-3}{2}}\nab^q\trchbc|_{p,S(u,\ub_*)},\\
        \OO_q^{*p,S}(\zec)(u)&:=|r^{2+q-\frac{2}{p}}|u|^{\frac{s-3}{2}}\nab^q\etac|_{p,S(u,\ub_*)},\\
        \OO_q^{*p,S}(\gac)(u)&:=|r^{1+q-\frac{2}{p}}|u|^{\frac{s-3}{2}}\nab^q\gac|_{p,S(u,\ub_*)},\\
        \OO_q^{*p,S}(\inc)(u)&:=|r^{1+q-\frac{2}{p}}|u|^{\frac{s-3}{2}}\nab^q\inc|_{p,S(u,\ub_*)},\\
        \OO_q^{*p,S}(\Lc)(u)&:=|r^{2+q-\frac{2}{p}}|u|^{\frac{s-3}{2}}\nab^q\Lc|_{p,S(u,\ub_*)},\\
        \OO_q^{*p,S}(\Jbc)(u)&:=|r^{1+q-\frac{2}{p}}|u|^{\frac{s-1}{2}}\nab^q\Jbc|_{p,S(u,\ub_*)},\\
        \OO_q^{*p,S}(\mubmc)(u)&:=|r^{3+q-\frac{2}{p}}|u|^{\frac{s-3}{2}}\nab^q\mubmc|_{p,S(u,\ub_*)}.
\end{align*}
We also define the flux of curvature components on $\Cb_*$ as follows:
\begin{align*}
    \RRb(\Cb_*):=\RRb_{1}\big|_{\Cb_*}+\RRb_{0}\big|_{\Cb_*},
\end{align*}
where $\RRb_{q}$, $q=0,1$ are defined in Section \ref{Rnorms}.
\subsubsection{\texorpdfstring{$\OO_{(0)}$}{} norms and \texorpdfstring{$\Rk_{(0)}$}{} norms (Initial data for the foliation of \texorpdfstring{$\KK_{(0)}$}{})}\label{'norms}
In this section, all the norms are defined with respect to the initial layer foliation $(u_{(0)},\ub)$ of $\kk_{(0)}$. We define
\begin{align*}
    \OO_{(0)}:=\sum_{q=0}^2\sum_{\Gac_{(0)}}\OOO_q(\Gac_{(0)}),
\end{align*}
where
\begin{equation*}
    \OOO_q(\Gac_{(0)}):=\sup_{p\in[2,4]}\sup_{\kk_{(0)}}\OOO_q^{p,S}(\Gac_{(0)}),
\end{equation*}
with
\begin{align*}
\OOO_q^{p,S}(\Gac_{(0)}):=\left|r_{(0)}^{\frac{s+1}{2}-\frac{2}{p}}(r_{(0)}\nab_{(0)})^q\Gac_{(0)}\right|_{p,S_{(0)}(u_{(0)},\ub)},
\end{align*}
and
\begin{equation*}
    \Gac_{(0)}\in\left\{\trchc_{(0)},\trchbc_{(0)},\hchc_{(0)},\hchbc_{(0)},\etac_{(0)},\etabc_{(0)},\omc_{(0)},\ombc_{(0)},r_{(0)}^{-1}\gac_{(0)},r_{(0)}^{-1}\inc_{(0)},r^{-1}_{(0)}\bcc_{(0)},r^{-1}_{(0)}\Omc_{(0)}\right\}.
\end{equation*}
The initial curvature flux on $\Si_0\sm K$ is defined by:
\begin{equation*}
    \Rk_{(0)}^2:=\int_{\Si_0\setminus K}\sum_{l=0}^2 r_{(0)}^{s}\left(|\dd_{(0)}^l\ac_{(0)}|^2+|\dd_{(0)}^l\bc_{(0)}|^2+|\dd_{(0)}^l(\rhoc_{(0)},\sic_{(0)})|^2+|\dd_{(0)}^l\bbc_{(0)}|^2+|\dd_{(0)}^l\aac_{(0)}|^2\right),
\end{equation*}
where
\begin{equation*}
    \dd_{(0)}:=(r_{(0)}\nab_{(0)},r_{(0)}(\nab_{(0)})_4,r_{(0)}(\nab_{(0)})_3).
\end{equation*}
\subsubsection{\texorpdfstring{$\osc$}{} norms (Oscillation control)}\label{Oscnorms}
We have two foliations $(u,\ub)$ and $(u_{(0)},\ub)$ in the initial layer region $\KK_{(0)}$ which will be compared using $\osc$ norms. We denote $(f,\la)$ the change of frame from $(u,\ub)$ to $(u_{(0)},\ub)$ introduced in Lemma \ref{changelemma}. We define
\begin{equation*}
    \osc:=\osc(f)+\osc(u)+\osc(x)+\osc(r),
\end{equation*}
where
\begin{align*}
    \osc(f)&:=\sup_{\kk_{(0)}}|(r_{(0)})^{\frac{s-1}{2}}\dk_{(0)}^{\leq 1}f|,\\
    \osc(u)&:=\sup_{\kk_{(0)}}|(r_{(0)})^{\frac{s-3}{2}}(u'-u)|,\\
    \osc(x)&:=\sup_{\kk_{(0)}}\sum_{A=1}^2|(r_{(0)})^{\frac{s-1}{2}}(x_{(0)}^A-x^A)|,\\
    \osc(r)&:=\sup_{\kk_{(0)}}|(r_{(0)})^{\frac{s-3}{2}} (r_{(0)}-r)|,
\end{align*}
with 
\begin{equation*}
    \dk_{(0)}:=\big\{r_{(0)}(\nab_{(0)})_3,r_{(0)}\nab_{(0)},r_{(0)}(\nab_{(0)})_4\big\}.
\end{equation*}
\subsection{Smallness constants}
Before starting our main theorem, we first introduce the following constants that will be used throughout this paper.
\begin{itemize}
    \item The constant $M>0$ and the constant $a$ are the mass and angular momentum of the initial Kerr spacetime relative to which our initial perturbation is measured.
    \item $R_0>0$ measures the radius of the compact set $K\subset\Si_0$ relative to which the external region is defined by $\MM:=D^+(\Si_0\sm K)$, the future domain of dependence of $\Si_0\sm K$.
    \item The size of the initial data layer norm is measured by $\ep_0>0$.
    \item The size of the bootstrap assumption norms are measured by $\ep>0$.
    \item $\de_0>0$ measures the height of the initial layer region $\KK_{(0)}$.
\end{itemize}
In what follows $M$ and $a$ are fixed constants such that $0\leq |a|\leq M$. We define
\begin{equation}
    \EE_0:=\frac{M}{R_0},\qquad \ep:=\EE_0^{-\frac{1}{3}}\ep_0.
\end{equation}
We take $\ep_0$ small enough and $R_0$ large enough such that
\begin{equation}\label{epep0}
    \ep^2\ll \ep\EE_0 \ll \ep\EE_0^\f12 \ll \ep_0 \ll \ep \ll \EE_0\ll \de_0<1.
\end{equation}
\begin{remark}
In \eqref{epep0}, $A\ll B$ means that $CA<B$ where $C$ is the largest universal constant among all the constants involved in the proof via $\les$.
\end{remark} 
\subsection{Main theorem}\label{secmain}
The goal of this paper is to prove the following theorem, which is similar to Theorem 4.1 in \cite{Caciotta}.
\begin{theorem}[Main theorem, version 2]\label{maintheorem}
Let $\g_{a,M}$ a Kerr metric in $\MM_{Kerr}$ with $|a|\leq M$. Let 
$\kk_{(0)}$ the initial layer region and $\Phi_{(0)}$ the map defined in Section \ref{initiallayer}. We also assume that\footnote{The norms $\OO_{(0)}$ and $\Rk_{(0)}$ are defined in Section \ref{'norms}.}
\begin{align}
\begin{split}\label{l2condition}
    \OO_{(0)}\leq\ep_0,\qquad\Rk_{(0)}\leq\ep_0.
\end{split}
\end{align}
Then, $\KK_{(0)}$ has a unique future development $(\MM,\g)$, defined {inside its future domain of dependence}, with the following properties:
\begin{enumerate}
    \item $(\MM,\g)$ can be foliated by a double null foliation $(u,\ub)$. Moreover, the outgoing cones $C_u$ are complete for all $u\leq u_0$.
    \item The norms $\OO$, $\RR$ satisfy
    \begin{equation}\label{finalest}
    \OO\les\ep_0,\qquad\RR\les\ep_0.
    \end{equation}
\end{enumerate}
\end{theorem}
\begin{remark}
    {We denote $(u_{\ub_*},\ub_{\ub_*})$ the double null foliation associated to the bootstrap region $V(u_0,\ub_*)$. The bounds \eqref{finalest} allow us to show that $(u_{\ub_*},\ub_{\ub_*})$ converges to $(u_\infty,\ub_\infty)$\footnote{Here, $(u_\infty,\ub_\infty)$ denotes the double null foliation of final spacetime.} as $\ub_{*}\to\infty$ with all the bootstrap bounds valid on $V(u_0, +\infty)$ for the quantities associated to the double null foliation $(u_\infty, \ub_\infty)$.}
\end{remark}
The proof of Theorem \ref{maintheorem} is given in Section \ref{proofmain}. It hinges on a sequence of basic theorems stated in Section \ref{Mnumber}, concerning estimates for $\OO$ and $\RR$ norms.
\begin{remark}
Theorem \ref{maintheorem} is proved in Section \ref{proofmain} for $s\in [4,6]$, and extended to $s\in(3,4)$ and $s>6$ in Appendices \ref{secc} and \ref{secd}.
\end{remark}
\subsection{Intermediate results}\label{Mnumber}
\begin{thmM0}
Assume that
\begin{equation}\label{assum13}
    \OO_{(0)}\leq\ep_0,\qquad \Rk_{(0)}\leq\ep_0,\qquad\OO\leq\ep,\qquad \osc\leq\ep.
\end{equation}
Then, we have
\begin{equation}\label{con13}
\Rk_0\les\ep_0.
\end{equation}
\end{thmM0}
Theorem M0 is proved in Section \ref{osc}. The proof is based on null frame transformation formulae introduced by Klainerman and Szeftel in \cite{KS:Kerr1}, see Proposition \ref{transformation} in Section \ref{secnullframe}.
\begin{thmM1}
Assume that
\begin{equation}
    \OO_{(0)}\leq\ep_0,\qquad \Rk_0\leq\ep_0,\qquad \OO\leq\ep,\qquad \RR\leq\ep,\qquad \osc\leq\ep.
\end{equation}
Then, we have
\begin{equation}
    \RR\les \ep_0.
\end{equation}
\end{thmM1}
Theorem M1 is proved in Section \ref{curvatureestimates}. The proof is based on the $r^p$--weighted estimate method introduced by Dafermos and Rodnianski in \cite{Da}.
\begin{thmM2}
Assume that $\Cb_*$ is foliated by the geodesic type foliation of Definition \ref{geodesicfoliation} and that we have the following assumptions:
\begin{align*}
     \OO_{(0)}(\Cb_*\cap\Si_0)\leq\ep_0,\qquad\RRb(\Cb_*)\les\ep_0,\qquad {\OO}^{*}(\Cb_*)\leq\ep.
\end{align*}
Then, we have
\begin{align*}
     \OO^*(\Cb_*)\les\ep_0.
\end{align*}
\end{thmM2}
Theorem M2 is proved in Section \ref{lastslice}. The proof is done by integrating the transport equations along the null generator $\Lb$ of the last slice $\Cb_*$ and applying elliptic estimates on the $2$--spheres of the geodesic type foliation of $\Cb_*$.
\begin{thmM3}
Assume that
\begin{align}\label{assM313}
        \OO_{(0)}\leq\ep_0,\qquad\RR\leq\ep_0,\qquad\OO^*(\Cb_*)\leq \ep_0, \qquad\OO\leq\ep, \qquad\osc\leq\ep.
\end{align}
Then, we have
\begin{equation*}
    \OO\les\ep_0,\qquad \osc\les\ep_0.
\end{equation*}
\end{thmM3}
Theorem M3 is proved in Section \ref{Ricciestimates}. The proof is done by integrating the transport equations along the outgoing and ingoing null cones and applying elliptic estimate on the $2$--spheres of the double null foliation of the spacetime $\kk$.
\begin{thmM4}
We consider the spacetime $\KK$ and its double null foliation $(u,\ub)$ which satisfies the assumptions:
\begin{equation}
\OO\les\ep_0,\qquad \RR\les \ep_0,\qquad \osc\les\ep_0,\qquad \OO^*(\Cb_*)\les\ep_0.
\end{equation}
Then, we can extend the spacetime $\KK=V(u_0,\ub_*)$ and the double null foliation to a new spacetime $\wideparen{\KK}=\wideparen{V}(u_0,\ub_*+\nu)$, where $\nu>0$ is sufficiently small, and an associated double null foliation $(\upa,\ub)$. Moreover, the new foliation $(\upa,\ub)$ is geodesic on the new last slice $\Cb_{**}:=\Cb_{\ub_*+\nu}$ and the new corresponding norms satisfy
\begin{align}
    \wideparen{\OO}\les\ep_0,\qquad\wideparen{\RR}\les\ep_0,\qquad \wideparen{\osc}\les\ep_0,\qquad \wideparen{\OO}^*(\Cb_{**})\les\ep_0.
\end{align}
\end{thmM4}
Theorem M4 is proved in Section \ref{ext}. The proof is based on local existence type arguments.
\begin{remark}
When compared to \cite{Caciotta}, we have the following similarities and differences:
\begin{enumerate}
\item We treat the general decay $s>3$ in Theorem \ref{maintheorem} while the main result of \cite{Caciotta} corresponds to the particular case $s>7$ in Theorem \ref{maintheorem}.
\item In \cite{Caciotta}, fourth order derivatives of Ricci coefficients and third order derivatives of curvature are estimated. In this paper, we only estimate first order derivatives of both Ricci coefficients and curvature respectively in Theorems M2--M3 and M1.
\item To estimate the curvature components in Theorem M1, instead of using the classcial vectorfield method as in \cite{Caciotta}, we use $r^p$--weighted estimates of \cite{Da}.
\item On the last slice $\Cb_*$, a canonical foliation is used in \cite{Caciotta} while we use a geodesic type foliation in Theorem M2.
\item The estimates of Ricci coefficients and the extension argument, i.e. the proof of Theorems M3 and M4, are similar to \cite{Caciotta}.
\end{enumerate}
\end{remark}
\subsection{Proof of the main theorem}\label{proofmain}
Now, we use Theorems M0--M4 to prove Theorem \ref{maintheorem}.
\begin{definition}\label{bootstrap}
Let $\aleph(\ub_*)$ the set of spacetimes $\KK$ associated with a double null foliation $(u,\ub)$ which satisfy the following properties:
\begin{enumerate}
    \item $\KK=V(u_0,\ub_*)$.
    \item The foliation on $\Cb_*$ is the geodesic type foliation of Definition \ref{geodesicfoliation}.
    \item We have the following bootstrap assumptions:
\begin{align}
    \OO\leq\epsilon,\qquad\RR\leq\epsilon,\qquad\osc\leq\ep,\qquad \OO^*(\Cb_*)\leq\ep.\label{B2}
\end{align}
\end{enumerate}
\end{definition}
\begin{definition}\label{defboot}
We denote $\UU$ the set of values $\ub_*$ such that $\aleph(\ub_*)\ne\emptyset$.
\end{definition}
Applying the local existence result of Theorem 10.2.1 in \cite{ch-kl} and the assumption $\OO_{(0)}\leq\ep_0$, we deduce that \eqref{B2} holds if $\ub_*$ is sufficiently close to $R_0$. So, we have $\mathcal{U}\ne\emptyset$.\\ \\
Define $\ub_*$ to be the supremum of the set $\mathcal{U}$. We want to prove $\ub_*=\infty$. We assume by contradiction that $\ub_*$ is finite. In particular, we may assume that $\ub_*\in\mathcal{U}$. We consider the region $\KK=V(u_0,\ub_*)$. Recall that we have
\begin{equation*}
    \OO_{(0)}\leq\ep_0,\qquad \Rk_{(0)}\leq\ep_0.
\end{equation*}
Combining with \eqref{B2} and applying Theorem M0, we obtain
\begin{equation}\label{Rk0est}
    \Rk_0\les\ep_0.
\end{equation}
Then, we have from Theorem M1
\begin{equation}\label{RR}
    \RR\les\ep_0.
\end{equation}
Applying Theorem M2, we obtain
\begin{align}\label{OO**}
    \OO^{*}(\Cb_*)\les\ep_0.
\end{align}
Combining \eqref{B2}, \eqref{RR}, \eqref{OO**} and Theorem M3, we deduce
\begin{equation}
    \OO\les\ep_0,\qquad \osc\les\ep_0.
\end{equation}
Next, applying Theorem M4, we can extend $\KK$ to $\wideparen{\KK}:=\wideparen{V}(u_0,\ub_*+\nu)$ endowed with a new double null foliation $(\upa,\ub)$ and
\begin{equation*}
    \wideparen{\OO}(\wideparen{\KK})\les\ep_0,\qquad\widetilde{\RR}(\wideparen{\KK})\les\ep_0,\qquad \wideparen{\osc}(\wideparen{\KK})\les\ep_0,\qquad \wideparen{\OO}^*(\Cb_{**}\cap\wideparen{\kk})\les\ep_0,
\end{equation*}
where $\Cb_{**}=\Cb_{\ub_*+\nu}$ is the new last slice. Thus, the region $\wideparen{V}(u_0,\ub_*+\nu)$ with the new double null foliation $(\wideparen{u},\ub)$ satisfies all the properties in Definition \ref{bootstrap}, and so $\aleph(\ub_*+\nu)\ne\emptyset$, a contradiction. So, we have $\ub_*=\infty$, which implies property 1 of Theorem \ref{maintheorem}. Moreover, we also have
\begin{equation}
    \OO\les\ep_0,\qquad \RR\les\ep_0,
\end{equation}
in the whole external region, which implies property 2 of Theorem \ref{maintheorem}. This concludes the proof of Theorem \ref{maintheorem}.
\subsection{Conclusions}
The following proposition is a consequence of Theorem \ref{maintheorem}.
\begin{proposition}\label{gdiff}
Under the assumptions of Theorem \ref{maintheorem}, we have the following estimate in the external region\footnote{See Figure \ref{figconclu} for a description of the external region.}:
\begin{equation}
    |\g-\g_K|_{\infty,S(u,\ub)}\les\frac{\ep_0}{r|u|^{\frac{s-3}{2}}}.\label{gdiffO}
\end{equation}
\end{proposition}
\begin{figure}
  \centering
  \includegraphics[width=1\textwidth]{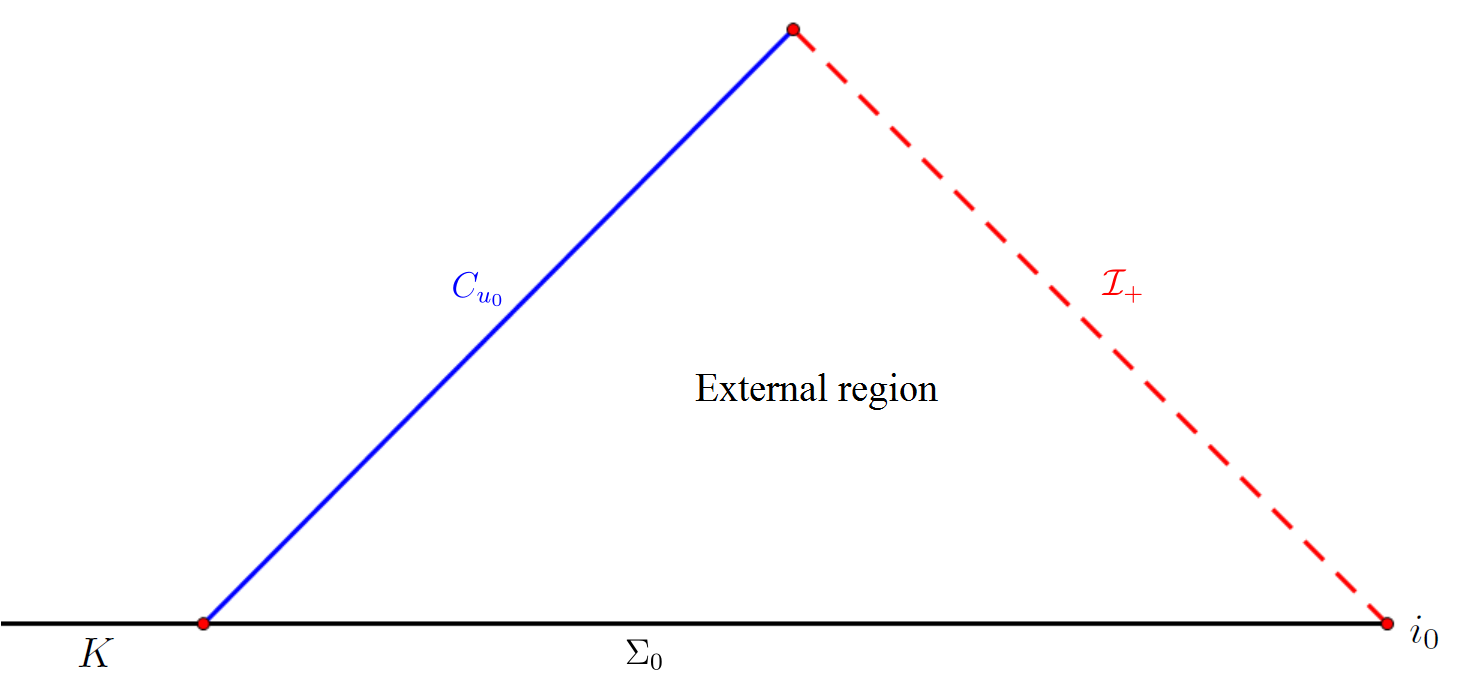}
  \caption{External region}\label{figconclu}
\end{figure}
\begin{proof}
We recall that we have
\begin{align*}
    \g=-2\Om^2(d\ub\otimes du+du\otimes d\ub)+\ga_{AB}(dx^A-\bu^Adu)\otimes(dx^B-\bu^Bdu),
\end{align*}
and
\begin{align*}
    \g_K=-2\Om_K^2(d\ub\otimes du+du\otimes d\ub)+(\ga_K)_{AB}(dx^A-\bu_K^Adu)\otimes(dx^B-\bu_K^Bdu).
\end{align*}
We compare these two metrics as follows:
\begin{align*}
    \g_{\ub\ub}-(\g_K)_{\ub\ub}&=0,\\
    \g_{\ub u}-(\g_K)_{\ub u}&=-4\Om^2+4\Om_K^2=-4\Omc(\Om+\Om_K)= O\left(\frac{\ep_0}{r|u|^{\frac{s-3}{2}}}\right),\\
    \g_{\ub A}-(\g_K)_{\ub A}&=0,\\
    \g_{uu}-(\g_K)_{uu}&=\gac_{AB}\bu^A\bu^B+(\ga_K)_{AB}(\bu^A\bu^B-\bu_K^A\bu_K^B)=O\left(\frac{\ep_0}{r|u|^{\frac{s-3}{2}}}\right),\\
    \g_{uA}-(\g_K)_{uA}&=-\gac_{AB}\bu^B-(\ga_K)_{AB}\bcc^B=O\left(\frac{\ep_0}{|u|^{\frac{s-3}{2}}}\right),\\
    \g_{AB}-(\g_K)_{AB}&=\gac_{AB}=O\left(\frac{\ep_0 r}{|u|^{\frac{s-3}{2}}}\right).
\end{align*}
This concludes the proof of Proposition \ref{gdiff}.
\end{proof}
\begin{remark}
Theorem \ref{maintheorem} contains also a number of important conclusions following from \eqref{finalest}: peeling properties, complete future infinity, Bondi mass formula and so on, see Chapter 8 in \cite{Kl-Ni}. See 
also Section 3.8 in \cite{KS:main} for various conclusions of the Kerr stability in the entire domain of outer communication.
\end{remark}
\section{Curvature estimates (Theorem M1)}\label{curvatureestimates}
In this section, we prove Theorem M1 by the $r^p$--weighted estimate method introduced in \cite{Da} and applied to Bianchi equations in \cite{Hol10a}, \cite{KS} and \cite{GKS}. For convenience, we recall the statement below.
\begin{thmM1}
Assume that
\begin{equation}\label{asm1}
    \OO_{(0)}\leq\ep_0,\qquad\Rk_0\les\ep_0,\qquad \OO\leq\ep,\qquad \RR\leq\ep,\qquad \osc\leq\ep.
\end{equation}
Then, we have
\begin{equation}
    \RR\les\ep_0.
\end{equation}
\end{thmM1}
In the sequel, we denote\footnote{See Figure \ref{fig4} for a description of $V(u,\ub)$, $\cuv$ and $\ucuv$.}
\begin{equation}
   V:=V(u,\ub),\qquad C_u^V:=C_u\cap V,\qquad\Cb_\ub^V:=\Cb_\ub\cap V.
\end{equation}
The following lemma allows us to estimate the linear terms and error terms in Propositions \ref{Bianchieq} and \ref{Bianchieqdkb}.
\begin{figure}
  \centering
  \includegraphics[width=1\textwidth]{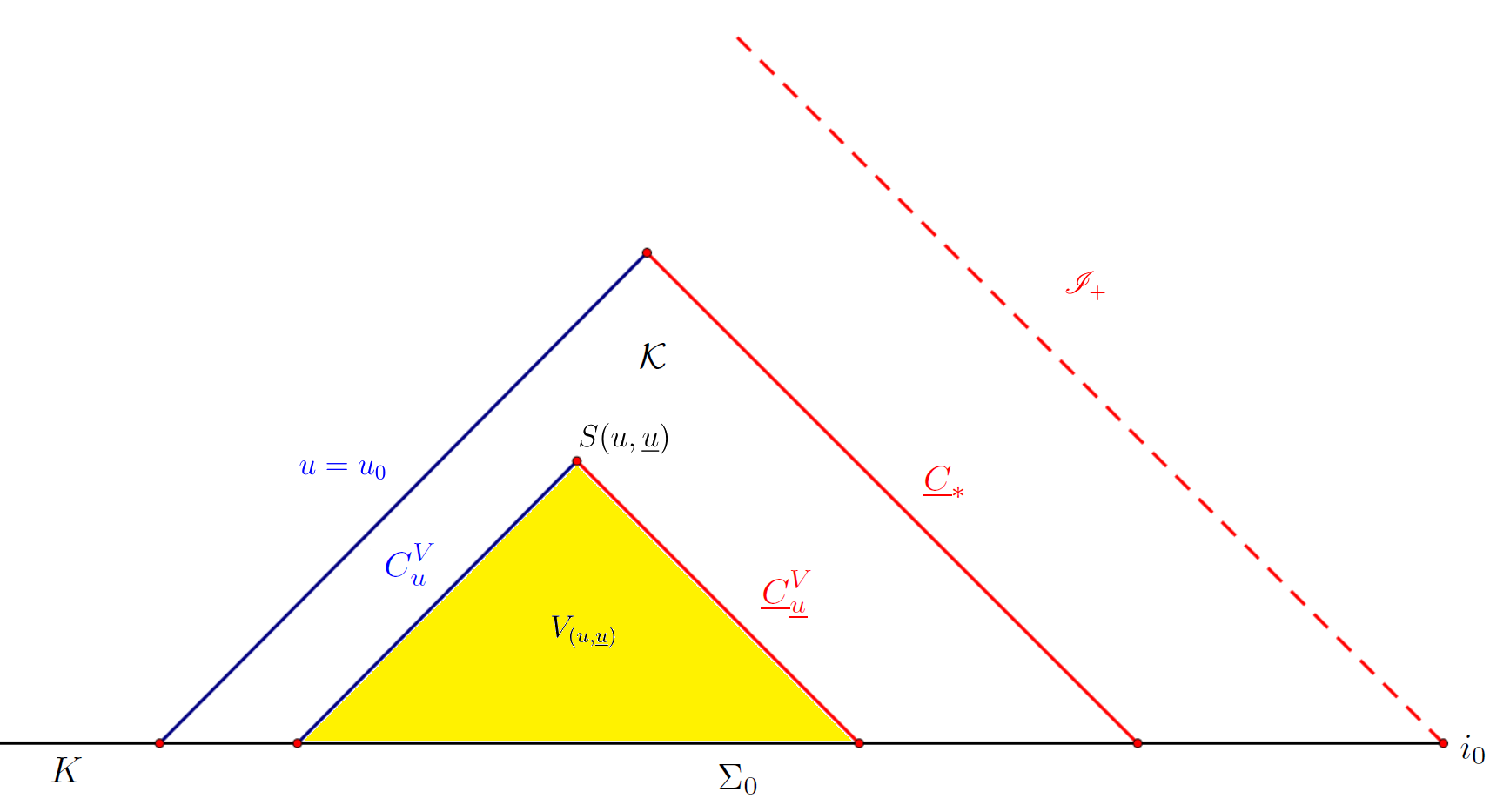}
  \caption{Domain of integration}\label{fig4}
\end{figure}
\begin{lemma}\label{estGagba}
Under the assumptions $\OO\leq\ep$ and $\RR\leq\ep$, we have
\begin{align}
\begin{split}
    r^2|u|^{\frac{s-3}{2}}|\Gag|_{\infty,S}&\les\ep,\qquad\qquad\quad r|u|^{\frac{s-1}{2}}|\Gab|_{\infty,S}\les\ep\\
    |r^{2-\frac{2}{p}}|u|^\frac{s-3}{2}\Gag^{(1)}|_{p,S}&\les\ep,\qquad\quad\, |r^{1-\frac{2}{p}}|u|^\frac{s-1}{2}\Gab^{(1)}|_{p,S}\les\ep,\qquad\quad p\in [2,4].
\end{split}
\end{align}
\end{lemma}
\begin{proof}
It follows directly from the assumption \eqref{asm1}, Definition \ref{gammag} and Proposition \ref{standardsobolev}.
\end{proof}
\subsection{General Bianchi pairs}
The following lemma provides the general structure of Bianchi pairs. It will be used repeatedly in Section \ref{curvatureestimates}. See Lemma 8.24 in \cite{KS} for an analog under the assumption of axial polarization and Lemma 15.3.8 in \cite{GKS} for an analog in the context of non integrable foliations.
\begin{lemma}\label{keypoint}
Let $k=1,2$ and $a_{(1)}$, $a_{(2)}$ real numbers. Then, we have the following properties.
\begin{enumerate}
    \item Let $\psi_{(1)}, h_{(1)}\in\sk_k$ and $\psi_{(2)},h_{(2)}\in\sk_{k-1}$ satisfying
    \begin{align}
    \begin{split}\label{bianchi1}
        \nab_3(\psi_{(1)})+a_{(1)}\trchb\,\psi_{(1)}&=-kd_k^*(\psi_{(2)})+h_{(1)},\\
        \nab_4(\psi_{(2)})+a_{(2)}\trch\,\psi_{(2)}&=d_k(\psi_{(1)})+h_{(2)}.
    \end{split}
    \end{align}
Then, the pair $(\psi_{(1)},\psi_{(2)})$ satisfies for any real number $p$
\begin{align}
\begin{split}\label{div}
       &\Div(r^p |\psi_{(1)}|^2e_3)+k\Div(r^p|\psi_{(2)}|^2e_4)\\
       +&\left(2a_{(1)}-1-\frac{p}{2}\right)r^{p}\trchb|\psi_{(1)}|^2+k\left(2a_{(2)}-1-\frac{p}{2}\right)r^{p}\trch|\psi_{(2)}|^2\\
       =&2k r^p\sdiv(\psi_{(1)}\cdot\psi_{(2)})
       +2r^p\psi_{(1)}\cdot h_{(1)}+2kr^p\psi_{(2)}\cdot h_{(2)}-2r^p\omb|\psi_{(1)}|^2\\
       -&2kr^p\om|\psi_{(2)}|^2+pr^{p-1}\left(e_3(r)-\frac{r}{2}\trchb\right)|\psi_{(1)}|^2+kpr^{p-1}\left(e_4(r)-\frac{r}{2}\tr\chi\right)|\psi_{(2)}|^2.
\end{split}
\end{align}
    \item Let $\psi_{(1)}, h_{(1)}\in\sk_{k-1}$ and $\psi_{(2)},h_{(2)}\in\sk_k$ satisfying
    \begin{align}
        \begin{split}\label{bianchi2}
        \nab_3(\psi_{(1)})+a_{(1)}\trchb\,\psi_{(1)}&=d_k(\psi_{(2)})+h_{(1)},\\
        \nab_4(\psi_{(2)})+a_{(2)}\trch\,\psi_{(2)}&=-kd_k^*(\psi_{(1)})+h_{(2)}.
        \end{split}
    \end{align}
Then, the pair $(\psi_{(1)},\psi_{(2)})$ satisfies for any real number $p$
\begin{align}
\begin{split}\label{div2}
       &k\Div(r^p |\psi_{(1)}|^2e_3)+\Div(r^p|\psi_{(2)}|^2e_4)\\
       +&k\left(2a_{(1)}-1-\frac{p}{2}\right)r^{p}\trchb|\psi_{(1)}|^2+\left(2a_{(2)}-1-\frac{p}{2}\right)r^{p}\trch|\psi_{(2)}|^2\\
       =&2 r^p\sdiv(\psi_{(1)}\cdot\psi_{(2)})
       +2kr^p\psi_{(1)}\cdot h_{(1)}+2r^p\psi_{(2)}\cdot h_{(2)}-2k r^p\omb|\psi_{(1)}|^2\\
       -&2r^p\om|\psi_{(2)}|^2+k pr^{p-1}\left(e_3(r)-\frac{r}{2}\trchb\right)|\psi_{(1)}|^2+pr^{p-1}\left(e_4(r)-\frac{r}{2}\trch\right)|\psi_{(2)}|^2.
\end{split}
\end{align}
\end{enumerate}
\end{lemma}
\begin{proof}
See Lemma 4.2 in \cite{ShenMink}.
\end{proof}
\begin{remark}
    Note that the linearized Bianchi equations can be written as systems of equations of the type \eqref{bianchi1} and \eqref{bianchi2}. In particular
    \begin{itemize}
        \item the Bianchi pair $(\ac,\bc)$ satisfies \eqref{bianchi1} with $k=2$, $a_{(1)}=\frac{1}{2}$, $a_{(2)}=2$,
        \item the Bianchi pair $(\bc,(\rhoc,-\sic))$ satisfies \eqref{bianchi1} with $k=1$, $a_{(1)}=1$, $a_{(2)}=\frac{3}{2}$,
        \item the Bianchi pair $((\rhoc,\sic),\bbc)$ satisfies \eqref{bianchi2} with $k=1$, $a_{(1)}=\frac{3}{2}$, $a_{(2)}=1$,
        \item the Bianchi pair $(\bbc,\aac)$ satisfies \eqref{bianchi2} with $k=2$, $a_{(1)}=2$, $a_{(2)}=\frac{1}{2}$.
    \end{itemize}
\end{remark}
\begin{proposition}\label{keyintegral}
For $\psi_{(1)}$, $\psi_{(2)}$ and $h_{(1)}$, $h_{(2)}$ satisfying \eqref{bianchi1} or \eqref{bianchi2}, we have the following properties for all $(u,\ub)\in\kk$:\\\\
(a) In the case $2+p-4a_{(1)}>0$ and $4a_{(2)}-2-p>0$, we have
\begin{align}
\begin{split}\label{caseone}
&\int_{\cuv}r^p |\psi_{(1)}|^2+\int_\ucuv r^p|\psi_{(2)}|^2 +\int_{V}r^{p-1}|\psi_{(1)}|^2+r^{p-1}|\psi_{(2)}|^2 \\ 
\les &\int_{\Sigma_0 \cap V} r^p |\psi_{(1)}|^2 + r^p|\psi_{(2)}|^2+\int_{V} r^p|\psi_{(1)}||h_{(1)}|+r^p|\psi_{(2)}||h_{(2)}|.
\end{split}
\end{align}
(b) In the case $2+p-4a_{(1)}>0$ and $4a_{(2)}-2-p=0$, we have
\begin{align}
\begin{split}\label{casetwo}
&\int_{\cuv}r^p |\psi_{(1)}|^2+\int_\ucuv r^p|\psi_{(2)}|^2 +\int_{V} r^{p-1}|\psi_{(1)}|^2\\
\les &\int_{\Sigma_0 \cap V} r^p |\psi_{(1)}|^2+r^p|\psi_{(2)}|^2+\int_{V} r^p|\psi_{(1)}||h_{(1)}|+r^p|\psi_{(2)}||h_{(2)}|.
\end{split}
\end{align}
(c) In the case $2+p-4a_{(1)}\leq 0$ and $4a_{(2)}-2-p\geq 0$, we have
\begin{align}
\begin{split}\label{casethree}
&\int_{\cuv}r^p|\psi_{(1)}|^2+\int_\ucuv r^p|\psi_{(2)}|^2\\ 
\les&\int_{\Sigma_0\cap V}r^p|\psi_{(1)}|^2+r^p|\psi_{(2)}|^2+\int_{V}r^{p-1}|\psi_{(1)}|^2+r^p|\psi_{(1)}||h_{(1)}|+r^p|\psi_{(2)}||h_{(2)}|.
\end{split}
\end{align}
(d) In the case $2+p-4a_{(1)}> 0$ and $4a_{(2)}-2-p\leq 0$, we have
\begin{align}
\begin{split}\label{casefour}
&\int_{\cuv}r^p|\psi_{(1)}|^2+\int_\ucuv r^p|\psi_{(2)}|^2+\int_{V}r^{p-1}|\psi_{(1)}|^2\\
\les&\int_{\Sigma_0\cap V}r^p|\psi_{(1)}|^2+r^p|\psi_{(2)}|^2+\int_{V}r^{p-1}|\psi_{(2)}|^2+r^p|\psi_{(1)}||h_{(1)}|+r^p|\psi_{(2)}||h_{(2)}|.
\end{split}
\end{align}
\end{proposition}
\begin{remark}
In the sequel, For a sum of terms of the form $\Ga\cdot R$, we ignore the terms having the same or even better decay. For example, we write
\begin{align*}
    \Gag\cdot \b+\Gag\cdot\a =\Gag\cdot\b,\qquad \Gag\c(\rhoc,\sic)+\Gag\c\bb=\Gag\c\bb,
\end{align*}
since $\a$ decays better than $\b$ and $\rhoc$ and $\sic$ decay better than $\bb$.
\end{remark}
\begin{proof}
We obtain from Proposition \ref{expressionGaR} that
\begin{align*}
    \om=\Gag+\O_2^1,\qquad \omb=\Gab+\O_2^1.
\end{align*}
Applying Lemma \ref{dint}, we obtain 
\begin{align}\label{e3r}
    e_3(r)-\frac{r}{2}\trchb=\frac{\ov{\Om\trchb}}{2\Om}r-\frac{r}{2}\trchb=\frac{r}{2\Om}(\ov{\Om\trchb}-\Om\trchb)=r\Gag+\O_2^1,
\end{align}
and similarly
\begin{align}\label{e4r}
    e_4(r)-\frac{r}{2}\trch= r\Gag+\O_2^1.
\end{align}
Applying Lemma \ref{estGagba}, we infer
\begin{align}
\begin{split}\label{Gaapsi}
&\int_V r^p\left(|\omb| |\psi_{(1)}|^2+|\om||\psi_{(2)}|^2\right)\\
&+\int_V r^{p-1}\left(e_3(r)-\frac{r}{2}\trchb\right)|\psi_{(1)}|^2+r^{p-1}\left(e_4(r)-\frac{r}{2}\trchb\right)|\psi_{(2)}|^2\\
\les& \int_V r^p\left(\Gab+\O_2^1\right)|\psi_{(1)}|^2+\int_V r^p\left(\Gag+\O_2^1\right)|\psi_{(2)}|^2\\
\les& M\int_V r^{p-2}|\psi_{(1)}|^2+\ep\int_V r^{p-1}|u|^{-1}|\psi_{(1)}|^2+M\int_V r^{p-2}|\psi_{(2)}|^2\\
\les& M\int_{u_0(\ub)}^u\left(|u|^{-2}\int_\cuv r^{p}|\psi_{(1)}|^2\right) du+M\int_{|u|}^\ub\left(|\ub|^{-2}\int_\ucuv r^{p}|\psi_{(2)}|^2\right)d\ub\\
\les& M\int_{u_0(\ub)}^u |u|^{-2}du \left(\sup_u \int_\cuv r^{p}|\psi_{(1)}|^2\right)+M\int_{|u|}^\ub|\ub|^{-2}d\ub\left(\sup_\ub\int_\ucuv r^{p}|\psi_{(2)}|^2\right)\\
\les& \EE_0\left(\sup_u \int_\cuv r^{p}|\psi_{(1)}|^2\right)+\EE_0\left(\sup_\ub\int_\ucuv r^{p}|\psi_{(2)}|^2\right).
\end{split}
\end{align}
Integrating \eqref{div} or \eqref{div2} and recalling that 
\begin{align*}
\trch-\frac{2}{r}=\Gag+\O_2^1,\qquad\trchb+\frac{2}{r}=\Gag+\O_2^1,
\end{align*}
we obtain
\begin{align*}
&\int_{\cuv}r^p |\psi_{(1)}|^2+\int_\ucuv r^p|\psi_{(2)}|^2 +\int_{V} (2+p-4a_{(1)})r^{p-1}|\psi_{(1)}|^2+(4a_{(2)}-2-p)r^{p-1}|\psi_{(2)}|^2 \\ 
\les &\int_{\Sigma_0 \cap V} r^p |\psi_{(1)}|^2+r^p|\psi_{(2)}|^2+\int_{V} r^p|\psi_{(1)}||h_{(1)}|+r^p|\psi_{(2)}||h_{(2)}|\\
&+\EE_0\sup_u\int_{\cuv}r^p |\psi_{(1)}|^2+\EE_0\sup_{\ub}\int_\ucuv r^p|\psi_{(2)}|^2.
\end{align*}
Taking the supremum in $u$ and $\ub$ and applying \eqref{Gaapsi}, we obtain for $\ep$ and $\EE_0$ small enough
\begin{align}
\begin{split}\label{psipsipsi}
&\sup_u\int_{\cuv}r^p |\psi_{(1)}|^2+\sup_\ub\int_\ucuv r^p|\psi_{(2)}|^2\\ +&\int_{V}(2+p-4a_{(1)})r^{p-1}|\psi_{(1)}|^2+(4a_{(2)}-2-p)r^{p-1}|\psi_{(2)}|^2 \\ 
\les &\int_{\Sigma_0 \cap V} r^p |\psi_{(1)}|^2+r^p|\psi_{(2)}|^2+\int_{V} r^p|\psi_{(1)}||h_{(1)}|+r^p|\psi_{(2)}||h_{(2)}|.
\end{split}
\end{align}
This concludes the proof of Proposition \ref{keyintegral}.
\end{proof}
The following lemma allows us to compare $\dkb\psic$ and $\widecheck{\dkb\psi}$.
\begin{lemma}\label{totallycheck}
    We denote for $\psi\in\{\a,\b,\rho,\si,\bb,\aa\}$
    \begin{align}\label{totallycheckdf}
        \widecheck{\psi^{(1)}}:=(\psic,\widecheck{\dkb\psi}).
    \end{align}
    Then, we have for $p<7$
    \begin{align}
    \begin{split}\label{compare}
        \int_\cuv r^p |\psic^{(1)}|^2&\les \int_\cuv r^p |\widecheck{\psi^{(1)}}|^2 +\frac{\ep_0^2}{|u|^{s-p}},\\
        \int_\ucuv r^p |\psic^{(1)}|^2&\les \int_\ucuv r^p |\widecheck{\psi^{(1)}}|^2 +\frac{\ep_0^2}{|u|^{s-p}}.
    \end{split}
    \end{align}
\end{lemma}
\begin{remark}
    Note that Proposition \ref{Bianchieqdkb} involves the quantities $\dkbaac$, $\dkbbbc$, $\dkbrhoc$, $\dkbsic$ $\dkbbc$ and $\dkbac$, while the curvature norms $\RR_1$ and $\RRb_1$ involve $\dkb\aac$, $\dkb\bbc$, $\dkb\rhoc$, $\dkb\sic$ $\dkb\bc$ and $\dkb\ac$. We thus need to compare $\dkb\psic$ and $\widecheck{\dkb\psi}$ which is the purpose of Lemma \ref{totallycheck}.
\end{remark}
\begin{proof}
For $\psi\in\{\rho,\si\}$ which is a scalar function, we have
\begin{align*}
    \widecheck{\dkb\psi}=\dkb\psic.
\end{align*}
The estimate \eqref{compare} is trivial in this case. Next, we focus on the case $\psi\in\{\a,\b,\bb,\aa\}$. We have from Proposition \ref{useful}
\begin{align*}
    \widecheck{\dkb\psi}=\dkb\psi-\dkb_K\psi_K=\dkb\psic+(\dkb-\dkb_K)\psi_K=\dkb\psic+r\Lc\c\psi_K.
\end{align*}
Recalling that $\psi_K=\O_4^2$, we deduce
    \begin{align*}
        \int_{\cuv} r^p|\dkb\psic|^2 &\les\int_\cuv r^p |\widecheck{\dkb\psi}|^2+\int_\cuv r^p|\Gag^{(1)}\c\O_3^2|^2\\
        &\les \int_\cuv r^p |\widecheck{\dkb\psi}|^2+M^4\int_\cuv r^{p-6}|\Gag^{(1)}|^2\\
        &\les \int_\cuv r^p |\widecheck{\dkb\psi}|^2+M^4\int_{|u|}^\ub r^{p-8} d\ub\int_S |r\Gag^{(1)}|^2\\
        &\les\int_\cuv r^p |\widecheck{\dkb\psi}|^2+M^4\ep^2\int_{|u|}^\ub r^{p-8}\frac{d\ub}{|u|^{s-3}}\\
        &\les\int_\cuv r^p|\widecheck{\dkb\psi}|^2+M^4\ep^2|u|^{p-s-4}\\
        &\les \int_\cuv r^p|\widecheck{\dkb\psi}|^2+\frac{\EE_0^4\ep^2}{|u|^{s-p}}\\
        &\les \int_\cuv r^p|\widecheck{\dkb\psi}|^2+\frac{\ep_0^2}{|u|^{s-p}},
    \end{align*}
    which implies the first inequality in \eqref{compare}. The second inequality in \eqref{compare} is similar. This concludes the proof of Lemma \ref{totallycheck}.
\end{proof}
\begin{remark}
The proof of Theorem M1 is analogous to Theorem M1 in \cite{ShenMink}. In particular, the estimates for the error terms $\err[...]$ in Propositions \ref{Bianchieq} and \ref{Bianchieqdkb} have already been done in \cite{ShenMink}. In the proof of Theorem M1 below, we thus focus on the linear terms $\fl[...]$.
\end{remark}
\subsection{Estimates for the Bianchi pair \texorpdfstring{$(\ac,\bc)$}{}}
\begin{proposition}\label{estab}
We have the following estimate:
\begin{equation}
\RR_0[\ac]^2+\RRb_0[\bc]^2+\RR_1[\ac]^2+\RRb_1[\bc]^2\les\ep_0^2.\label{abs}
\end{equation}
\end{proposition}
\begin{proof}
We recall the following equations from Propositions \ref{Bianchieq} and \ref{Bianchieqdkb}:
\begin{align}
    \begin{split}\label{bianchiab}
\nab_3\dkbaco+\f12\tr\chib\,\widecheck{\a^{(1)}}=&-2d^*_2(\widecheck{\b^{(1)}})+\fl[\nab_3\widecheck{\a^{(1)}}]+\err[\nab_3\widecheck{\a^{(1)}}],\\
\nab_4\widecheck{\b^{(1)}}+2\tr\chi\,\widecheck{\b^{(1)}}=&d_2(\widecheck{\a^{(1)}})+\fl[\nab_4\widecheck{\b^{(1)}}]+\err[\nab_4\widecheck{\b^{(1)}}].
    \end{split}
\end{align}
Applying \eqref{caseone} with $\psi_{(1)}=\widecheck{\a^{(1)}}$, $\psi_{(2)}=\widecheck{\b^{(1)}}$, $a_{(1)}=\f12$, $a_{(2)}=2$, $h_{(1)}=\fl[\nab_3\widecheck{\a^{(1)}}]+\err[\nab_3\widecheck{\a^{(1)}}]$, $h_{(2)}=\fl[\nab_4\widecheck{\b^{(1)}}]+\err[\nab_4\widecheck{\b^{(1)}}]$ and $p=s$ we obtain
\begin{align}
\begin{split}
   &\int_{\cuv} r^s |\widecheck{\a^{(1)}}|^2 + \int_{\ucuv} r^s |\widecheck{\b^{(1)}}|^2+\int_{V} r^{s-1} |\widecheck{\a^{(1)}}|^2 \\ 
    \les &\int_{\Si_0\cap V}r^s |\widecheck{\a^{(1)}}|^2+r^s |\widecheck{\b^{(1)}}|^2+\int_{V} r^s|\widecheck{\a^{(1)}}| \Big|\fl[\nab_3\widecheck{\a^{(1)}}]+\err[\nab_3\widecheck{\a^{(1)}}]\Big|\\
    &+\int_V r^s|\widecheck{\b^{(1)}}|\Big|\err[\nab_4\widecheck{\b^{(1)}}]+\err[\nab_4\widecheck{\b^{(1)}}]\Big|.\label{AB}
\end{split}
\end{align}
We recall from Proposition \ref{Bianchieq}\footnote{We ignore the terms which decay better.}
\begin{align}
\begin{split}\label{flab}
    &\int_V r^s |\dkbaco|\fl[\nab_3\dkbaco]|+r^s|\dkbbco||\fl[\nab_4\dkbbco]|\\
    \les& \int_V r^s\left(\frac{M}{r^3}|\Gag^{(1)}||\dkbaco|+\frac{M^2}{r^3}|\dkbaco||\dko\rhoc|+\frac{M}{r^2}|\dko\bc|^2\right).
\end{split}
\end{align}
We estimate \eqref{flab} term by term. For the first term:
\begin{align*}
    \int_V M r^{s-3}|\dkbaco||\Gag^{(1)}|&\les M\int_{u_0(\ub)}^u du \left(\int_\cuv r^s|\dkbaco|^2\right)^\f12\left(\int_\cuv r^{s-6}|\Gag^{(1)}|^2\right)^\f12\\
    &\les M\ep\int_{u_0(\ub)}^u du \left(\int_{|u|}^\ub r^{s-8}|r\Gag^{(1)}|^2_{2,S} d\ub\right)^\f12\\
    &\les M\ep^2\int_{u_0(\ub)}^u du \left(\int_{|u|}^\ub r^{s-8}|u|^{3-s} d\ub\right)^\f12\\
    &\les M\ep^2\int_{u_0(\ub)}^u |u|^{-2}du\\
    &\les \ep_0^2.
\end{align*}
For the second term:
\begin{align*}
    \int_V M^2 r^{s-3}|\dkbaco||\rhoc^{(1)}|&\les M^2\int_{u_0(\ub)}^u du \left(\int_\cuv r^s|\dkbaco|^2\right)^\f12\left(\int_\cuv r^2|\rhoc^{(1)}|^2\right)^\f12 r^{\frac{s}{2}-4}\\
    &\les M^2\int_{u_0(\ub)}^u du \frac{\ep^2}{|u|^{\frac{s-2}{2}}} |u|^{\frac{s}{2}-4}\\
    &\les M^2 \ep^2 |u|^{-2}\\
    &\les \ep_0^2. 
\end{align*}
For the last term:
\begin{align*}
    \int_V M r^{s-2}|\dkbbco|^2 &\les M \int_{|u|}^\ub\frac{d\ub}{r^2}\int_\ucuv r^s |\dkbbco|^2\les M\ep^2\int_{|u|}^\ub\frac{d\ub}{r^2}\les\ep_0^2.
\end{align*}
Injecting all the estimates into \eqref{flab}, we deduce
\begin{equation}\label{estflab}
\int_V r^s|\dkbaco||\fl[\nab_3\dkbaco]|+r^s|\dkbbco||\fl[\nab_4\dkbbco]|\les\ep_0^2.
\end{equation}
Next, we proceed as in Proposition 4.7 in \cite{ShenMink} to obtain
\begin{equation}\label{esterrab}
    \int_V r^s \left(|\dkbaco||\err[\nab_3\dkbaco]|+|\dkbbco||\err[\nab_4\dkbbco]|\right)\les\ep_0^2.
\end{equation}
Combining \eqref{estflab} and \eqref{esterrab}, we obtain
\begin{equation}\label{estabfinal}
\int_{\cuv} r^s |\dkbaco|^2 + \int_{\ucuv} r^s |\dkbbco|^2 + \int_{V}r^{s-1} |\dkbaco|^2 \les\ep_0^2.
\end{equation}
Applying Lemma \ref{totallycheck}, we deduce \eqref{abs}. This concludes the proof of Proposition \ref{estab}.
\end{proof}
\subsection{Estimates for the Bianchi pair \texorpdfstring{$(\bc,(\rhoc,\sic))$}{}{}}
\begin{proposition}\label{estbr}
We have the following estimate:
\begin{equation}
\RR_0[\bc]^2+\RRb_0[(\rhoc,\sic)]^2+\RR_1[\bc]^2+\RRb_1[(\rhoc,\sic)]^2\les \ep_0^2.
\end{equation}
\end{proposition}
\begin{proof}
We recall from Propositions \ref{Bianchieq} and \ref{Bianchieqdkb}:
\begin{align*}
\nab_4(\dkbrhoco,-\dkbsico)+\frac{3}{2}\tr\chi(\dkbrhoco,-\dkbsico)&=d_1\bc+\fl[\nab_4\dkbrhoco,-\nab_4\dkbsico]+\err[\nab_4\dkbrhoco,-\nab_4\dkbsico],\\
\nab_3\dkbbco+\tr\chib\,\dkbbco&=-d_1^*(\dkbrhoco,-\dkbsico)+\fl[\nab_3\dkbbco]+\err[\nab_3\dkbbco],
\end{align*}
where 
\begin{equation*}
\fl[\nab_4\dkbrhoco,-\nab_4\dkbsico]=\Gag^{(1)}\c\O_3^1+\O_3^2\c\rhoc^{(1)},\qquad \fl[\nab_3\dkbbco]=\Gag^{(1)}\c\O_3^1+\O_2^1\c\bc^{(1)}+\O_3^2\c\bbc^{(1)}.
\end{equation*}
Applying \eqref{casetwo} with $\psi_{(1)}=\dkbbco$, $\psi_{(2)}=(\dkbrhoco,-\dkbsico)$, $a_{(1)}=1$, $a_{(2)}=\frac{3}{2}$, $h_{(1)}=\fl[\nab_3\dkbbco]+\err[\nab_3\dkbbco]$, $h_{(2)}=\fl[\nab_4\dkbrhoco,-\nab_4\dkbsico]+\err[\nab_4\dkbrhoco,-\nab_4\dkbsico]$ and $p=4$, we infer
\begin{align}
\begin{split}\label{brint}
&\int_{\cuv} r^4|\dkbbco|^2+\int_{\ucuv}r^{4}|(\dkbrhoco,\dkbsico)|^2 \\
\les &\int_{\Si_0\cap V}\left(r^4|\dkbbco|^2+r^{4}|(\dkbrhoco,\dkbsico)|^2\right)+\int_V r^4\left|\dkbbco\c\left(\fl[\nab_3\dkbbco]+\err[\nab_3\dkbbco]\right)\right|\\
&+\int_V r^{4} \left|(\dkbrhoco,\dkbsico)\c \left(\fl[\nab_4\dkbrhoco,\nab_4\dkbsico]+\err[\nab_4\dkbrhoco,\nab_4\dkbsico]\right)\right|\\
\les &|u|^{4-s}\ep_0^2+\int_V r^4\left|\dkbbco\c\left(\fl[\nab_3\dkbbco]+\err[\nab_3\dkbbco]\right)\right|\\
&+\int_V r^{4}\left|(\dkbrhoco,\dkbsico)\c \left(\fl[\nab_4\dkbrhoco,\nab_4\dkbsico]+\err[\nab_4\dkbrhoco,\nab_4\dkbsico]\right)\right|.
\end{split}
\end{align}
Notice that we have
\begin{align}
\begin{split}\label{estflbrho}
&\int_V r^4 |\dkbbco||\fl[\nab_3\dkbbco]|+r^4|(\dkbrhoco,\dkbsico)||\fl[\nab_4\dkbrhoco,-\nab_4\dkbsico]|\\
\les&\int_V Mr^2|\bc^{(1)}|^2+M^2r|\bc^{(1)}||\bbc^{(1)}|+Mr|\Gag^{(1)}||\rhoc^{(1)}|\\
\les& M\int_{|u|}^\ub\frac{d\ub}{r^2}\int_\ucuv r^4|\bc^{(1)}|^2+M^2\int_{|u|}^\ub \frac{d\ub}{r^\frac{s}{2}}\int_{\ucuv} |r^\frac{s}{2}\bc^{(1)}||r\bbc^{(1)}|\\
&+M\int_{|u|}^\ub \frac{d\ub}{r^2} \int_\ucuv |r\Gag^{(1)}||r^2\rhoc^{(1)}|\\
\les& M\int_{|u|}^\ub\frac{d\ub}{r^2}\frac{\ep^2}{|u|^{s-4}}+M^2\int_{|u|}^\ub \frac{d\ub}{r^\frac{s}{2}}\left(\int_{\ucuv} |r^\frac{s}{2}\bc^{(1)}|\right)^\frac{1}{2}\left(\int_\ucuv|r\bbc^{(1)}|^2\right)^\frac{1}{2}\\
&+M\int_{|u|}^\ub \frac{d\ub}{r^2}\left(\int_{u_0(\ub)}^u du |r\Gag^{(1)}|^2_{2,S}\right)^\frac{1}{2}\left(\int_\ucuv |r^2\rhoc^{(1)}|^2\right)^\frac{1}{2}\\
\les &\frac{M\ep^2}{r|u|^{s-4}}+M^2\ep\int_{|u|}^\ub \frac{d\ub}{r^\frac{s}{2}}\frac{\ep}{|u|^{\frac{s-2}{2}}}+M\int_{|u|}^\ub \frac{d\ub}{r^2} \frac{\ep}{|u|^\frac{s-4}{2}}\frac{\ep}{|u|^\frac{s-4}{2}}\\
\les &\frac{M \ep^2}{r|u|^{s-4}}+\frac{M^2\ep^2}{|u|^{s-2}}+\frac{M\ep^2}{r|u|^{s-4}}\\
\les &\frac{\ep_0^2}{|u|^{s-4}}.
\end{split}
\end{align}
Next, we proceed as in Proposition 4.8 in \cite{ShenMink} to obtain
\begin{align}
\begin{split}\label{esterrbrho}
\int_V r^{4}|\dkbbco||\err[\nab_3\dkbbco]|+r^4|(\dkbrhoco,\dkbsico)|| \err[\nab_4\dkbrhoco,\nab_4\dkbsico]|\les\frac{\ep_0^2}{|u|^{s-4}}.
\end{split}
\end{align}
Combining \eqref{brint}, \eqref{estflbrho} and \eqref{esterrbrho}, we deduce
\begin{equation}
    \int_{\cuv} r^{4} |\dkbbco|^2+\int_{\ucuv}r^{4}|(\dkbrhoco,\dkbsico)|^2 \les\frac{\ep_0^2}{|u|^{s-4}}.
\end{equation}
Combining with Lemma \ref{totallycheck}, this concludes the proof of Proposition \ref{estbr}.
\end{proof}
\subsection{Estimates for the Bianchi pair \texorpdfstring{$((\rhoc,\sic),\bbc)$}{}}\label{ssec9.3}
\begin{proposition}\label{estrb}
We have the following estimates:
\begin{equation}\label{estrhobb}
\RR_0[(\rhoc,\sic)]^2+\RRb_0[\bbc]^2+\RR_1[(\rhoc,\sic)]^2+\RRb_1[\bbc]^2\les\ep_0^2.
\end{equation}
\end{proposition}
\begin{proof}
We recall the following equations from Propositions \ref{Bianchieq} and \ref{Bianchieqdkb}:
\begin{align*}
\nab_4\dkbbbco+\tr\chi\,\dkbbbco&=d_1^*(\dkbrhoco,\dkbsico)+\fl[\nab_4\dkbbbco]+\err[\nab_4\dkbbbco],\\
\nab_3(\dkbrhoco,\dkbsico)+\frac{3}{2}\trchb\,(\dkbrhoco,\dkbsico)&=-d_1\dkbbbco+\fl[\nab_3\dkbrhoco,\nab_3\dkbsico]+\err[\nab_3\dkbrhoco,\nab_3\dkbsico],
\end{align*}
where
\begin{align*}
\fl[\nab_3\dkbrhoco,\nab_3\dkbsico]=\Gab^{(1)}\c\O_3^1+\O_3^2\c\aac^{(1)},\qquad \fl[\nab_4\dkbbco]=\Gag^{(1)}\c\O_3^1+\O_2^1\c\bbc^{(1)}.
\end{align*}
Applying \eqref{casethree} with $\psi_{(1)}=(\dkbrhoco,\dkbsico)$, $\psi_{(2)}=\dkbbbco$, $a_{(1)}=1$, $a_{(2)}=\frac{3}{2}$, $h_{(1)}=\fl[\nab_3\dkbrhoco,\nab_3\dkbsico]+\err[\nab_3\dkbrhoco,\nab_3\dkbsico]$, $h_{(2)}=\fl[\nab_4\dkbbbco]+\err[\nab_4\dkbbbco]$ and $p=2$, we obtain
\begin{align*}
&\int_{\cuv} r^2|(\dkbrhoco,\dkbsico)|^2+\int_{\ucuv}r^2|\dkbbbco|^2 \\
\les &\int_{\Si_0\cap V}r^2\left(|(\dkbrhoco,\dkbsico)|^2+|\dkbbbco|^2\right)+\int_V r|(\dkbrhoco,\dkbsico)|^2\\
&+\int_V r^2|(\dkbrhoco,\dkbsico)||\fl[\nab_3\dkbrhoco,\nab_3\dkbsico]+\err[\nab_3\dkbrhoco,\nab_3\dkbsico]|\\
&+\int_V r^2|\dkbbbco||\fl[\nab_4\dkbbbco]+\err[\nab_4\dkbbbco]|.
\end{align*}
First, we have
\begin{align}
\begin{split}\label{initialbr}
    \int_{\Si_0\cap V}r^2\left(|(\dkbrhoco,\dkbsico)|^2+|\dkbbbco|^2\right)&\les |u|^{2-s}\int_{\Si_0\cap V}r^s\left(|(\dkbrhoco,\dkbsico)|^2+|\dkbbbco|^2\right)\les\frac{\ep_0^2}{|u|^{s-2}}.
\end{split}
\end{align}
Notice that we have from Proposition \ref{estbr} that
\begin{align}
\begin{split}\label{bulkr}
\int_V r|(\dkbrhoco,\dkbsico)|^2 \les \int_{|u|}^{\ub}r^{-3}d\ub\int_\ucuv r^4|(\dkbrhoco,\dkbsico)|^2\les\int_{|u|}^{\ub}r^{-3}d\ub \frac{\ep_0^2}{|u|^{s-4}}
\les\frac{\ep_0^2}{|u|^{s-2}}.
\end{split}
\end{align}
Moreover, we have
\begin{align}
\begin{split}\label{estflrhob}
&\int_V r^2|(\dkbrhoco,\dkbsico)||\fl[\nab_3\dkbrhoco,\nab_3\dkbsico]|+r^2|\dkbbbco||\fl[\nab_4\dkbbbco]|\\
\les& \int_V Mr^{-1}|\rhoc^{(1)}||\Gab^{(1)}|+M^2r^{-1}|\rhoc^{(1)}||\aac^{(1)}|\\
\les & M\int_{|u|}^\ub \frac{d\ub}{r^3}\int_{\ucuv} |r^2\rhoc^{(1)}||\Gab^{(1)}|+M^2\int_{|u|}^\ub \frac{d\ub}{r^3}\int_{\ucuv}|r^2\rhoc^{(1)}||\aac^{(1)}|\\
\les& M\int_{|u|}^\ub\frac{d\ub}{r^3} \left(\int_\ucuv |r^2\rhoc^{(1)}|^2\right)^\frac{1}{2}\left(\int_{u_0(\ub)}^u du|\Gab^{(1)}|^2_{2,S}\right)^\frac{1}{2}\\
&+M^2\int_{|u|}^\ub \frac{d\ub}{r^3}\left(\int_{\ucuv}|r^2\rhoc^{(1)}|^2\right)^\frac{1}{2}\left(\int_\ucuv|\aac^{(1)}|^2\right)^\frac{1}{2}\\
\les& M\int_{|u|}^\ub\frac{d\ub}{r^3} \frac{\ep}{|u|^{\frac{s-4}{2}}}\frac{\ep}{|u|^\frac{s-2}{2}}+M^2\int_{|u|}^\ub \frac{d\ub}{r^3} \frac{\ep}{|u|^\frac{s-4}{2}}\frac{\ep}{|u|^\frac{s}{2}}\\
\les&\frac{M\ep^2}{|u|^{s-1}}+\frac{M^2\ep^2}{|u|^{s}}\\
\les&\frac{\ep_0^2}{|u|^{s-2}}.
\end{split}
\end{align}
Next, we have from Proposition 4.12 in \cite{ShenMink} that
\begin{equation}\label{esterrrhob}
    \int_V r^{2}|(\dkbrhoco,\dkbsico)||\err[\nab_3\dkbrhoco,\nab_3\dkbsico]|+r^2|\dkbbbco||\err[\nab_4\dkbbbco]|\les\frac{\ep_0^2}{|u|^{s-2}}.
\end{equation}
Combining \eqref{initialbr}, \eqref{bulkr}, \eqref{estflrhob} and \eqref{esterrrhob}, we obtain
\begin{align*}
    \int_{\cuv}r^2|(\dkbrhoco,\dkbsico)|^2+\int_{\ucuv}r^2|\dkbbbco|^2\les\frac{\ep_0^2}{|u|^{s-2}}.
\end{align*}
Combining with Lemma \ref{totallycheck}, this concludes the proof of Proposition \ref{estrb}.
\end{proof}
\subsection{Estimates for the Bianchi pair \texorpdfstring{$(\bbc,\aac)$}{}}\label{ssec9.4}
\begin{proposition}\label{estba}
We have the following estimate:
\begin{equation}\label{estbbaa}
    \RR_0[\bbc]^2+\RRb_0[\aac]^2+\RR_1[\bbc]^2+\RRb_1[\aac]^2\les\ep_0^2.
\end{equation}
\end{proposition}
\begin{proof}
We recall the following Bianchi equations from Propositions \ref{Bianchieq} and \ref{Bianchieqdkb}:
\begin{align}
\begin{split}
   \nab_4\dkbaaco+\frac{1}{2}\trch\,\dkbaaco&=-\nab\hot\dkbbbco+\fl[\nab_4\dkbaaco]+\err[\nab_4\dkbaaco],\\
\nab_3\bbc+2\trchb\,\dkbbbco&=-\sdiv\dkbaaco+\fl[\nab_3\dkbbbco]+\err[\nab_3\dkbbbco],\label{babian}
\end{split}
\end{align}
where
\begin{align*}
    \fl[\nab_4\dkbaaco]=\Gab^{(1)}\c\O_3^1+\O_2^1\c\aac^{(1)},\qquad \fl[\nab_3\dkbbbco]=\Gab^{(1)}\c\O_4^2+\O_2^1\c\bbc^{(1)}.
\end{align*}
Applying \eqref{casethree} with $\psi_{(1)}=\dkbbbco$, $\psi_{(2)}=\dkbaaco$, $a_{(1)}=2$, $a_{(2)}=\frac{1}{2}$, $h_{(1)}=\fl[\nab_3\dkbbbco]+\err[\nab_3\dkbbbco]$, $h_{(2)}=\fl[\nab_4\dkbaaco]+\err[\nab_4\dkbaaco]$ and $p=0$, we obtain
\begin{align}
\begin{split}
&\int_{\cuv}|\dkbbbco|^2+\int_{\ucuv} |\dkbaaco|^2 \\
\les&\int_{\Si_0\cap V}|\dkbbbco|^2+|\dkbaaco|^2+\int_V r^{-1}|\dkbbbco|^2\\
&+\int_V |\dkbbbco||\fl[\nab_3\dkbbbco]+\err[\nab_3\dkbbbco]|+|\dkbaaco||\fl[\nab_4\dkbaaco]+\err[\nab_4\dkbaaco]|.\label{bap}
\end{split}
\end{align}
First, we have
\begin{align*}
    \int_{\Si_0\cap V}|\dkbbbco|^2+|\dkbaaco|^2\les |u|^{-s}\int_{\Si_0\cap V}r^s\left(|\dkbbbco|^2+|\dkbaaco|^2\right)\les \frac{\ep_0^2}{|u|^s}.
\end{align*}
We recall from Proposition \ref{estrb} that
\begin{align*}
    \int_V r^{-1}|\dkbbbco|^2\les\int_{|u|}^\ub \frac{d\ub}{r^3}\int_{\ucuv} |r\bbc^{(1)}|^2\les \int_{|u|}^\ub \frac{d\ub}{r^3}\frac{\ep_0^2}{|u|^{s-2}}\les \frac{\ep_0^2}{|u|^s}.
\end{align*}
Next, we estimate
\begin{align*}
    &\int_V |\dkbbbco||\fl[\nab_3\dkbbbco]|+|\dkbaaco||\fl[\nab_4\dkbaaco]|\\
    \les& \int_V Mr^{-3}|\aac^{(1)}||\Gab^{(1)}|+Mr^{-2}|\aac^{(1)}|^2\\
    \les&M\int_{|u|}^\ub\frac{d\ub}{r^3}\left(\int_{\ucuv}|\aac^{(1)}|^2\right)^\frac{1}{2}\left(\int_{u_0(u)}^u|\Gab^{(1)}|^2_{2,S}\right)^\frac{1}{2}+M\int_{|u|}^\ub\frac{d\ub}{r^2}\int_{\ucuv}|\aac^{(1)}|^2\\
    \les&M\int_{|u|}^\ub\frac{d\ub}{r^3}\frac{\ep}{|u|^\frac{s}{2}}\frac{\ep}{|u|^{\frac{s-2}{2}}}+M\int_{|u|}^\ub\frac{d\ub}{r^2}\frac{\ep^2}{|u|^s}\\
    \les&\frac{\ep_0^2}{|u|^s}.
\end{align*}
We recall from Proposition 4.13 in \cite{ShenMink}
\begin{equation*}
    \int_V |\dkbbbco||\err[\nab_3\dkbbbco]|+|\dkbaaco||\err[\nab_4\dkbaaco]|\les\frac{\ep_0^2}{|u|^s}.
\end{equation*}
Hence, we obtain
\begin{align*}
    \int_{\cuv}|\dkbbbco|^2+\int_{\ucuv}|\dkbaaco|^2\les\frac{\ep_0^2}{|u|^s}.
\end{align*}
This concludes the proof of Proposition \ref{estba}.
\end{proof}
\subsection{End of the proof of Theorem M1}\label{proofrp}
\begin{proposition}\label{estRS}
We have the following estimate:
\begin{equation}
    \RR_0^{S}+\RRb_0^{S}\les \RR_0+\RR_1+\RRb_0+\RRb_1.
\end{equation}
\end{proposition}
\begin{proof}
It is sufficient to prove that
\begin{equation}
\RR_0^{S}\lesssim\RR_0+\RR_1,\qquad\RRb_0^{S}\les\RRb_0+\RRb_1,
\end{equation}
which is a direct consequence of Propositions \ref{Bianchieq}, \ref{Bianchieqdkb} and \ref{sobolevkn}.\footnote{Note that the initial conditions in Proposition \ref{sobolevkn} can be controlled as in Lemma 6.6 of \cite{ShenMink}.}
\end{proof}
Remark that Propositions \ref{estab}--\ref{estRS} imply
\begin{equation*}
\RR^2\les \ep_0^2.
\end{equation*}
This concludes the proof of Theorem M1.
\section{Ricci coefficients estimates (Theorem M3)}\label{Ricciestimates}
The goal of this section is to prove Theorem M3, which we recall below for convenience.
\begin{thmM3}
Assume that
\begin{align}\label{assM31}
    \OO_{(0)}\leq\ep_0,\qquad \RR\les\ep_0,\qquad\OO^*(\Cb_*)\les\ep_0,\qquad \OO\leq\ep,\qquad \osc\leq\ep.
\end{align}
Then, we have
\begin{equation}\label{conM34}
    \OO\les\ep_0,\qquad \quad\osc\les\ep_0.
\end{equation}
\end{thmM3}
\begin{remark}
Throughout Section \ref{Ricciestimates}, we denote
\begin{align*}
    S_0(\ub)&:=S(u_0(\ub),\ub),\qquad S_*(u):=S(u,\ub_*),\qquad S:=S(u,\ub).
\end{align*}
\end{remark}
\subsection{Preliminaries}
\subsubsection{Null frame transformation}\label{secnullframe}
In order to control $\osc$, we will need to compare the double null frame of $\kk$ and the one of $\kk_{(0)}$ in $\kk_{(0)}$. To this end, we will rely on the null frame transformations introduced in \cite{KS} and \cite{KS:Kerr1}.
\begin{lemma}\label{changelemma}
Given two null frames $(e_3,e_4,e_1,e_2)$ and $(e'_3,e'_4,e'_1,e'_2)$ associated to double-null foliations, assume that they have the same generator $\Lb$ for the ingoing direction. Then, a null frame transformation from the null frame $(e_3,e_4,e_1,e_2)$ to $(e'_3,e'_4,e'_1,e'_2)$ can be written in the form:
\begin{align}
    \begin{split}\label{change}
        e'_4&=\lambda\left(e_4+f^Be_B+\frac{1}{4}|f|^2e_3\right),\\
        e'_3&=\lambda^{-1}e_3,\\
        e'_A&=e_A+\frac{1}{2}f_Ae_3,
    \end{split}
\end{align}
where 
\begin{align}\label{lambdaOmOm'}
    \la=\frac{\Om}{\Om'}
\end{align}
is a scalar function and $f$ is a $1$--form. Moreover, the inverse transform of \eqref{change} is given by
\begin{align}
\begin{split}
e_4&=\la^{-1}e'_4-f^A e_A'+\frac{\la}{4}|f|^2 e_3',\\
e_3&=\la e'_3,\\
e_A&=e'_A-\frac{\la}{2}f_Ae'_3. \label{change'}
\end{split}
\end{align} 
\end{lemma}
\begin{proof}
Applying Lemma 3.1 in \cite{KS:Kerr1} with $\fb=0$, we obtain \eqref{change}. Notice that \eqref{change'} is a direct consequence of \eqref{change}. Recalling
\begin{align*}
    e'_3=2\Om' \Lb,\qquad e_3=2\Om \Lb,
\end{align*}
we deduce immediately \eqref{lambdaOmOm'}. This concludes the proof of Lemma \ref{changelemma}.
\end{proof}
\begin{proposition}\label{transformation}
Under the null frame transformation \eqref{change}, some of the Ricci coefficients transform as follows:
\begin{align*}
\la^{-1}\trch'&=\trch+\sdiv'f+\lot,\\
\curl'f&=\lot,\\
\la \chib'&=\chib,\\
\eta'&=\eta+\frac{1}{2}\la\nab_{e'_3}'f+\lot,\\
\etab'&=\etab+\frac{1}{4}\trchb\, f+\lot,\\
\la^{-2}\xi'&=\xi+\frac{1}{2}\nab_{\la^{-1}e'_4}'f+\frac{1}{4}\trch\, f+\lot ,\\
\la^{-1}\om'&=\om+\lot,
\end{align*}
where the terms denoted by $\lot$ have the following schematic form:
\begin{align*}
    \lot=\Gab\cdot f+\O_2^1\c f.
\end{align*}
The curvature components transform as follows:
\begin{align*}
\lambda^{-2}\alpha'&=\alpha+(f\widehat{\otimes}\beta-{^*f}\widehat{\otimes}{^*\beta})+\left(f\widehat{\otimes}f-\frac{1}{2}{^*f}\widehat{\otimes}{^*f}\right)\rho+\frac{3}{2}(f\widehat{\otimes}{^*f}){^*\rho}+l.o.t.,\\
\lambda^{2}\aa'&=\aa,\\
\lambda^{-1}\beta'&=\beta+\frac{3}{2}(f\rho+{^*f}\sigma)+\lot,\\
\lambda\bb'&=\bb-\frac{1}{2}\aa\cdot f,\\
\rho'&=\rho-f\cdot\bb+\lot,\\
\sigma'&=\sigma-f\cdot{^*\bb}+\lot,
\end{align*}
where the terms denoted by $\lot$ have the following schematic form:
\begin{equation}
   \lot=O(f^3)(\rho,\sigma)+O(f^2)(\alpha,\beta,\bb,\aa).
\end{equation}
\end{proposition}
\begin{proof}
All the transformation formulae, with the exception of the one for $\om$, follow immediately from Proposition 3.3 in \cite{KS:Kerr1}, in view of the fact that $\fb=0$. See Proposition 6.2 in \cite{ShenMink} for the proof of the transformation formula for $\om$ which uses in addition that $\la=\frac{\Om}{\Om'}$.
\end{proof}
\subsubsection{Difference of the pull backs of Kerr values}
The goal of this section is to compare the difference of the pull backs of Kerr values by $\Phi$ and $\Phi_{(0)}$ in $\kk_{(0)}$ which will appear in the control of $\osc$. In the remainder of this paper, for any quantity $X$, we denote the quantity $X_{(0)}$ associated to the initial data layer region $\kk_{(0)}$ by
\begin{align}\label{primenotation1}
    X':=X_{(0)},
\end{align}
to simplify the notations. For example:
\begin{equation}\label{primenotation2}
    u':=u_{(0)},\qquad \eta':=\eta_{(0)},\qquad \OO':=\OO_{(0)},\qquad \mathfrak{R}':=\mathfrak{R}_{(0)},\qquad S':=S_{(0)},\qquad \kk':=\kk_{(0)}.
\end{equation}
\begin{lemma}\label{estKK'}
Let $X$ be a scalar defined in $\kk'$. We denote
\begin{align*}
    X'_K:={\Phi'}^*(X_{Kerr}),\qquad X_K:=\Phi^*(X_{Kerr}),
\end{align*}
and we assume that
\begin{align*}
    X_{K}=\O_q^k.
\end{align*}
Then, we have:
\begin{align*}
    \sup_{S}|X'_K-X_K|\les \frac{M^k\ep}{r^{q+\frac{s-1}{2}}}\les\frac{\ep_0}{r^{q-k+\frac{s-1}{2}}}.
\end{align*}
\end{lemma}
\begin{proof}
Let $p\in\KK$ which has the coordinates $(u,\ub,x^1,x^2)$ and resp. $(u',\ub,{x'}^1,{x'}^2)$ in two different foliations. Then, we have
\begin{align}
\begin{split}\label{estKK'scalar}
    |X'_K(p)-X_K(p)|&\les |{\Phi'}^*(X_{Kerr})(p)-{\Phi}^*(X_{Kerr})(p)|\\
    &\les|X_{Kerr}(\Phi'(p))-X_{Kerr}(\Phi(p))|\\
    &\les\left|\frac{\pr X_{Kerr}}{\pr u_{Kerr}}(u'-u)\right|+\left|\frac{\pr X_{Kerr}}{\pr x^A_{Kerr}}({x'}^A-x^A)\right|\\
    &\les\frac{M^k}{r^{q+1}}|u'-u|+\frac{M^k}{r^q}\left|{x'}^A-x^A\right|\\
    &\les\frac{M^k}{r^{q+1}}\frac{\osc(u)}{r^\frac{s-3}{2}}+\frac{M^k}{r^q}\frac{\osc(x)}{r^\frac{s-1}{2}}\\
    &\les\frac{M^k}{r^{q+1}}\frac{\ep}{r^\frac{s-3}{2}}+\frac{M^k}{r^q}\frac{\ep}{r^\frac{s-1}{2}}\\
    &\les\frac{\ep_0}{r^{q-k+\frac{s-1}{2}}},
\end{split}
\end{align}
where we used \eqref{assM31} to bound $\osc(u)$ and $\osc(x)$. This concludes the proof of Lemma \ref{estKK'}.
\end{proof}
\subsubsection{Evolution lemma}
\begin{lemma}[Evolution lemma]\label{evolution}
Assume that the spacetime $\KK$ is foliated by $(u,\ub)$.\\ \\
(1) Let $U,F$ be $k$-covariant $S$--tangent tensor fields satisfying the outgoing evolution equation
\begin{equation}
    \nab_4 U +\la_0\trch U= F,
\end{equation}
where $\lambda_0\geq 0$. Denoting $\lambda_1=2(\lambda_0-\frac{1}{p})$, we have along $C_u$
\begin{equation}\label{transubU}
|r^{\la_1}U|_{p,S}\les |r^{\la_1}U|_{p,S_*(u)}+\int_{\ub(p)}^{\ub_*(p)}  |r^{\la_1}F|_{p,S}(u,\ub')d\ub',
\end{equation}
where $\ub_*(p)$ is such that $C_{u(p)}\cap \Cb_*=S(u(p),\ub_*(p))$ and where $S_*(u)=S(u(p),\ub_*(p))$.\\\\
(2) Let $V,\Fb$ be k-covariant $S$--tangent tensor fields satisfying the ingoing evolution equation
\begin{equation}
\nab_3 V+\lambda_0\trchb V=\Fb.
\end{equation}
Denoting $\lambda_1 = 2(\lambda_0 - \frac{1}{p})$, we have along $\Cb_\ub$
\begin{equation}\label{transuV}
    |r^{\lambda_1}V|_{p,S} \les |r^{\lambda_1}V|_{p,S_0(\ub)}+ \int_{u_0(\ub)}^{u}  |r^{\lambda_1}\underline{F}|_{p,S}(u',\ub)du',
\end{equation}
where $S(u_0(\ub),\ub) := C_{u_0(\ub)} \cap {\Cb_\ub} \subset \mathcal{K}$ is the unique sphere of $\Cb_\ub$ which is located in the future of $\Si_0$ and touches $\Si_0$ and where $S_0(\ub)=S(u_0(\ub),\ub)$.
\end{lemma}
\begin{proof}
See Lemma 4.1.5 in \cite{Kl-Ni}.
\end{proof}
\subsubsection{Bootstrap control in the initial data layer}
\begin{lemma}\label{estGagbO}
We denote
\begin{equation}
    \Gaw:=\Gag\cup\Gab\cup\Gag'\cup\Gab'.
\end{equation}
Then, we have the following estimates in $\KK'$:
\begin{align}
\begin{split}
    \sup_{\kk'} r^{\frac{s+1}{2}}|\Gaw|&\les\ep,\\
    \sup_{S\subset \kk'} r^\frac{s+1}{2}|r^{-\frac{2}{p}}\Gaw^{(1)}|_{p,S}&\les\ep,\qquad p\in [2,4],
\end{split}
\end{align}
Moreover, we have 
\begin{align*}
    f\in r\Gaw,\qquad\la\in r\Gaw, \qquad u'-u\in r^2\Gaw,\qquad r'-r\in r^2\Gaw.
\end{align*}
\end{lemma}
\begin{proof}
It follows directly from Lemma \ref{estGagba} and the assumption \eqref{assM31}.
\end{proof}
\subsection{Estimates for \texorpdfstring{$\OO_{[1]}$}{}, \texorpdfstring{$\OO_\ga$}{} and \texorpdfstring{$\osc$}{} }\label{Ricci01}
In the remainder of Section \ref{Ricciestimates}, we always assume $p\in [2,4]$.
\subsubsection{Estimate for \texorpdfstring{$\okk_{0,1}(\Omomc)$}{}}
\begin{proposition}\label{estOOb0}
We have the following estimate for $q=0,1$:
\begin{equation}\label{OObep0}
|r^{\frac{s+1}{2}-\frac{2}{p}}(r\nab)^q\Omomc|_{p,S_0(\ub)}\les\ep_0,
\end{equation}
where we recall that $S_0(\ub)$ is the unique leaf of $\Cb_\ub$ which is located in the future of $\Si_0$ and touches $\Si_0$.
\end{proposition}
\begin{proof}
We recall from Proposition \ref{transformation} and $\la=\frac{\Om}{\Om'}$ that
\begin{align}\label{changeOmom}
    \Om'\om'=\Om\om+f\c\Gaw+\O_2^1\c f.
\end{align}
Notice from Proposition \ref{decayGamma} and Lemma \ref{estKK'} that
\begin{align*}
    |\Om_K\om_K-\Om_K'\om_K'|\les\frac{\ep_0}{r^\frac{s+1}{2}}.
\end{align*}
Hence, we have
\begin{align*}
    |\widecheck{\Om'\om'}-\widecheck{\Om\om}|&\les|\Om'\om'-\Om\om|+|\Om'_K\om'_K-\Om_K\om_K|\\
    &\les|\Gaw\c f|+|\O_2^1\c f|+\frac{\ep_0}{r^\frac{s+1}{2}}\\
    &\les\frac{\ep_0}{r^\frac{s+1}{2}}.
\end{align*}
Thus, we infer
\begin{align*}
    |r^{\frac{s+1}{2}-\frac{2}{p}}\widecheck{\Om\om}|_{p,S_0(\ub)}\les |r^{\frac{s+1}{2}-\frac{2}{p}}\widecheck{\Om'\om'}|_{p,S_0(\ub)}+\ep_0\les\ep_0.
\end{align*}
We have from \eqref{changeOmom} and \eqref{change'}
\begin{align*}
    e_A(\Om\om)&=e'_A(\Om'\om')-\frac{\la}{2}f_Ae_3'(\Om'\om')+\Gaw^{(1)}\c f+\Gaw\c f^{(1)}+\O_2^1\c f^{(1)}\\
    &=e'_A(\Om'\om')+\Gaw\c\Gaw^{(1)}+\O_2^1\c\Gaw^{(1)}.
\end{align*}
Notice that we have from Lemma \ref{estKK'}
\begin{equation*}
    |e_A(\Om_K\om_K)-e_A'(\Om'_K\om'_K)|\les\frac{\ep_0}{r^\frac{s+3}{2}}. 
\end{equation*}
Hence, we have
\begin{align*}
    |r^{-\frac{2}{p}}e_A(\Omomc)|_{p,S_0(\ub)}\les |e'_A(\widecheck{\Om'\om'})|_{\infty,S_0(\ub)}+\frac{M}{r^2}\frac{\ep}{r^\frac{s+1}{2}}+\frac{\ep^2}{r^{s+1}}+\frac{\ep_0}{r^\frac{s+3}{2}}\les\frac{\ep_0}{r^\frac{s+3}{2}},
\end{align*}
which implies
\begin{equation*}
    |r^{\frac{s+3}{2}-\frac{2}{p}}e_A(\Omomc)|_{p,S_0(\ub)}\les\ep_0.
\end{equation*}
This concludes the proof of Proposition \ref{estOOb0}.
\end{proof}
\subsubsection{Estimates for \texorpdfstring{$\OO_{0}(\widecheck{\Om\om})$}{} and \texorpdfstring{$\OO_{0}(\widecheck{\Om\omb})$}{}}\label{omcombc}
\begin{proposition}\label{propomcombc}
We have the following estimates:
\begin{align}
    |r^{2+\frac{s-3}{6}-\frac{2}{p}}|u|^{\frac{s-3}{3}}\widecheck{\Om\om}|_{p,S}&\les\ep_0,\label{estomc}\\
    |r^{2-\frac{2}{p}}|u|^{\frac{s-3}{2}}\widecheck{\Om\omb}|_{p,S}&\les\ep_0.\label{estombc}
\end{align}
\end{proposition}
\begin{proof}
We recall from Proposition \ref{nullstructure}
\begin{align*}
\nab_4(\widecheck{\Om\omb})=&\O_0^0\c\rhoc+\Gag\c\O_2^1+\Gag\c\Gag+r\Gag\c\rhoc.
\end{align*}
Applying Lemma \ref{evolution}, we infer
\begin{align*}
    |r^{-\frac{2}{p}}\widecheck{\Om\omb}|_{p,S}&\les|r^{-\frac{2}{p}}\widecheck{\Om\omb}|_{p,S_*(u)}+\int_{\ub}^{\ub_*}\left(|r^{-\frac{2}{p}}\rhoc|_{p,S}+\frac{M}{r^2}|r^{-\frac{2}{p}}\Gag|_{p,S}+|r^{-\frac{2}{p}}\Gag\c\Gag|_{p,S}\right) d\ub \\
    &\les\frac{\ep_0}{r^2|u|^{\frac{s-3}{2}}}+\int_{\ub}^{\ub_*}\left(\frac{\ep_0}{r^3|u|^{\frac{s-3}{2}}}+\frac{\ep M}{r^4|u|^{\frac{s-3}{2}}}+\frac{\ep^2}{r^4|u|^{s-3}}\right)d\ub\\
    &\les\frac{\ep_0}{r^2|u|^{\frac{s-3}{2}}},
\end{align*}
which implies \eqref{estombc}. Next, we recall from Proposition \ref{nullstructure}
\begin{align*}
\nab_3(\widecheck{\Om\om})=&\O_0^0\c\rhoc+\Gag\c\O_2^1+\Gag\c\Gag+r\Gag\c\rhoc.
\end{align*}
Applying Lemma \ref{evolution} and Proposition \ref{estOOb0}, we deduce
\begin{align*}
    |r^{-\frac{2}{p}}\widecheck{\Om\om}|_{p,S}&\les|r^{-\frac{2}{p}}\widecheck{\Om\om}|_{p,S_0(\ub)}+\int_{u_0(\ub)}^{u}\left(|r^{-\frac{2}{p}}\rhoc|_{p,S}+\frac{M}{r^2}|r^{-\frac{2}{p}}\Gag|_{p,S}+|r^{-\frac{2}{p}}\Gag\c\Gag|_{p,S}\right)du \\
    &\les\frac{\ep_0}{r^{\frac{s+1}{2}}}+\int_{u_0(\ub)}^{u}\left(\frac{\ep_0}{r^3|u|^{\frac{s-3}{2}}}+\frac{\ep M}{r^3|u|^{\frac{s-1}{2}}}+\frac{\ep^2}{r^4|u|^{s-3}}\right)du\\
    &\les\frac{\ep_0}{r^{2+\frac{s-3}{6}}|u|^\frac{s-3}{3}},
\end{align*}
which implies \eqref{estomc}. This concludes the proof of Proposition \ref{propomcombc}.
\end{proof}
\subsubsection{Estimates for \texorpdfstring{$\OO_1(\widecheck{\Om\om})$}{} and \texorpdfstring{$\OO_1(\widecheck{\Om\omb})$}{}}\label{nabomcombc}
\begin{proposition}\label{propnabomcombc}
We have the following estimates:
\begin{align}
    |r^{3+\frac{s-3}{6}-\frac{2}{p}}|u|^{\frac{s-3}{3}}\nab\widecheck{\Om\om}|_{p,S}(u,\ub)&\les\ep_0,\label{estomc1}\\
    |r^{2-\frac{2}{p}}|u|^{\frac{s-1}{2}}\nab\widecheck{\Om\omb}|_{p,S}(u,\ub)&\les\ep_0.\label{estombc1}
\end{align}
\end{proposition}
\begin{proof}
We recall from Propositions \ref{nullstructure}
\begin{align*}
    \Om\nab_4(\widecheck{\Om\omb})=\f12\widecheck{\Om^2\rho}+\Gag\c \O_2^1+\Gag\c\Gag.
\end{align*}
Differentiating it by $r\nab$ and applying Corollary \ref{commutation}, we infer
\begin{align}\label{nab4ombc}
    \nab_4(r\nab\widecheck{\Om\omb})=\frac{\Om}{2}r\nab\rhoc+\Gag^{(1)}\c \O_2^1+\Gag\c\Gag^{(1)}.
\end{align}
Since we cannot control $\nab\rhoc$ in $L^p$ directly, we use a renormalization argument. We recall from Proposition \ref{Bianchieq}
\begin{align*}
    \nab_4\bbc+\trch\,\bbc=-\nab\rhoc+{^*\nab}\sic+\Gab^{(1)}\c \O_3^1+r^{-1}\Gag\c\Gab^{(1)},
\end{align*}
which implies from \eqref{e4r}
\begin{align}\label{nab4rbbc}
    \nab_4(\Om r\bbc)+\f12\trch (\Om r\bbc)=-\Om r\nab\rhoc+\Om r{^*\nab}\sic+\Gab^{(1)}\c\O_2^1+\Gag\c\Gab^{(1)}.
\end{align}
We define $\ombc^\dagger$ by
\begin{align*}
    \nabla_4\ombc^\dagger=\f12\Om\sic,\qquad\quad \ombc^\dagger=0\quad \mbox{ on }\Cb_*.
\end{align*}
Applying Proposition \ref{commkn}, we deduce
\begin{align}\label{nab4ombcdagger}
    \Om\nab_4({r{^*\nab}\ombc^\dagger})&=r{^*\nab}(\Om\nab_4\ombc^\dagger)+[\Om\nab_4,r{^*\nab}]\ombc^\dagger=\f12\Om r{^*\nabla}\sic+\Gag\c r\nab\ombc^\dagger+\Gag\c\Gag^{(1)}.
\end{align}
We introduce a new quantity
\begin{equation*}
    \vkab:=-r\nabla\widecheck{\Om\omb}+r{^*\nabla\ombc^\dagger}-\f12 r\bbc.
\end{equation*}
Then, we obtain from \eqref{nab4ombc}, \eqref{nab4rbbc} and \eqref{nab4ombcdagger}
\begin{align}
\begin{split}\label{eqvarkappa}
        \nab_4\underline{\varkappa}&=-\nab_4(r\nab\widecheck{\Om\omb})+\nab_4(r{^*\nab\ombc^\dagger})-\f12\nab_4 (r\bbc)\\
        &=\Gab^{(1)}\c O\left(\frac{M}{r^2}\right)+\Gag\c\Gab^{(1)}+\Gag\c r\nab\ombc^\dagger.    
\end{split}
\end{align}
Applying Lemma \ref{evolution}, we obtain
\begin{align*}
|r^{-\frac{2}{p}}\underline{\varkappa}|_{p,S}&\les|r^{-\frac{2}{p}}\underline{\varkappa}|_{p,S_*(u)}+\int_{\ub}^{\ub_*}\left(|r^{-\frac{2}{p}}\Gag\c\Gab^{(1)}|_{p,S}+\frac{M}{r^2}|r^{-\frac{2}{p}}\Gab^{(1)}|_{p,S}+|r^{-\frac{2}{p}}\Gag\c r\nab\ombc^\dagger|_{p,S}\right)d\ub\\
&\les\frac{\ep_0}{r|u|^\frac{s-1}{2}}+\int_{\ub}^{\ub_*}\left(\frac{\ep^2}{r^3|u|^{s-2}}+\frac{M}{r^3}\frac{\ep}{|u|^\frac{s-1}{2}}+\frac{\ep}{r^2|u|^{\frac{s-3}{2}}}|r^{1-\frac{2}{p}}\nab\ombc^\dagger|_{p,S}\right)d\ub\\
&\les\frac{\ep_0}{r|u|^\frac{s-1}{2}}+\int_{\ub}^{\ub_*}\left(\frac{\ep}{r^2|u|^{\frac{s-3}{2}}}|r^{1-\frac{2}{p}}\nab\ombc^\dagger|_{p,S}\right)d\ub.
\end{align*}
Denoting $\langle\omb\rangle:=(\widecheck{\Om\omb},\ombc^\dagger)$, then\footnote{Recall that $d_1^*(f,g)=-\nabla f+{^*\nab}g.$}
\begin{align}
    \begin{split}
    d_1^*\langle\omb\rangle&=r^{-1}\underline{\varkappa}+\frac{1}{2}\bbc.\label{hoho}
    \end{split}
\end{align}
Applying Proposition \ref{ellipticest} to \eqref{hoho}, we obtain
\begin{align*}
|r^{1-\frac{2}{p}}\nab\Omombc|_{p,S}+|r^{1-\frac{2}{p}}\nab\ombc^\dagger|_{p,S}&\les|r^{-\frac{2}{p}}\underline{\varkappa}|_{p,S}+|r^{1-\frac{2}{p}}\bb|_{p,S}\\
&\les\frac{\ep_0}{r|u|^{\frac{s-1}{2}}}+\int_{\ub}^{\ub_*}\left(\frac{\ep}{r^2|u|^{\frac{s-3}{2}}}|r^{1-\frac{2}{p}}\nab\ombc^\dagger|_{p,S}\right)d\ub.
\end{align*}
By Gronwall inequality, we deduce
\begin{align*}
    |r^{1-\frac{2}{p}}\nab\Omombc|_{p,S}+|r^{1-\frac{2}{p}}\nab\ombc^\dagger|_{p,S}\les\frac{\ep_0}{r|u|^{\frac{s-1}{2}}},
\end{align*}
which implies \eqref{estombc1}. \\ \\
Next, we recall from Proposition \ref{nullstructure}
\begin{align*}
    \Om\nab_3(\Omomc)=\frac{1}{2}\widecheck{\Om^2\rho}+\Gag\c\O_2^1+\Gag\c\Gag.
\end{align*}
Differentiating it by $r\nab$ and applying Corollary \ref{commutation}, we infer
\begin{align}\label{nab3omc}
    \nab_3(r\nab\widecheck{\Om\om})=\frac{\Om}{2}r\nab\rhoc+\Gag^{(1)}\c \O_2^1+\Gag\c\Gag^{(1)}.
\end{align}
We recall from Proposition \ref{Bianchieq}
\begin{align*}
    \nab_3\bc+\trchb\,\bc=\nab\rhoc+{^*\nab}\sic+\Gag^{(1)}\c\O_3^1+r^{-1}\Gag\c\Gab^{(1)},
\end{align*}
which implies from \eqref{e3r}
\begin{align}\label{nab3Omrb}
    \nab_3(\Om r\bc)+\f12\trchb (\Om r\bc)=\Om r\nab\rhoc+\Om r{^*\nab}\sic+\Gag^{(1)}\c\O_2^1+\Gag\c\Gab^{(1)}.
\end{align}
We define $\omc^\dagger$ by
\begin{align*}
    \nab_3\omc^\dagger=\f12\Om\sic,\qquad\quad\omc^\dagger=0\quad \mbox{ on }S_0(\ub)\;\;\; \forall \ub.
\end{align*}
Applying Proposition \ref{commkn}, we deduce
\begin{align}\label{nab3omcdagger}
    \nab_3({r{^*\nabla}\omc^\dagger})=\f12 r{^*\nabla}\sic+\Gab\c r\nab\omc^\dagger+\Gag\c\Gab^{(1)}.
\end{align}
Combining \eqref{nab3omc}, \eqref{nab3Omrb} and \eqref{nab3omcdagger}, we obtain
\begin{align}
\begin{split}\label{eqvarka}
    \nab_3\varkappa=\Gab^{(1)}\c\O_2^1+\Gag\c\Gab^{(1)}+\Gab\c r\nab\omc^\dagger,    
\end{split}
\end{align}
where
\begin{equation}
    \varkappa:=r\nabla\widecheck{\Om\om}+r{^*\nabla\omc^\dagger}-\f12 r\bc.
\end{equation}
Applying Lemma \ref{evolution}, we infer
\begin{align*}
    |r^{1-\frac{2}{p}}\varkappa|_{p,S}&\les |r^{1-\frac{2}{p}}\varkappa|_{p,S_0(\ub)}+\int_{-\ub}^u\left(|r^{-\frac{2}{p}}\Gag\c\Gab^{(1)}|_{p,S}+\frac{M}{r^2}|r^{-\frac{2}{p}}\Gab^{(1)}|_{p,S}+|r^{-\frac{2}{p}}\Gab\c r\nab\omc^\dagger|_{p,S}\right)du \\
    &\les\frac{\ep_0}{r^\frac{s-1}{2}}+\int_{-\ub}^u \left(\frac{\ep^2}{r^3|u|^{s-2}}+\frac{M\ep}{r^3|u|^\frac{s-1}{2}}+\frac{\ep}{r|u|^\frac{s-1}{2}}|r^{1-\frac{2}{p}}\nab\omc^\dagger|_{p,S}\right)d\ub \\
    &\les \frac{\ep_0}{r^{2+\frac{s-3}{6}}|u|^\frac{s-3}{3}}+\int_{-\ub}^u \frac{\ep}{r|u|^\frac{s-1}{2}}|r^{1-\frac{2}{p}}\nab\omc^\dagger|_{p,S}d\ub.
\end{align*}
Applying Proposition \ref{ellipticest}, we obtain
\begin{align*}
    |r^{2-\frac{2}{p}}\nab\Omomc|_{p,S}+|r^{2-\frac{2}{p}}\nab\omc^\dagger|_{p,S}&\les |r^{1-\frac{2}{p}}\varkappa|_{p,S}+|r^{2-\frac{2}{p}}\bc|_{p,S}+\int_{-\ub}^u \frac{\ep}{r|u|^\frac{s-1}{2}}|r^{1-\frac{2}{p}}\nab\omc^\dagger|_{p,S} d\ub \\
    &\les \frac{\ep_0}{r^{2+\frac{s-3}{6}}|u|^\frac{s-3}{3}}+\int_{-\ub}^u \frac{\ep}{r|u|^\frac{s-1}{2}}|r^{1-\frac{2}{p}}\nab\omc^\dagger|_{p,S}d\ub.
\end{align*}
By Gronwall inequality, we deduce
\begin{align*}
    |r^{2-\frac{2}{p}}\nab\Omomc|_{p,S}+|r^{2-\frac{2}{p}}\nab\omc^\dagger|_{p,S}\les\frac{\ep_0}{r^{2+\frac{s-3}{6}}|u|^\frac{s-3}{3}},
\end{align*}
which implies \eqref{estomc1}. This concludes the proof of Proposition \ref{propnabomcombc}.
\end{proof}
\subsubsection{Estimate for \texorpdfstring{$\OO_{0,1}(\Omc)$}{}}
\begin{proposition}\label{estOmc0}
We have the following estimate for $q=0,1$:
\begin{equation}
    \left|r^{1-\frac{2}{p}}(r\nab)^q|u|^{\frac{s-3}{2}}\Omc\right|_{p,S} \les\ep_0.
\end{equation}
\end{proposition}
\begin{proof}
We recall from Proposition \ref{metriceqre}
\begin{align*}
    \nab_4\Omc=&\O_0^0\c\widecheck{\Om\om}+\O_1^1\c\Gag+r\Gag\c\Gag,
\end{align*}
which implies from Corollary \ref{commutation}
\begin{align*}
    \nab_4(r\nab\Omc)=&\O_0^0\c r\nab(\widecheck{\Om\om})+\O_1^1\c\Gag^{(1)}+r\Gag\c\Gag^{(1)}.
\end{align*}
Applying Lemma \ref{evolution}, we have
\begin{align*}
|r^{-\frac{2}{p}}\Omc|_{p,S}&\les|r^{-\frac{2}{p}}\Omc|_{p,S_*(u)}+\int_\ub^{\ub_*}\left(|r^{-\frac{2}{p}}\omc|_{p,S}+\frac{M}{r}|r^{-\frac{2}{p}}\Gag|_{p,S}+|r^{1-\frac{2}{p}}\Gag\c\Gag|\right)d\ub\\
&\les\frac{\ep_0}{r|u|^\frac{s-3}{2}}+\int_\ub^{\ub_*}\left(\frac{\ep_0}{r^{2}|u|^\frac{s-3}{2}}+\frac{M}{r^2}\frac{\ep}{r|u|^{\frac{s-3}{2}}}+\frac{\ep^2}{r^3|u|^{s-3}}\right)d\ub\\
&\les \frac{\ep_0}{r|u|^\frac{s-3}{2}},
\end{align*}
and similarly
\begin{align*}
|r^{1-\frac{2}{p}}\nab\Omc|_{p,S}&\les|r^{1-\frac{2}{p}}\nab\Omc|_{p,S_*(u)}+\int_\ub^{\ub_*}\left(|r^{1-\frac{2}{p}}\nab\Omomc|_{p,S}+\frac{M}{r}|r^{-\frac{2}{p}}\Gag^{(1)}|_{p,S}+|r^{1-\frac{2}{p}}\Gag\c\Gag^{(1)}|\right)d\ub\\
&\les \frac{\ep_0}{r|u|^\frac{s-3}{2}}.
\end{align*}
This concludes the proof of Proposition \ref{estOmc0}.
\end{proof}
\subsubsection{Estimate for \texorpdfstring{$\osc(u)$}{}}
\begin{proposition}\label{osc0ucontrol}
    We have the following estimate:
    \begin{equation}
        \sup_{\kk'}r^\frac{s-3}{2}|u'-u|\les\ep_0.
    \end{equation}
\end{proposition}
\begin{proof}
We have from \eqref{change}
\begin{align*}
    (\Om'e'_4-\Om'e'_3)(u'-u)=&-1-\left(\Om'e'_4-\Om'e'_3\right)u\\
    =&-1-\Om\left(e_4+f^B e_B+\frac{1}{4}|f|^2 e_3\right)u+\frac{{\Om'}^2}{\Om^2}\\
    =&-\frac{1}{4\Om}|f|^2+\frac{{\Om'}^2-\Om^2}{\Om^2}\\
    =&\O_0^0\c(\Om'-\Om)+r^2\Gaw\c\Gaw.
\end{align*}
Applying Propositions \ref{estOmc0}, \ref{standardsobolev} and Lemma \ref{estKK'}, we obtain
\begin{equation*}
    |\Om'-\Om|\les|\Omc'-\Omc|+\left|\left(\Om_K'-\f12\right)-\left(\Om_K-\f12\right)\right| \les \frac{\ep_0}{r^\frac{s-1}{2}}.
\end{equation*}
Integrating it along $e'_4-e'_3$, we deduce
\begin{align*}
    |u'-u|_{\infty,S'}\les\int_{w}^{\ub_*}\left(\frac{\ep_0}{{r}^{\frac{s-1}{2}}} +\frac{\ep^2}{r^{s-1}}\right) dw\les \frac{\ep_0}{{r}^{\frac{s-3}{2}}}.
\end{align*}
This concludes the proof of Proposition \ref{osc0ucontrol}.
\end{proof}
\subsubsection{Estimate for \texorpdfstring{$\OO_{0,1}(\widecheck{\Om\trch})$}{}}\label{Omtrch}
\begin{proposition}\label{propOmtrch}
We have the following estimate for $q=0,1$:
\begin{align}
    \left|r^{2-\frac{2}{p}}|u|^{\frac{s-3}{2}}(r\nab)^q\widecheck{\Om\trch}\right|_{p,S}&\les\ep_0.\label{estOmtrch} 
\end{align}
\end{proposition}
\begin{proof}
We recall from Proposition \ref{nullstructure}
\begin{align}\label{Omtrchraych}
\nab_4(\widecheck{\Om\trch})+\trch\,\widecheck{\Om\trch}=\O_1^0\c\widecheck{\Om\om}+\Gag\c\O_2^1+\Gag\c\Gag.
\end{align}
Applying Propositions \ref{propomcombc} and  Lemma \ref{evolution}, we infer
\begin{align*}
    |r^{2-\frac{2}{p}}\widecheck{\Om\trch}|_{p,S}&\les |r^{2-\frac{2}{p}}\widecheck{\Om\trch}|_{p,S_*(u)}+\int_\ub^{\ub_*}\left( |r^{1-\frac{2}{p}}\widecheck{\Om\om}|_{p,S}+\frac{M}{r^2}|r^{2-\frac{2}{p}}\Gag|_{p,S}+|r^{2-\frac{2}{p}}\Gag\c\Gag|_{p,S}\right)d\ub\\
    &\les\frac{\ep_0}{|u|^\frac{s-3}{2}}+\int_\ub^{\ub_*}\left( \frac{\ep_0}{r^{1+\frac{s-3}{6}}|u|^{\frac{s-3}{3}}}+\frac{M}{r^2}\frac{\ep}{|u|^\frac{s-3}{2}}+\frac{\ep^2}{r^2|u|^{s-3}}\right)d\ub\\
    &\les\frac{\ep_0}{|u|^{\frac{s-3}{2}}}.
\end{align*}
Next, differentiating \eqref{Omtrchraych} by $r\nab$ and applying Corollary \ref{commutation}, we obtain
\begin{align*}
\nab_4(r\nab\widecheck{\Om\trch})+\trch(r\nab\widecheck{\Om\trch})=\O_0^0\c\nab\widecheck{\Om\om}+\Gag^{(1)}\c\O_2^1+\Gag\c\Gag^{(1)}.
\end{align*}
Applying Lemma \ref{evolution}, we deduce
\begin{align*}
    &\;\;\;\;\,|r^{3-\frac{2}{p}}\nab\widecheck{\Om\trch}|_{p,S}\\
    &\les |r^{3-\frac{2}{p}}\nab\widecheck{\Om\trch}|_{p,S_*(u)}+\int_{\ub}^{\ub_*}\left(|r^{2-\frac{2}{p}}\nab\Omomc|_{p,S}+M|r^{-\frac{2}{p}}\Gag^{(1)}|_{p,S}+|r^{2-\frac{2}{p}}\Gag\c\Gag^{(1)}|_{p,S}\right)d\ub\\
    &\les \frac{\ep_0}{|u|^\frac{s-3}{2}}+\int_{\ub}^{\ub_*}\left(\frac{\ep_0}{r^{1+\frac{s-3}{6}}|u|^\frac{s-3}{3}}+\frac{M}{r^2|u|^\frac{s-3}{2}}+\frac{\ep^2}{r^2|u|^{s-3}}\right)d\ub\\
    &\les\frac{\ep_0}{|u|^\frac{s-3}{2}}.
\end{align*}
This concludes the proof of Proposition \ref{propOmtrch}.
\end{proof}
\subsubsection{Estimate for \texorpdfstring{$\okk_{0}(\widecheck{\Om^{-1}\trchb})$}{}}
\begin{proposition}\label{OOb0Omtrchb}
We have the following estimate:
\begin{align}
    \begin{split}
        |r^{\frac{s+1}{2}-\frac{2}{p}}\widecheck{\Om^{-1}\trchb}|_{p,S_0(\ub)}\les \ep_0.
    \end{split}
\end{align}
\end{proposition}
\begin{proof}
We have from Proposition \ref{transformation}
\begin{align*}
    {\Om'}^{-1}\trchb'=\Om^{-1}\trchb,
\end{align*}
which implies
\begin{equation}\label{diffeOmtrchb}
\widecheck{{\Om'}^{-1}\trchb'}=\widecheck{\Om^{-1}\trch}+\Om^{-1}_K\tr_K\chib_K-{\Om'}_K^{-1}\tr_K\chib'_K.
\end{equation}
Proceeding as in \eqref{estKK'scalar} and applying Propositions \ref{standardnull} and \ref{osc0ucontrol}
\begin{align*}
    \left|\Om^{-1}_K\tr_K\chib_K-{\Om'}_K^{-1}\tr_K\chib'_K\right|&\les \left|\pr_u (\Om^{-1}\tr\chib)_{Kerr}\right||u'-u|+\left|\pr_{x^A}(\Om^{-1}\tr\chib)_{Kerr}\right||{x'}^A-x^A|\\
    &\les\frac{1}{r^2}|u'-u|+\frac{M^2}{r^3} |r\Gaw|\\
    &\les \frac{1}{r^2}\frac{\ep_0}{r^\frac{s-3}{2}}+\frac{M^2}{r^3}\frac{\ep}{r^\frac{s-1}{2}}\\
    &\les \frac{\ep_0}{r^\frac{s+1}{2}}.
\end{align*}
Injecting it into \eqref{diffeOmtrchb}, we deduce
\begin{align*}
    \sup_{\kk'}r^\frac{s+1}{2}|\widecheck{\Om^{-1}\trchb}|\les  \sup_{\kk'}r^\frac{s+1}{2}|\widecheck{{\Om'}^{-1}\trchb'}|+\sup_{\kk'}r^\frac{s+1}{2}|\Om^{-1}_K\tr_K\chib_K-{\Om'}_K^{-1}\tr_K\chib'_K|\les\ep_0.
\end{align*}
This concludes the proof of Proposition \ref{OOb0Omtrchb}.
\end{proof}
\subsubsection{Estimate for \texorpdfstring{$\OO_{0}(\widecheck{\Om\trchb})$}{}}\label{nabtrchctrchbc}
\begin{proposition}\label{propnabtrchtrchbc}
We have the following estimate:
\begin{align}
    |r^{2-\frac{2}{p}}|u|^{\frac{s-3}{2}}\widecheck{\Om\trchb}|_{p,S}\les\ep_0.\label{esttrchbc1}
\end{align}
\end{proposition}
\begin{proof}
Next, we recall from Proposition \ref{nullstructure}
\begin{align*}
\nab_3(\widecheck{\Om\trchb})+\trchb\,\widecheck{\Om\trchb}=\O_1^0\c\widecheck{\Om\omb}+\Gag\c\O_2^1+\Gag\c\Gab.
\end{align*}
Notice that we have from Propositions \ref{estOmc0} and \ref{OOb0Omtrchb}
\begin{align*}
    |r^{2-\frac{2}{p}}\widecheck{\Om\trchb}|_{p,S_0(\ub)}\les|r^{2-\frac{2}{p}}\widecheck{\Om^{-1}\trchb}|_{p,S_0(\ub)}+|r^{2-\frac{2}{p}}\widecheck{\Om^2}|_{p,S_0(\ub)}+|r^{2-\frac{2}{p}}\widecheck{\Om\trchb}\,\widecheck{\Om^2}|_{p,S_0(\ub)}\les\frac{\ep_0}{|u|^{\frac{s-3}{2}}}.
\end{align*}
Applying Lemma \ref{evolution}, we infer
\begin{align*}
    |r^{2-\frac{2}{p}}\widecheck{\Om\trchb}|_{p,S}&\les |r^{2-\frac{2}{p}}\widecheck{\Om\trchb}|_{p,S_0(\ub)}+\int_{u_0(\ub)}^{u}\left( |r^{1-\frac{2}{p}}\Omombc|_{p,S}+\frac{M}{r^2}|r^{2-\frac{2}{p}}\Gag|_{p,S}+|r^{2-\frac{2}{p}}\Gab\c\Gab|_{p,S}\right)du\\
    &\les\frac{\ep_0}{|u|^\frac{s-3}{2}}+\int_{u_0(\ub)}^{u}\left(\frac{\ep_0}{|u|^{\frac{s-1}{2}}}+\frac{M}{r^2}\frac{\ep}{|u|^\frac{s-3}{2}}+\frac{\ep^2}{|u|^{s-1}}\right)du\\
    &\les\frac{\ep_0}{|u|^{\frac{s-3}{2}}}.
\end{align*}
This concludes the proof of Proposition \ref{propnabtrchtrchbc}.
\end{proof}
\subsubsection{Estimate for \texorpdfstring{$\OO_{0,1}(\etac)$}{} }\label{etaetab}
\begin{proposition}\label{propetaetab}
We have the following estimate for $q=0,1$:
\begin{align}
    |r^{2-\frac{2}{p}}|u|^{\frac{s-3}{2}}(r\nab)^q\etac|_{p,S}&\les\ep_0.\label{estetac}
\end{align}
\end{proposition}
\begin{proof}
We recall from Proposition \ref{equationsmumub}
\begin{align}
\begin{split}\label{[mu]}
\nab_4\widecheck{[\mu]}+\trch\widecheck{[\mu]}=&\O_1^0\c\rhoc+\O_3^1\c\Gag^{(1)}+r^{-1}\Gab\c\Gag^{(1)}.
\end{split}
\end{align}
Applying Lemma \ref{evolution}, we obtain
\begin{align*}
|r^{2-\frac{2}{p}}\widecheck{[\mu]}|_{p,S}\les&|r^{2-\frac{2}{p}}\widecheck{[\mu]}|_{p,S_*(u)}+\int_{\ub}^{\ub_*}\left(|r^{1-\frac{2}{p}}\rhoc|_{p,S}+|r^{1-\frac{2}{p}}\Gab\c\Gag^{(1)}|_{p,S}+\frac{M}{r^3}|r^{2-\frac{2}{p}}\Gag^{(1)}|_{p,S}\right)d\ub\\
\les&\frac{\ep_0}{r|u|^{\frac{s-3}{2}}}+\int_{\ub}^{\ub_*}\left(\frac{\ep_0}{r^2|u|^\frac{s-3}{2}}+\frac{\ep^2}{r^2|u|^{s-1}}+\frac{M}{r^3}\frac{\ep}{|u|^\frac{s-3}{2}}\right)d\ub\\
\les&\frac{\ep_0}{r|u|^\frac{s-3}{2}}.
\end{align*}
We have from Proposition \ref{nullstructure}
\begin{align*}
    \sdiv\etac&=-\muc-\rhoc+\O_3^2\c\Gab^{(1)}+\Gag\c\Gab.
\end{align*}
We recall from \eqref{auxillaryquantities} that
\begin{align*}
[\mu]=\mu+\frac{1}{4}\trch\,\trchb,
\end{align*}
which implies
\begin{equation}
    \mumc=\muc+\O_1^0\c(\trchc,\trchbc,r^{-1}\Omc)+\Gag\c\Gag.
\end{equation}
Thus, we obtain
\begin{align}
\begin{split}\label{divetac}
    \sdiv\etac=-\mumc-\rhoc+\O_3^2+\O_1^0\c(\widecheck{\Om\trch},\widecheck{\Om\trchb},r^{-1}\Omc)+\Gag\c\Gab.
\end{split}
\end{align}
Recall the linearized torsion equation in Proposition \ref{nullstructure}:
\begin{align}
\begin{split}\label{curletac}
\curl\etac=\widecheck{\si}+\O_3^2\c\Gab^{(1)}+\Gab\c\Gag.
\end{split}
\end{align}
Combining \eqref{divetac} and \eqref{curletac}, applying Propositions \ref{ellipticest}, \ref{estOmc0} and \ref{propOmtrch}, we obtain for $q=0,1$:
\begin{align*}
    &|r^{2+q-\frac{2}{p}}\nab^q\etac|_{p,S}\\
    \les& |r^{3-\frac{2}{p}}(\widecheck{[\mu]},\rhoc,\sic)|_{p,S}+|r^{2-\frac{2}{p}}(\widecheck{\Om\trch},\widecheck{\Om\trchb},r^{-1}\Omc)|_{p,S}+\frac{M^2}{r^3}|r^{3-\frac{2}{p}}\Gab^{(1)}|_{p,S}+|r^{3-\frac{2}{p}}\Gag\c\Gab|_{p,S}\\
    \les& \frac{\ep_0}{|u|^{\frac{s-3}{2}}}.
\end{align*}
This concludes the proof of Proposition \ref{propetaetab}.
\end{proof}
\subsubsection{Estimate for \texorpdfstring{$\osc(f)$}{}}
\begin{proposition}\label{estf}
    We have the following estimate:
    \begin{equation}
        \sup_{\kk'}|r^{\frac{s-1}{2}} \dk^{\leq 1}f|\les \ep_0.
    \end{equation}
\end{proposition}
\begin{proof}
We have from Proposition \ref{transformation}
\begin{align*}
    \Om'\trch'=\Om\trch+\sdiv f+f\c \Gaw+\O_2^1\c f.
\end{align*}
Applying \eqref{change'} and Proposition \ref{nullstructure}, we have
\begin{align*}
    e_A(\Om\trch)&=e_A(\Om'\trch')-e_A(\sdiv f)+r^{-1}f\c\Gaw^{(1)}+r^{-1}\Gaw\c f^{(1)}+\O_3^1\c f^{(1)}\\
    &=e'_A(\Om'\trch')-\frac{\la}{2}f_A e'_3(\Om'\trch')-e_A(\sdiv f)+\Gaw\c\Gaw^{(1)}+\O_2^1\c\Gaw^{(1)}\\
    &=e'_A(\Om'\trch')+\frac{\Om'}{4}f_A\trch'\,\trchb'-e_A(\sdiv f)+\Gaw\c\Gaw^{(1)}+\O_2^1\c\Gaw^{(1)}.
\end{align*}
Recall from Lemma \ref{estKK'} that
\begin{align*}
    |e'_A(\Om'_K\tr_K\chi'_K)-e_A(\Om_K\tr_K\chi_K)|\les \frac{\ep_0}{r^\frac{s+3}{2}}.
\end{align*}
Thus, we infer
\begin{align*}
    &\left|r^{-\frac{2}{p}}\left(e_A(\sdiv f)-\frac{\Om'}{4}f_A\trch'\trchb'+e_A(\widecheck{\Om\trch})-e'_A(\widecheck{\Om'\trch'})\right)\right|_{p,S}\\
    \les& |e'_A(\Om'_K\tr_K\chi'_K)-e_A(\Om_K\tr_K\chi_K)|_{\infty,S}+\frac{M^2}{r^3}|r^{-\frac{2}{p}}\c\Gaw^{(1)}|+|r^{-\frac{2}{p}}\Gaw\c\Gaw^{(1)}|_{p,S}\\
    \les&\frac{\ep_0}{r^\frac{s+3}{2}}+\frac{M^2}{r^3}\frac{\ep}{r^\frac{s+1}{2}}+\frac{\ep^2}{r^{s+1}}\\
    \les& \frac{\ep_0}{r^\frac{s+3}{2}}.
\end{align*}
Recall that
\begin{align*}
    -\frac{\Om'}{4}f\trch'\trchb'=\frac{f}{2r^2}+\O_3^1\c f+\Gaw\c\Gaw.
\end{align*}
Hence, we deduce from Proposition \ref{propOmtrch}
\begin{align}
\begin{split}\label{divfest}
    &\;\;\;\,\,\, \left|r^{-\frac{2}{p}}\left(e_A(\sdiv f)+\frac{f}{2r^2}\right)\right|_{p,S}\\
    &\les|r^{-\frac{2}{p}}e_A(\widecheck{\Om\trch})|_{p,S}+|e'_A(\widecheck{\Om'\trch'})|_{\infty,S}+\frac{M}{r^2}|r^{-\frac{2}{p}}\c\Gaw^{(1)}|_{p,S}+|r^{-\frac{2}{p}}\Gaw\c\Gaw^{(1)}|_{p,S}+\frac{\ep_0}{r^\frac{s+3}{2}}\\
    &\les\frac{\ep_0}{r^\frac{s+3}{2}}+\frac{M}{r^2}\frac{\ep}{r^\frac{s+1}{2}}+\frac{\ep^2}{r^{s+1}}\\
    &\les\frac{\ep_0}{r^\frac{s+3}{2}}.
\end{split}
\end{align}
Next, we have from Proposition \ref{transformation}
\begin{align*}
    \curl f=\O_2^1\c\Gaw+r\Gaw\c\Gaw,
\end{align*}
which implies
\begin{align}
\begin{split}\label{curlfest}
    |r^{-\frac{2}{p}}e_A(\curl f)|_{p,S}&\les \frac{M^2}{r^3}|r^{-\frac{2}{p}}\c\Gaw^{(1)}|_{p,S}+|r^{-\frac{2}{p}}\Gaw\c\Gaw^{(1)}|_{p,S}\\
    &\les \frac{M^2}{r^3}\frac{\ep}{r^\frac{s+1}{2}}+\frac{\ep^2}{r^{s+1}}\\
    &\les \frac{\ep_0}{r^\frac{s+3}{2}}.
\end{split}
\end{align}
Combining \eqref{divfest} and \eqref{curlfest}, we deduce\footnote{Recall that $d_1^*d_1f=d_1^*(\sdiv f,\curl f)=-\nab\sdiv f+{^*\nab}\curl f$.}
\begin{align*}
\left|r^{-\frac{2}{p}}\left(d_1^*d_1 f-\frac{f}{2r^2}\right)\right|_{p,S}\les\frac{\ep_0}{r^\frac{s+3}{2}}.
\end{align*}
We have from Proposition \ref{standardnull}
\begin{align*}
    d_1^* d_1 f-\frac{f}{2r^2}&=(-\De_1+\mathbf{K})f-\frac{f}{2r^2}\\
    &=-\De_1f-\frac{f}{4}\trch\trchb-\frac{f}{2r^2}+r^{-1}\Gaw^{(1)}\c f+\O_3^1\c f\\
    &=-\De_1 f+\frac{f}{2r^2}+\O_2^1\c \Gaw^{(1)}+\Gaw\c\Gaw^{(1)}.
\end{align*}
Thus, we infer
\begin{align*}
    \left|r^{-\frac{2}{p}}\left(-\De_1+\frac{1}{2r^2}\right)f\right|_{p,S}&\les\left|r^{-\frac{2}{p}}\left(d_1^*d_1-\frac{1}{2r^2}\right)f\right|_{p,S}+\frac{M}{r^2}|r^{-\frac{2}{p}}\Gaw^{(1)}|_{p,S}+|r^{-\frac{2}{p}}\Gaw\c\Gaw^{(1)}|_{p,S}\\
    &\les \frac{\ep_0}{r^\frac{s+3}{2}}+\frac{M}{r^2}\frac{\ep}{r^\frac{s+1}{2}}+\frac{\ep^2}{r^{s+1}}\\
    &\les \frac{\ep_0}{r^\frac{s+3}{2}}.
\end{align*}
Applying Proposition \ref{ellipticest}, we deduce for $q=0,1,2$,
\begin{align*}
    |r^{-\frac{2}{p}}(r\nab)^q f|_{p,S}\les\frac{\ep_0}{r^\frac{s+3}{2}}.
\end{align*}
Then, applying Proposition \ref{standardsobolev}, we obtain
\begin{equation}\label{finftynab}
    \sup_{S\subset\kk'}|r^\frac{s+3}{2}(r\nab)^{\leq 1}f|_{\infty,S}\les\ep_0.
\end{equation}
Next, we recall from Proposition \ref{transformation}
\begin{align}\label{nab3feta}
    \la \nab_{e_3} f=2\eta'-2\eta+\O_1^1\c\Gaw+r\Gaw\c\Gaw.
\end{align}
Applying Lemma \ref{estKK'}, we obtain
\begin{align*}
    |\eta_K'(e'_A)-\eta_K(e_A)|\les\frac{\ep_0}{r^\frac{s+1}{2}}.
\end{align*}
Applying Proposition \ref{propetaetab} and \eqref{finftynab}, we deduce
\begin{align}
\begin{split}\label{nab3finfty}
    |r^{-\frac{2}{p}}\nab_{e_3}f|_{p,S}&\les |r^{-\frac{2}{p}}(\etac',\etac)|_{p,S}+M|r^{-1-\frac{2}{p}}\Gaw|_{p,S}+|r^{1-\frac{2}{p}}\Gaw\c\Gaw|_{p,S}+\frac{\ep_0}{r^\frac{s+1}{2}}\\
    &\les\frac{\ep_0}{r^\frac{s+1}{2}}+\frac{M\ep}{r^\frac{s+3}{2}}+\frac{\ep^2}{r^s}\\
    &\les\frac{\ep_0}{r^\frac{s+1}{2}}.
\end{split}
\end{align}
Finally, we recall from Proposition \ref{transformation}
\begin{equation*}
    \frac{1}{2}\nab_{\la^{-1}e_4}f=-\frac{1}{4}\trch\, f+\O_1^1\c\Gaw+r\Gaw\c\Gaw,
\end{equation*}
which implies from \eqref{finftynab}
\begin{equation}\label{nab4finfty}
    \sup_{\kk'} r^\frac{s+1}{2}|\nab_{e_4}f|\les \sup_{\kk'}r^\frac{s-1}{2}|f|+\sup_{\kk'}\left(\frac{M\ep}{r}+\frac{\ep^2}{r^\frac{s-3}{2}}\right)\les \ep_0.
\end{equation}
Combining \eqref{finftynab}, \eqref{nab3finfty} and \eqref{nab4finfty}, we deduce
\begin{equation*}
    \sup_{\kk'} r^\frac{s-1}{2}|\dk^{\leq 1}f|\les\ep_0.
\end{equation*}
This concludes the proof of Proposition \ref{estf}.
\end{proof}
\subsubsection{Estimate for \texorpdfstring{$\okk_{0,1}(\etabc)$}{}}
\begin{proposition}\label{OObetabc1}
We have the following estimate for $q=0,1$:
    \begin{equation}
        |r^{2-\frac{2}{p}}|u|^{\frac{s-3}{2}}(r\nab)^q\etabc|_{p,S_0(\ub)}\les\ep_0.\label{estOObetabc}
    \end{equation}
\end{proposition}
\begin{proof}
We recall from Proposition \ref{transformation}
\begin{align}
\begin{split}\label{transetab}
\etab&=\etab'-\frac{1}{4}\trchb\,f+\O_1^1\c\Gaw+r\Gaw\c\Gaw.
\end{split}
\end{align}
Recall that we have from Lemma \ref{estKK'}
\begin{align*}
|\etab_K(e_A)-\etab_K'(e'_A)|\les \frac{\ep_0}{r^\frac{s+1}{2}}.
\end{align*}
Thus, we obtain
\begin{align}
\begin{split}\label{etabcinfty}
    |\etabc_A|&\les |\etabc'_A|+|\etab_K(e_A)-\etab_K'(e'_A)|+\frac{M}{r}|\Gaw|+r|\Gaw\c\Gaw|\\
    &\les\frac{\ep_0}{r^\frac{s+1}{2}}+\frac{M}{r}\frac{\ep}{r^\frac{s+1}{2}}+\frac{\ep^2}{r^s}\\
    &\les\frac{\ep_0}{r^\frac{s+1}{2}}.
\end{split}
\end{align}
Differentiating \eqref{transetab} by $e_B$, we obtain
\begin{align*}
    e_B(\etab_A)=e'_B(\etab'_A)+\O_2^0\c f^{(1)}+\O_2^1\c\Gaw^{(1)}+\Gaw\c\Gaw^{(1)}.
\end{align*}
Applying Lemma \ref{estKK'}, we have
\begin{align}\label{transetabK}
    |e_B(\etab_K)_A-e'_B(\etab'_K)_A|\les\frac{\ep_0}{r^\frac{s+3}{2}}.
\end{align}
Hence, we obtain
\begin{align*}
    &\;\;\;\,\,\,|r^{-\frac{2}{p}}e_B(\etabc_A)|_{p,S}\\
    &\les |r^{-\frac{2}{p}}e'_B(\etabc_A')|_{p,S}+|e_B(\etab_K)_A-e'_B(\etab'_K)_A|+r^{-2}|f^{(1)}|_{\infty,S}+M|r^{-2-\frac{2}{p}}\Gaw^{(1)}|_{p,S}+|r^{-\frac{2}{p}}\Gaw\c\Gaw^{(1)}|_{p,S}\\
    &\les \frac{\ep_0}{r^\frac{s+3}{2}}+\frac{M\ep}{r^\frac{s+5}{2}}+\frac{\ep^2}{r^{s+1}}\\
    &\les \frac{\ep_0}{r^\frac{s+3}{2}}.
\end{align*}
Combining with \eqref{etabcinfty}, we infer for $q=0,1$
\begin{equation*}
    |r^{\frac{s+1}{2}-\frac{2}{p}}(r\nab)^q\etabc|_{p,S}\les\ep_0.
\end{equation*}
This concludes the proof of Proposition \ref{OObetabc1}.
\end{proof}
\subsubsection{Estimate for \texorpdfstring{$\OO_{0,1}(\etabc)$}{}}
\begin{proposition}\label{propetabc1}
We have the following estimate for $q=0,1$:
    \begin{equation}
        |r^{2-\frac{2}{p}}|u|^{\frac{s-3}{2}}(r\nab)^q\etabc|_{p,S}\les\ep_0.\label{estetabc}
    \end{equation}
\end{proposition}
\begin{proof}
We have from Propositions \ref{nullstructure}, \ref{propOmtrch}, \ref{OOb0Omtrchb} and \ref{OObetabc1} that
\begin{align}
\begin{split}\label{mumbcinitial}
    |r^{2-\frac{2}{p}}\mubmc|_{p,S_0(\ub)}&\les |r^{2-\frac{2}{p}}\mubc|_{p,S_0(\ub)}+|r^{1-\frac{2}{p}}(\trchc,\trchbc)|_{p,S_0(\ub)}\\
    &\les |r^{2-\frac{2}{p}}\sdiv\etabc|_{p,S_0(\ub)}+|r^{2-\frac{2}{p}}\rhoc|_{p,S_0(\ub)}+|r^{1-\frac{2}{p}}(\widecheck{\Om\trch},\widecheck{\Om^{-1}\trchb})|_{p,S_0(\ub)}\\
    &+M^2|r^{-1-\frac{2}{p}}\Gab^{(1)}|_{p,S_0(\ub)}+|r^{2-\frac{2}{p}}\Gab\c\Gag|_{p,S_0(\ub)}\\
    &\les \frac{\ep_0}{r^\frac{s-1}{2}}.
\end{split}
\end{align}
We recall from Proposition \ref{equationsmumub}
\begin{align*}
\nab_3\mubmc+\trchb\,\mubmc&=O(r^{-1})\rhoc+\O_3^1\c\Gab^{(1)}+r^{-1}\Gab\c\Gab^{(1)}.
\end{align*}
Applying Lemma \ref{evolution} and \eqref{mumbcinitial}, we obtain
\begin{align*}
    |r^{2-\frac{2}{p}}\widecheck{[\mub]}|_{p,S}&\les |r^{2-\frac{2}{p}}\widecheck{[\mub]}|_{p,S_0(\ub)}+\int_{u_0(\ub)}^u\left( |r^{1-\frac{2}{p}}\rhoc|_{p,S}+\frac{M}{r^2}|r^{1-\frac{2}{p}}\Gab^{(1)}|_{p,S}+|r^{1-\frac{2}{p}}\Gab\c\Gab^{(1)}|_{p,S}\right)du\\
    &\les \frac{\ep_0}{r^\frac{s-1}{2}}+\int_{u_0(\ub)}^u\left(\frac{\ep_0}{r^2|u|^\frac{s-3}{2}}+\frac{M}{r^2}\frac{\ep}{|u|^{\frac{s-1}{2}}}+\frac{\ep^2}{r^2|u|^{s-1}}\right)du\\
    &\les \frac{\ep_0}{r|u|^\frac{s-3}{2}}.
\end{align*}
We have from Proposition \ref{nullstructure}
\begin{align}
    \begin{split}\label{hodgeetabc}
        \sdiv\etabc=&-\widecheck{[\mub]}-\rhoc+\O_3^2\c\Gab^{(1)}+\O_1^0\c(\trchc,\trchbc,r^{-1}\Omc)+\Gag\c\Gab,\\
        \curl\etabc=&\sic+\O_3^2\c\Gab^{(1)}+\Gag\c\Gab.
    \end{split}
\end{align}
Applying Proposition \ref{ellipticest}, we deduce for $q=0,1$:
\begin{align*}
|r^{2+q-\frac{2}{p}}\nab^q\etabc|_{p,S}&\les|r^{3-\frac{2}{p}}(\mubmc,\rhoc)|_{p,S}+|r^{2-\frac{2}{p}}(\trchc,\trchbc,r^{-1}\Omc)|_{p,S}+M^2|r^{-\frac{2}{p}}\Gab^{(1)}|_{p,S}+|r^{3-\frac{2}{p}}\Gag\c\Gab|_{p,S}\\
&\les\frac{\ep_0}{|u|^{\frac{s-3}{2}}}+\frac{M^2}{r|u|^\frac{s-1}{2}}+\frac{\ep^2}{|u|^{s-2}}\\
&\les\frac{\ep_0}{|u|^{\frac{s-3}{2}}}.
\end{align*}
This concludes the proof of Proposition \ref{propetabc1}.
\end{proof}
\subsubsection{Estimate for \texorpdfstring{$\OO_{0,1,2}(\Omc)$}{}}\label{Omc}
\begin{proposition}\label{propOmc}
We have the following estimates:
\begin{align}
|r^{1-\frac{2}{p}}|u|^{\frac{s-3}{2}}(r\nab)^q\Omc|_{p,S}\les\ep_0,\qquad q=0,1,2.
\end{align}
\end{proposition}
\begin{proof}
Recall that 
\begin{align*}
    \eta+\etab =2\nab\log\Om,
\end{align*}
and
\begin{align*}
    \eta_K+\etab_K=2\nab\log\Om_K.
\end{align*}
Taking the difference, we obtain
\begin{equation*}
    \De\log\left(\frac{\Om}{\Om_K}\right)=\f12\sdiv(\etac+\etabc),
\end{equation*}
Applying Propositions \ref{ellipticest}, \ref{propetaetab} and \ref{propetabc1}, we have for $q=1,2$:
\begin{align}
\begin{split}
\left|r^{1+q-\frac{2}{p}}\nab^q\left(\log\left(\frac{\Om}{\Om_K}\right)\right)\right|_{p,S}(u,\ub) \les\frac{\ep_0}{|u|^{\frac{s-3}{2}}}.\label{OmOmK}
\end{split}
\end{align}
Applying Proposition \ref{estOmc0}, we deduce that
\begin{align*}
    \left|r\log\left(\frac{\Om}{\Om_K}\right)\right|_{\infty,S}\leq \left|r\left(\frac{\Om}{\Om_K}-1\right)\right|_{\infty,S}\les |r\Omc|_{\infty,S}\les \frac{\ep_0}{|u|^\frac{s-3}{2}}
\end{align*}
so that \eqref{OmOmK} also holds true for $q=0$. Hence, we have
\begin{align*}
    \left\|\log\frac{\Om}{\Om_K}\right\|_{W^{2,4}(S)}\les \frac{\ep_0}{r^\f12 |u|^{\frac{s-3}{2}}}.
\end{align*}
We recall that the Sobolev space $W^{2,4}(S)$ forms a Banach algebra, i.e.
\begin{align}\label{Banachalgebra}
    \|uv\|_{W^{2,4}(S)}\les \|u\|_{W^{2,4}(S)}\|v\|_{W^{2,4}(S)},\qquad \forall u,v\in W^{2,4}(S),
\end{align}
see for example Theorem 4.39 in \cite{adams}. Notice that
\begin{align}\label{expOmc}
    \Omc=\Om_K\left(\frac{\Om}{\Om_K}-1\right)=\Om_K\left(\exp\left(\log\left(\frac{\Om}{\Om_K}\right)\right)-1\right)=\Om_K\sum_{k=1}^\infty \frac{1}{k!}\left(\log\frac{\Om}{\Om_K}\right)^k.
\end{align}
Combining \eqref{OmOmK}, \eqref{Banachalgebra} and \eqref{expOmc}, we deduce
\begin{align*}
    |(r\nab)^{\leq 2}\Omc|_{4,S}\les\sum_{k=1}^\infty \left|(r\nab)^{\leq 2}\log\left(\frac{\Om}{\Om_K}\right)\right|_{4,S}^k
    \les\frac{\ep_0}{r^\f12 |u|^{\frac{s-3}{2}}},
\end{align*}
which implies for $p\in [2,4]$:
\begin{align*}
    |r^{1-\frac{2}{p}}(r\nab)^q\Omc|_{p,S}\les \frac{\ep_0}{|u|^{\frac{s-3}{2}}},\qquad q=0,1,2.
\end{align*}
This concludes the proof of Proposition \ref{propOmc}.
\end{proof}
\subsubsection{Estimate for \texorpdfstring{$\okk_1(\widecheck{\Om^{-1}
\trchb})$}{}}
\begin{proposition}\label{OOb0Omtrchb1}
We have the following estimate:
\begin{align}
\begin{split}
    |r^{\frac{s+3}{2}-\frac{2}{p}}\nab(\widecheck{\Om^{-1}\trchb})|_{p,S_0(\ub)}\les \ep_0.
\end{split}
\end{align}
\end{proposition}
\begin{proof}
We have from Proposition \ref{transformation}
\begin{align*}
    {\Om'}^{-1}\trchb'=\Om^{-1}\trchb.
\end{align*}
Differentiating it by $e_A$ and applying \eqref{change'}, we obtain
\begin{align*}
e_A(\Om^{-1}\trchb)=e'_A({\Om'}^{-1}\trchb')-\frac{\la}{2}f_Ae'_3({\Om'}^{-1}\trchb').
\end{align*}
Applying Propositions \ref{standardnull}, \ref{estf} and Lemma \ref{estKK'}, we infer
\begin{align*}
    |r^{-\frac{2}{p}}e_A(\widecheck{\Om^{-1}\trchb})|_{p,S}&\les |r^{-\frac{2}{p}}e'_A(\widecheck{{\Om'}^{-1}\trchb'})|_{p,S}+|e_A(\Om^{-1}\trchb)_K-e'_A(\Om^{-1}\trchb)_K'|+|r^{-2-\frac{2}{p}}f_A|_{p,S}\\
    &\les \frac{\ep_0}{r^\frac{s+3}{2}}.
\end{align*}
This concludes the proof of Proposition \ref{OOb0Omtrchb1}.
\end{proof}
\subsubsection{Estimate for \texorpdfstring{$\OO_{1}(\widecheck{\Om\trchb})$}{}}
\begin{proposition}\label{proptrchbc1}
We have the following estimates:
\begin{align}
    |r^{3-\frac{2}{p}}|u|^{\frac{s-3}{2}}\nab\widecheck{\Om\trchb}|_{p,S}\les\ep_0.
\end{align}
\end{proposition}
\begin{proof}
Notice that we have
\begin{align*}
    \widecheck{\Om\trchb}=\Om^2(\Om^{-1}{\trchb})-\Om^2_K(\Om^{-1}_K\tr_K\chib_K)=\widecheck{\Om^2}(\Om^{-1}\trchb)+\Om^2_K(\widecheck{\Om^{-1}\trchb}),
\end{align*}
which implies
\begin{align*}
{\nab\widecheck{\Om\trchb}}=\O_1^0\c\nab\Omc+\O_0^0\c\trchbc+\Gag\c\Gag^{(1)}.
\end{align*}
Hence, we have from Propositions \ref{estOmc0} and \ref{OOb0Omtrchb1}
\begin{align*}
    |r^{3-\frac{2}{p}}\nab\widecheck{\Om\trchb}|_{p,S_0(\ub)}&\les|r^{3-\frac{2}{p}}\nab\widecheck{\Om^{-1}\trchb}|_{p,S_0(\ub)}+|r^{2-\frac{2}{p}}\nab\Omc|_{p,S_0(\ub)}+|r^{3-\frac{2}{p}}\Gag\c\Gag^{(1)}|_{p,S_0(\ub)}\\
    &\les\frac{\ep_0}{|u|^{\frac{s-3}{2}}}.
\end{align*}
Next, we recall from Proposition \ref{nullstructure}
\begin{align*}
\nab_3(\widecheck{\Om\trchb})+\trchb\,\widecheck{\Om\trchb}=\O_1^0\c\widecheck{\Om\omb}+\Gag\c\O_2^1+\Gab\c\Gab.
\end{align*}
Differentiating it by $r\nab$ and applying Corollary \ref{commutation}, we deduce
\begin{align*}
\nab_3(r\nab\widecheck{\Om\trchb})+\trchb\,r\nab\widecheck{\Om\trchb}=\O_0^0\c\nab\widecheck{\Om\omb}+\Gag^{(1)}\c\O_2^1+\Gab\c\Gab^{(1)}.
\end{align*}
Applying Lemma \ref{evolution} and Proposition \ref{propnabomcombc}, we infer
\begin{align*}
    &\;\;\;\,\,\, |r^{3-\frac{2}{p}}\nab\widecheck{\Om\trchb}|_{p,S}\\
    &\les |r^{3-\frac{2}{p}}\nab\widecheck{\Om\trchb}|_{p,S_0(\ub)}+\int_{u_0(\ub)}^{u}\left(|r^{2-\frac{2}{p}}\nab\widecheck{\Om\omb}|_{p,S}+\frac{M}{r^2}|r^{2-\frac{2}{p}}\Gag|_{p,S}+|r^{2-\frac{2}{p}}\Gab\c\Gab^{(1)}|_{p,S}\right)du\\
    &\les\frac{\ep_0}{|u|^\frac{s-3}{2}}+\int_{u_0(\ub)}^{u}\left(\frac{\ep_0}{|u|^{\frac{s-1}{2}}}+\frac{M}{r^2}\frac{\ep}{|u|^\frac{s-3}{2}}+\frac{\ep^2}{|u|^{s-1}}\right)du\\
    &\les\frac{\ep_0}{|u|^{\frac{s-3}{2}}}.
\end{align*}
This concludes the proof of Proposition \ref{proptrchbc1}.
\end{proof}
\subsubsection{Estimates for \texorpdfstring{$\OO_{0,1}(\hchc)$}{} and \texorpdfstring{$\OO_{0,1}(\hchbc)$}{}}\label{hchc}
\begin{proposition}\label{prophchc}
We have the following estimates for $q=0,1$:
\begin{align}
\begin{split}\label{esthchc}
    |r^{2-\frac{2}{p}}|u|^{\frac{s-3}{2}}(r\nab)^q\hchc|_{p,S}&\les  \ep_0,\\
    |r^{1-\frac{2}{p}}|u|^{\frac{s-1}{2}}(r\nab)^q\hchbc|_{p,S}&\les\ep_0.
\end{split}
\end{align}
\end{proposition}
\begin{proof}
We recall from Proposition \ref{nullstructure}
\begin{align*}
    \sdiv\hchc=-\bc+\O_0^0\c\nab\trchc+\O_1^0\c\zec+\O_3^2\c\Gag^{(1)}+\Gag\c\Gag.
\end{align*}
Applying Propositions \ref{ellipticest}, \ref{propOmtrch}, \ref{propetaetab} and \ref{propetabc1}, we deduce for $q=0,1$:
\begin{align*}
    |r^{2+q-\frac{2}{p}}\nab^q\hchc|_{p,S}&\les |r^{3-\frac{2}{p}}\nab\trchc|_{p,S}+|r^{2-\frac{2}{p}}\zec|_{p,S}+|r^{3-\frac{2}{p}}\bc|_{p,S}+\frac{M^2}{r^2}|r^{2-\frac{2}{p}}\Gag^{(1)}|_{p,S}+|r^{3-\frac{2}{p}}\Gag\c\Gag|_{p,S} \\
    &\les\frac{\ep_0}{|u|^{\frac{s-3}{2}}}.
\end{align*}
Next, we recall from Proposition \ref{nullstructure}
\begin{align*}
\sdiv\hchbc=\bbc+\O_0^0\c\nab\trchbc+\O_1^0\c\zec+\O_3^2\c\Gab^{(1)}+\Gab\c\Gag.
\end{align*}
Applying Propositions \ref{ellipticest}, \ref{proptrchbc1}, \ref{propetaetab} and \ref{propetabc1}, we deduce for $q=0,1$:
\begin{align*}
    |r^{1+q-\frac{2}{p}}\nab^q\hchbc|_{p,S}&\les |r^{2-\frac{2}{p}}\nab\trchbc|_{p,S}+|r^{1-\frac{2}{p}}\zec|_{p,S}+|r^{2-\frac{2}{p}}\bbc|_{p,S}+\frac{M^2}{r^2}|r^{1-\frac{2}{p}}\Gab^{(1)}|_{p,S}+|r^{2-\frac{2}{p}}\Gab\c\Gag|_{p,S} \\
    &\les\frac{\ep_0}{|u|^{\frac{s-1}{2}}}.
\end{align*}
This concludes the proof of Proposition \ref{prophchc}.
\end{proof}
\subsubsection{Estimates for \texorpdfstring{$\OO_{0,1}(\bcc)$}{}}
\begin{proposition}\label{propbcc}
We have the following estimate for $q=0,1$:
\begin{align}
|r^{2-\frac{2}{p}}|u|^{\frac{s-3}{2}}(r\nab)^q\,\bcc^A|_{p,S}\les\ep_0.\label{estbcc}
\end{align}
\end{proposition}
\begin{proof}
We have from Proposition \ref{metriceqre}
    \begin{align}\label{bccA}
        \nab_4(\bcc^A)=\O_1^0\c\zec+\O_3^2\c\Gag+\Gag\c\Gag.
    \end{align}
    Applying Lemma \ref{evolution}, we deduce from Propositions \ref{propetaetab} and \ref{propetabc1}\footnote{Recall that $\bcc^A=0$ on $\Cb_*$ and $\ze=\frac{\eta-\etab}{2}$.}
    \begin{align*}
        |r^{-\frac{2}{p}}\bcc^A|_{p,S}&\les \int_{\ub}^{\ub_*} \left(|r^{-1-\frac{2}{p}}\zec|_{p,S}+M^2|r^{-3-\frac{2}{p}}\Gag|_{p,S}+|r^{-\frac{2}{p}}\Gag\c\Gag|_{p,S}
        \right)d\ub\\
        &\les \int_{\ub}^{\ub_*} \left(\frac{\ep_0}{r^3|u|^\frac{s-3}{2}}+\frac{M^2\ep}{r^5|u|^\frac{s-3}{2}}+\frac{\ep^2}{r^4|u|^{s-3}}
        \right)d\ub\\
        &\les \frac{\ep_0}{r^2|u|^\frac{s-3}{2}}+\frac{M^2\ep}{r^4|u|^\frac{s-3}{2}}+\frac{\ep^2}{r^3|u|^{s-3}}\\
        &\les \frac{\ep_0}{r^2|u|^\frac{s-3}{2}},
    \end{align*}
    which implies \eqref{estbcc} in the case $q=0$. Differentiating \eqref{bccA} by $r\nab$ and applying Corollary \ref{commutation}, we obtain
    \begin{align*}
        \nab_4(r\nab\bcc^A)=\O_1^0\c r\nab\zec+\O_3^2\c\Gag^{(1)}+\Gag\c\Gag^{(1)}.
    \end{align*}
    Applying Lemma \ref{evolution}, we deduce from Propositions \ref{propetaetab} and \ref{propetabc1}
    \begin{align*}
        |r^{-\frac{2}{p}}(r\nab)\bcc^A|_{p,S}&\les \int_{\ub}^{\ub_*} \left(|r^{-\frac{2}{p}}\nab\zec|_{p,S}+M^2|r^{-3-\frac{2}{p}}\Gag^{(1)}|_{p,S}+|r^{-\frac{2}{p}}\Gag\c\Gag^{(1)}|_{p,S}
        \right)d\ub\\
        &\les \frac{\ep_0}{r^2|u|^\frac{s-3}{2}},
    \end{align*}
    which implies \eqref{estbcc} in the case $q=1$. This concludes the proof of Proposition \ref{propbcc}.
\end{proof}
\subsubsection{Estimate for \texorpdfstring{$\osc(x)$}{}}
\begin{proposition}\label{proposcx}
We have the following estimates
\begin{equation}
    \sup_{\kk'}r^\frac{s-1}{2}|{x'}^A-x^A|\les\ep_0.
\end{equation}
\end{proposition}
\begin{proof}
We have from \eqref{change} and Proposition \ref{decayGamma}
\begin{align}
\begin{split}\label{xx'evo}
    (\Om'e'_4-\Om'e'_3)({x'}^A-x^A)&=-{\bu'}^A-\Om'e'_4(x^A)+\Om'e'_3(x^A)\\
    &=-{\bu'}^A-\Om\left( e_4+f^B e_B +\frac{1}{4}|f|^2 e_3\right)x^A +\frac{{\Om'}^2}{\Om^2}\bu^A\\
    &=\bu^A-{\bu'}^A-\Om f^B e_B(x^A)-\frac{1}{4}|f|^2\bu^A+\left(\frac{{\Om'}^2}{\Om^2}-1\right)\bu^A\\
    &=\bu^A-{\bu'}^A+\O_1^0\c f+\O_1^1\c\Gaw+r\Gaw\c\Gaw.
\end{split}
\end{align}
Notice from Lemma \ref{estKK'} that
\begin{align*}
    |{b'}^A-b^A|\les|\bcc'^A|+|\bcc^A|+|(b'_K)^A-b_K^A|\les \frac{\ep_0}{r^\frac{s+1}{2}}.
\end{align*}
Hence, we obtain from Proposition \ref{estf}
\begin{align*}
|(\Om'e'_4-\Om'e'_3)({x'}^A-x^A)|\les r^{-1}|f|+\frac{\ep_0}{r^\frac{s+1}{2}}+\frac{M}{r}|\Gaw|+r|\Gaw\c\Gaw|\les \frac{\ep_0}{r^\frac{s+1}{2}}.
\end{align*}
Integrating it along $e'_4-e'_3$, we deduce\footnote{Recall that ${x'}^A=x^A$ on $\Cb_*\cap\Si_0$ which is chosen in Section \ref{bootregion}. {Together with the equation for $\Om'e_3'({x'}^A-x^A)$, we can easily deduce the estimate for ${x'}^A-x^A$ on $\Cb_*\cap\KK'$.}}
\begin{align*}
    |{x'}^A-x^A|_{\infty,S'}\les {|{x'}^A-x^A|_{\infty,\Cb_*\cap\KK'}+}\int_{w}^{\ub_*} \frac{\ep_0}{{r}^{\frac{s+1}{2}}}dw\les\frac{\ep_0}{r^\frac{s-1}{2}}.
\end{align*}
This concludes the proof of Proposition \ref{proposcx}.
\end{proof}
\subsubsection{Estimates for \texorpdfstring{$\OO_{0,1}(\gac)$}{} and \texorpdfstring{$\OO_{0,1}(\inc)$}{}}
\begin{proposition}\label{propgac}
We have the following estimates for $q=0,1$:
\begin{align}
|r^{-1+q-\frac{2}{p}}|u|^{\frac{s-3}{2}}\nab^q\gac_{AB}|_{p,S}\les&\ep_0,\label{estgac}\\
|r^{-1+q-\frac{2}{p}}|u|^{\frac{s-3}{2}}\nab^q\inc_{AB}|_{p,S}\les&\ep_0.\label{estinc}
\end{align}
\end{proposition}
\begin{proof}
We recall from Proposition \ref{metriceqre}
\begin{align*}
\nab_4(\gac_{AB})-\trch(\gac_{AB})=&\O_{-2}^0\c(\hchc,\trchc)+\O_{-1}^0\c\Omc+r^3\Gag\c\Gag.
\end{align*}
Differentiating it by $r\nab$ and applying Corollary \ref{commutation}, we deduce
\begin{align*}
    \nab_4(r\nab\gac_{AB})-\trch(r\nab\gac_{AB})=&\O_{-3}^0\c\nab(\hchc,\trchc)+\O_{-2}^0\c\nab\Omc+r^3\Gag\c\Gag^{(1)}.
\end{align*}
Applying Lemma \ref{evolution} and Propositions \ref{propOmtrch}, \ref{propOmc} and \ref{prophchc}, we obtain for $q=0,1$
\begin{align*}
    &|r^{-2-\frac{2}{p}}(r\nab)^q\gac_{AB}|_{p,S}\\
    \les&|r^{-2-\frac{2}{p}}(r\nab)^q\gac_{AB}|_{p,S_*(u)}+\int_{\ub}^{\ub_*}\left(|r^{-\frac{2}{p}}(r\nab)^q(\hchc,\trchc,r^{-1}\Omc)|_{p,S}+|r^{-1-\frac{2}{p}}\Gag\c\Gag^{(1)}|_{p,S}\right)d\ub\\
    \les & \frac{\ep_0}{r|u|^\frac{s-3}{2}}+\int_\ub^{\ub_*}\left(\frac{\ep_0}{r^2|u|^\frac{s-3}{2}}+\frac{\ep^2}{r^3|u|^{s-3}}\right)d\ub\\
    \les &\frac{\ep_0}{r|u|^\frac{s-3}{2}},
\end{align*}
which implies \eqref{estgac}. The proof of \eqref{estinc} is similar and left to the reader. This concludes the proof of Proposition \ref{propgac}.
\end{proof}
\subsubsection{Estimate for \texorpdfstring{$\osc(r)$}{}}
\begin{proposition}\label{proposcr}
We have the following estimate:
\begin{equation}
    \sup_{\kk'}r^\frac{s-3}{2}|{r'}-r|\les\ep_0.
\end{equation}
\end{proposition}
\begin{proof}
Notice that we have
\begin{align*}
    {\det\ga_{AB}}-{\det(\ga_{K})_{AB}}&=\left(\ga_{11}\ga_{22}-\ga_{12}^2\right)-\left((\ga_{K})_{11}(\ga_K)_{22}-(\ga_K)_{12}^2\right)\\
    &=\gac_{11}(\ga_K)_{22}+(\ga_K)_{11}\gac_{22}-2\gac_{12}(\ga_K)_{12}+\gac_{11}\gac_{22}-\gac_{12}^2\\
    &=\O_{-2}^0\c\gac_{AB}+r^6\Gag\c\Gag. 
\end{align*}
Hence, we obtain
\begin{align*}
    4\pi\left|r^2-r_K^2\right|&=\left|\int_{\mathbb{S}^2} \left(\sqrt{\det\ga_{AB}}-\sqrt{\det(\ga_K)_{AB}}\right)d\si_{\mathbb{S}^2}\right| \\
    &\les\int_{\mathbb{S}^2} \left|\frac{{\det\ga_{AB}}-{\det(\ga_{K})_{AB}}}{\sqrt{\det\ga_{AB}}+\sqrt{\det(\ga_{K})_{AB}}}\right|d\si_{\mathbb{S}^2}\\
    &\les |\gac_{AB}|+r^4|\Gag\c\Gag|\\
    &\les \frac{r\ep_0}{|u|^\frac{s-3}{2}}+r^4\frac{\ep^2}{r^4|u|^{s-3}}\\
    &\les \frac{r\ep_0}{|u|^\frac{s-3}{2}},
\end{align*}
which implies
\begin{equation}\label{rcheck}
    \sup_{\kk}|r-r_K|\les \frac{\ep_0}{|u|^\frac{s-3}{2}}.
\end{equation}
Similarly, we have
\begin{equation}\label{r'check}
    \sup_{\kk'}|r'-r_K'|\les \frac{\ep_0}{r^\frac{s-3}{2}}.
\end{equation}
Proceeding as in \eqref{estKK'scalar}, we obtain for any $p\in\kk'$
\begin{align*}
    |r'_K(p)-r_K(p)|&=|r_{Kerr}(\Phi'(p))-r_{Kerr}(\Phi(p))|\\
    &\les|\pr_u(r_{Kerr})||u'-u|+|\pr_A(r_{Kerr})||{x'}^A-x^A|\\
    &\les |u'-u|\\
    &\les \frac{\ep_0}{r^\frac{s-3}{2}},
\end{align*}
where we used Proposition \ref{osc0ucontrol}. Combining with \eqref{rcheck} and \eqref{r'check}, we infer
\begin{equation}
    \sup_{\kk'}r^\frac{s-3}{2}|r'-r|\les \ep_0.
\end{equation}
This concludes the proof of Proposition \ref{proposcr}.
\end{proof}
\subsection{Estimate for \texorpdfstring{$\OO_{\bGa}$}{}}\label{Ga}
\begin{proposition}\label{Ga0}
We have the following estimates:
\begin{align}
\begin{split}\label{estGa0}
    |r^{2-\frac{2}{p}}|u|^{\frac{s-3}{2}}(\Jc,\Lc)|_{p,S}&\les\ep_0,\\
    |r^{1-\frac{2}{p}}|u|^{\frac{s-1}{2}}\Jbc|_{p,S}&\les\ep_0.
\end{split}
\end{align}
\end{proposition}
\begin{proof}
We recall from Lemma \ref{bGa}
\begin{align*}
    \bGa_{\ub A}^B=&\f12 \ga^{BC}\pr_{\ub}(\ga_{AC}),\\
    \bGa_{uA}^B=&\f12\ga^{BC}\left(\pr_u(\ga_{AC})-\frac{\pr}{\pr x^A}\left(\ga_{CD} b^D\right)+\frac{\pr}{\pr x^C}\left(\ga_{AD}b^D\right)\right)+\frac{b^B}{4\Om^2}\left(2\frac{\pr}{\pr x^A}(\Om^2)-\frac{\pr}{\pr \ub}(\ga_{AD}b^D)\right),\\
    \bGa_{BC}^A=&\f12\ga^{AD}\left(\frac{\pr \ga_{BD}}{\pr x^C}+\frac{\pr \ga_{CD}}{\pr x^B}-\frac{\pr \ga_{BC}}{\pr x^D}\right) +\frac{b^A}{4\Om^2}\frac{\pr\ga_{BC}}{\pr\ub}.
\end{align*}
Hence, we have from Definition \ref{dfGamma}
\begin{align*}
    \Jc^B_A=\O_2^0\c \pr_\ub(\gac_{AB})+\O_3^0\c\gac_{AB}+r^{-1}\Gag\c\pr_\ub(\gac_{AB}).
\end{align*}
We have from Proposition \ref{metriceqre}
\begin{align*}
    \pr_{\ub}(\gac_{AB})=\trch(\gac_{AB})+\O_{-2}^0\c(\trchc,\hchc,r^{-1}\Omc)+r^3\Gag\c\Gag,
\end{align*}
which implies from Propositions \ref{propOmtrch}, \ref{prophchc}, \ref{propOmc} and \ref{propgac}
\begin{equation*}
    |r^{-\frac{2}{p}}\pr_{\ub}(\gac_{AB})|_{p,S}\les |r^{-1-\frac{2}{p}}\gac_{AB}|_{p,S}+ |r^{2-\frac{2}{p}}(\trchc,\hchc,r^{-1}\Omc)|_{p,S}+|r^{3-\frac{2}{p}}\Gag\c\Gag|_{p,S}\les \frac{\ep_0}{|u|^\frac{s-3}{2}}.
\end{equation*}
Thus, we infer from Proposition \ref{propgac}
\begin{align*}
    |r^{2-\frac{2}{p}}\Jc|_{p,S}&\les |r^{-1-\frac{2}{p}}\gac_{AB}|_{p,S}+|r^{-\frac{2}{p}}\pr_\ub(\gac_{AB})|_{p,S}\les\frac{\ep_0}{|u|^\frac{s-3}{2}}.
\end{align*}
Next, we have from Definition \ref{dfGamma}
\begin{align*}
\Lc_{BC}^A=\O_3^2\c\pr_\ub(\gac_{AB})+\O_2^0\c\dko\gac_{AB}+\O_{-1}^0\c\bcc^A+r^2\Gag\c\Gag^{(1)}.
\end{align*}
Thus, we deduce from Propositions \ref{propgac} and \ref{propbcc}
\begin{align*}
    |r^{1-\frac{2}{p}}\Lc_{BC}^A|_{p,S}&\les \frac{M^2}{r^2}|r^{-\frac{2}{p}}\pr_\ub(\gac_{AB})|_{p,S}+|r^{-1-\frac{2}{p}}\dko\gac_{AB}|_{p,S}+|r^{2-\frac{2}{p}}\bcc^A|_{p,S}+|r^{3-\frac{2}{p}}\Gag\c\Gag^{(1)}|_{p,S}\\
    &\les\frac{\ep_0}{|u|^\frac{s-3}{2}}+\frac{\ep^2}{r|u|^{s-3}}\\
    &\les\frac{\ep_0}{|u|^\frac{s-3}{2}},
\end{align*}
which implies
\begin{equation*}
    |r^{2-\frac{2}{p}}\Lc|_{p,S}\les \frac{\ep_0}{|u|^\frac{s-3}{2}}.
\end{equation*}
Finally, we have from Definition \ref{dfGamma}
\begin{align*}
    \Jbc_A^B=\O_{2}^0\c(\pr_u(\gac_{AB}),\pr_{\ub}(\gac_{AB}),r^{-1}\dko\gac_{AB})+\O_3^2\c\dko\Omc+\O_0^0\c\dko\bcc^A+\O_1^2\c\pr_{\ub}(\bcc^A)+r\Gag\c\Gag^{(1)}.
\end{align*}
We recall from Propositions \ref{metriceqre}, \ref{propOmtrch}, \ref{propOmc}, \ref{prophchc}, \ref{propbcc} and \ref{propgac} 
\begin{align*}
    |r^{-\frac{2}{p}}\pr_u(\gac_{AB})|_{p,S}&\les |r^{-1-\frac{2}{p}}\gac_{AB}|+|r^{2-\frac{2}{p}}(\hchbc,\trchbc,r^{-1}\Omc,\bcc^A)|_{p,S}+M^2|r^{-\frac{2}{p}}\Gag|_{p,S}+|r^{3-\frac{2}{p}}\Gag\c\Gag|_{p,S}\\
    &\les \frac{\ep_0}{|u|^\frac{s-3}{2}}+\frac{r\ep_0}{|u|^{\frac{s-1}{2}}}+\frac{M^2\ep}{r^2|u|^\frac{s-3}{2}}+\frac{\ep^2}{r|u|^{s-3}}\\
    &\les\frac{r\ep_0}{|u|^{\frac{s-1}{2}}},
\end{align*}
and
\begin{align*}
    |r^{-\frac{2}{p}}\pr_\ub(\bcc^A)|_{p,S}&\les |r^{-1-\frac{2}{p}}\zec|_{p,S}+M^2|r^{-3-\frac{2}{p}}\Gag|_{p,S}+|r^{-\frac{2}{p}}\Gag\c\Gag|_{p,S}\\
    &\les \frac{\ep_0}{r^3|u|^\frac{s-3}{2}}+\frac{M^2\ep}{r^5|u|^\frac{s-3}{2}}+\frac{\ep^2}{r^4|u|^{s-3}}\\
    &\les \frac{\ep_0}{r^3|u|^\frac{s-3}{2}}.
\end{align*}
Combining the above estimates, we deduce from Propositions \ref{propOmc}, \ref{propbcc} and \ref{propgac}
\begin{align*}
    |r^{1-\frac{2}{p}}\Jbc|_{p,S}&\les |r^{-1-\frac{2}{p}}(\pr_u(\gac_{AB}),\pr_{\ub}(\gac_{AB}),r^{-1}\dko\gac_{AB})|_{p,S}+M^2|r^{-2-\frac{2}{p}}\dko\Omc|_{p,S}\\
    &+|r^{1-\frac{2}{p}}\dko\bcc^A|_{p,S}+M^2|r^{-\frac{2}{p}}\bcc^A|_{p,S}+|r^2\Gag\c\Gag^{(1)}|_{p,S}\\
    &\les \frac{\ep_0}{|u|^\frac{s-1}{2}}+\frac{M^2\ep_0}{r^3|u|^\frac{s-3}{2}}+\frac{\ep_0}{r|u|^\frac{s-3}{2}}+\frac{M^2\ep_0}{r^2|u|^\frac{s-3}{2}}+\frac{\ep^2}{r^2|u|^{s-3}}\\
    &\les \frac{\ep_0}{|u|^\frac{s-1}{2}}.
\end{align*}
This concludes the proof of Proposition \ref{Ga0}.
\end{proof}
In view of Propositions \ref{estOOb0}-\ref{Ga0}, we obtain \eqref{conM34}. This concludes the proof of Theorem M3.
\section{Initialization and extension (Theorems M0, M2 and M4)}\label{initiallast}
\subsection{Control of the initial data layer (Theorem M0)}\label{osc}
The goal of this section is to prove Theorem M0, which we recall below for convenience.
\begin{thmM0}
Assume that\footnote{Recall that quantities with prime are associated to the foliation of the initial data layer, see \eqref{primenotation1} and \eqref{primenotation2}.}
\begin{equation}\label{assum1}
    \OO'\leq\ep_0,\qquad \Rk'\leq\ep_0,\qquad\OO\leq\ep,\qquad \osc\leq\ep.
\end{equation}
Then, we have
\begin{equation}\label{con1}
\Rk_0\les\ep_0.
\end{equation}
\end{thmM0}
\begin{proof}
We denote for $e_A$ an orthogonal frame on $S$
\begin{align*}
    R:=\{\a_{AB},\,\b_{A},\,\rho,\,\si,\,\bb_{A},\,\aa_{AB}\},
\end{align*}
and for $e_A'$ an orthogonal frame on $S'$
\begin{align*}
    R':=\{\a'_{AB},\,\b'_A,\,\rho',\,\si',\,\bb'_{A},\,\aa'_{AB}\}.
\end{align*}
The transformation formulae for curvature in Proposition \ref{transformation} can be written as follows:
\begin{align}\label{Rtrans}
    R=R'+f \c R'+\lot,
\end{align}
which implies
\begin{equation*}
    \Rc=\Rc'+R_K-R_K'+f\c (\Rc'+R_K')+\lot
\end{equation*}
We have from Lemma \ref{estKK'}
\begin{align*}
    |R_K-R'_K|\les \frac{M\ep}{r^\frac{s+5}{2}}.
\end{align*}
Hence, we obtain
\begin{align*}
    |\Rc|\les|\Rc'|+|f\c\Rc'|+\frac{M}{r^3}|f|+|R_K-R'_K|\les|\Rc'|+\frac{M\ep}{r^\frac{s+5}{2}},
\end{align*}
where we used $\osc(f)\leq\ep$. Then, we have
\begin{align*}
    \int_{\Si_0\setminus K}r^s|\Rc|^2\les\int_{\Si_0\setminus K} r^s|\Rc'|^2+r^s\left(\frac{M^2\ep^2}{r^{s+5}}\right)\les {\Rk'_0}^2+\frac{M^2\ep^2}{R_0^2}\les\ep_0^2.
\end{align*}
Moreover, we have from \eqref{Rtrans}
\begin{equation}\label{prRtrans}
    \pr R=(\pr R)'+\pr f\c R'+f\c(\pr R)'+\lot,
\end{equation}
where
\begin{equation*}
    \pr R:=\{(\nab_\mu\a)_{AB},\,(\nab_\mu\b)_{A},\,\nab_\mu\rho,\,\nab_\mu\si,\,(\nab_\mu\bb)_{A},\,(\nab_\mu\aa)_{AB}\},\qquad \mu=1,2,3,4.
\end{equation*}
We have from Lemma \ref{estKK'}
\begin{align*}
    |(\pr R)_K-(\pr R)'_K|\les\frac{M\ep}{r^\frac{s+7}{2}}.
\end{align*}
Moreover, we have from Proposition \ref{useful}
\begin{equation*}
    \widecheck{\pr R}=\pr\Rc+(\pr-\pr_K)R_K=\pr\Rc+(\Jbc,\Lc,\Jc)\c\O_3^1,
\end{equation*}
which implies
\begin{align*}
    \int_{\Si_0\sm K}r^{s+2}|\pr\Rc|^2&\les\int_{\Si_0\sm K}r^{s+2}|\widecheck{\pr R}|^2+r^{s+2}\frac{M^2}{r^6}|\Gaw^{(1)}|^2\\
    &\les\int_{\Si_0\sm K}r^{s+2}|\widecheck{\pr R}|^2+\int_{R_0}^\infty \frac{M^2}{r^{4-s}}|\Gaw^{(1)}|_{2,S}^2\\
    &\les\int_{\Si_0\sm K}r^{s+2}|\widecheck{\pr R}|^2+\int_{R_0}^\infty \frac{M^2}{r^{4-s}}\frac{\ep^2}{r^{s-1}}\\
    &\les\int_{\Si_0\sm K}r^{s+2}|\widecheck{\pr R}|^2+\ep_0^2.
\end{align*}
Similarly, we have
\begin{equation*}
    \int_{\Si_0\sm K}r^{s+2}|\widecheck{(\pr R)'}|^2\les\int_{\Si_0\sm K}r^{s+2}|\pr'\Rc'|^2+\ep_0^2\les\ep_0^2.
\end{equation*}
Hence, we obtain from \eqref{prRtrans} and $\osc(f)\leq\ep$ that
\begin{align*}
\int_{\Si_0\sm K}r^{s+2}|\pr\Rc|^2 &\les\ep_0^2+\int_{\Si_0\sm K}r^{s+2}|\widecheck{\pr R}|^2\\
&\les\ep_0^2+\int_{\Si_0\sm K}r^{s+2}|\widecheck{(\pr R)'}|^2+r^{s+2}\frac{M^2\ep^2}{r^{s+7}}+|\pr f|^2 r^{s+2}|R'|^2+|f|^2 r^{s+2}|\pr'R'|^2\\
&\les\ep_0^2+\int_{\Si_0\sm K}r^{s}|\Rc'|^2+r^{s+2}|\widecheck{(\pr R)'}|^2+r^{s+2}|f|^2 |(\pr'R')_K|^2\\
&\les\ep_0^2+\int_{\Si_0\sm K}r^{s+2}\frac{\ep^2}{r^{s-1}}\frac{M^2}{r^8}\\
&\les\ep_0^2.
\end{align*}
Thus, we deduce
\begin{align*}
\Rk_0^2=\sum_{R}\int_{\Si_0\setminus K}r^s|\dk^{\leq 1}\Rc|^2\les\sum_{R}\int_{\Si_0\setminus K}r^s|\Rc|^2+r^{s+2}|\pr\Rc|^2\les\ep_0^2.
\end{align*}
This concludes the proof of Theorem M0.
\end{proof}
\subsection{Control of the last slice (Theorem M2)}\label{lastslice}
The goal of this section is to prove Theorem M2 which we recall below for convenience.
\begin{thmM2}
Let $\Cb_*$ endowed with the geodesic type foliation of Definition \ref{geodesicfoliation}. We assume that
\begin{align*}
    \OO'(\Cb_*\cap\Si_0)\leq\ep_0,\qquad\RRb(\Cb_*)\les\ep_0,\qquad {\OO}^{*}(\Cb_*)\leq\ep.
\end{align*}
Then, we have
\begin{align*}
    \OO^*(\Cb_*)\les\ep_0.
\end{align*}
\end{thmM2}
\subsubsection{Geodesic type foliation on \texorpdfstring{$\Cb_*$}{}}
\begin{definition}\label{geodesicfoliation}
A foliation $(u,\ub_*)$ on the last slice $\Cb_*$ is called a geodesic type foliation if it satisfies the following conditions:
\begin{equation}\label{geodesicfo}
    \Omc=0,\qquad\quad \bcc^A=0.
\end{equation}
\end{definition}
\begin{remark}
Notice that \eqref{geodesicfo} is equivalent to
\begin{align*}
        e_3(u)=\Om_K^{-1},\qquad e_3(x^A)=\Om^{-1}_K\bu_K^A.
    \end{align*}
In particular, the choice of the foliation is prescribed by $\Omc=0$ while $\bcc^A=0$ corresponds in fact to the transport of the $x^A$ coordinates. The local existence of geodesic type foliations as in Definition \ref{geodesicfoliation} is completely analogous to the one for the standard geodesic foliation, i.e. $\Om=\frac{1}{2}$, with a transport of the $x^A$ coordinates by $e_3(x^A)=0$.
\end{remark}
\begin{lemma}\label{ombzero}
   Let $\Cb_*$ endowed with the geodesic type foliation of Definition \ref{geodesicfoliation}. Then, we have 
    \begin{equation*}
        \ombc=0,\quad \mbox{ on }\Cb_*.
    \end{equation*}
\end{lemma}
\begin{proof}
    We have from Lemma \ref{checke4e3} and \eqref{geodesicfo}
    \begin{equation*}
        \Om e_3-\Om_K (e_3)_K=\bcc^A\pr_{x^A}=0,
    \end{equation*}
    which implies from \eqref{geodesicfo}
    \begin{equation}\label{e3equality}
        e_3=(e_3)_K.
    \end{equation}
    Hence, we have from \eqref{geodesicfo} and \eqref{e3equality}
    \begin{align*}
        \omb=-\frac{1}{2}\nab_3(\log\Om)=-\frac{1}{2}\nab_{(e_3)_K}(\log\Om_K)=\omb_K.
    \end{align*}
    This concludes the proof of Lemma \ref{ombzero}.
\end{proof}
\begin{remark}\label{newGag}
     Under the assumption $\OO^*(\Cb_*)\leq\ep$, we have on $\Cb_*$
\begin{align*}
    \dko\trchbc&\in\Gag,\qquad\quad\, \dko\gac\in r\Gag,\qquad\, \dko(\inc)\in r\Gag,\\
    \Lc&\in\Gag,\qquad\qquad\;\,\,\, \Jbc\in\Gab,\qquad\qquad \,\;\; r\widecheck{\mub}\in\Gag.
\end{align*}
\end{remark}
\begin{proposition}\label{nab3geo}
We have the following linearized equations on $\Cb_*$:
\begin{align}
\begin{split}\label{nab3all}
\Om\nab_3(\widecheck{\Om\trch})+\f12\Om\trchb\,\widecheck{\Om\trch}&=-2\Om^2\muc+\Gag\c\O_3^2+\Gag\c\Gag,\\
\sdiv(\hchc)&=-\bc+\O_0^0\c\nab\trchc+\O_1^0\c\zec+\O_3^2\c\Gag+\Gag\c\Gag,\\
\nab_3\,\widecheck{\Om\trchb}+\trchb\,\widecheck{\Om\trchb}&=\Gag\c\O_2^1+\Gab\c\Gab,\\
\sdiv(\widecheck{\hchb})&=\bbc+\O_0^0\c\nab\trchbc+\O_1^0\c\zec+\O_3^2\c\Gab+\Gab\c\Gag,\\
\Om\nab_3\mubmc+\Om\trchb\,\mubmc&=\Om\trchb\,\rhoc+\O_3^2\c (\rhoc,\bbc)+\O_3^1\c\Gag+\O_0^0\c\etabc\c\bbc+r^{-1}\Gab\c\Gab.
\end{split}
\end{align}
We also have the following equation:
\begin{align}
\begin{split}\label{nab3trchbc}
&\nab_3(r\De\widecheck{\Om\trchb})+\frac{3}{2}\trchb (r\De\widecheck{\Om\trchb})\\
=&\O_2^2\c\nab^2\hchbc+\O_3^1\c\Gag^{(1)}+r\Gab\c\nab^2\hchbc+\O_{-1}^0\c|\nab\hchbc|^2+r^{-1}\Gab\c\Gag^{(1)}.
\end{split}
\end{align}
\end{proposition}
\begin{proof}
Notice that \eqref{nab3all} follows directly from Proposition \ref{nullstructure} and Remark \ref{newGag}. Next, we have from Proposition \ref{nullstructureA}
\begin{align*}
\Om\nab_3\,\widecheck{\Om\trchb}+(\Om\tr\chib)\,\widecheck{\Om\trchb}=&-4\Om\tr\chib\,\widecheck{\Om\omb}+\fl[\nab_3\,\trchc]+\err[\nab_3\,\trchc],\\
\fl[\nab_3\,\trchbc]:=&-2\widecheck{\Om\hchb}(\Om_K\hchb_K)-4\Om_K\omb_K\widecheck{\Om\trchb}-\bcc^A\nab_{\pr_{x^a}}\tr_K\chib_K,\\
\err[\nab_3\,\trchbc]:=&\f12(\widecheck{\Om\trchb})^2-|\widecheck{\Om\hchb}|^2.
\end{align*}
Recalling from Lemma \ref{ombzero} that $\Omc=0$, $\ombc=0$ and $\bcc^A=0$ on $\Cb_*$, we obtain
\begin{align}
\Om\nab_3\,\widecheck{\Om\trchb}+(\Om\tr\chib)\,\widecheck{\Om\trchb}=&\O_3^2\c\hchbc+\O_2^1\c\trchbc+\f12(\widecheck{\Om\trchb})^2-|\widecheck{\Om\hchb}|^2.
\end{align}
Differentiating it by $\nab$ and applying Proposition \ref{commkn}, we deduce
\begin{align*}
    \Om\nab_3(\nab\,\widecheck{\Om\trchb})+(\Om\trchb)\nab\,\widecheck{\Om\trchb}&=\O_3^2\c\nab\hchbc+\O_2^1\c\nab\trchbc+\O_0^0\c \nab|\hchbc|^2+r^{-1}\Gag\c\Gag+[\Om\nab_3,\nab]\widecheck{\Om\trchb}\\
    &=\O_3^2\c\nab\hchbc+\O_3^1\c\Gag+\O_0^0\c \nab|\hchbc|^2-\f12\Om\trchb\nab\,\widecheck{\Om\trchb}+r^{-1}\Gab\c\Gag,
\end{align*}
which implies
\begin{align*}
    \Om\nab_3(\nab\,\widecheck{\Om\trchb})+\frac{3}{2}\Om\trchb\nab\,\widecheck{\Om\trchb}=\O_3^2\c\nab\hchbc+\O_3^1\c\Gag+\O_0^0\c\nab|\hchbc|^2+r^{-1}\Gab\c\Gag.
\end{align*}
Differentiating it by $r\sdiv$ and applying Corollary \ref{commutation}, we infer
\begin{align*}
    \Om\nab_3(r\De\widecheck{\Om\trchb})+\frac{3}{2}\Om\trchb (r\De\widecheck{\Om\trchb})=\O_2^2\c\nab^2\hchbc+\O_3^1\c\Gag^{(1)}+r\Gab\c\nab^2\hchbc+\O_{-1}^0\c|\nab\hchbc|^2+r^{-1}\Gab\c\Gag^{(1)}.
\end{align*}
This concludes the proof of Proposition \ref{nab3geo}.
\end{proof}
\subsubsection{Estimate for \texorpdfstring{$\OO^*_1$}{}}
In the remainder of Section \ref{lastslice}, we denote
\begin{equation}
    S:=S(u,\ub_*),\qquad\quad S_*:=\Cb\cap\Si_0.
\end{equation}
\begin{proposition}\label{OO*}
For a geodesic type foliation on $\Cb_*$ satisfying
\begin{equation}\label{assumOO*1}
    \OO'(\Cb_*\cap\Si_0)\leq\ep_0,\qquad \RRb(\Cb_*)\les\ep_0,\qquad \OO^*(\Cb_*)\leq\ep,
\end{equation}
we have
\begin{align}
    \OO^*_1\les\ep_0.
\end{align}
\end{proposition}
\begin{proof}
Notice that the assumption $\OO'(\Cb_*\cap\Si_0)\leq \ep_0$ controls the initial data on the last sphere $S_*$. The idea of the proof is to transport the estimates in the direction of $\nab_3$ along $\Cb_*$ using Proposition \ref{nab3geo} and Lemma \ref{evolution}. The proof is similar to Section \ref{Ricciestimates}, and in fact easier since $\Omc=0$, $\ombc=0$ and $\bcc^A=0$ on $\Cb_*$.\\ \\
{\bf Step 1.} Estimate for $\OO^*_{0,1,2}(\trchbc)$.\\ \\
The estimates for $\OO_0^*(\trchbc)$ and $\OO_1^*(\trchbc)$ follow directly from Propositions \ref{OOb0Omtrchb} and \ref{proptrchbc1} in the particular case $\Omc=0$ and $\ombc=0$. Then, we focus on $\OO_2^*(\trchbc)$.\\\\
We recall from \eqref{nab3trchbc}
\begin{align*}
\nab_3(r\De\widecheck{\Om\trchb})+\frac{3}{2}\trchb (r\De\widecheck{\Om\trchb})=\O_2^2\c\nab^2\hchbc+\O_3^1\c\Gag^{(1)}+r\Gab\c\nab^2\hchbc+\O_{-1}^0\c|\nab\hchbc|^2+r^{-1}\Gab\c\Gag^{(1)}.
\end{align*}
Moreover, we have from Proposition \ref{nab3geo} and Remark \ref{newGag}
\begin{align*}
\nab^2\hchbc=\nab\bbc+r^{-2}\Gag^{(1)},\qquad\quad\nab\hchbc=\bbc+r^{-1}\Gag.
\end{align*}
Hence, we obtain
\begin{align*}
\nab_3(r\De\widecheck{\Om\trchb})+\frac{3}{2}\trchb (r\De\widecheck{\Om\trchb})=\O_2^2\c\nab\bbc+r\Gab\c\nab\bbc+\O_{-1}^0\c|\bbc^{(0)}|^2+\O_3^1\c\Gag^{(1)}+r^{-1}\Gab\c\Gag^{(1)}.
\end{align*}
We recall from Proposition \ref{Bianchieq}
\begin{align*}
    \nab_3(\rhoc,\sic)+\frac{3}{2}\trchb\,(\rhoc,\sic)=-d_1\bbc+\Gag^{(1)}\c\O_3^1+\O_3^2\c\aac+\Gag\c\aac.
\end{align*}
We deduce
\begin{align*}
    \nab_3\,\Xib+\frac{3}{2}\trchb\,\Xib=\O_3^2\c\aac+\O_{-1}^0\c|\bbc^{(0)}|^2+\O_3^1\c\Gag^{(1)}+\Gag\c\aac+r^{-1}\Gab\c\Gag^{(1)},
\end{align*}
where $\Xib$ is the following renormalized quantity
\begin{align*}
    \Xib=r\De\widecheck{\Om\trchb}+\O_2^2\c(\rhoc,\sic)+r\Gab\c(\rhoc,\sic).
\end{align*}
Applying Lemma \ref{evolution}, we deduce for $p\in [2,4]$
\begin{align*}
    &|r^{3-\frac{2}{p}}\Xib|_{p,S}\\
    \les& |r^{3-\frac{2}{p}}\Xib|_{p,S_*}+\int_{u_0(\ub_*)}^u M^2|r^{-\frac{2}{p}}\aac|_{p,S}+|r^{4-\frac{2}{p}}|\bbc^{(0)}|^2|_{p,S}+M|r^{-\frac{2}{p}}\Gag^{(1)}|+|r^{3-\frac{2}{p}}\Gag\c\aac|_{p,S}+|r^{2-\frac{2}{p}}\Gab\c\Gag^{(1)}|_{p,S}\\
    \les&\frac{\ep_0}{r^\frac{s-3}{2}}+\int_{u_0(\ub_*)}^u \frac{M^2\ep_0}{r|u|^\frac{s+1}{2}}+\frac{M}{r^2|u|^\frac{s-3}{2}}+\frac{\ep_0}{|u|^{s-1}}+\frac{\ep^2}{r|u|^{s-2}}+|r^{4-\frac{2}{p}}|\bbc^{(0)}|^2|_{p,S}\\
    \les&\frac{\ep_0}{|u|^\frac{s-3}{2}}+\int_{u_0(\ub_*)}^u r^{2}|r^\frac{3}{4}\bbc^{(0)}|_{8,S}^2\\
    \les&\frac{\ep_0}{|u|^\frac{s-3}{2}}+\int_{u_0(\ub_*)}^u |r\bbc^{(1)}|_{2,S}^2\\
    \les&\frac{\ep_0}{|u|^\frac{s-3}{2}}+\frac{\ep^2}{|u|^{s-2}}\\
    \les&\frac{\ep_0}{|u|^\frac{s-3}{2}},
\end{align*}
where we used the Sobolev inequality $|r^{-\frac{1}{4}}U|_{8,S}\les|r^{-1}U^{(1)}|_{2,S}$ for any $S$--tangent tensor $U$. Hence, we infer from Proposition \ref{ellipticest}
\begin{align*}
    |r^{4-\frac{2}{p}}\nab^2\widecheck{\Om\trchb}|_{p,S}\les|r^{3-\frac{2}{p}}\Xib|_{p,S}+\frac{M^2}{r^2}|r^{3-\frac{2}{p}}(\rhoc,\sic)|_{p,S}+|r^{4-\frac{2}{p}}\Gab\c(\rhoc,\sic)|_{p,S}\les \frac{\ep_0}{|u|^\frac{s-3}{2}}.
\end{align*}
Thus, we obtain
\begin{align}\label{esttrchbc012}
    \OO^*_{0,1,2}(\trchbc)\les\ep_0.
\end{align}
{\bf Step 2.} Estimate for $\OO_{0,1}^*(\mubmc)$.\\ \\
The estimate for $\OO_0^*(\mubmc)$ follows directly from Proposition \ref{propetabc1} in the particular case $\Omc=0$ and $\ombc=0$. Next, we recall from \eqref{nab3all}
\begin{align*}
    \Om\nab_3\mubmc+\Om\trchb\mubmc=\Om\trchb\,\rhoc+\O_3^2\c (\rhoc,\bbc)+\O_3^1\c\Gag+\O_0^0\c\etabc\c\bbc+r^{-1}\Gab\c\Gab.
\end{align*}
Differentiating it by $r\nab$ and applying Corollary \ref{commutation}, we infer
\begin{align*}
\nab_3(r\nab\mubmc)+\trchb(r\nab\mubmc)=\O_0^0\c(\nab\rhoc,\nab\bbc)+\O_0^0\c\rhoc\c\bbc+r\Gag\c\nab\bbc+\O_3^1\c\Gag^{(1)}+r^{-1}\Gab\c\Gab^{(1)}.
\end{align*}
We deduce
\begin{align*}
\nab_3\Ub+\trchb\,\Ub=\O_0^0\c\rhoc\c\bbc+\O_3^1\c\Gag^{(1)}+r^{-1}\Gab\c\Gab^{(1)},
\end{align*}
where the renormalized quantity $\Ub$ is given by
\begin{align*}
    \Ub=r\nab\mubmc+\O_0^0\c(\bc,\rhoc,\sic)+r\Gag\c(\rhoc,\sic).
\end{align*}
Applying Lemma \ref{evolution}, we have
\begin{align*}
    |r^{2-\frac{2}{p}}\Ub|_{p,S}&\les |r^{2-\frac{2}{p}}\Ub|_{p,S_*}+\int_{u_0(\ub_*)}^u |r^{2-\frac{2}{p}}\rhoc\c\bbc|_{p,S}+M|r^{-1-\frac{2}{p}}\Gag^{(1)}|_{p,S}+|r^{1-\frac{2}{p}}\Gab\c\Gab^{(1)}|_{p,S}\\
    &\les \frac{\ep_0}{r^\frac{s-1}{2}}+\frac{M\ep}{r^2|u|^\frac{s-3}{2}}+\frac{\ep^2}{r|u|^{s-2}}+\int_{u_0(\ub_*)}^u |r^\frac{3}{2}\rhoc\c\bbc|_{4,S}\\
    &\les\frac{\ep_0}{r|u|^\frac{s-3}{2}}+\int_{u_0(\ub_*)}^u r^2|r^{-\frac{1}{4}}\rhoc|_{8,S}|r^{-\frac{1}{4}}\bbc|_{8,S}\\
    &\les \frac{\ep_0}{r|u|^\frac{s-3}{2}}+\int_{u_0(\ub_*)}^u |\rhoc^{(1)}|_{2,S}|\bbc^{(1)}|_{2,S}\\
    &\les \frac{\ep_0}{r|u|^\frac{s-3}{2}}+r^{-2}\left(\int_{u_0(\ub_*)}^u |r\rhoc^{(1)}|_{2,S}^2\right)^\f12 \left(\int_{u_0(\ub_*)}^u |r\bbc^{(1)}|_{2,S}^2\right)^\f12 \\
    &\les \frac{\ep_0}{r|u|^\frac{s-3}{2}}+\frac{\ep^2}{r^2|u|^{s-2}}\\
    &\les \frac{\ep_0}{r|u|^\frac{s-3}{2}},
\end{align*}
where we used the Sobolev inequality $|r^{-\frac{1}{4}}U|_{8,S}\les|r^{-1}U^{(1)}|_{2,S}$ for any $S$--tangent tensor $U$. Hence, we obtain
\begin{align*}
    |r^{3-\frac{2}{p}}\nab\mubmc|_{p,S}\les |r^{2-\frac{2}{p}}\Ub|_{p,S}+|r^{2-\frac{2}{p}}(\bc,\rhoc,\sic)|_{p,S}+|r^{3-\frac{2}{p}}\Gag\c(\rhoc,\sic)|_{p,S}\les \frac{\ep_0}{r|u|^\frac{s-3}{2}}.
\end{align*}
Thus, we have
\begin{equation}\label{estmubc}
    \OO_{0,1}^*(\mubmc)\les\ep_0.
\end{equation}
{\bf Step 3.} Estimate for $\OO^*_{0,1}(\trchc)$.\\ \\
We recall from \eqref{nab3all} that
\begin{align}\label{nab3eqOmtrch}
\Om\nab_3(\widecheck{\Om\trch})=-2\widecheck{\Om^2[\mu]}+\Gag\c\O_3^2+\Gag\c\Gag.
\end{align}
We recall from \eqref{defmassaspect} and \eqref{auxillaryquantities}
\begin{align*}
    [\mu]+[\mub]=\f12\trch\,\trchb-2\De\log\Om+\hch\c\hchb-2\rho,
\end{align*}
which implies
\begin{align*}
    \widecheck{\Om^2[\mu]}=-\widecheck{\Om^2[\mub]}+\f12\widecheck{\Om\trch}\,\Om\trchb+\O_1^0\c\trchbc-2\Om^2\rhoc+\O_3^2\c\Gab+\Gag\c\Gab.
\end{align*}
Injecting it into \eqref{nab3eqOmtrch}, we deduce
\begin{align}\label{nab3Omtrchmub}
\Om\nab_3(\widecheck{\Om\trch})+\Om\trchb\,\widecheck{\Om\trch}=2\widecheck{\Om^2[\mub]}+4\Om^2\rhoc+\O_1^0\c\trchbc+\O_3^2\c\Gab+\Gab\c\Gag.
\end{align}
Applying Lemma \ref{evolution}, we infer
\begin{align*}
    &\;\;\;\,\,\,|r^{2-\frac{2}{p}}\widecheck{\Om\trch}|_{p,S}\\
    &\les |r^{2-\frac{2}{p}}\widecheck{\Om\trch}|_{p,S_*}+\int_{u_0(\ub_*)}^u |r^{2-\frac{2}{p}}\mubmc|_{p,S}+|r^{2-\frac{2}{p}}\rhoc|_{p,S}+|r^{1-\frac{2}{p}}\trchbc|_{p,S}+M^2|r^{-1-\frac{2}{p}}\Gab|_{p,S}+|r^{2-\frac{2}{p}}\Gab\c\Gag|_{p,S}\\
    &\les \frac{\ep_0}{r^\frac{s-1}{2}}+\int_{u_0(\ub_*)}^u \frac{\ep_0}{r|u|^\frac{s-3}{2}}+\frac{M^2\ep}{r^2|u|^\frac{s-1}{2}}+\frac{\ep^2}{r|u|^{s-2}}\\
    &\les \frac{\ep_0}{|u|^\frac{s-3}{2}}.
\end{align*}
Next, we differentiate \eqref{nab3Omtrchmub} by $\nab$ and apply Proposition \ref{commkn} to deduce
\begin{align*}
&\;\;\;\,\,\,\Om\nab_3(\nab\widecheck{\Om\trch})+\Om\trchb(\nab\widecheck{\Om\trch})\\
&=2\Om^2\nab\mubmc+4\Om^2\nab\rhoc-\Om\chib\c\nab\widecheck{\Om\trch}+\O_1^0\c\nab\trchbc+\Gab^{(1)}\c\O_4^2+r^{-1}\Gab\c\Gag^{(1)}\\
&=2\Om^2\nab\mubmc+4\Om^2\nab\rhoc-\f12\Om\trchb(\nab\widecheck{\Om\trch})+\O_1^0\c\nab\trchbc+\Gab^{(1)}\c\O_4^2+r^{-1}\Gab\c\Gag^{(1)},
\end{align*}
which implies from $\Omc=0$ that
\begin{equation}\label{nab3trchre}
\nab_3(\nab\widecheck{\trch})+\frac{3}{2}\trchb(\nab\widecheck{\trch})=2\nab\mubmc+4\nab\rhoc+\O_1^0\c\nab\trchbc+\Gab^{(1)}\c\O_4^2+r^{-1}\Gab\c\Gag^{(1)}.
\end{equation}
Next, we define $\chi^\dagger$ by the following transport equation:
\begin{align}\label{transchidagger}
    \Om\nab_3\chi^\dagger+\f12\Om\trchb \,\chi^\dagger =4\Om\sic\quad\mbox{ on } \Cb_*,\qquad \chi^\dagger=0\quad \mbox{ on }S_*.
\end{align}
Differentiating it by ${^*\nab}$ and applying Proposition \ref{commkn}, we infer
\begin{align*}
    \Om\nab_3 {^*\nab\chi^\dagger}+\f12\Om\trchb({^*\nab\chi^\dagger})+\f12{^*\nab}(\Om\trchb)\chi^\dagger&=4{^*\nab}(\Om\sic)+[\Om\nab_3,{^*\nab}]\chi^\dagger\\
    &=4{^*\nab}(\Om\sic)-\f12\Om\trchb({^*\nab\chi^\dagger})-\Om\hchb\c{^*\nab\chi^\dagger},
\end{align*}
which implies\footnote{We assume here that $\chi^\dagger\in\Gag$. This assumption will be improved at the end of Step 3.}
\begin{equation}\label{nab3chdagger}
    \nab_3 {^*\nab\chi^\dagger}+\trchb{^*\nab\chi^\dagger}=4{^*\nab}\sic+\O_4^2\c\Gag^{(1)}+r^{-1}\Gab\c\Gag^{(1)}.
\end{equation}
We recall from Proposition \ref{Bianchieq}
\begin{equation}\label{nab3b}
    \nab_3\bc+\frac{3}{2}\trchb\,\bc=\frac{1}{2}\trchb\,\bc+\nab\rhoc+{^*\nab\sic}+\Gag^{(1)}\c\O_3^1+r^{-1}\Gag\c\Gab^{(1)}.
\end{equation}
Combining \eqref{nab3trchre}, \eqref{nab3chdagger} and \eqref{nab3b}, we deduce
\begin{align*}
    \nab_3X +\trchb X=-2\trchb\,\bc+2\nab\mubmc+\O_1^0\c\nab\trchbc+\Gag^{(1)}\c\O_3^1+r^{-1}\Gab\c\Gag^{(1)},
\end{align*}
where
\begin{equation*}
    X:=\nab\trchc+{^*\nab}\chi^\dagger-4\bc.
\end{equation*}
Applying Lemma \ref{evolution} and \eqref{estmubc}, we obtain
\begin{align*}
    &\;\;\;\,\,\, |r^{3-\frac{2}{p}}X|_{p,S}\\
    &\les |r^{3-\frac{2}{p}}X|_{p,S_*}+\int_{u_0(\ub_*)}^u |r^{2-\frac{2}{p}}\bc|_{p,S}+|r^{3-\frac{2}{p}}\nab\mubmc|_{p,S}+|r^{3-\frac{2}{p}}\nab\trchbc|_{p,S}+M|r^{-\frac{2}{p}}\Gag^{(1)}|_{p,S}+|r^{2-\frac{2}{p}}\Gab\c\Gag^{(1)}|_{p,S}\\
    &\les \frac{\ep_0}{r^\frac{s-3}{2}}+\int_{u_0(\ub_*)}^u \frac{\ep_0}{r|u|^\frac{s-3}{2}}+\frac{M\ep}{r^2|u|^\frac{s-3}{2}}+\frac{\ep^2}{r|u|^{\frac{s-1}{2}}}\\
    &\les \frac{\ep_0}{|u|^\frac{s-3}{2}}.
\end{align*}
Then, we have from Proposition \ref{ellipticest}
\begin{align*}
    |r^{3-\frac{2}{p}}\nab(\trchc,\chi^\dagger)|_{p,S}\les |r^{3-\frac{2}{p}}X|_{p,S}+|r^{3-\frac{2}{p}}\bc|_{p,S}\les\frac{\ep_0}{|u|^\frac{s-3}{2}}.
\end{align*}
Next, applying Lemma \ref{evolution} to \eqref{transchidagger}, we deduce
\begin{align*}
    |r^{1-\frac{2}{p}}\chi^\dagger|_{p,S}\les \int_{u_0(\ub_*)}^u |r^{1-\frac{2}{p}}\sic|_{p,S}\les \frac{\ep_0}{r|u|^\frac{s-3}{2}}.
\end{align*}
Hence, we improve the assumption $\chi^\dagger\in\Gag$ and obtain
\begin{equation}\label{trchcstar}
    \OO_{0,1}(\trchc)\les\ep_0.
\end{equation}
{\bf Step 4.} Estimates for $\OO^*_{0,1}(\zec)$, $\OO_{0,1}^*(\hchc)$ and $\OO_{0,1}^*(\hchbc)$.\\ \\
We have from \eqref{estmubc}
\begin{align*}
    |r^{4-\frac{2}{p}}\nab\mubmc|_{p,S}\les \frac{\ep_0}{|u|^\frac{s-3}{2}},
\end{align*}
which implies from \eqref{auxillaryquantities} and \eqref{trchcstar}
\begin{equation*}
    |r^{4-\frac{2}{p}}\nab\mubc|_{p,S}\les |r^{4-\frac{2}{p}}\nab\mubmc|_{p,S}+|r^{3-\frac{2}{p}}\nab(\trchc,\trchbc)|_{p,S}\les\frac{\ep_0}{|u|^\frac{s-3}{2}}.
\end{equation*}
We proceed as in Proposition \ref{propetabc1} to deduce
\begin{align*}
    |r^{2-\frac{2}{p}}(r\nab)^q\etabc|_{p,S}\les\frac{\ep_0}{|u|^\frac{s-3}{2}},\qquad q=0,1.
\end{align*}
Recalling that
\begin{align*}
    2\zec=\etac-\etabc,\qquad \etac+\etabc=\widecheck{\De\log\Om}=0,
\end{align*}
we obtain
\begin{align*}
    |r^{2-\frac{2}{p}}(r\nab)^q\zec|_{p,S}=|r^{2-\frac{2}{p}}(r\nab)^q\etab|_{p,S}\les\frac{\ep_0}{|u|^\frac{s-3}{2}},\qquad q=0,1.
\end{align*}
Hence, we have
\begin{equation}\label{estzec}
    \OO^*_{0,1}(\zec)\les\ep_0.
\end{equation}
Finally, we apply Proposition \ref{prophchc} in the particular case of $\Omc=0$ to deduce
\begin{equation}\label{OOhchchchbc}
    \OO_{0,1}^*(\hchc)+\OO_{0,1}^*(\hchbc)\les\ep_0.
\end{equation}
Combining the estimates \eqref{esttrchbc012}, \eqref{estmubc}, \eqref{trchcstar}, \eqref{estzec} and \eqref{OOhchchchbc}, this concludes the proof of Proposition \ref{OO*}.
\end{proof}
\subsubsection{Estimate for \texorpdfstring{$\OO_\ga^*$}{}}
\begin{proposition}\label{propgastar}
For a geodesic type foliation on $\Cb_*$ in the sense of Definition \ref{geodesicfoliation}, we assume that
\begin{equation}\label{assumOO*ga}
    \OO'(\Cb_*\cap\Si_0)\leq\ep_0,\qquad \RRb(\Cb_*)\les\ep_0,\qquad \OO^*(\Cb_*)\leq\ep.
\end{equation}
Then, we have
\begin{align}
    \OO^*_\ga\les\ep_0.
\end{align}
\end{proposition}
\begin{proof}
We recall from Proposition \ref{metriceqre}
\begin{align}\label{dugac}
    \nab_3(\gac_{AB})-\trchb(\gac_{AB})=&\O_{-2}^0\c(\hchbc,\trchbc)+\O_3^2\c\gac_{AB}.
\end{align}
Differentiating it by $\nab$ and applying Corollary \ref{commutation}, we deduce
\begin{align*}
     \Om\nab_3(\nab\gac_{AB})-\Om\trchb(\nab\gac_{AB})&=\nab(\Om\trchb)\gac_{AB}+\O_{-2}^0\c(\nab\hchbc,\nab\trchbc)+\O_3^2\c\nab\gac_{AB}+[\Om\nab_3,\nab]\gac_{AB}\\
     &=-\f12\Om\trchb\nab\gac_{AB}+\O_{-2}^0\c(\nab\hchbc,\nab\trchbc)+\O_1^2\c\Gag+r^2\Gag\c\Gag,
\end{align*}
which implies\footnote{We use the convention introduced in Remark \ref{newGag}.}
\begin{align}\label{nab3nabgac}
    \nab_3(\nab\gac_{AB})-\f12\trchb(\nab\gac_{AB})=\O_{-2}^0\c(\nab\hchbc,\nab\trchbc)+\O_1^2\c\Gag+r^2\Gag\c\Gag.
\end{align}
Applying Lemma \ref{evolution}, we obtain
\begin{align*}
    &\;\;\;\,\,\,|r^{-1-\frac{2}{p}}\nab\gac_{AB}|_{p,S}\\
    &\les |r^{-1-\frac{2}{p}}\nab\gac_{AB}|_{p,S_*}+\int_{u_0(\ub_*)}^u\left(|r^{1-\frac{2}{p}}\nab(\hchbc,\trchbc)|_{p,S}+M^2|r^{-2-\frac{2}{p}}\Gag|_{p,S}+|r^{1-\frac{2}{p}}\Gag\c\Gag|_{p,S}\right)du\\
    &\les \frac{\ep_0}{r^\frac{s-1}{2}}+\int_{u_0(\ub_*)}^u\left(\frac{\ep_0}{r|u|^\frac{s-1}{2}}+\frac{M^2}{r^4|u|^\frac{s-3}{2}}+\frac{\ep^2}{r^3|u|^{s-3}}\right)du\\
    &\les\frac{\ep_0}{r|u|^\frac{s-3}{2}}.
\end{align*}
Differentiating \eqref{nab3nabgac} by $r\nab$ and applying Corollary \ref{commutation}, we infer
\begin{align*}
    \nab_3(r\nab^2\gac_{AB})-\f12\trchb(r\nab^2\gac_{AB})=\O_{-3}^0\c(\nab^2\hchbc,\nab^ 2\trchbc)+\O_1^2\c\Gag^{(1)}+r^2\Gag\c\Gag^{(1)}.
\end{align*}
Next, we have from \eqref{nab3all}
\begin{align*}
\nab^2\hchbc=\nab\bbc+\O_0^0\c\nab^2\trchbc+\O_1^0\c\nab\zec+\O_4^2\c\Gab^{(1)}+r^{-1}\Gab\c\Gag^{(1)}.
\end{align*}
Hence, we obtain
\begin{align*}
    \nab_3(r\nab^2\gac_{AB})-\f12\trchb(r\nab^2\gac_{AB})=\O_{-3}^0\c(\nab\bbc,\nab^ 2\trchbc,r^{-1}\nab\zec)+\O_1^2\c\Gab^{(1)}+r^2\Gab\c\Gag^{(1)}.
\end{align*}
We recall from Proposition \ref{Bianchieq}
\begin{equation*}
    \nab_3(\rhoc,\sic)+\frac{3}{2}\trchb\,(\rhoc,\sic)=-d_1\bbc+\Gag^{(1)}\c\O_3^1+\O_3^2\c\aac+\Gag\c\aac.
\end{equation*}
We obtain
\begin{equation*}
    \nab_3 G-\f12 \trchb\, G =\O_{-3}^0\c(\nab^ 2\trchbc,r^{-1}\nab\zec,r^{-1}\rhoc,r^{-1}\sic)+\O_0^1\c\Gag^{(1)}+\O_0^2\c\aac+r^3\Gag\c\aac+r^2\Gab\c\Gag^{(1)},
\end{equation*}
where the renormalized quantity $G$ is given by 
\begin{equation*}
    G=r\nab^2\gac_{AB}+\O_{-3}^0\c(\rhoc,\sic).
\end{equation*}
Applying Lemma \ref{evolution}, we infer
\begin{align*}
    |r^{-1-\frac{2}{p}}G|_{p,S}&\les|r^{-1-\frac{2}{p}}G|_{p,S_*}+\int_{u_0(\ub_*)}^u |r^{2-\frac{2}{p}}\nab^2\trchbc|_{p,S}+|r^{1-\frac{2}{p}}(\nab\zec,\rhoc,\sic)|_{p,S}\\
    &+\int_{u_0(\ub_*)}^u M|r^{-1-\frac{2}{p}}\Gag^{(1)}|_{p,S}+M^2|r^{-1-\frac{2}{p}}\aac|_{p,S}+|r^{2-\frac{2}{p}}\Gag\c\aac|_{p,S}+|r^{1-\frac{2}{p}}\Gab\c\Gag^{(1)}|_{p,S}\\
    &\les \frac{\ep_0}{r^\frac{s-1}{2}}+\int_{u_0(\ub_*)}^u \frac{\ep_0}{r^2|u|^\frac{s-3}{2}}+\frac{M\ep}{r^3|u|^\frac{s-3}{2}}+\frac{M^2\ep}{r^2|u|^\frac{s+1}{2}}+\frac{\ep^2}{r|u|^{s-1}}+\frac{\ep^2}{r^2|u|^{s-2}}\\
    &\les \frac{\ep_0}{r|u|^\frac{s-3}{2}}.
\end{align*}
Hence, we deduce
\begin{equation*}
    |r^{1-\frac{2}{p}}\nab^2\gac_{AB}|_{p,S}\les |r^{-\frac{2}{p}}G|_{p,S}+|r^{3-\frac{2}{p}}(\rhoc,\sic)|_{p,S}\les \frac{\ep_0}{|u|^\frac{s-3}{2}}.
\end{equation*}
Thus, we have
\begin{equation}
    \OO_{0,1,2}^*(\gac)\les\ep_0.
\end{equation}
The estimate for $\OO_{0,1,2}(\inc)$ is similar. This concludes the proof of Proposition \ref{propgastar}.
\end{proof}
\subsubsection{Estimate for \texorpdfstring{$\OO^*_\bGa$}{}}
\begin{proposition}\label{propGastar}
For a geodesic type foliation on $\Cb_*$ in the sense of Definition \ref{geodesicfoliation}, we assume that
\begin{equation}
    \OO'(\Cb_*\cap\Si_0)\leq\ep_0,\qquad \RRb(\Cb_*)\les\ep_0,\qquad \OO^*(\Cb_*)\leq\ep.
\end{equation}
Then, we have
\begin{align}
    \OO^*_\Ga\les\ep_0.
\end{align}
\end{proposition}
\begin{proof}
    We recall from Lemma \ref{bGa}
    \begin{align*}
    \bGa_{uA}^B=&\f12\ga^{BC}\left(\pr_u(\ga_{AC})-\frac{\pr}{\pr x^A}\left(\ga_{CD} b^D\right)+\frac{\pr}{\pr x^C}\left(\ga_{AD}b^D\right)\right)+\frac{b^B}{4\Om^2}\left(2\frac{\pr}{\pr x^A}(\Om^2)-\frac{\pr}{\pr \ub}(\ga_{AD}b^D)\right),\\
    \bGa_{BC}^A=&\f12\ga^{AD}\left(\frac{\pr \ga_{BD}}{\pr x^C}+\frac{\pr \ga_{CD}}{\pr x^B}-\frac{\pr \ga_{BC}}{\pr x^D}\right) +\frac{b^A}{4\Om^2}\frac{\pr\ga_{BC}}{\pr\ub}.
    \end{align*}
Applying Definition \ref{dfGamma}, and since $\Omc=0$ and $\bcc=0$, we infer
\begin{align*}
\Jbc&=\O_0^0\c\left(\pr_u\gac,\nab\gac,\pr_\ub\gac,r^{-1}\gac\right)+r\Gag\c\Gab,\\
\Lc&=\O_0^0\c(\nab\gac,\pr_\ub\gac,r^{-1}\gac)+r\Gag\c\Gag,
\end{align*}
where we used Remark \ref{newGag}. Applying Proposition \ref{metriceqre}, we deduce
\begin{align*}
\Jbc&=\O_0^0\c\left(\trchbc,\hchbc,\trchc,\hchc,r^{-1}\gac,\nab\gac\right)+r\Gag\c\Gab,\\
\Lc&=\O_0^0\c(\trchc,\hchc,r^{-1}\gac,\nab\gac)+r\Gag\c\Gag.
\end{align*}
Differentiating it by $r\nab$, we obtain
\begin{align*}
r\nab\Jbc&=\O_0^0\c\left(r\nab\trchbc,r\nab\hchbc,r\nab\trchc,r\nab\hchc,\nab\gac,r\nab^2\gac\right)+r\Gag\c\Gab^{(1)},\\
r\nab\Lc&=\O_0^0\c(r\nab\trchc,r\nab\hchc,\nab\gac,r\nab^2\gac)+r\Gag\c\Gag^{(1)}.
\end{align*}
Combining with Propositions \ref{OO*} and \ref{propgastar}, we deduce for $q=0,1$
\begin{align*}
    |r^{1-\frac{2}{p}}(r\nab)^q\Jbc|_{p,S}&\les\frac{\ep_0}{|u|^\frac{s-1}{2}}+|r^{2-\frac{2}{p}}\Gag\c\Gab^{(1)}|_{p,S}\les \frac{\ep_0}{|u|^\frac{s-1}{2}},\\
    |r^{2-\frac{2}{p}}(r\nab)^q\Lc|_{p,S}&\les\frac{\ep_0}{|u|^\frac{s-3}{2}}+|r^{3-\frac{2}{p}}\Gag\c\Gag^{(1)}|_{p,S}\les\frac{\ep_0}{|u|^\frac{s-3}{2}}.
\end{align*}
This concludes the proof of Proposition \ref{propGastar}.
\end{proof}
Combining Propositions \ref{OO*}, \ref{propgastar} and \ref{propGastar}, this concludes the proof of Theorem M2.
\subsection{Extension argument (Theorem M4)}\label{ext}
In this section, we prove Theorem M4 which we recall below for convenience.
\begin{thmM4}
We consider the spacetime $\KK$ and its double null foliation $(u,\ub)$ which satisfies the assumptions:
\begin{equation}\label{assumptiontilde}
\OO\les\ep_0,\qquad \RR\les \ep_0,\qquad \osc\les\ep_0,\qquad \OO^*(\Cb_*)\les\ep_0.
\end{equation}
Then, we can extend the spacetime $\KK=V(u_0,\ub_*)$ and the double null foliation to a new spacetime $\wideparen{\KK}=\wideparen{V}(u_0,\ub_*+\nu)$, where $\nu>0$ is sufficiently small, and an associated double null foliation $(\upa,\ub)$. Moreover, the new foliation $(\upa,\ub)$ is geodesic on the new last slice $\Cb_{**}:=\Cb_{\ub_*+\nu}$ and the new corresponding norms satisfy
\begin{align}
    \wideparen{\OO}\les\ep_0,\qquad\wideparen{\RR}\les\ep_0,\qquad \wideparen{\osc}\les\ep_0,\qquad \wideparen{\OO}^*(\Cb_{**})\les\ep_0.
\end{align}
\end{thmM4}
\begin{proof}
\begin{figure}
  \centering
  \includegraphics[width=1\textwidth]{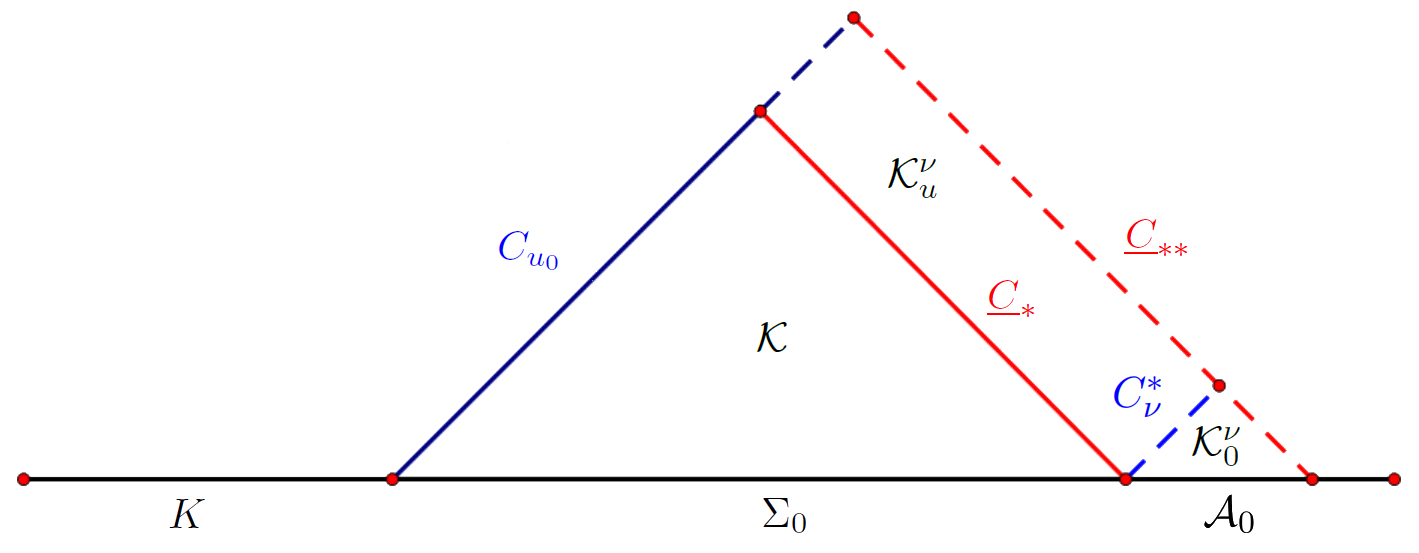}
  \caption{Extension argument}\label{fig5}
\end{figure}
We define the region
\begin{align*}
    \AA_0:=\{p\in\Si_0/\,w(p)\in[\ub_*,\ub_*+\nu]\}.
\end{align*}
For $\nu>0$ small enough, we consider the future domain of dependence $\kk_0^\nu$ of $\AA_0$, with
\begin{equation*}
    \kk_0^\nu:=V'(u_*,\ub_*+\nu)\subset \kk'.
\end{equation*}
The boundary of $\KK_0^\nu$ contains, see Figure \ref{fig5}:
\begin{enumerate}
    \item the part $\AA_0$ of $\Si_0$;
    \item a part of an ingoing cone denoted by
\begin{equation*}
    \Cb_{**}(u_*-\nu):=\Cb(\ub_*+\nu;[u_*-\nu,u_*]),
\end{equation*}
which is a part of the new last slice $\Cb_{**}:=\Cb_{\ub_*+\nu}$, where
\begin{equation*}
    u_*-\nu:=u|_{\Cb_{\ub_*+\nu}\cap \Sigma_0},\quad u_*=u|_{\Cb_{\ub_*}\cap\Sigma_0};
\end{equation*}
\item a part of the outgoing cone emanating from $S_{0}(u_*)=\Cb_{\ub_*}\cap\Sigma_0$ and in $V'(u_*,\ub_*+\nu)$, denoted by $C^*_\nu$.
\end{enumerate}
For the double null foliation $(u',\ub)$ in $\kk_0^\nu$, we have:
\begin{equation}\label{tildeOOnu}
    \OO'(\kk_0^\nu)\les\ep_0,\qquad\qquad \RR'(\kk_0^\nu)\les\ep_0,
\end{equation}
for $\nu>0$ small enough by local existence. By the local existence theorem for the characteristic Cauchy problem in \cite{luk}, we can extend the spacetime $\KK\cup\KK_0^\nu$ to the future domain of dependence of $C^*_\nu\cup\Cb_*$.\footnote{Even if it means reducing the value of $\nu$.} Moreover, we extend the outgoing cones of $\KK$ to this new domain, construct the ingoing cones from $C_\nu^*$ and denote $(\widetilde{u},\ub)$ the global foliation in $\widetilde{\KK}:=\KK\cup \kk_u^\nu\cup \kk_0^\nu$. As a consequence of Theorem 1 of \cite{luk}, we have
\begin{equation}\label{tildeOOnuu}
    \widetilde{\OO}(\kk^\nu_u)\les \ep_0,\qquad\qquad \widetilde{\RR}(\kk^\nu_u)\les \ep_0.
\end{equation}
Combining \eqref{assumptiontilde}, \eqref{tildeOOnu} and \eqref{tildeOOnuu}, we obtain
\begin{align*}
    \widetilde{\OO}(\widetilde{\kk})\les\ep_0,\qquad\qquad\widetilde{\RR}(\widetilde{\kk})\les\ep_0.
\end{align*}\\
The foliation $(\widetilde{u},\ub)$ on the new last slice $\Cb_{**}$ of $V(u_0,\ub_*+\nu)$ is in general not of the geodesic type of Definition \ref{geodesicfoliation}. Starting from $S_{**}:=\Cb_{**}\cap\Si_0$, we construct a geodesic type foliation $(\upa,\ub_*)$ on $\Cb_{**}(-\ub_*-\nu<\upa<u_1)$, where $u_1$ is sufficiently close to $-\ub_*-\nu$. We also denote
\begin{align*}
\Vpa(u_1):=\Vpa(u_1,\ub_*+\nu):=J^-(\wideparen{S}(u_1,\ub_*+\nu)).
\end{align*}
Proceeding as in Section \ref{bootregion}, we can construct a double null foliation $(\upa,\ub)$ in $\Vpa(u_1)$. Let $U$ the set of $u_1\leq u_0$ such that we have the following estimates:
\begin{align}
    \wideparen{\OO}(\Vpa(u_1))\leq\ep,\qquad\wideparen{\RR}(\Vpa(u_1))\leq\ep,\qquad \wideparen{\osc}(\Vpa(u_1))\leq\ep,\qquad \wideparen{\OO}^*(\Cb_{**}\cap\Vpa(u_1))\leq\ep.\label{B2paren}
\end{align}
We also denote
\begin{equation}\label{u2sup}
    u_2:=\sup U.
\end{equation}
Assume that $u_2<u_0$. Applying Theorems M0--M3 in $\Vpa(u_2)$, we obtain
\begin{align*}
    \wideparen{\OO}(\Vpa(u_2))\les\ep_0,\qquad \wideparen{\RR}(\Vpa(u_2))\les\ep_0,\qquad \wideparen{\osc}(\Vpa(u_2))\les\ep_0,\qquad \wideparen{\OO}^*(\Cb_{**}\cap\Vpa(u_2))\les\ep_0.
\end{align*}
Thus, by the continuity of the norms, we obtain $u_2+\de\in U$ for $\de>0$ small enough, a contradiction with \eqref{u2sup}. Hence, we have $u_2=u_0$ which implies
\begin{align*}
    \wideparen{\OO}(\Vpa(u_0))\les\ep_0,\qquad \wideparen{\RR}(\Vpa(u_0))\les\ep_0,\qquad \wideparen{\osc}(\Vpa(u_0))\les\ep_0,\qquad \wideparen{\OO}^*(\Cb_{**}\cap\Vpa(u_0))\les\ep_0.
\end{align*}
This concludes the proof of Theorem M4.
\end{proof}
\appendix
\renewcommand{\appendixname}{Appendix~\Alph{section}}
\section{Proof of results in Section \ref{seclinearequation}}\label{appa}
\subsection{Proof of Proposition \ref{nullstructure}}\label{pfnull}
We prove below a more precise version of Proposition \ref{nullstructure}.
\begin{proposition}\label{nullstructureA}
We have the following linearized null structure equations:
\begin{align*}
\Om\nab_4\widecheck{\Om\trch}+(\Om\tr\chi)\,\widecheck{\Om\trch}=&-4\Om\tr\chi\,\widecheck{\Om\om}+\fl[\nab_4\,\trchc]+\err[\nab_4\,\trchc],\\
\fl[\nab_4\,\trchc]:=&-2\widecheck{\Om\hch}(\Om_K\hch_K)-4\Om_K\om_K\widecheck{\Om\trch},\\
\err[\nab_4\,\trchc]:=&\f12(\widecheck{\Om\trch})^2-|\widecheck{\Om\hch}|^2,\\ \\
\Om\nab_3\,\widecheck{\Om\trchb}+(\Om\tr\chib)\,\widecheck{\Om\trchb}=&-4\Om\tr\chib\,\widecheck{\Om\omb}+\fl[\nab_3\,\trchc]+\err[\nab_3\,\trchc],\\
\fl[\nab_3\,\trchbc]:=&-2\widecheck{\Om\hchb}(\Om_K\hchb_K)-4\Om_K\omb_K\widecheck{\Om\trchb}-\bcc^A\pr_A(\Om_K\tr_K\chib_K),\\
\err[\nab_3\,\trchbc]:=&\f12(\widecheck{\Om\trchb})^2-|\widecheck{\Om\hchb}|^2,\\ \\
\Om\nab_4(\widecheck{\Om\trchb})=&-2\widecheck{\Om^2[\mub]}+\fl[\nab_4(\trchbc)]+\err[\nab_4(\trchbc)],\\
\fl[\nab_4(\trchbc)]:=&4\Om_K\etab_K(\widecheck{\Om\etab}),\\
\err[\nab_4(\trchbc)]:=&2|\widecheck{\Om\etab}|^2,\\\\
\Om\nab_3(\widecheck{\Om\trch})=&-2\widecheck{\Om^2[\mu]}+\fl[\nab_4(\trchc)]+\err[\nab_4(\trchc)],\\
\fl[\nab_4(\trchc)]:=&4\Om_K\eta_K(\widecheck{\Om\eta})+\bcc^A\pr_A(\Om_K\tr_K\chi_K),\\
\err[\nab_4(\trchc)]:=&2|\widecheck{\Om\eta}|^2.
\end{align*}
We also have the linearized Codazzi equations:
\begin{align*}
\sdiv(\hchc)=&\f12\nab\widecheck{\trch}+\f12\trch\,\zec-\bc+\fl[\sdiv(\widecheck{\hch})]+\err[\sdiv(\widecheck{\hch})],\\
\fl[\sdiv(\widecheck{\hch})]:=&\f12\trchc\,\ze_K-\zec\c\hch_K-\ze_K\c\hchc+(\ga^{AB}-(\ga_K)^{AB})\nab_A(\hch_K)_{B\bullet}-\ga_K^{AB}\Lc_{AB}^C(\hch_K)_{C\bullet},\\
\err[\sdiv(\widecheck{\hch})]:=&-\zec\cdot\hchc,\\ \\
\sdiv(\hchbc)=&\f12\nab\trchbc-\f12\trchb\,\zec+\bbc+\fl[\sdiv(\widecheck{\hch})]+\err[\sdiv(\widecheck{\hch})],\\
\fl[\sdiv(\widecheck{\hch})]:=&-\f12\trchbc\,\ze_K+\zec\c\hchb_K+\ze_K\c\hchbc+(\ga^{AB}-(\ga_K)^{AB})\nab_A(\hchb_K)_{B\bullet}-\ga_K^{AB}\Lc_{AB}^C(\hchb_K)_{C\bullet},\\
\err[\sdiv(\widecheck{\hch})]:=&\zec\cdot\hchbc.
\end{align*}
The linearized mass aspect functions are given by:
\begin{align*}
    \sdiv\etac&=-\muc-\rhoc+\fl[\sdiv(\etac)]+\err[\sdiv(\etac)],\\
    \fl[\sdiv(\etac)]&=(\ga^{AB}-(\ga_K)^{AB})\nab_A(\eta_K)_{B}-(\ga_K)^{AB}\Lc_{AB}^C(\eta_K)_{C}+\f12\hch_K\c\hchb+\f12\hchc\c\hchb_K,\\
    \err[\sdiv(\etac)]&=\f12\hch\c\hchbc,\\ \\
    \sdiv\etabc&=-\mubc-\rhoc+\fl[\sdiv(\etabc)]+\err[\sdiv(\etabc)],\\
    \fl[\sdiv(\etabc)]&=(\ga^{AB}-(\ga_K)^{AB})\nab_A(\etab_K)_{B}-(\ga_K)^{AB}\Lc_{AB}^C(\etab_K)_{C}+\f12\hch_K\c\hchb+\f12\hchc\c\hchb_K,\\
    \err[\sdiv(\etabc)]&=\f12\hch\c\hchbc.
\end{align*}
The linearized torsion equations are given by:
\begin{align*}
\curl\etac=&\widecheck{\si}+\fl[\curl\etac]+\err[\curl\etac],\\
\fl[\curl\etac]:=& \f12\widecheck{\hchb}\wedge\hch_K+\f12\hchb_K\wedge\widecheck{\hch}-(\in^{AB}-\in_K^{AB})\nab_A(\eta_K)_B+(\in_K)^{AB}\Lc_{AB}^C(\eta_K)_C,\\
\err[\curl\etac]:=&\f12\widecheck{\hchb}\wedge\widecheck{\hch},\\ \\
\curl\etabc=&-\widecheck{\si}+\fl[\curl\etabc]+\err[\curl\etabc],\\
\fl[\curl\etabc]:=&-\f12\widecheck{\hchb}\wedge\hch_K-\f12\hchb_K\wedge\widecheck{\hch}-(\in^{AB}-\in_K^{AB})\nab_A(\etab_K)_B+(\in_K)^{AB}\Lc_{AB}^C(\etab_K)_C,\\
\err[\curl\etabc]:=&-\f12\hchbc\wedge\hchc.
\end{align*}
Moreover, we have
\begin{align*}
\Om\nab_4(\widecheck{\Om\omb})=&\f12\widecheck{\Om^2\rho}+\fl[\nab_4\ombc]+\err[\nab_4\ombc],\\
    \fl[\nab_4\ombc]:=&\frac{3}{2}(\widecheck{\Om\eta}-\widecheck{\Om\etab})\cdot(\Om_K\eta_K-\Om_K\etab_K)-\frac{1}{4}(\widecheck{\Om\eta}-\widecheck{\Om\etab})\cdot(\Om_K\eta_K+\Om_K\etab_K)\\
    &-\frac{1}{4}(\Om_K\eta_K-\Om_K\etab_K)\cdot(\widecheck{\Om\eta}+\widecheck{\Om\etab})-\frac{1}{4}(\widecheck{\Om\eta}+\widecheck{\Om\etab})\cdot(\eta_K+\etab_K),\\
    \err[\nab_4\ombc]:=&\frac{3}{4}|\widecheck{\Om\eta}-\widecheck{\Om\etab}|^2-\frac{1}{4}(\widecheck{\Om\eta}-\widecheck{\Om\etab})\cdot(\widecheck{\Om\eta}+\widecheck{\Om\etab})-\frac{1}{8}|\widecheck{\Om\eta}+\widecheck{\Om\etab}|^2,\\ \\
\Om\nab_3(\widecheck{\Om\om})=&\f12\widecheck{\Om^2\rho}+\fl[\nab_3\omc]+\err[\nab_3\omc],\\
    \fl[\nab_3\omc]:=&\frac{3}{2}(\widecheck{\Om\eta}-\widecheck{\Om\etab})\cdot(\Om_K\eta_K-\Om_K\etab_K)+\frac{1}{4}(\widecheck{\Om\eta}-\widecheck{\Om\etab})\cdot(\Om_K\eta_K+\Om_K\etab_K)\\
    &+\frac{1}{4}(\Om_K\eta_K-\Om_K\etab_K)\cdot(\widecheck{\Om\eta}+\widecheck{\Om\etab})-\frac{1}{4}(\widecheck{\Om\eta}+\widecheck{\Om\etab})\cdot(\eta_K+\etab_K)+\bcc^A\nab_{\pr_{A}}(\Om_K\om_K),\\
    \err[\nab_3\omc]:=&\frac{3}{4}|\widecheck{\Om\eta}-\widecheck{\Om\etab}|^2+\frac{1}{4}(\widecheck{\Om\eta}-\widecheck{\Om\etab})\cdot(\widecheck{\Om\eta}+\widecheck{\Om\etab})-\frac{1}{8}|\widecheck{\Om\eta}+\widecheck{\Om\etab}|^2.
\end{align*}
\end{proposition}
\begin{remark}
    Combining Propositions \ref{nullstructureA} and \ref{decayGamma} and Definitions \ref{bigOnotation} and \ref{gammag}, we obtain immediately Proposition \ref{nullstructure}.
\end{remark}
\begin{proof}[Proof of Proposition \ref{nullstructure}]
We recall from Corollary \ref{cornull}
\begin{align*}
\Om\nab_4(\Om\tr\chi) +\frac{1}{2}(\Om\trch)^2=&-\Om^2|\hch|^2-4\Om\om(\Om\trch),
\end{align*}
which implies
\begin{align*}
\Om_K\nab_{(e_4)_K}(\Om_K\tr_K\chi_K)+\frac{1}{2}(\Om_K\tr_K\chi_K)^2 =&-\Om_K^2|\hch_K|^2-4\Om_K\om_K(\Om_K\tr_K\chi_K).
\end{align*}
Taking the difference, we infer
\begin{align*}
\Om\nab_4(\widecheck{\Om\tr\chi})=&-\frac{1}{2}(\Om\tr\chi)^2-\Om^2|\hch|^2-4\Om\om (\Om\trch)+\f12(\Om_K\tr_K\chi_K)^2+\Om_K^2|\hch_K|^2+4\Om_K\om_K(\tr_K\chi_K) \\
=&-\Om\trch\widecheck{\Om\trch}+\f12\widecheck{\Om\trch}^2-2\widecheck{\Om\hch}(\Om_K\hch_K)-|\widecheck{\Om\hch}|^2-4\Om_K\om_K\widecheck{\Om\trch}\\
&-4\Om_K\tr_K\chi_K\widecheck{\Om\om}-4\widecheck{\Om\om}\,\widecheck{\Om\trch}.
\end{align*}
Similarly, we have
\begin{align*}
\Om\nab_3(\widecheck{\Om\trchb})=&-\Om\trchb\widecheck{\Om\trchb}+\f12\widecheck{\Om\trchb}^2-2(\widecheck{\Om\hchb})\Om_K\hchb_K-|\widecheck{\Om\hchb}|^2-4\Om_K\om_K\widecheck{\Om\trchb}\\
&-4\Om_K\tr_K\chib_K\widecheck{\Om\omb}-4\widecheck{\Om\omb}\,\widecheck{\Om\trchb}-\bcc^A\pr_A(\Om_K\tr_K\chib_K).
\end{align*}
Next, we have from Corollary \ref{cornull}
\begin{align*}
\Om\nab_4(\Om\tr\chib)=-2\Om^2[\mub]+2\Om^2|\etab|^2,
\end{align*}
which implies
\begin{align*}
\Om_K\nab_{(e_4)_K}(\Om_K\tr_K\chib_K)=-2\Om_K^2[\mub]_K+2\Om_K^2|\etab_K|^2.
\end{align*}
Taking the difference, we infer
\begin{align*}
\Om\nab_4(\widecheck{\Om\trchb})=-2\widecheck{\Om^2[\mub]}+4\Om_K\etab_K(\widecheck{\Om\etab})+2|\widecheck{\Om\etab}|^2.
\end{align*}
Similarly, we have
\begin{align*}
\Om\nab_3(\widecheck{\Om\trch})=-2\widecheck{\Om^2[\mu]}+4\Om_K\eta_K(\widecheck{\Om\eta})+2|\widecheck{\Om\eta}|^2+\bcc^A\pr_A(\Om_K\tr_K\chi_K).
\end{align*}
Now we consider the Codazzi equation
\begin{align*}
    \sdiv\hch=\f12(\nab\trch+\trch\,\ze)-\ze\c\hch-\b,
\end{align*}
which implies
\begin{align*}
\sdiv_K\hch_K=\f12(\nab_K\tr_K\chi_K+\tr_K\chi_K\,\ze_K)-\ze_K\c\hch_K-\b_K.
\end{align*}
Taking the difference, we obtain from Proposition \ref{useful}
\begin{align*}
    \sdiv(\hchc)=&\sdiv\hch-\sdiv\hch_K\\
    =&\sdiv\hch-\sdiv_K\hch_K-(\sdiv-\sdiv_K)\hch_K \\
    =&\f12\nab\widecheck{\trch}+\f12\trch\,\zec+\f12\trchc\,\ze_K-\zec\c\hch_K-\ze_K\c\hchc-\zec\c\hchc-\bc-\widecheck{\ga^{AB}}\nab_A(\hch_K)_{B\bullet}\\
    &-\ga_K^{AB}\Lc_{AB}^D\,(\hch_K)_{\bullet D}-\ga_K^{AB}\Lc_{A\bullet}^D\,(\hch_K)_{BD}.
\end{align*}
Similarly, we have
\begin{align*}
    \sdiv(\widecheck{\hchb})=&\f12\nab\trchbc-\f12\trchb\,\zec-\f12\trchbc\,\ze_K+\zec\c\hchb_K+\ze_K\c\hchbc+\zec\c\hchbc+\bbc-\widecheck{\ga^{AB}}\nab_A(\hchb_K)_{B\bullet}\\
    &-\ga_K^{AB}\Lc_{AB}^D(\hchb_K)_{\bullet D}-\ga_K^{AB}\Lc_{A\bullet}^D(\hchb_K)_{BD}.
\end{align*}
Next, we recall from \eqref{defmassaspect}
\begin{equation*}
    \mu=-\sdiv\eta+\frac{1}{2}\hch\c\hchb-\rho.
\end{equation*}
Hence, we have
\begin{align*}
    \mu_K=-\sdiv_K\eta_K+\f12 \hch_K\c\hchb_K-\rho_K.
\end{align*}
Taking the difference, we obtain
\begin{align*}
    \muc&=-\sdiv\etac+(\sdiv-\sdiv_K)\eta_K+\f12\hch_K\c\hchb+\f12\hchc\c\hchb_K+\f12\hch\c\hchbc-\rhoc\\
    &=-\sdiv\etac-\rhoc+(\ga^{AB}-(\ga_K)^{AB})\nab_A(\eta_K)_{B}-(\ga_K)^{AB}\Lc_{AB}^C(\eta_K)_{C}\\
    &+\f12\hch_K\c\hchb+\f12\hchc\c\hchb_K+\f12\hch\c\hchbc.
\end{align*}
Similarly, we have
\begin{align*}
    \mubc&=-\sdiv\etabc-\rhoc+(\ga^{AB}-(\ga_K)^{AB})\nab_A(\etab_K)_{B}-(\ga_K)^{AB}\Lc_{AB}^C(\etab_K)_{C}\\
    &+\f12\hch_K\c\hchb+\f12\hchc\c\hchb_K+\f12\hch\c\hchbc.
\end{align*}
We consider the torsion equation
\begin{align*}
\curl\eta=-\curl\etab=\sigma+\frac{1}{2}\hchb\wedge\hch,
\end{align*}
which implies
\begin{align*}
    \curl_K\eta_K=-\curl_K\etab_K=\sigma_K+\frac{1}{2}\hchb_K\wedge\hch_K.
\end{align*}
Taking the difference, we infer from Proposition \ref{useful} that
\begin{align*}
    \curl\etac=&\curl\eta-\curl\eta_K=\curl\eta-\curl_K\eta_K-(\curl-\curl_K)\eta_K\\
    =&\widecheck{\si}+\f12 \widecheck{\hchb}\wedge\widecheck{\hch}+ \f12\widecheck{\hchb}\wedge\hch_K+\f12\hchb_K\wedge\widecheck{\hch}-\widecheck{\in^{AB}}\nab_A(\eta_K)_B+\in_K^{AB}\Lc_{AB}^C(\eta_K)_C.
\end{align*}
Similarly, we have
\begin{align*}
    -\curl\etabc=&\widecheck{\si}+\f12 \widecheck{\hchb}\wedge\widecheck{\hch}+ \f12\widecheck{\hchb}\wedge\hch_K+\f12\hchb_K\wedge\widecheck{\hch}+\widecheck{\in^{AB}}\nab_A(\etab_K)_B-\in^{AB}_K\Lc_{AB}^C(\etab_K)_C.
\end{align*}
Finally, we recall from Proposition \ref{standardnull}
\begin{align*}
    \Om\nab_4(\Om\omb)=&\f12\Om^2\rho+\frac{3}{4}|\Om\eta-\Om\etab|^2-\frac{1}{4}(\Om\eta-\Om\etab)\cdot(\Om\eta+\Om\etab)-\frac{1}{8}|\Om\eta+\Om\etab|^2,
\end{align*}
which implies
\begin{align*}
    \Om_K\nab_{(e_4)_K}\omb_K=&\frac{1}{2}\Om_K^2\rho_K+\frac{3}{4}|\Om_K\eta_K-\Om_K\etab_K|^2-\frac{1}{4}(\Om_K\eta_K-\Om_K\etab_K)\cdot(\Om_K\eta_K+\Om_K\etab_K)\\
    &-\frac{1}{8}|\Om_K\eta_K+\Om_K\etab_K|^2.
\end{align*}
Taking the difference, we obtain
\begin{align*}
        \Om\nab_4(\widecheck{\Om\omb})=&\frac{3}{4}\Om^2|\eta-\etab|^2-\frac{1}{4}\Om^2(\eta-\etab)\cdot(\eta+\etab)-\frac{1}{8}\Om^2|\eta+\etab|^2+\frac{1}{2}\Om^2\rho \\
        &-\frac{3}{4}\Om_K^2|\eta_K-\etab_K|^2+\frac{1}{4}\Om_K^2(\eta_K-\etab_K)\cdot(\eta_K+\etab_K)+\frac{1}{8}\Om_K^2|\eta_K+\etab_K|^2-\frac{1}{2}\Om_K^2\rho_K \\
        =&\frac{3}{4}|\widecheck{\Om\eta}-\widecheck{\Om\etab}|^2+\frac{3}{2}(\widecheck{\Om\eta}-\widecheck{\Om\etab})\cdot(\Om_K\eta_K-\Om_K\etab_K)-\frac{1}{4}(\widecheck{\Om\eta}-\widecheck{\Om\etab})\cdot(\widecheck{\Om\eta}+\widecheck{\Om\etab})\\
        &-\frac{1}{4}(\widecheck{\Om\eta}-\widecheck{\Om\etab})\cdot(\Om_K\eta_K+\Om_K\etab_K)-\frac{1}{4}(\Om_K\eta_K-\Om_K\etab_K)\cdot(\widecheck{\Om\eta}+\widecheck{\Om\etab})-\frac{1}{8}|\widecheck{\Om\eta}+\widecheck{\Om\etab}|^2\\
        &-\frac{1}{4}(\widecheck{\Om\eta}+\widecheck{\Om\etab})\cdot(\eta_K+\etab_K)+\f12\widecheck{\Om^2\rho}.
\end{align*}
Similarly, we have
\begin{align*}
        \Om\nab_3(\widecheck{\Om\om})=&\frac{3}{4}|\widecheck{\Om\eta}-\widecheck{\Om\etab}|^2+\frac{3}{2}(\widecheck{\Om\eta}-\widecheck{\Om\etab})\cdot(\Om_K\eta_K-\Om_K\etab_K)+\frac{1}{4}(\widecheck{\Om\eta}-\widecheck{\Om\etab})\cdot(\widecheck{\Om\eta}+\widecheck{\Om\etab})\\
        &+\frac{1}{4}(\widecheck{\Om\eta}-\widecheck{\Om\etab})\cdot(\Om_K\eta_K+\Om_K\etab_K)+\frac{1}{4}(\Om_K\eta_K-\Om_K\etab_K)\cdot(\widecheck{\Om\eta}+\widecheck{\Om\etab})-\frac{1}{8}|\widecheck{\Om\eta}+\widecheck{\Om\etab}|^2\\
        &-\frac{1}{4}(\widecheck{\Om\eta}+\widecheck{\Om\etab})\cdot(\eta_K+\etab_K)+\f12\widecheck{\Om^2\rho}+\bcc^A\nab_{\pr_{x^a}}(\Om_K\om_K).
\end{align*}
This concludes the proof of Proposition \ref{nullstructureA}.
\end{proof}
\subsection{Proof of Proposition \ref{equationsmumub}}\label{equationsmu}
We recall from Corollary \ref{cornull} that
\begin{align*}
    \Om\nab_4\mum+\Om\trch\mum&=\Om [F]+\Om(\trch\,\rho-2\eta\c\b),\\
    [F]&:=\hch\c(\nab\hot\eta)+(\eta-\etab)\c(\nab\trch+\trch\,\ze)+\f12\trch(|\eta|^2-|\etab|^2)\\
    &-\f12\left(\trchb|\hch|^2+\trch(\hch\c\hchb)\right)+2\eta\c\hch\c\etab,
\end{align*}
which implies
\begin{align*}
    \Om_K(\nab_K)_{(e_4)_K}\mum_K+\Om_K\tr_K\chi_K\mum_K=\Om_K [F]_K+\Om_K(\trch_K\,\rho_K-2\eta_K\c\b_K).
\end{align*}
Taking the difference, we obtain
\begin{align*}
    \Om\nab_4\mumc+\Om\trch\mumc&=\widecheck{\Om [F]}+\Om(\trch\,\rho-2\eta\c\b)-\Om_K(\trch_K\,\rho_K-2\eta_K\c\b_K)\\
    &=\widecheck{\Om[F]}+\widecheck{\Om\trch}\,\rho_K+\Om\trch\,\rhoc-2\widecheck{\Om\eta}\c\b_K-2\Om_K\eta_K\c\bc-2\widecheck{\Om\eta}\c\bc\\
    &=\Om\trch\,\rhoc+\widecheck{\Om[F]}+\O_3^1\c\Gag+\O_3^2\c\bc-2\widecheck{\Om\eta}\c\bc.
\end{align*}
Next, we compute
\begin{align*}
    {\Om [F]}=&-2\hch\c d_2^*\eta+(\eta-\etab)\c\nab\trch+(\O_3^2+\Gag)\c(\O_1^0+\Gag)\c(\O_3^2+\Gag)\\
    &+(\O_1^0+\Gag)\c(\O_3^2+\Gag)\c(\O_3^2+\Gab)+(\O_3^2+\Gag)\c(\O_3^2+\Gag)\c(\O_3^2+\Gag)\\
    =&-2\hch\c d_2^*\eta+(\eta-\etab)\c\nab\trch+\O_7^4+\O_4^2\c\Gab+r^{-1}\Gag\c\Gab.
\end{align*}
Hence, we deduce from Proposition \ref{useful} that
\begin{align*}
\widecheck{\Om[F]}=&-2\hchc\c (d_2^*)_K\eta_K-2\hch_K\c\widecheck{d_2^*\eta}-2\hchc\c\widecheck{d_2^*\eta}+(\etac-\etabc)\c\nab_K\tr_K\chi_K+(\eta_K-\etab_K)\c\nab\trchc\\
&+(\etac-\etabc)\c\nab\trchc+\O_4^2\c\Gab+r^{-1}\Gag\c\Gab\\
=&\O_3^2\c\widecheck{d_2^*\eta}-2\hchc\c\widecheck{d_2^*\eta}+\O_3^2\c\nab\trchc+(\etac-\etabc)\c\nab\trchc+\O_4^2\c\Gab+r^{-1}\Gag\c\Gab\\
=&\O_3^2\c (d_2^*\etac+\Lc\c\O_3^2)-2\hchc\c(d_2^*\etac+\Lc\c\O_3^2)+\O_3^2\c\nab\trchc+(\etac-\etabc)\c\nab\trchc+\O_4^2\c\Gab+r^{-1}\Gag\c\Gab\\
=&\O_3^2\c (d_2^*\etac)-2\hchc\c(d_2^*\etac)+\O_6^4\c\Lc+r^{-1}\Gag\c\Lc+\O_3^2\c\nab\trchc+(\etac-\etabc)\c\nab\trchc+\O_4^2\c\Gab+r^{-1}\Gag\c\Gab.
\end{align*}
Thus, we obtain
\begin{align*}
\Om\nab_4\mumc+\Om\trch\mumc&=\Om\trch\,\rhoc+\O_3^2\c (d_2^*\etac,\nab\trchc,\bc)+\O_3^1\c(\Gag,\Lc)-2\hchc\c(d_2^*\etac)+r^{-1}\Gag\c\Lc\\
&+(\etac-\etabc)\c\nab\trchc+r^{-1}\Gag\c\Gab-2\widecheck{\Om\eta}\c\bc.
\end{align*}
Similarly, we have
\begin{align*}
\Om\nab_3\mubmc+\Om\trchb\mubmc&=\Om\trchb\,\rhoc+\O_3^2\c (d_2^*\etabc,\nab\trchbc,\bbc)+\O_3^1\c(\Gag,\Lc)-2\hchbc\c(d_2^*\etabc)+r^{-1}\Gab\c\Lc\\
&+(\etabc-\etac)\c\nab\trchbc+r^{-1}\Gab\c\Gab+2\widecheck{\Om\etab}\c\bbc.
\end{align*}
This concludes the proof of Proposition \ref{equationsmumub}.
\subsection{Proof of Propositions \ref{Bianchieq} and \ref{Bianchieqdkb}}\label{pfBianchi}
We first prove the following lemma.
\begin{lemma}\label{dkbcomm}
    Let $k,k'=0,1,2$, $\la_0$ a real number and $\nab$ a first order elliptic operator. Then, we have the following properties:
\begin{enumerate}
    \item Assume that $\psi,h\in\sk_k$ and $\psi'\in\sk_{k'}$ satisfy the following equation:
    \begin{equation}\label{A1eq}
        \nab_3\psi+\la_0\trchb\,\psi=\nab\psi'+h.
    \end{equation}
        Then, we have
    \begin{equation}\label{nab3dkbpsi}
        \nab_3\dkb\psi+\la_0\trchb\,\dkb\psi=\nab\dkb\psi'+h^{[1]},
    \end{equation}
    where
    \begin{equation*}
        h^{[1]}=\dkb h+(\Gag+\O_3^2)h+(\Gag+\O_3^2)\c{\psi'}^{(1)}+(\Gab+\O_3^2)\c\psi^{(1)}+\Gab^{(1)}\c\psi.
    \end{equation*}
    \item Assume that $\psi,h\in \sk_k$ and $\psi'\in \sk_{k'}$ satisfy the following equation:
    \begin{equation}
        \nab_4\psi +\la_0\trch\,\psi=\nab\psi'+h.
    \end{equation}
Then, we have
\begin{equation}\label{nab4dkbpsi}
\nab_4\dkb\psi+\la_0\trch\,\dkb\psi=\nab\dkb\psi'+h^{[1]},
\end{equation}
where
\begin{equation*}
    h^{[1]}=\dkb h+(\Gag+\O_3^2)h+(\Gag+\O_3^2)\c({\psi'}^{(1)},\psi^{(1)})+\Gag^{(1)}\c\psi.
    \end{equation*}
\end{enumerate}
\end{lemma}
\begin{proof}
We have from \eqref{A1eq}
    \begin{align*}
        \Om\nab_3\psi+\la_0\Om\trchb\,\psi=\Om\nab\psi'+\Om h.
    \end{align*}
Differentiating it by $\dkb$, we obtain
\begin{align*}
\Om\nab_3(\dkb\psi)+\la_0r\nab(\Om\trchb)\psi+\la_0\Om\trchb(\dkb\psi)=r\nab\Om\c\nab\psi'+\Om\nab(\dkb\psi')+\dkb(\Om h)+[\Om\nab_3,\dkb]\psi,
\end{align*}
which implies from Corollary \ref{commutation}
\begin{align*}
    &\;\;\;\,\,\,\Om \nab_3(\dkb\psi)+\la_0\Om\trchb (\dkb\psi)\\
    &=\Om \nab(\dkb\psi')+(\Gag+\O_3^2)\c\nab\psi'+(\Gag^{(1)}+\O_3^2)\c\psi+(\Gag+\O_3^2)h+\Om\dkb h\\
    &+(\Gab+\O_3^2)\c\dkb\psi+(\Gab^{(1)}+\O_3^2)\c\psi\\
    &=\Om\nab(\dkb\psi')+(\Gag+\O_3^2)h+(\Gab+\O_3^2)\c{\psi'}^{(1)}+(\Gag+\O_3^2)\c\psi^{(1)}+\Gab^{(1)}\c\psi+\Om\dkb h.
\end{align*}
Hence, we obtain
\begin{align}
\begin{split}\label{nab3psieq}
&\nab_3(\dkb\psi)+\la_0\trchb (\dkb\psi)\\
=&\nab(\dkb\psi')+(\Gag+\O_3^2)h+(\Gag+\O_3^2)\c{\psi'}^{(1)}+(\Gab+\O_3^2)\c\psi^{(1)}+\Gab^{(1)}\c\psi+\dkb h,
\end{split}
\end{align}
which implies \eqref{nab3dkbpsi}. The proof of \eqref{nab4dkbpsi} is similar. This concludes the proof of Lemma \ref{dkbcomm}.
\end{proof}
We next prove the following lemma concerning the linearization of Bianchi equations.
\begin{lemma}\label{linbianchi}
    Let $k,k'=0,1,2$, $\la_0$ a real number and $\nab$ a first order elliptic operator. Then, we have the following properties:
\begin{enumerate}
    \item Assume that $\psi,h\in\sk_k$ and $\psi'\in\sk_{k'}$ satisfy the following equation:
    \begin{equation}
        \nab_3\psi+\la_0\trchb\,\psi=\nab\psi'+h.
    \end{equation}
        Then, we have
    \begin{equation}\label{nab3psic}
\nab_3\psic+\la_0\trchb\,\psic=\nab\psi'+\widecheck{h}+\Jbc\c\psi_K+\Lc\c\psi_K'+\Gag\c\dk^{\leq 1}\psi_K.
    \end{equation}
    \item Assume that $\psi,h\in \sk_k$ and $\psi'\in \sk_{k'}$ satisfy the following equation:
    \begin{equation}
        \nab_4\psi +\la_0\trch\,\psi=\nab\psi'+h.
    \end{equation}
Then, we have
\begin{equation}\label{nab4psic}
\nab_4\psic+\la_0\trch\,\psic=\nab\psi'+\widecheck{h}+\Jc\c\psi_K+\Lc\c\psi_K'+\Gag\c\dk^{\leq 1}\psi_K.
\end{equation}
\end{enumerate}
\end{lemma}
\begin{proof}
    Assume that we have
    \begin{align*}
        \nab_3\psi+\la_0\trchb\,\psi=\nab\psi'+h.
    \end{align*}
    Then, we obtain
    \begin{align*}
        (\nab_K)_{(e_3)_K}\psi_K+\la_0\tr_K\chib_K\,\psi_K=\nab_K\psi_K'+h_K.
    \end{align*}
    Taking the difference, we deduce
    \begin{align}\label{nab3psiK}
        \nab_3\psi-(\nab_K)_{(e_3)_K}\psi_K+\la_0(\trchb\,\psi-\tr_K\chib_K\,\psi_K)=\nab\psi'-\nab_K\psi_K'+\widecheck{h}.
    \end{align}
    Recall from \eqref{diff3} that
    \begin{align*}
        e_3-(e_3)_K&=\Om^{-1}(\Om e_3-\Om_K (e_3)_K)+\left(\frac{\Om_K}{\Om}-1\right)(e_3)_K\\
        &=\Om^{-1}\bcc^A\pr_{x^A}+\Om^{-1}\Omc(e_3)_K.
    \end{align*}
    Thus, we obtain
    \begin{align*}
        (\nab_K)_{e_3-(e_3)_K}\psi_K=\Gag\c\pr_{x^A} (\psi_K)+r\Gag\c (\nab_K)_{(e_3)_K}\psi_K=\Gag\c\dk\psi_K.
    \end{align*}
    Combining with \eqref{nab3psiK}, we deduce
    \begin{align*}
        \nab_3\psic+\la_0\trchb\,\psic&=\nab_3(\psi-\psi_K)+\la_0\trchb\,(\psi-\psi_K)\\
        &=\left((\nab_K)_{(e_3)_K}-\nab_3\right)\psi_K+\la_0(\tr_K\chib_K-\trchb)\psi_K+\nab\psic'+(\nab-\nab_K)\psi'_K+\widecheck{h}\\
        &=\left((\nab_K)_{(e_3)_K-e_3}+(\nab_K-\nab)\right)\psi_K-\la_0\,\trchbc\,\psi_K+\nab\psic'+(\nab-\nab_K)\psi'_K+\widecheck{h}\\
        &=\nab\psi'+\widecheck{h}-(\nab-\nab_K)_{e_3}\psi_K+(\nab-\nab_K)\psi'_K+\Gag\c\dk^{\leq 1}\psi_K.
    \end{align*}
     Recalling Lemma \ref{diffnab}, we have
     \begin{align*}
         (\nab-\nab_K)_{e_3}\psi_K&=\Jbc\c\psi_K+\Lc\c\psi_K,\\
         (\nab-\nab_K)\psi'_K&=\Lc\c\psi_K.
     \end{align*}
     Thus, we obtain
     \begin{align*}
\nab_3\psic+\la_0\trchb\,\psic=\nab\psi'+\widecheck{h}+(\Jbc,\Lc)\c\psi_K+\Lc\c\psi_K'+\Gag\c\dk^{\leq 1}\psi_K,
     \end{align*}
     which implies \eqref{nab3psic}. The proof of \eqref{nab4psic} is similar. This concludes the proof of Lemma \ref{linbianchi}.
\end{proof}
We are now ready to prove the following proposition, which implies directly Propositions \ref{Bianchieq} and \ref{Bianchieqdkb}.
\begin{proposition}\label{BianchieqA}
We have the following linearized Bianchi equations:
\begin{align*}
\nab_4\aac+\f12\tr\chi\,\aac
=&\nab\hot\bbc+\fl[\nab_4\aac]+\err[\nab_4\aac],\\
\fl[\nab_4\aac]:=&\Gab^{(1)}\c\O_3^1+\O_2^1\c\aac+\O_3^2\c(\bbc,\rhoc,\sic),\\
\err[\nab_4\aac]:=&\Gag\c(\aac,\bbc)+\Gab\c(\rhoc,\sic),\\ \\
\nab_4\dkbaac+\f12\trchb\dkbaac=&-\nab\hot\dkbbbc+\fl[\nab_4\dkbaac]+\err[\nab_4\dkbaac],\\
\fl[\nab_4\dkbaac]=&\Gab^{(1)}\c\O_3^1+\O_2^1\c\aac^{(1)}+\O_3^2\c(\bbc^{(1)},\rhoc^{(1)},\sic^{(1)}),\\
\err[\nab_4\dkbaac]=&\Gag^{(1)}\c(\aac,\bbc)+\Gab^{(1)}\c(\rhoc,\sic)+\Gag\c(\aac^{(1)},\bbc^{(1)})+\Gab\c(\rhoc^{(1)},\sic^{(1)}),\\ \\
\nab_{3}\bbc+2\tr\chib\,\bbc=&-\sdiv\aac+\fl[\nab_{3}\bbc]+\err[\nab_{3}\bbc],\\
\fl[\nab_{3}\bbc]:=&\Gab^{(1)}\c\O_4^2+\O_2^1\c\bbc+\O_3^2\c\aac,\\
\err[\nab_3\bbc]:=&\Gab\c\bbc+\Gag\c\aac,\\ \\
\nab_3\widecheck{\dkb\bb}+2\trchb\widecheck{\dkb\bb}=&-\sdiv\widecheck{\dkb\aa}+\fl[\nab_3\widecheck{\dkb\bb}]+\err[\nab_3\widecheck{\dkb\bb}],\\
\fl[\nab_3\widecheck{\dkb\bb}]:=&\Gab^{(1)}\c\O_4^2+\O_2^1\c\bbc^{(1)}+\O_3^2\c\aac^{(1)},\\
\err[\nab_3\widecheck{\dkb\bb}]:=&\Gab^{(1)}\c\bbc+\Gab\c\bbc^{(1)}+\Gag^{(1)}\c\aac+\Gag\c\aac^{(1)},\\\\
\nab_4\bbc+\trch\,\bbc=&-\nab\rhoc+{^*\nab}\sic+\fl[\nab_4\bbc]+\err[\nab_4\bbc],\\
\fl[\nab_4\bbc]:=&\Gag^{(1)}\c\O_3^1+\O_2^1\c\bbc+\O_3^2\c(\rhoc,\sic,\bc),\\
\err[\nab_4\bbc]:=&\Gag\c(\bbc,\rhoc,\sic)+\Gab\c\bc,\\ \\
\nab_4\dkbbbc+\trch\dkbbbc:=&-\nab\dkbrhoc+{^*\nab}\dkbsic+\fl[\nab_4\dkbbbc]+\err[\nab_4\dkbbbc],\\
\fl[\nab_4\dkbbbc]:=&\Gag^{(1)}\c\O_3^1+\O_2^1\c\bbc^{(1)}+\O_3^2\c(\rhoc^{(1)},\sic^{(1)},\bc^{(1)}),\\
\err[\nab_4\dkbbbc]:=&\Gag^{(1)}\c(\bbc,\rhoc,\sic)+\Gag\c(\bbc^{(1)},\rhoc^{(1)},\sic^{(1)})+\Gab^{(1)}\c\bc+\Gab\c\bc^{(1)},\\ \\
\nab_3(\rhoc,\sic)+\frac{3}{2}\trchb\,(\rhoc,\sic)&=-d_1\bbc+\fl[\nab_3\rhoc]+\err[\nab_3\rhoc],\\
\fl[\nab_3\rhoc]:=&\Gag^{(1)}\c\O_3^1+\O_3^2\c(\aac,\bbc), \\
\err[\nab_3\rhoc]:=&\Gag\c(\aac,\bbc),\\ \\
\nab_3(\dkbrhoc,\dkbsic)+\frac{3}{2}\trchb(\dkbrhoc,\dkbsic)=&-d_1\dkb\bbc+\fl[\nab_3(\dkbrhoc,\dkbsic)]+\err[\nab_3(\dkbrhoc,\dkbsic)],\\
\fl[\nab_3(\dkbrhoc,\dkbsic)]:=&\Gab^{(1)}\c\O_3^1+\O_3^2\c(\aac^{(1)},\bbc^{(1)},\rhoc^{(1)},\sic^{(1)}),\\
\err[\nab_3(\dkbrhoc,\dkbsic)]:=&\Gag^{(1)}\c(\aac,\bbc)+\Gag\c(\aac^{(1)},\bbc^{(1)})+\Gab\c(\rhoc^{(1)},\sic^{(1)})+\Gab^{(1)}\c(\rhoc,\sic),\\\\
\nab_4(\rhoc,-\sic)+\frac{3}{2}\trch(\rhoc,-\sic)
=&d_1\bc+\fl[\nab_4\rhoc]+\err[\nab_4\rhoc],\\
\fl[\nab_4\rhoc]:=&\Gag^{(1)}\c\O_3^1+\O_3^2\c(\ac,\bc),\\
\err[\nab_4\rhoc]:=&\Gab\c\ac+\Gag\c\bc,\\ \\
\nab_4(\dkbrhoc,-\dkbsic)+\frac{3}{2}\trch(\dkbrhoc,-\dkbsic)=&d_1\dkbbc+\fl[\nab_4(\dkbrhoc,-\dkbsic)]+\err[\nab_4(\dkbrhoc,\dkbsic)],\\
\fl[\nab_4(\dkbrhoc,-\dkbsic)]:=&\Gag^{(1)}\c\O_3^1+\O_3^2\c(\ac^{(1)},\bc^{(1)},\rhoc^{(1)},\sic^{(1)}),\\
\err[\nab_4(\dkbrhoc,-\dkbsic)]:=&\Gab^{(1)}\c\ac+\Gab\c\ac^{(1)}+\Gag^{(1)}\c(\bc,\rhoc,\sic)+\Gag\c(\bc^{(1)},\rhoc^{(1)},\sic^{(1)}),\\ \\
\nab_3\bc+\trchb\,\bc=&\nab\rhoc+{^*\nab\sic}+\fl[\nab_3\bc]+\err[\nab_3\bc],\\
\fl[\nab_3\bc]:=&\Gag^{(1)}\c\O_3^1+\O_2^1\c\bc+\O_3^2\c(\bbc,\rhoc,\sic),\\
\err[\nab_3\bc]:=&\Gab\c\bc+\Gag\c(\bbc,\rhoc,\sic),\\ \\
\nab_3\dkb\bc+\trchb\,\dkb\bc=&\nab\dkb\rhoc+{^*\nab\dkb\sic}+\fl[\nab_3\dkbbc]+\err[\nab_3\dkbbc],\\
\fl[\nab_3\dkbbc]:=&\Gag^{(1)}\c\O_3^1+\O_2^1\c\bc^{(1)}+\O_3^2\c(\bbc^{(1)},\rhoc^{(1)},\sic^{(1)}),\\
\err[\nab_3\dkbbc]:=&\Gab^{(1)}\c\bc+\Gab\c\bc^{(1)}+\Gag^{(1)}\c(\bbc,\rhoc,\sic)+\Gag\c(\bbc^{(1)},\rhoc^{(1)},\sic^{(1)}),\\ \\
\nab_4\bc+2\trch\,\bc=&\sdiv\ac+\fl[\nab_4\bc]+\err[\nab_4\bc],\\
\fl[\nab_4\bc]:=&\Gag^{(1)}\c\O_4^2+\O_2^1\c(\bc,\ac),\\
\err[\nab_4\bc]:=&\Gag\c(\bc,\ac),\\ \\
\nab_4\dkb\bc+2\trch\dkb\bc=&\sdiv\dkb\ac+\fl[\nab_4\dkbbc]+\err[\nab_4\dkbbc],\\
\fl[\nab_4\dkbbc]:=&\Gag^{(1)}\c\O_4^2+\O_2^1\c(\bc^{(1)},\ac^{(1)}),\\
\err[\nab_4\dkbbc]:=&\Gag^{(1)}\c(\bc,\ac)+\Gag\c(\bc^{(1)},\ac^{(1)}),\\ \\
\nab_3\ac+\f12\trchb\,\ac=&\nab\hot\bc+\fl[\nab_3\ac]+\err[\nab_3\ac],\\
\fl[\nab_3\ac]:=&\Gag^{(1)}\c\O_3^1+\O_2^1\c\ac+\O_3^2\c(\rhoc,\sic,\bc),\\
\err[\nab_3\ac]:=&\Gab\c\ac+\Gag\c\bc,\\ \\
\nab_3\dkb\ac+\f12\trchb\dkb\ac=&\nab\hot\dkb\bc+\fl[\nab_3\dkbac]+\err[\nab_3\dkbac],\\
\fl[\nab_3\dkbac]:=&\Gag^{(1)}\c\O_3^1+\O_2^1\c\ac^{(1)}+\O_3^2\c(\rhoc^{(1)},\sic^{(1)},\bc^{(1)}),\\
\err[\nab_3\dkbac]:=&\Gab^{(1)}\c\ac+\Gab\c\ac^{(1)}+\Gag^{(1)}\c\bc+\Gag\c\bc^{(1)}.
\end{align*}
\end{proposition}
\begin{proof}
We recall from Proposition \ref{standardbianchi}
\begin{align*}
\nab_4\aa+\f12\tr\chi\,\aa&=-\nab\hot\bb+h[\aa_4],\\
h[\aa_4]&:=4\om\aa-3(\hchb\rho-{^*\hchb}\si)+(\ze-4\etab)\hot\bb.
\end{align*}
Applying Lemma \ref{dkbcomm}, we deduce
\begin{align*}
    \nab_4\dkb\aa+\f12\trch\dkb\aa&=-\nab\hot\dkb\bb+h[\aa_4]^{[1]},
\end{align*}
where
\begin{align*}
    h[\aa_4]^{[1]}&=\dkb h[\aa_4]+(\Gag+\O_3^2)\c h[\aa_4]+(\Gag+\O_3^2)\c(\aa^{(1)},\bb^{(1)})+\Gag^{(1)}\c\aa\\
    &=\dkb h[\aa_4]+(\Gag+\O_3^2)\c h[\aa_4]+(\Gag+\O_3^2)\c(\aac^{(1)},\bbc^{(1)})+\Gag^{(1)}\c\aac\\
    &+\Gag\c\O_5^3+\Gag\c\O_4^2+\O_7^4+\Gag^{(1)}\c\O_5^3.
\end{align*}
We compute
\begin{align}
\begin{split}\label{haa4compute}
h[\aa_4]&=4\om\aa-3(\hchb\rho-{^*\hchb}\si)+(\ze-4\etab)\hot\bb\\
&=(\O_2^1+\Gag)(\O_5^3+\aac)+(\O_3^2+\Gab)\c(\O_3^1+\rhoc+\sic)+(\O_3^2+\Gag)\c(\O_4^2+\bbc)\\
&=\O^4_7+\Gag\c \O_5^3+\O_2^1\c\aac+\Gag\c\aac+\O_6^3+\Gab\c\O_3^1+\O_3^2\c(\rhoc,\sic)+\Gab\c(\rhoc,\sic)\\
&+\O_7^4+\Gag\c\O_4^2+\O_3^2\c\bbc+\Gag\c\bbc\\
&=\O_6^3+\Gab\c\O_3^1+\O_2^1\c\aac+\O_3^2\c(\bbc,\rhoc,\sic)+\Gag\c(\aac,\bbc)+\Gab\c(\rhoc,\sic),
\end{split}
\end{align}
which implies
\begin{align*}
h[\aa_4]^{[1]}&=\O_6^3+\Gab^{(1)}\c\O_3^1+\O_2^1\c\aac^{(1)}+\O_3^2\c(\bbc^{(1)},\rhoc^{(1)},\sic^{(1)})\\
&+\Gag^{(1)}\c(\aac,\bbc)+\Gab^{(1)}\c(\rhoc,\sic)+\Gag\c(\aac^{(1)},\bbc^{(1)})+\Gab\c(\rhoc^{(1)},\sic^{(1)}).
\end{align*}
Applying Lemma \ref{linbianchi}, we obtain
\begin{align*}
\nab_4\aac+\f12\trch\,\aac&=-\nab\hot\bbc+\widecheck{h[\aa_4]}+\Jc\c\aa_K+\Lc\c\bb_K+\Gag\c\dk^{\leq 1}\aa_K\\
&=-\nab\hot\bbc+\widecheck{h[\aa_4]}+\Gag^{(1)}\c\O_5^3+\Gag^{(1)}\c\O_4^2+\Gag\c\O_5^3\\
&=-\nab\hot\bbc+\widecheck{h[\aa_4]}+\Gag^{(1)}\c\O_4^2,
\end{align*}
and similarly
\begin{align*}
\nab_4\widecheck{\dkb\aa}+\f12\trchb\widecheck{\dkb\aa}=-\nab\hot\widecheck{\dkb\bb}+\widecheck{h[\aa_4]^{[1]}}+\Gag^{(1)}\c\O_4^2.
\end{align*}
Hence, we infer
\begin{align*}
\nab_4\aac+\f12\trchb\aac&=-\nab\hot\bbc+\widecheck{h[\aa_4]}+\Gag^{(1)}\c\O_4^2\\
&=-\nab\hot\bbc+\Gab^{(1)}\c\O_3^1+\O_2^1\c\aac+\O_3^2\c(\bbc,\rhoc,\sic)+\Gag\c(\aac,\bbc)+\Gab\c(\rhoc,\sic).
\end{align*}
Similarly, we have
\begin{align*}
    \nab_4\widecheck{\dkb\aa}+\f12\trchb\widecheck{\dkb\aa}&=-\nab\hot\widecheck{\dkb\bb}+\widecheck{h[\aa_4]^{[1]}}+\Gag^{(1)}\c\O_4^2\\
    &=-\nab\hot\widecheck{\dkb\bb}+\Gab^{(1)}\c\O_3^1+\O_2^1\c\aac^{(1)}+\O_3^2\c(\bbc^{(1)},\rhoc^{(1)},\sic^{(1)})\\
&+\Gag^{(1)}\c(\aac,\bbc)+\Gab^{(1)}\c(\rhoc,\sic)+\Gag\c(\aac^{(1)},\bbc^{(1)})+\Gab\c(\rhoc^{(1)},\sic^{(1)}).
\end{align*}
Next, we recall from Proposition \ref{standardbianchi}
\begin{align*}
\nab_3\bb+2\trchb\,\bb&=-\sdiv\aa+h[\bb_3],\\
h[\bb_3]&:=-2\omb\bb+(2\zeta-\eta)\cdot\aa.
\end{align*}
Applying Lemma \ref{dkbcomm}, we deduce
\begin{align*}
    \nab_3\dkb\bb+2\trchb\dkb\bb&=-\sdiv\dkb\aa+h[\bb_3]^{[1]},
\end{align*}
where
\begin{align*}
    h[\bb_3]^{[1]}&=\dkb h[\bb_3]+(\Gag+\O_3^2)\c h[\bb_3]+(\Gag+\O_3^2)\c\aa^{(1)}+(\Gab+\O_3^2)\c\bb^{(1)}+\Gab^{(1)}\c\bb\\
    &=\dkb h[\bb_3]+(\Gag+\O_3^2)\c h[\bb_3]+(\Gag+\O_3^2)\c\aac^{(1)}+(\Gab+\O_3^2)\c\bbc^{(1)}+\Gab^{(1)}\c\bbc\\
    &+\Gab\c\O_5^3+\Gab\c\O_4^2+\O_7^4+\Gab^{(1)}\c\O_4^2.
\end{align*}
We compute
\begin{align*}
h[\bb_3]&=-2\omb\bb+(2\zeta-\eta)\cdot\aa\\
&=(\O_2^1+\Gab)\c(\O_4^2+\bbc)+(\O_3^2+\Gag)\c(\O_5^3+\aac)\\
&=\O_6^3+\Gab\c\O_4^2+\O_2^1\c\bbc+\Gab\c\bbc+\O_8^5+\Gag\c\O_5^3+\O_3^2\c\aac+\Gag\c\aac\\
&=\O_6^3+\Gab\c\O_4^2+\O_2^1\c\bbc+\O_3^2\c\aac+\Gab\c\bbc+\Gag\c\aac,
\end{align*}
which implies
\begin{align*}
h[\bb_3]^{[1]}&=\O_6^3+\Gab^{(1)}\c\O_4^2+\O_2^1\c\bbc^{(1)}+\O_3^2\c\aac^{(1)}+\Gab^{(1)}\c\bbc+\Gab\c\bbc^{(1)}+\Gag^{(1)}\c\aac+\Gag\c\aac^{(1)}.
\end{align*}
Applying Lemma \ref{linbianchi}, we obtain
\begin{align*}
    \nab_3\bbc+2\trchb\,\bbc&=-\sdiv\aac+\widecheck{h[\bb_3]}+\Jbc\c\bb_K+\Lc\c\aac_K+\Gag\c\dk^{\leq 1}\bb_K\\
    &=-\sdiv\aac+\widecheck{h[\bb_3]}+\Gab^{(1)}\c\O_4^2+\Gag^{(1)}\c\O_5^3+\Gag\c\O_4^2\\
    &=-\sdiv\aac+\widecheck{h[\bb_3]}+\Gab^{(1)}\c\O_4^2,
\end{align*}
and similarly
\begin{equation*}
    \nab_3\widecheck{\dkb\bb}+2\trchb\widecheck{\dkb\bb}=-\sdiv\widecheck{\dkb\aa}+\widecheck{h[\bb_3]^{[1]}}+\Gab^{(1)}\c\O_4^2.
\end{equation*}
Hence, we deduce
\begin{align*}
\nab_3\bbc+2\trchb\,\bbc=-\sdiv\aac+\Gab^{(1)}\c\O_4^2+\O_2^1\c\bbc+\O_3^2\c\aac+\Gab\c\bbc+\Gag\c\aac.
\end{align*}
Similarly, we have
\begin{align*}
    \nab_3\widecheck{\dkb\bb}+2\trchb\widecheck{\dkb\bb}&=-\sdiv\widecheck{\dkb\aa}+\Gab^{(1)}\c\O_4^2+\O_2^1\c\bbc^{(1)}+\O_3^2\c\aac^{(1)}\\
    &+\Gab^{(1)}\c\bbc+\Gab\c\bbc^{(1)}+\Gag^{(1)}\c\aac+\Gag\c\aac^{(1)}.
\end{align*}
Next, we recall from Proposition \ref{standardbianchi}
\begin{align*}
    \nab_4\bb+\tr\chi\,\bb&=-\nabla\rho+{^*\nabla\sigma}+h[\bb_4],\\
    h[\bb_4]&:=2\omega\bb+2\hchb\cdot\beta-3(\etab\rho-{^*\etab}\sigma).
\end{align*}
Applying Lemma \ref{dkbcomm}, we infer
\begin{align*}
    \nab_4\dkb\bb+\trch\dkb\bb&=-\nab\dkb\rho+{^*\nab\dkb\si}+h[\bb_4]^{[1]},
\end{align*}
where
\begin{align*}
    h[\bb_4]^{[1]}&=\dkb h[\bb_4]+(\Gag+\O_3^2)\c h[\bb_4]+(\Gag+\O_3^2)\c (\bb^{(1)},\rho^{(1)},\si^{(1)})+\Gag^{(1)}\c \bb^{(1)}\\
    &=\dkb h[\bb_4]+(\Gag+\O_3^2)\c h[\bb_4]+(\Gag+\O_3^2)\c (\bbc^{(1)},\rhoc^{(1)},\sic^{(1)})+\Gag^{(1)}\c \bbc^{(1)}\\
    &+\Gag\c \O_4^2+\Gag\c \O_3^1+\O_6^3+\Gag^{(1)}\c\O_4^2.
\end{align*}
We compute
\begin{align*}
{h[\bb_4]}&=(\O_2^1+\Gag)\c(\O_4^2+\bbc)+(\O_3^2+\Gab)\c(\O_4^2+\bc)+(\O_3^2+\Gag)\c(\O_3^1+\rhoc+\sic)\\
&=\O_6^3+\Gag\c\O_4^2+\O_2^1\c\bbc+\Gag\c\bbc+\O_7^4+\Gab\c\O_4^2+\O_3^2\c\bc+\Gab\c\bc\\
&+\O_6^3+\Gag\c\O_3^1+\O_3^2\c(\rhoc,\sic)+\Gag\c(\rhoc,\sic)\\
&=\O_6^3+\Gag\c\O_3^1+\O_2^1\c\bbc+\O_3^2\c(\rhoc,\sic,\bc)+\Gag\c(\bbc,\rhoc,\sic)+\Gab\c\bc,
\end{align*}
which implies
\begin{align*}
h[\bb_4]^{[1]}&=\O_6^3+\Gag^{(1)}\c\O_3^1+\O_2^1\c\bbc^{(1)}+\O_3^2\c(\rhoc^{(1)},\sic^{(1)},\bc^{(1)})\\
&+\Gag^{(1)}\c(\bbc,\rhoc,\sic)+\Gag\c(\bbc^{(1)},\rhoc^{(1)},\sic^{(1)})+\Gab^{(1)}\c\bc+\Gab\c\bc^{(1)}.
\end{align*}
Applying Lemma \ref{linbianchi}, we obtain
\begin{align*}
    \nab_4\bbc+\trch\,\bbc&=-\nab\rhoc+{^*\nab}\sic+\widecheck{h[\bb_4]}+\Jc\c\bb_K+\Lc\c(\rho_K,\si_K)+\Gag\c\dk^{\leq 1}\bb_K\\
    &=-\nab\rhoc+{^*\nab}\sic+\widecheck{h[\bb_4]}+\Gag^{(1)}\c\O_3^1.
\end{align*}
Hence, we infer
\begin{align*}
     \nab_4\bbc+\trch\,\bbc=-\nab\rhoc+{^*\nab}\sic+\Gag^{(1)}\c\O_3^1+\O_2^1\c\bbc+\O_3^2\c(\rhoc,\sic,\bc)+\Gag\c(\bbc,\rhoc,\sic)+\Gab\c\bc.
\end{align*}
Similarly, we have
\begin{align*}
\nab_4\dkbbbc+\trch\dkbbbc&=-\nab\dkbrhoc+{^*\nab}\dkbsic+\Gag^{(1)}\c\O_3^1+\O_2^1\c\bbc^{(1)}+\O_3^2\c(\rhoc^{(1)},\sic^{(1)},\bc^{(1)})\\
&+\Gag^{(1)}\c(\bbc,\rhoc,\sic)+\Gag\c(\bbc^{(1)},\rhoc^{(1)},\sic^{(1)})+\Gab^{(1)}\c\bc+\Gab\c\bc^{(1)}.
\end{align*}
Next, we recall from Proposition \ref{standardbianchi}
\begin{align*}
    \nab_3(\rho,\si)+\frac{3}{2}\tr\chib(\rho,\si)&=-d_1\bb+h[\rho_3,\si_3],\\
    h[\rho_3,\si_3]&=-\f12\hch\cdot(\aa,-{^*\aa})+\ze\cdot(\bb,-{^*\bb})-2\eta\cdot(\bb,-{^*\bb}).
\end{align*}
Applying Lemma \ref{dkbcomm}, we obtain
\begin{align*}
    \nab_3\dkb(\rho,\si)+\frac{3}{2}\trchb\dkb(\rho,\si)=-d_1\dkb\bb+h[\rho_3,\si_3]^{[1]},
\end{align*}
where
\begin{align*}
h[\rho_3,\si_3]^{[1]}&=\dkb h[\rho_3,\si_3]+(\Gag+\O_3^2)\c h[\rho_3,\si_3]+(\Gab+\O_3^2)\c(\rho^{(1)},\si^{(1)})+(\Gag+\O_3^2)\c\bb^{(1)}+\Gab^{(1)}\c(\rho,\si)\\
    &=\dkb h[\rho_3,\si_3]+(\Gag+\O_3^2)\c h[\rho_3,\si_3]+(\Gab+\O_3^2)\c(\rhoc^{(1)},\sic^{(1)})+(\Gag+\O_3^2)\c\bbc^{(1)}+\Gab^{(1)}\c(\rhoc,\sic)\\
    &+\Gab\c \O_3^1+\O_6^3+\Gag\c\O_4^2+\O_7^4+\Gab^{(1)}\c\O_3^1.
\end{align*}
We compute
\begin{align*}
    h[\rho_3,\si_3]&=(\O_3^2+\Gag)\c(\O_5^3+\aac)+(\O_3^2+\Gag)\c(\O_4^2+\bbc)\\
&=\O_8^5+\Gag\c\O_5^3+\O_3^2\c\aac+\Gag\c\aac+\O_7^4+\Gag\c\O_4^2+\O_3^2\c\bbc+\Gag\c\bbc\\
    &=\O_7^4+\Gag\c\O_4^2+\O_3^2\c(\aac,\bbc)+\Gag\c(\aac,\bbc),
\end{align*}
which implies
\begin{align*}
h[\rho_3,\si_3]^{[1]}&=\O_6^3+\Gab^{(1)}\c\O_3^1+\O_3^2\c(\aac^{(1)},\bbc^{(1)},\rhoc^{(1)},\sic^{(1)})\\
&+\Gag^{(1)}\c(\aac,\bbc)+\Gag\c(\aac^{(1)},\bbc^{(1)})+\Gab\c(\rhoc^{(1)},\sic^{(1)})+\Gab^{(1)}\c(\rhoc,\sic).
\end{align*}
Applying Lemma \ref{linbianchi}, we deduce
\begin{align*}
    \nab_3(\rhoc,\sic)+\frac{3}{2}\trchb\,(\rhoc,\sic)&=-d_1\bbc+\widecheck{h[\rho_3,\si_3]}+\Lc\c\bb_K+\Gag\c(\rho_K,\si_K)\\
    &=-d_1\bbc+\widecheck{h[\rho_3,\si_3]}+\Gag^{(1)}\c\O_3^1\\
&=-d_1\bbc+\Gag^{(1)}\c\O_3^1+\O_3^2\c(\aac,\bbc)+\Gag\c(\aac,\bbc).
\end{align*}
Similarly, we have
\begin{align*}
\nab_3(\dkbrhoc,\dkbsic)+\frac{3}{2}\trchb(\dkbrhoc,\dkbsic)&=-d_1\dkbbbc+\Gab^{(1)}\c\O_3^1+\O_3^2\c(\aac^{(1)},\bbc^{(1)},\rhoc^{(1)},\sic^{(1)})\\
&+\Gag^{(1)}\c(\aac,\bbc)+\Gag\c(\aac^{(1)},\bbc^{(1)})+\Gab\c(\rhoc^{(1)},\sic^{(1)})+\Gab^{(1)}\c(\rhoc,\sic).
\end{align*}
We have from Proposition \ref{standardbianchi}
\begin{align*}
    \nab_4(\rho,-\si)+\frac{3}{2}\tr\chi\,(\rho,-\si)&=d_1\b+h[\rho_4,\si_4],\\
    h[\rho_4,\si_4]&=-\frac{1}{2}\hchb\cdot(\a,{^*\a})+(\ze+2\etab)\cdot(\b,{^*\b}).
\end{align*}
Applying Lemma \ref{dkbcomm}, we have
\begin{align*}
    \nab_4\dkb(\rho,-\si)+\frac{3}{2}\trch\dkb(\rho,-\si)=d_1\dkb\b+h[\rho_4,\si_4]^{[1]},
\end{align*}
where
\begin{align*}
    h[\rho_4,\si_4]^{[1]}&=\dkb h[\rho_4,\si_4]+(\Gag+\O_3^2)\c h[\rho_4,\si_4]+(\Gag+\O_3^2)\c (\b^{(1)},\rho^{(1)},\si^{(1)})+\Gag^{(1)}\c (\rho,\si)\\
    &=\dkb h[\rho_4,\si_4]+(\Gag+\O_3^2)\c h[\rho_4,\si_4]+(\Gag+\O_3^2)\c (\bc^{(1)},\rhoc^{(1)},\sic^{(1)})+\Gag^{(1)}\c (\rhoc,\sic)\\
    &+\Gag\c \O_3^1+\O_6^3+\Gag^{(1)}\c\O_3^1.
\end{align*}
We compute
\begin{align*}
h[\rho_4,\si_4]&=(\O_3^2+\Gab)\c(\O_5^3+\ac)+(\O_3^2+\Gag)\c(\O_4^2+\bc)\\
&=\O_8^5+\Gab\c\O_5^3+\O_3^2\c\ac+\Gab\c\ac+\O_7^4+\Gag\c\O_4^2+\O_3^2\c\bc+\Gag\c\bc\\
&=\O_7^4+\Gag\c\O_4^2+\O_3^2\c(\ac,\bc)+\Gab\c\ac+\Gag\c\bc,
\end{align*}
which implies
\begin{align*}
h[\rho_4,\si_4]^{[1]}&=\O_6^3+\Gag^{(1)}\c\O_3^1+\O_3^2\c(\ac^{(1)},\bc^{(1)},\rhoc^{(1)},\sic^{(1)})\\
&+\Gab^{(1)}\c\ac+\Gab\c\ac^{(1)}+\Gag^{(1)}\c(\bc,\rhoc,\sic)+\Gag\c(\bc^{(1)},\rhoc^{(1)},\sic^{(1)}).
\end{align*}
Applying Lemma \ref{linbianchi}, we deduce
\begin{align*}
    \nab_4(\rhoc,-\sic)+\frac{3}{2}\trch(\rhoc,-\sic)&=d_1\bc+\widecheck{h[\rho_4,\si_4]}+\Jc\c(\rho_K,\si_K)+\Lc\c\b_K+\Gag\c\dk^{\leq 1}(\rho_K,\si_K)\\
    &=d_1\bc+\widecheck{h[\rho_4,\si_4]}+\Gag^{(1)}\c\O_3^1\\
    &=d_1\bc+\Gag^{(1)}\c\O_3^1+\O_3^2\c(\ac,\bc)+\Gab\c\ac+\Gag\c\bc.
\end{align*}
Similarly, we have
\begin{align*}
\nab_4(\dkbrhoc,-\dkbsic)+\frac{3}{2}\trch(\dkbrhoc,-\dkbsic)&=d_1\dkbbc+\Gag^{(1)}\c\O_3^1+\O_3^2\c(\ac^{(1)},\bc^{(1)},\rhoc^{(1)},\sic^{(1)})\\
&+\Gab^{(1)}\c\ac+\Gab\c\ac^{(1)}+\Gag^{(1)}\c(\bc,\rhoc,\sic)+\Gag\c(\bc^{(1)},\rhoc^{(1)},\sic^{(1)}).
\end{align*}
Next, we recall from Proposition \ref{standardbianchi}
\begin{align*}
    \nab_3\b+\tr\chib\,\b&=\nab\rho+{^*\nabla}\si+h[\b_3],\\
    h[\b_3]&=2\omb\b+2\hch\cdot\bb+3(\eta\rho+{^*\eta}\si).
\end{align*}
Applying Lemma \ref{dkbcomm}, we deduce
\begin{align*}
    \nab_3\dkb\b+\trchb\dkb\b=\nab\dkb\rho+{^*\nab\dkb\si}+h[\b_3]^{[1]},
\end{align*}
where
\begin{align*}
    h[\b_3]^{[1]}&=\dkb h[\b_3]+(\Gag+\O_3^2)\c h[\b_3]+(\Gab+\O_3^2)\c \b^{(1)}+(\Gag+\O_3^2)\c (\rho^{(1)},\si^{(1)})+\Gag^{(1)}\c\b \\
    &=\dkb h[\b_3]+(\Gag+\O_3^2)\c h[\b_3]+(\Gab+\O_3^2)\c\bc^{(1)}+(\Gag+\O_3^2)\c (\rhoc^{(1)},\sic^{(1)})+\Gag^{(1)}\c\bc \\
    &+\Gab\c\O_4^2+\O_7^4+\Gag\c\O_3^1+\O_6^3+\Gag^{(1)}\c\O_4^2.
\end{align*}
We compute
\begin{align*}
{h[\b_3]}&=(\O_2^1+\Gab)\c(\O_4^2+\bc)+(\O_3^2+\Gag)\c(\O_4^2+\bbc)+(\O_3^2+\Gag)\c(\O_3^1+\rhoc+\sic)\\
&=\O_6^3+\Gab\c\O_4^2+\O_2^1\c\bc+\Gab\c\bc+\O_7^4+\Gag\c\O_4^2+\O_3^2\c\bbc+\Gag\c\bbc\\
&+\O_6^3+\Gag\c\O_3^1+\O_3^2\c(\rhoc,\sic)+\Gag\c(\rhoc,\sic)\\
&=\O_6^3+\Gag\c\O_3^1+\O_2^1\c\bc+\O_3^2\c(\bbc,\rhoc,\sic)+\Gab\c\bc+\Gag\c(\bbc,\rhoc,\sic),
\end{align*}
which implies
\begin{align*}
h[\b_3]^{[1]}&=\O_6^3+\Gag^{(1)}\c\O_3^1+\O_2^1\c\bc^{(1)}+\O_3^2\c(\bbc^{(1)},\rhoc^{(1)},\sic^{(1)},\bc^{(1)})\\
&+\Gab^{(1)}\c\bc+\Gab\c\bc^{(1)}+\Gag^{(1)}\c(\bbc,\rhoc,\sic)+\Gag\c(\bbc^{(1)},\rhoc^{(1)},\sic^{(1)}).
\end{align*}
Applying Lemma \ref{linbianchi}, we infer
\begin{align*}
    \nab_3\bc+\trchb\,\bc&=\nab\rhoc+{^*\nab\sic}+\widecheck{h[\b_3]}+\Lc\c(\rho_K,\si_K)+\Jbc\c\b_K+\Gag\c\dk^{\leq 1}\b_K\\
    &=\nab\rhoc+{^*\nab\sic}+\widecheck{h[\b_3]}+\Gag^{(1)}\c\O_3^1.
\end{align*}
Hence, we obtain
\begin{align*}
    \nab_3\bc+\trchb\,\bc=\nab\rhoc+{^*\nab\sic}+\Gag^{(1)}\c\O_3^1+\O_2^1\c\bc+\O_3^2\c(\bbc,\rhoc,\sic)+\Gab\c\bc+\Gag\c(\bbc,\rhoc,\sic).
\end{align*}
Similarly, we have
\begin{align*}
    \nab_3\dkbbc+\trchb\,\dkbbc&=\nab\dkbrhoc+{^*\nab\dkbsic}+\Gag^{(1)}\c\O_3^1+\O_2^1\c\bc^{(1)}+\O_3^2\c(\bbc^{(1)},\rhoc^{(1)},\sic^{(1)})\\
&+\Gab^{(1)}\c\bc+\Gab\c\bc^{(1)}+\Gag^{(1)}\c(\bbc,\rhoc,\sic)+\Gag\c(\bbc^{(1)},\rhoc^{(1)},\sic^{(1)}).
\end{align*}
Next, we recall from Proposition \ref{standardbianchi}
\begin{align*}
\nab_4\b+2\tr\chi\,\b&=\sdiv\alpha+h[\b_4]\\
h[\b_4]&=-2\om\b+(2\ze+\etab)\a.
\end{align*}
Applying Lemma \ref{dkbcomm}, we deduce
\begin{align*}
    \nab_4\dkb\b+2\trch\dkb\b=\sdiv\dkb\a+h[\b_4]^{[1]},
\end{align*}
where
\begin{align*}
    h[\b_4]^{[1]}&=\dkb h[\b_4]+(\Gag+\O_3^2)\c h[\b_4]+(\Gag+\O_3^2)\c(\b^{(1)},\a^{(1)})+\Gag^{(1)}\c\b\\
    &=\dkb h[\b_4]+(\Gag+\O_3^2)\c h[\b_4]+(\Gag+\O_3^2)\c(\bc^{(1)},\ac^{(1)})+\Gag^{(1)}\c\bc\\
    &+\Gag\c\O_4^2+\O_7^4+\Gag^{(1)}\c\O_4^2.
\end{align*}
We compute
\begin{align*}
    h[\b_4]&=(\O_2^1+\Gag)\c(\O_4^2+\bc)+(\O_3^2+\Gag)\c(\O_5^3+\ac)\\
&=\O_6^3+\Gag\c\O_4^2+\O_2^1\c\bc+\Gag\c\bc+\O_8^5+\Gag\c\O_5^3+\O_3^2\c\ac+\Gag\c\ac\\
    &=\O_6^3+\Gag\c\O_4^2+\O_2^1\c(\bc,\ac)+\Gag\c(\bc,\ac),
\end{align*}
which implies
\begin{align*}
h[\b_4]^{[1]}=\O_6^3+\Gag^{(1)}\c\O_4^2+\O_2^1\c(\bc^{(1)},\ac^{(1)})+\Gag\c(\bc^{(1)},\ac^{(1)})+\Gag^{(1)}\c(\bc,\ac).
\end{align*}
Applying Lemma \ref{linbianchi}, we obtain
\begin{align*}
\nab_4\bc+2\trch\,\bc&=\sdiv\ac+\widecheck{h[\b_4]}+\Jc\c\b_K+\Lc\c\a_K+\Gag\c\b_K\\
&=\sdiv\ac+\widecheck{h[\b_4]}+\Gag^{(1)}\c\O_4^2\\
&=\sdiv\ac+\Gag^{(1)}\c\O_4^2+\O_2^1\c(\bc,\ac)+\Gag\c(\bc,\ac).
\end{align*}
Similarly, we have
\begin{align*}
\nab_4\dkbbc+2\trch\dkbbc=\sdiv\dkbac+\Gag^{(1)}\c\O_4^2+\O_2^1\c(\bc^{(1)},\ac^{(1)})+\Gag^{(1)}\c(\bc,\ac)+\Gag\c(\bc^{(1)},\ac^{(1)}).
\end{align*}
Finally, we recall from Proposition \ref{standardbianchi}
\begin{align*}
    \nab_3\a+\f12\trchb\,\a&=\nab\hot\b+h[\a_3],\\
    h[\a_3]&=4\omb\a-3(\hch\rho+{^*\hch}\si)+(\ze+4\eta)\hot\b.
\end{align*}
Applying Lemma \ref{dkbcomm}, we obtain
\begin{align*}
    \nab_3\dkb\a+\f12\trchb\dkb\a=\nab\hot\dkb\b+h[\a_3]^{[1]},
\end{align*}
where
\begin{align*}
    h[\a_3]^{[1]}&=\dkb h[\a_3]+(\Gag+\O_3^2)\c h[\a_3]+(\Gab+\O_3^2)\c\a^{(1)}+(\Gag+\O_3^2)\c\b^{(1)}+\Gab^{(1)}\c\a\\
    &=\dkb h[\a_3]+(\Gag+\O_3^2)\c h[\a_3]+(\Gab+\O_3^2)\c\ac^{(1)}+(\Gag+\O_3^2)\c\bc^{(1)}+\Gab^{(1)}\c\ac\\
    &+\Gab\c\O_5^3+\O_8^5+\Gag\c\O_4^2+\O_7^4+\Gab^{(1)}\c\O_5^3.
\end{align*}
We compute
\begin{align*}
h[\a_3]&=(\O_2^1+\Gab)\c(\O_5^3+\ac)+(\O_3^2+\Gag)\c(\O_3^1+\rhoc+\sic)+(\O_3^2+\Gag)\c(\O_4^2+\bc)\\
&=\O_7^4+\Gab\c\O_5^3+\O_2^1\c\ac+\Gab\c\ac+\O_6^3+\Gag\c\O_3^1+\O_3^2\c(\rhoc,\sic)\\
&+\O_7^4+\Gag\c\O_4^2+\O_3^2\c\bc+\Gag\c\bc\\
&=\O_6^3+\Gag\c\O_3^1+\O_2^1\c\ac+\O_3^2\c(\rhoc,\sic,\bc)+\Gab\c\ac+\Gag\c\bc,
\end{align*}
which implies
\begin{align*}
h[\a_3]^{[1]}&=\O_6^3+\Gag^{(1)}\c\O_3^1+\O_2^1\c\ac^{(1)}+\O_3^2\c(\rhoc^{(1)},\sic^{(1)},\bc^{(1)})\\
&+\Gab^{(1)}\c\ac+\Gab\c\ac^{(1)}+\Gag^{(1)}\c\bc+\Gag\c\bc^{(1)}.
\end{align*}
Applying Lemma \ref{linbianchi}, we infer
\begin{align*}
\nab_3\ac+\f12\trchb\,\ac&=\nab\hot\bc+\widecheck{h[\a_3]}+\Jbc\c\a_K+\Lc\c\b_K+\Gag\c\a_K\\
&=\nab\hot\bc+\widecheck{h[\a_3]}+\Gab^{(1)}\c\O_5^3+\Gag^{(1)}\c\O_4^2+\Gag\c\O_5^3\\
&=\nab\hot\bc+\widecheck{h[\a_3]}+\Gag^{(1)}\c\O_4^2\\
&=\nab\hot\bc+\Gag^{(1)}\c\O_3^1+\O_2^1\c\ac+\O_3^2\c(\rhoc,\sic,\bc)+\Gab\c\ac+\Gag\c\bc.
\end{align*}
Similarly, we have
\begin{align*}
\nab_3\dkbac+\f12\trchb\dkbac&=\nab\hot\dkbbc+\Gag^{(1)}\c\O_3^1+\O_2^1\c\ac^{(1)}+\O_3^2\c(\rhoc^{(1)},\sic^{(1)},\bc^{(1)})\\
&+\Gab^{(1)}\c\ac+\Gab\c\ac^{(1)}+\Gag^{(1)}\c\bc+\Gag\c\bc^{(1)}.
\end{align*}
This concludes the proof of Proposition \ref{BianchieqA}.
\end{proof}
\subsection{Proof of Proposition \ref{metriceqre}}\label{pfmetric}
\begin{proposition}\label{metriceq}
We have the following equations:
\begin{align*}
\Om\nab_4(\gac_{AB})=&\Om\trch\,\gac_{AB}+\widecheck{\Om\trch}(\ga_K)_{AB}+2\Om\hchc_{AB}+2\Omc(\hch_K)_{AB},\\
\Om\nab_4(\Omc)=&-2\Om(\widecheck{\Om\om})-2\Omc(\Om_K\om_K),\\
\Om\nab_4(\bcc^A)=&-4\left(\widecheck{\Om^2}\ze_B\ga^{AB}+\Om_K^2\zec_B\ga^{AB}+\Om_K^2(\ze_K)_B
\widecheck{\ga^{AB}}\right),\\
\Om\nab_4(\inc_{AB})=&\Om\trch\,\inc_{AB}+\widecheck{\Om\trch}(\in_K)_{AB}.
\end{align*}
In the case $\Omc=0$ and $\bcc=0$, we have the following equations:
\begin{align*}
\Om\nab_3(\gac_{AB})-\Om\trchb\,\gac_{AB}=&\widecheck{\Om\trchb}(\ga_K)_{AB}+2\Om\hchbc_{AB}-\pr_A(\bu_K^C)\gac_{BC}-\pr_B(\bu_K^C)\gac_{AC},\\
\Om\nab_3(\inc_{AB})-\Om\trchb\,\inc_{AB}=&\widecheck{\Om\trchb}(\in_K)_{AB}-\pr_A(\bu_K^C)\inc_{BC}-\pr_B(\bu_K^C)\inc_{AC}.
\end{align*}
\end{proposition}
\begin{remark}
    Note that Proposition \ref{metriceqre} follows directly from Propositions \ref{decayGamma} and \ref{metriceq} and Definitions \ref{bigOnotation} and \ref{gammag}.
\end{remark}
\begin{proof}[Proof of Proposition \ref{metriceq}]
Firstly, recall from Proposition \ref{metricequations} that
\begin{align*}
    \Om\nab_4 (\ga_{AB})=2\Om\chi_{AB},
\end{align*}
which implies
\begin{align*}
        \Om\nab_4(\gac_{AB})=&\Om\nab_4(\ga_{AB}-(\ga_K)_{AB})\\
        =&\Om\nab_4(\ga_{AB})-\Om\nab_{(e_4)_K}(\ga_K)_{AB}\\
        =&2\Om\chi_{AB}-2\Om_K(\chi_K)_{AB}\\
        =&\Om\trch \,\gac_{AB}+\widecheck{\Om\trch}(\ga_K)_{AB}+2\Om\hch_{AB}-2\Om_K(\hch_K)_{AB}\\
        =&\Om\trch\,\gac_{AB}+\widecheck{\Om\trch}(\ga_K)_{AB}+2\Om\hchc_{AB}+2\Omc(\hch_K)_{AB}.
\end{align*}
Next, we compute
\begin{align*}
\Om\nab_4(\Omc)=&\Om\nab_4(\Om-\Om_K)\\
=&\Om\nab_4\Om-\Om_K\nab_{(e_4)_K}\Om_K\\
    =&-2\om\Om^2+2\om_K\Om_K^2\\
    =&-2\widecheck{\Om\om}\Om-2\Omc(\Om_K\om_K).
\end{align*}
Recall from Proposition \ref{metricequations} that
\begin{align*}
    \nab_4 (\bu^A) =-4\Om\ze^A=-4\Om\ze_B\ga^{AB}.
\end{align*}
Hence, we have
\begin{align*}
    \Om\nab_4\bcc^A=&\Om\nab_4(\bu^A)-\Om_K\nab_{(e_4)_K}(\bu^A_K)\\
   =&-4\Om^2\ze_B\ga^{AB}+4\Om_K^2(\ze_K)_B(\ga_K)^{AB}\\
    =&-4\left(\widecheck{\Om^2}\ze_B\ga^{AB}+\Om_K^2\zec_B\ga^{AB}+\Om_K^2(\ze_K)_B
    \widecheck{\ga^{AB}}\right).
\end{align*}
Next, we recall from Proposition \ref{metricequations}
\begin{equation*}
    \Om\nab_4(\in_{AB})=\Om\trch \in_{AB},
\end{equation*}
which implies
\begin{align*}
    \Om\nab_4(\inc_{AB})=&\Om\nab_4(\in_{AB})-\Om_K\nab_{(e_4)_K}(\in_K)_{AB}\\
    =&\Om\trch\in_{AB}-\Om_K\tr_K\chi_K(\in_K)_{AB}\\
    =&\Om\trch\,\inc_{AB}+\widecheck{\Om\trch}(\in_K)_{AB}.
\end{align*}
Then, we recall from Proposition \ref{metricequations}
\begin{equation*}
    \Om e_3(\ga_{AB})=2\Om\chib_{AB}-\pr_A(\bu^C)\ga_{BC}-\pr_B(\bu^C)\ga_{AC},
\end{equation*}
which implies
\begin{equation*}
    \Om_K(e_3)_K((\ga_K)_{AB})=2\Om_K(\chib_K)_{AB}-\pr_A(b_K^c)(\ga_K)_{BC}-\pr_B(b_K^c)(\ga_K)_{AC}.
\end{equation*}
Thus, we infer
\begin{align*}
    \Om e_3(\gac_{AB})=&\Om e_3(\ga_{AB})-\Om_K(e_3)_K(\ga_K)_{AB}-\widecheck{\Om e_3}(\ga_K)_{AB}\\
    =&2\Om\chib_{AB}-\pr_A(\bu^C)\ga_{BC}-\pr_B(\bu^C)\ga_{AC}-2\Om_K(\chib_K)_{AB}\\
    &+\pr_A(b_K^C)(\ga_K)_{BC}+\pr_B(b_K^C)(\ga_K)_{AC}-\bcc^C\pr_{C} (\ga_K)_{AB}\\
    =&\Om\trchb\,\ga_{AB}-\Om_K\tr_K\chib_K\,(\ga_K)_{AB}+2(\Om\hchb_{AB}-\Om_K(\hchb_K)_{AB})\\
    &-\pr_A(\bu^C)\ga_{BC}-\pr_B(\bu^C)\ga_{AC}+\pr_A(b_K^C)(\ga_K)_{BC}+\pr_B(b_K^C)(\ga_K)_{AC}-\bcc^C\pr_{C} (\ga_K)_{AB}\\
    =&\Om\trchb\,\gac_{AB}+\widecheck{\Om\trchb}(\ga_K)_{AB}+2\Om\hchbc_{AB}+2\Omc(\hchb_K)_{AB}\\
    &-\pr_A(\bcc^C)\ga_{BC}-\pr_A(\bu_K^C)\gac_{BC}-\pr_B(\bcc^C)\ga_{AC}-\pr_B(\bu_K^C)\gac_{AC}-\bcc^C\pr_C(\ga_K)_{AB}.
\end{align*}
In the particular case $\Omc=0$ and $\bcc=0$, we deduce
\begin{equation*}
    \Om e_3(\gac_{AB})-\Om\trchb\,\gac_{AB}=\widecheck{\Om\trchb}(\ga_K)_{AB}+2\Om\hchbc_{AB}-\pr_A(\bu_K^C)\gac_{BC}-\pr_B(\bu_K^C)\gac_{AC}.
\end{equation*}
Similarly, we have from Proposition \ref{metricequations}
\begin{align*}
    \Om\nab_3(\in_{AB})=\Om\trchb\in_{AB}-\pr_A(\bu^C)\in_{BC}-\pr_B(\bu^C)\in_{BC},
\end{align*}
which yields
\begin{align*}
    \Om_K(e_3)_K(\in_K)_{AB}=\Om_K\tr_K\chib_K(\in_K)_{AB}-\pr_A(\bu_K^C)(\in_K)_{BC}-\pr_B(\bu_K^C)(\in_K)_{BC}.
\end{align*}
Hence, we have
\begin{align*}
    \Om\nab_3(\inc_{AB})=&\Om e_3(\in_{AB})-\Om_K (e_3)_K (\in_K)_{AB}-\widecheck{\Om e_3}(\in_K)_{AB}\\
    =&\Om\trchb\in_{AB}-\pr_A(\bu^C)\in_{BC}-\pr_B(\bu^C)\in_{BC}-\Om_K\tr_K\chib_K(\in_K)_{AB}\\
    &+\pr_A(\bu_K^C)(\in_K)_{BC}+\pr_B(\bu_K^C)(\in_K)_{BC}-\bcc^C\pr_C(\in_K)_{AB}\\
    =&\Om\trchb\,\inc_{AB}+\widecheck{\Om\trchb}(\in_K)_{AB}-\pr_A(\bcc^C)\in_{BC}-\pr_A(\bu_K^C)\inc_{BC}\\
    &-\pr_B(\bcc^C)\in_{AC}-\pr_B(\bu_K^C)\inc_{AC}-\bcc^C\pr_C(\ga_K)_{AB},
\end{align*}
which implies from $\Omc=0$ and $\bcc=0$
\begin{align*}
    \Om\nab_3(\inc_{AB})=\Om\trchb\,\inc_{AB}+\widecheck{\Om\trchb}(\in_K)_{AB}-\pr_A(\bu_K^C)\inc_{BC}-\pr_B(\bu_K^C)\inc_{AC}.
\end{align*}
This concludes the proof of Proposition \ref{metriceq}.
\end{proof}
\section{Proof of the case \texorpdfstring{$s\in(3,4)$}{}}\label{secc}
In Sections \ref{curvatureestimates}--\ref{initiallast}, we have provided the proof of Theorem \ref{maintheorem} in the case $s\in[4,6]$. In this appendix, we outline the proof in the case $s\in(3,4)$, which is similar to the case of $s\in[4,6]$.
\subsection{Fundamental norms}\label{sec34}
The definitions of some of the norms are different from Section \ref{fundamentalnorms}. In the sequel, we only mention the norms that differ from the ones in Section \ref{fundamentalnorms}.\\ \\
We define the $\RR_0^S[\bc]$ and $\RR_0^S[\rhoc,\sic]$ as follows:
\begin{align*}
    \RR_0^{S}[\bc]&:=\sup_{\KK}\sup_{p\in [2,4]}|r^{\frac{s+3}{2}-\frac{2}{p}}\bc|_{p,S},\\
    \RRb_0^S[\rhoc,\sic]&:=\sup_{\KK}\sup_{p\in [2,4]}|r^{\frac{s+2}{2}-\frac{2}{p}}|u|^{\frac{1}{2}} (\rhoc,\sic)|_{p,S}.
\end{align*}
The following flux of curvature are different from the case $s\in[4,6]$ for $q=0,1$:
\begin{align*}
         \RR_{q}[\bc](u,\ub)&:= \Vert r^{\frac{s}{2}}(r\nab)^q\bc\Vert_{2,\cuv},\\
         \RRb_{q}[(\rhoc,\sic)](u,\ub)&:=\Vert r^{\frac{s}{2}}(r\nab)^q(\rhoc,\sic)\Vert_{2,\ucuv}.
\end{align*}
The other norms are the same as in the case $s\in[4,6]$. 
\subsection{Estimates for Ricci coefficients and curvature components}\label{ssecc2}
We discuss the curvature estimates of Section \ref{curvatureestimates} in this case. Recall that we have four Bianchi pairs: $(\ac,\bc)$, $(\bc,(\rhoc,\sic))$, $((\rhoc,\sic),\bbc)$ and $(\bbc,\aac)$. In Section \ref{curvatureestimates}, we took respectively $p=s,4,2,0$, to estimate the Bianchi pairs. In the case $s\in(3,4)$, we take $p=s,s,2,0$ to estimate $\RR$. The method is then exactly the same as in Section \ref{curvatureestimates}. \\ \\
Concerning the control of Ricci coefficients, we can proceed by the same method as in Section \ref{Ricciestimates}. Recall that Proposition \ref{propnabomcombc} used the fact that there exists a constant $\de>0$ such that
\begin{align*}
    \sup_{p\in [2,4]}|r^{3+\de-\frac{2}{p}}\bc|_{p,S}\les\ep_0.
\end{align*}
This still holds true since we have $\frac{s+3}{2}>3$ for $s>3$. The other propositions still hold true since we only used $s>3$ in their proofs. 
\subsection{Conclusion}\label{ssecc3}
The proof of Theorems M0, M2 and M4 remain exactly the same since all the arguments also apply to the case $s>3$. Consequently, we have Theorem \ref{maintheorem} in this case. Hence, we deduce that Theorem \ref{maintheorem} holds true for $s\in(3,6]$.
\section{Proof of the case \texorpdfstring{$s>6$}{}}\label{secd}
In this appendix, we prove Theorem \ref{maintheorem} in the case $s>6$. Then, we compare the result to the peeling decay for curvature components obtained in \cite{Caciotta,knpeeling}.
\subsection{Fundamental norms}\label{ssecd1}
The definitions of $\RR$--norms for $\ac$ and $\bc$ are different from Section \ref{Rnorms}. We denote
\begin{align}
    s_0:=\min\left\{s,\frac{29}{4}\right\}.
\end{align}
We define for $q=0,1$:
\begin{align}
    \begin{split}\label{newbeta}
    \RR_q[\bc]&:=|u|^\frac{s-4}{2}\|r^2(r\nab)^q\bc\|_{2,\cuv},\\
    \RRb_q[\bc]&:=\|r^{3}|u|^{\frac{s-6}{2}}(r\nab)^q\bc\|_{2,\ucuv},\\
    \RR_q[\ac]&:=\|r^3|u|^\frac{s-6}{2}(r\nab)^q\ac\|_{2,\cuv},\\
    \RRb_q[\ac]&:=|u|^\frac{s-s_0}{2}\|r^{\frac{s_0}{2}}(r\nab)^q\ac\|_{2,\ucuv}.
    \end{split}
\end{align}
We also define:
\begin{equation}
    \RRb_0^S[\bc]:=\sup_\kk\sup_{p\in [2,4]}|r^{4-\frac{2}{p}}|u|^\frac{s-5}{2}\bc|_{p,S}.
\end{equation}
The norms $\RR_0^S[\ac]$ are defined as follows:
\begin{align}
\begin{split}\label{peelinga}
    \RRb_0^S[\ac]&:=\sup_\kk|r^{\frac{s+1}{2}}\ac|_{2,S},\qquad\qquad\quad\, s\in(6,7),\\
    \RR_0^S[\ac]&:=\sup_\kk|r^{4}(\log r)^{-\frac{1}{2}}\ac|_{2,S},\qquad\,\, s=7, \\
    \RR_0^S[\ac]&:=\sup_\kk|r^{4}|u|^\frac{s-7}{2}\ac|_{2,S},\qquad\quad\;\;\, s>7.
\end{split}
\end{align}
We also define:
\begin{equation}
    \RRb_0^S[\aac]:=\sup_\kk\sup_{p\in [2,4]} |r^{1-\frac{2}{p}}|u|^\frac{s+1}{2}\aac|_{p,S}.
\end{equation}
All the other norms are defined as in Section \ref{Rnorms}.
\subsection{Curvature estimates}\label{secd3}
In Section \ref{curvatureestimates}, we took respectively $p=s,4,2,0$, to estimate the Bianchi pairs $(\ac,\bc)$, $(\bc,(\rhoc,-\sic))$, $((\rhoc,\sic),\bbc)$ and $(\bbc,\aac)$. In the case $s>6$, we take respectively $p=6,4,2,0$ for the Bianchi pairs. Proceeding as in Section \ref{curvatureestimates}, we deduce from Propositions \ref{estab}--\ref{estba} that
\begin{equation}
\sum_{q=0}^1\left(\RR_q[\ac]+\RR_q[\bc]+\RRb_q[\bc]+\RR_q[\rhoc,\sic]+\RR_q[\bbc]+\RRb_q[\rhoc,\sic]+\RRb_q[\bbc]+\RRb_q[\aac]\right)\les\ep_0,
\end{equation}\label{oldest}
which implies from Proposition \ref{sobolevkn} that
\begin{equation}\label{C8}
    \RR_0^S[\bc]+\RR_0^S[\rhoc,\sic]+\RR_0^S[\bbc]\les\ep_0.
\end{equation}
In view of \eqref{C8}, it remains to estimate $\ac$ and $\aac$. Next, we state the following \emph{Teukolsky equation} first derived by Teukolsky in the linearized setting in \cite{teukolsky}.
\begin{proposition}
We have the following Teukolsky equations for $\a$ and $\aa$:
\begin{equation}
\begin{split}\label{teukolsky}
\Om\nab_3(\Om\nab_4(r\Om^2\a))+2r\Om^2d_2^*d_2(\Om^2\a)&=-\frac{2\Om}{r}\nab_3(r\Om^2\a)+\dko\left((\O_3^2+\Gag)\cdot(\O_4^2+\b)\right), \\
\Om\nab_4(\Om\nab_3(r\Om^2\aa))+2r\Om^2d_2^*d_2(\Om^2\aa)&=\frac{2\Om}{r}\nab_4(r\Om^2\aa)+\dko\left((\O_3^2+\Gab)\cdot(\O_4^2+\aa)\right).
\end{split}
\end{equation}
\end{proposition}
\begin{proof}
    See for example Propositions 3.4.6 and 3.4.7 in \cite{DHRT}. 
\end{proof}
\begin{lemma}\label{teulm}
We define the following quantities:
\begin{align}
    \begin{split}
        \amr&:=\frac{1}{r^4}\nab_4(r^5\a)\in \sk_2,\qquad\quad \as:=rd_2\a \in \sk_1,\\
        \aamr&:=\frac{1}{r^4}\nab_3(r^5\aa)\in \sk_2,\qquad\quad \aas:=rd_2\aa \in \sk_1.
    \end{split}
\end{align}
Then, we have the following equations:
\begin{align}
    \begin{split}\label{teu}
    \nab_3\amr&=-2d_2^*\as+\frac{4\a}{r}+\dko\left((\O_2^1+\Gag)\c(\O_4^2+\bc)\right),\\
    \nab_4\as+\frac{5}{2}\trch\,\as&=d_2\amr+\dko\left((\O_2^1+\Gag)\c(\O_4^2+\bc)\right),
    \end{split}
\end{align}
and
\begin{align}
\begin{split}\label{teuaa}
\nab_4\aamr&=-2d_2^*\aas+\frac{4\aa}{r}+\dko\left((\O_2^1+\Gab)\c(\O_4^2+\aac)\right),\\
\nab_3\aas+\frac{5}{2}\trchb\,\aas&=d_2\aamr+\dko\left((\O_2^1+\Gab)\c(\O_4^2+\aac)\right).
\end{split}
\end{align}
\end{lemma}
\begin{proof}
    It follows directly from Lemma C.3 in \cite{ShenMink}.\footnote{Note that $\Gaa$ in \cite{ShenMink} decays better than $\O_2^1+\Gab$.}
\end{proof}
\begin{corollary}\label{coroteu}
We have the following linearized equations:
\begin{align*}
    \nab_3\amrc&=-2d_2^*\asc+\frac{4\ac}{r}+\fl[\a]+\err[\a],\\
    \nab_4\asc+\frac{5}{2}\trch\,\asc&=d_2\amrc+\fl[\a]+\err[\a],\\
    \fl[\a]&=\O_4^2\c\Gag^{(1)}+\O_2^1\c\bc^{(1)},\\
    \err[\a]&=\Gag^{(1)}\c\bc+\Gag\c\bc^{(1)},
\end{align*}
and
\begin{align*}
    \nab_4\aamrc&=-2d_2^*\aasc+\frac{4\aasc}{r}+\fl[\aa]+\err[\aa],\\
    \nab_3\aasc+\frac{5}{2}\trchb\,\aasc&=d_2\aamrc+\fl[\aa]+\err[\aa],\\
    \fl[\aa]&=\O_4^2\c\Gab^{(1)}+\O_2^1\c\aac^{(1)},\\
    \err[\aa]&=\Gab^{(1)}\c\aac+\Gab\c\aac^{(1)}.
\end{align*}
\end{corollary}
\begin{proof}
    It follows directly from Lemmas \ref{teulm} and \ref{useful}.
\end{proof}
\begin{lemma}\label{keyidentity}
We have the following identity for any real number $p$:
\begin{align}
\begin{split}\label{integralkey}
&\Div(r^p|\amrc|^2e_3)+2\Div(r^p|\asc|^2e_4)+(p+2)r^{p-1}|\amrc|^2+2(8-p)r^{p-1}|\asc|^2\\
=&4r^pd_1(\amrc\c\asc)+8r^{p-1}\amrc\c\ac+2r^p(\amrc+2\asc)\c(\fl[\a]+\err[\a])+r^p(\O_2^1+\Gab)\c(|\amrc|^2,|\asc|^2).
\end{split}
\end{align}
\end{lemma}
\begin{proof}
Applying Lemma \ref{keypoint} with $\psi_{(1)}=\amrc$, $\psi_{(2)}=\asc$, $a_{(1)}=0$, $a_{(2)}=\frac{5}{2}$, $h_{(1)}=\frac{4\ac}{r}+\fl[\a]+\err[\a]$, $h_{(2)}=\fl[\a]+\err[\a]$ and $k=2$, we infer
\begin{align*}
    &\Div(r^p|\amrc|^2e_3)+2\Div(r^p|\asc|^2e_4)-\left(1+\frac{p}{2}\right)r^p\trchb|\amrc|^2+2\left(4-\frac{p}{2}\right)r^p\trch|\asc|^2\\
    =&4r^pd_1(\amrc\c\asc)+8r^{p-1}\amrc\c\ac+2r^p(\amrc+2\asc)\c(\fl[\a]+\err[\a])+r^p(\O_2^1+\Gab)\c(|\amrc|^2,|\asc|^2),
\end{align*}
which concludes the proof of Lemma \ref{keyidentity}.
\end{proof}
We recall the following Poincar\'e inequality.
\begin{proposition}\label{poincare}
    For any $\a\in\sk_2$, we have the following Poincar\'e inequality:
    \begin{align*}
        |rd_2\a|^2_{2,S} \geq c_2 |\a|^2_{2,S},
    \end{align*}
    for sphere $S=S(u,\ub)$, where $c_2$ is a constant satisfying:
    \begin{equation}\label{poinc2}
        c_2=2-O(\ep).
    \end{equation}
\end{proposition}
\begin{proof}
See for example Proposition 9.3.2 in \cite{DHRT}.
\end{proof}
\begin{proposition}\label{keyint}
We have the following estimate:
\begin{align*}
    \int_{\cuv}r^{s_0}|\amrc|^2 +\int_{\ucuv}r^{s_0}|\dko\ac|^2+\int_V r^{s_0-1}\left(|\amrc|^2+|\dko\ac|^2\right)\les\frac{\ep_0^2}{|u|^{s-s_0}}.
\end{align*}
\end{proposition}
\begin{proof}
Integrating \eqref{integralkey} with $p=s_0$ and proceeding as in Proposition \ref{estab}, we obtain
\begin{align*}
    &\int_{C_u}r^{s_0}|\amrc|^2 +\int_{\Cb_\ub}r^{s_0}|\asc|^2+\int_V r^{s_0-1}\left((s_0+2)|\amrc|^2+2(8-s_0)|\asc|^2-8\amrc\c\ac\right)\\
    \les&\int_{\Si_0\cap V}r^{s_0}(|\amrc|^2+|\asc|^2)+\int_V r^{s_0}|(\amrc,\asc)||\fl[\a]+\err[\a]|+r^{s_0}(\O_2^1+\Gab)(|\amrc|^2+|\asc|^2).
\end{align*}
First, we have
\begin{align*}
    \int_{\Si_0\cap V}r^{s_0}(|\amrc|^2+|\asc|^2)\les|u|^{s_0-s}\int_{\Si_0\cap V}r^s(|\amrc|^2+|\asc|^2)\les\frac{\ep_0^2}{|u|^{s-s_0}}.
\end{align*}
Next, we estimate
\begin{align*}
    &\int_Vr^{s_0}|(\amrc,\asc)||\fl[\a]|\\
    \les&\int_Vr^{s_0}|(\amrc,\asc)|\left|\O_4^2\c\Gag^{(1)}+\O_2^1\c\bc^{(1)}\right|\\
    \les&\int_VM^2r^{s_0-4}|(\amrc,\asc)||\Gag^{(1)}|+Mr^{s_0-2} |(\amrc,\asc)||\bc^{(1)}|\\
    \les&M^2\int_{-\ub}^u \frac{du}{r^{4-\frac{s_0}{2}}}\left(\int_\cuv r^{s_0}|\amrc|^2\right)^\f12\left(\int_\cuv |\Gag^{(1)}|^2\right)^\f12+M^2\int_{|u|}^\ub\frac{d\ub}{r^{5-\frac{s_0}{2}}}\left(\int_\ucuv r^{s_0}|\asc|^2\right)^\f12\left(\int_\ucuv r^2|\Gag^{(1)}|^2\right)^\f12\\
    &+M\int_{-\ub}^u \frac{du}{r^{4-\frac{s_0}{2}}}\left(\int_\cuv r^{s_0}|\amrc|^2\right)^\f12\left(\int_\cuv r^4|\bc^{(1)}|^2\right)^\f12+M\int_{|u|}^\ub \frac{d\ub}{r^{5-\frac{s_0}{2}}}\left(\int_\ucuv r^{s_0}|\asc|^2\right)^\f12\left(\int_\ucuv r^6|\bc^{(1)}|^2\right)^\f12\\
    \les& M^2\int_{-\ub}^u \frac{du}{r^{4-\frac{s_0}{2}}} \frac{\ep}{|u|^{\frac{s-s_0}{2}}}\left(\int_{|u|}^\ub d\ub |\Gag^{(1)}|^2_{2,S}\right)^\f12+M^2\int_{|u|}^\ub \frac{d\ub}{r^{5-\frac{s_0}{2}}}\frac{\ep}{|u|^\frac{s-s_0}{2}}\left(\int_{-\ub}^u du |r\Gag^{(1)}|_{2,S}^2\right)^\f12\\
    &+M\int_{-\ub}^u \frac{du}{r^{4-\frac{s_0}{2}}}\frac{\ep}{|u|^\frac{s-s_0}{2}}\frac{\ep}{|u|^\frac{s-4}{2}}+M\int_{|u|}^\ub \frac{d\ub}{r^{5-\frac{s_0}{2}}}\frac{\ep}{|u|^\frac{s-s_0}{2}}\frac{\ep}{|u|^{\frac{s-6}{2}}}\\
    \les& M^2 \int_{-\ub}^u \frac{du}{|u|^{4-\frac{s_0}{2}}} \frac{\ep}{|u|^{\frac{s-s_0}{2}}} \frac{\ep}{|u|^\frac{s-2}{2}}+M^2\int_{|u|}^\ub \frac{d\ub}{r^{5-\frac{s_0}{2}}}\frac{\ep}{|u|^\frac{s-s_0}{2}}\frac{\ep}{|u|^\frac{s-4}{2}}+\frac{\ep^2M}{|u|^{s-s_0+1}}\\
    \les&\frac{\ep^2 M^2}{|u|^{s-s_0+2}}+\frac{\ep^2 M}{|u|^{s-s_0+1}}\\
    \les&\frac{\ep_0^2}{|u|^{s-s_0}}.
\end{align*}
Proceeding as in Proposition C.5 in \cite{ShenMink}, we infer
\begin{align*}
    \int_V r^{s_0}|(\amrc,\asc)||\err[\a]|\les\frac{\ep_0^2}{|u|^{s-s_0}}.
\end{align*}
We have from Proposition \ref{useful}
\begin{equation*}
    \asc=\widecheck{rd_2\a}=rd_2\ac+(rd_2-r_K(d_2)_K)\a_K=rd_2\ac+\Gag^{(1)}\c\O_4^3.
\end{equation*}
Combining with Proposition \ref{poincare}, we obtain
\begin{align*}
     &\int_V r^{s_0-1}\left((s_0+2)|\amrc|^2+2(8-s_0)|\asc|^2-8\amrc\c\ac\right)\\
     \gtrsim&\int_V r^{s_0-1}\left((s_0+2)|\amrc|^2+2c_2(8-s_0)|\ac|^2-8\amrc\c\ac\right)-\int_V r^{s_0-1}|\Gag^{(1)}\c\O_4^3|^2.
\end{align*}
Since the discriminant satisfies\footnote{Here, we used $6<s_0\leq\frac{29}{4}$ and $c_2=2-O(\ep)$.}
\begin{align*}
    \De=8^2-4(s_0+2)\c 2c_2(8-s_0)<0,
\end{align*}
we obtain
\begin{equation*}
    \int_V r^{s_0-1}\left((s_0+2)|\amrc|^2+2c_2(8-s_0)|\ac|^2-8\amrc\c\ac\right)\gtrsim\int_V r^{s_0-1}(|\amrc|^2+|\dko\ac|^2).
\end{equation*}
Finally, we also have
\begin{align*}
    \int_V r^{s_0}(\O_2^1+\Gab)(|\amrc|^2+|\asc|^2)\les (\ep+\EE_0)\int_V r^{s_0-1}(|\amrc|^2+|\asc|^2).
\end{align*}
Combining the above estimates, we deduce
\begin{align*}
    &\int_{C_u}r^{s_0}|\amrc|^2 +\int_{\Cb_\ub}r^{s_0}|\dko\ac|^2+\int_V r^{s_0-1}\left(|\amrc|^2+|\dko\ac|^2\right)\\
    \les& \frac{\ep_0^2}{|u|^{s-s_0}}+(\ep+\EE_0)\int_V r^{s_0-1}(|\amrc|^2+|\asc|^2)+\int_{\ucuv} r^{s_0} |\Gag^{(1)}\c\O_4^3|^2+\int_V r^{s_0-1}|\Gag^{(1)}\c\O_4^3|^2\\
    \les&\frac{\ep_0^2}{|u|^{s-s_0}}+(\ep+\EE_0)\int_V r^{s_0-1}(|\amrc|^2+|\asc|^2)+M^6\int_{-\ub}^u \frac{du}{r^{8-s_0}}|\Gag^{(1)}|^2_{2,S}+M^6\int_{-\ub}^u du\int_{|u|}^\ub  \frac{d\ub}{r^{9-s_0}}|\Gag^{(1)}|^2_{2,S}\\
    \les&\frac{\ep_0^2}{|u|^{s-s_0}}+(\ep+\EE_0)\int_V r^{s_0-1}(|\amrc|^2+|\asc|^2)+\frac{M^6\ep^2}{|u|^{6+s-s_0}}\\
    \les&\frac{\ep_0^2}{|u|^{s-s_0}}+(\ep+\EE_0)\int_V r^{s_0-1}(|\amrc|^2+|\asc|^2),
\end{align*}
which implies for $\ep$ and $\EE_0$ small enough
\begin{equation*}
    \int_{C_u}r^{s_0}|\amrc|^2 +\int_{\Cb_\ub}r^{s_0}|\dko\ac|^2+\int_V r^{s_0-1}\left(|\amrc|^2+|\dko\ac|^2\right)\les \frac{\ep_0^2}{|u|^{s-s_0}}.
\end{equation*}
This concludes the proof of Proposition \ref{keyint}.
\end{proof}
\begin{proposition}\label{finalpeeling}
    We have the following estimate:
    \begin{equation}
        \RRb_0[\ac]+\RRb_1[\ac]+\RR_0^S[\ac]\les\ep_0.
    \end{equation}
\end{proposition}
\begin{proof}
We have from Propositions \ref{ellipticest} and \ref{keyint}
\begin{equation}\label{RRbqalpha}
    \RRb_0[\ac]^2+\RRb_1[\ac]^2\les |u|^{s-s_0}\int_{\ucuv} |\ac|^2+|\dkb\ac|^2\les\ep_0^2.
\end{equation}
Recalling that
    \begin{align*}
        \amrc=\frac{1}{r^4}\nab_4(r^5\ac)+\O_4^3\c\Gag^{(1)},
    \end{align*}
we infer from Lemma \ref{dint} that
\begin{align*}
\Om e_4\left(\int_S|r^4\ac|^2\right)&=\int_S 2\Om\nab_4(r^4\ac)\cdot (r^4\ac)+\Om\trch |r^4\ac|^2\\
&=\int_S 2\Om r^4\ac\cdot(\nab_4(r^4\ac)+r^3\ac)+\Om\left(\trch-\frac{2}{r}\right)|r^4\ac|^2\\
&=\int_S 2\Om r^4\ac\cdot (\nab_4(r^5\ac )r^{-1}+e_4(r^{-1})r^5\ac+r^3\ac )+(\O_2^1+\Gag)|r^4\ac|^2\\
&=\int_S 2\Om r^4\ac\cdot (\nab_4(r^5\ac)r^{-1})+(\O_2^1+\Gag)|r^4\ac|^2\\
&=\int_S 2\Om r^4\ac\cdot r^3\amrc+(\O_2^1+\Gag)|r^4\ac|^2+ (r^4\ac)\c\O^3_1\c\Gag^{(1)}.
\end{align*}
Hence, we deduce
\begin{align*}
    \left|e_4\left(|r^4\ac|_{2,S}^2\right)\right|\les |r^4\ac|_{2,S}|r^3\amr|_{2,S}+(M+\ep)r^{-2}|r^4\ac|_{2,S}^2+|r^4\ac|_{2,S}|\O_1^3\c\Gag^{(1)}|_{2,S},
\end{align*}
which implies
\begin{equation}\label{e4acest}
    \left|e_4(|r^4\ac|_{2,S})\right|\les |r^3\amr|_{2,S}+(M+\ep)r^{-2}|r^4\ac|_{2,S}+|\O_1^3\c\Gag^{(1)}|_{2,S}.
\end{equation}
Integrating \eqref{e4acest} along $C_u$ and applying the initial assumption and Proposition \ref{keyint}, we infer
\begin{align*}
    |r^4\ac|_{2,S}&\les \frac{\ep_0}{r^\frac{s-7}{2}}+\int_{|u|}^\ub d\ub |r^3\amrc|_{2,S}+\int_{|u|}^\ub \frac{(M+\ep)d\ub}{r^2}|r^4\ac|_{2,S}+\int_{|u|}^\ub d\ub \frac{M^3}{r^2|u|^\frac{s-3}{2}}\\
    &\les \frac{\ep_0}{|u|^\frac{s-7}{2}}+\left(\int_{|u|}^\ub r^{6-s_0} d\ub \right)^\f12 \left(\int_{|u|}^\ub r^{s_0}|\amrc|^2\right)^\f12 +\int_{|u|}^\ub \frac{(M+\ep)d\ub}{r^2}|r^4\ac|_{2,S}.
\end{align*}
Applying Gronwall inequality, we obtain
\begin{equation}\label{eq5}
    |r^4\ac|_{2,S}\les\frac{\ep_0}{|u|^\frac{s-7}{2}}+\left(\int_{|u|}^\ub r^{6-s_0} d\ub \right)^\f12 \left(\int_{|u|}^\ub r^{s_0}|\amrc|^2\right)^\f12 .
\end{equation}
Notice that we have
$$
\int_{|u|}^\ub r^{6-s_0}d\ub=\left\{
\begin{aligned}
&\frac{1}{7-s_0}(\ub^{7-s_0}-|u|^{7-s_0}),   \qquad\qquad s_0\in (6,7),\\
&\log\left(\frac{\ub}{|u|}\right), \qquad\qquad\qquad\qquad\quad\;\;\, s_0=7,\\
&\frac{1}{s_0-7}(|u|^{7-s_0}-\ub^{7-s_0}), \qquad\qquad\,  s_0>7.
\end{aligned}
\right.
$$
Hence, we obtain from \eqref{eq5}
\begin{equation}\label{eq6}
|r^4\a|_{2,S}\les\left\{
\begin{aligned}
&\ep_0 r^{\frac{7-s}{2}},   \qquad\qquad s\in (6,7),\\
&\ep_0 (\log r)^\frac{1}{2}, \qquad\;\;\, s=7,\\
&\ep_0 |u|^{\frac{7-s}{2}}, \qquad\quad\;\,  s>7,
\end{aligned}
\right.
\end{equation}
which implies
\begin{equation}\label{RRSac}
    \RR_0^S[\ac]\les\ep_0.
\end{equation}
Combining \eqref{RRbqalpha} and \eqref{RRSac}, this concludes the proof of Proposition \ref{finalpeeling}.
\end{proof}
\begin{proposition}\label{finalaa}
    We have the following estimate:
    \begin{equation*}
        \RRb_0^S[\aac]\les\ep_0.
    \end{equation*}
\end{proposition}
\begin{proof}
We recall from Corollary \ref{coroteu}
\begin{align*}
    \nab_4\aamrc&=-2d_2^*\aasc+\frac{4\aasc}{r}+\fl[\aa]+\err[\aa],\\
    \nab_3\aasc+\frac{5}{2}\trchb\,\aasc&=d_2\aamrc+\fl[\aa]+\err[\aa],\\
    \fl[\aa]&=\O_4^2\c\Gab^{(1)}+\O_2^1\c\aac^{(1)},\\
    \err[\aa]&=\Gab^{(1)}\c\aac+\Gab\c\aac^{(1)}.
\end{align*}
Applying Lemma \ref{keypoint} with $\psi_{(1)}=\aamrc$, $\psi_{(2)}=\aasc$, $a_{(1)}=\frac{5}{2}$, $a_{(2)}=0$, $h_{(1)}=\fl[\aa]+\err[\aa]$, $h_{(2)}=\frac{4\aasc}{r}+\fl[\aa]+\err[\aa]$, $k=2$ and $p=-2$, we obtain
\begin{align*}
    &2\Div(r^{-2}|\aasc|^2e_3)+\Div(r^{-2}|\aamrc|^2e_4)
       +10r^{-2}\trchb|\aasc|^2\\
       =&2 r^{-2}d_1(\aasc\cdot\aamrc)
       +r^{-2}(\aasc,\aamrc)\c(\fl[\aa]+\err[\aa])+8r^{-3}\aamrc\cdot\aac+r^{-2}(\O_2^1+\Gab)(|\aamrc|^2,|\aasc|^2).
\end{align*}
We compute for $c_2$ defined in Proposition \ref{poincare}
\begin{align*}
2c_2\Div(r^{-2}|\aac|^2e_3)&=2c_2\nab_3(r^{-2}|\aac|^2)+2c_2r^{-2}|\aac|^2\Div(e_3)\\
&=2c_2\nab_3(r^{-12}|r^5\aac|^2)+2c_2r^{-2}|\aac|^2(\trchb-2\omb)\\
&=20c_2r^{-3}|\aac|^2+4c_2r^{-3}\aamrc\cdot\aac+r^{-2}(\O_2^1+\Gab)|\aac|^2+\O_7^3\c\Gab^{(1)}\c\aac,
\end{align*}
 where we used
\begin{align}
\begin{split}\label{totallye3aa}
    \aamrc&=\frac{1}{r^4}\nab_3(r^5\aa)-\frac{1}{r_K^4}(\nab_K)_{(e_3)_K}(r_K^5\aa_K)\\
    &=\frac{1}{r^4}\nab_3(r^5\aa)-\frac{1}{r^4}\nab_3(r^5\aa_K)+\frac{1}{r^4}\nab_3(r^5\aa_K)-\frac{1}{r_K^4}(\nab_K)_{(e_3)_K}(r_K^5\aa_K)\\
    &=\frac{1}{r^4}\nab_3(r^5\aac)+\O_4^3\c\Gab^{(1)}.
\end{split}
\end{align}
Combining the above identities, we deduce
\begin{align*}
    &2\Div(r^{-2}(|\aasc|^2-c_2|\aac|^2)e_3)+\Div(r^{-2}|\aamrc|^2)+10r^{-2}\trchb|\aasc|^2+20c_2r^{-3}|\aac|^2\\
    =&2r^{-2}d_1(\aasc\c\aamrc)+(8-4c_2)r^{-3}\aamrc\c\aac+r^{-2}(\aasc,\aamrc)\c(\fl[\aa]+\err[\aa])+\O_7^3\c\Gab^{(1)}\c\aac+r^{-2}(\O_2^1+\Gab)|\aac|^2.
\end{align*}
Integrating it in $V$, we obtain from Proposition \ref{poincare}
\begin{align*}
    \int_\cuv r^{-2}(|\aasc|^2-c_2|\aac|^2)+\int_{\ucuv}r^{-2}|\aamrc|^2&\les\int_{\Si_0\cap V} r^{-2}(|\aasc|^2 +|\aamrc|^2)+\int_V r^{-3}|\aasc|^2+\ep r^{-3}|\aamrc||\aac|\\
    &+\int_V r^{-2}|(\aamrc,\aasc)||\fl[\aa]+\err[\aa]|+|\O_7^3\c\Gab^{(1)}||\aac|+ r^{-2}(\O_2^1+\Gab)|\aac|^2.
\end{align*}
We first have from Proposition \ref{useful}
\begin{equation*}
    \aasc=\widecheck{rd_2\aa}=rd_2\aac+(rd_2-r_K(d_2)_K)\aa_K=rd_2\aac+\Gag^{(1)}\c\O_4^3.
\end{equation*}
Noticing that
\begin{align*}
    \int_\cuv r^{-2}|\Gag^{(1)}\c\O_4^3|^2\les\int_{|u|}^\ub d\ub \frac{M^6}{r^{12}}|r\Gag^{(1)}|^2_{2,S}\les M^6\int_{|u|}^\ub d\ub\frac{\ep^2}{r^{12}|u|^{s-3}}\les\frac{\ep_0^2}{|u|^{s+2}},
\end{align*}
Hence, we obtain from Proposition \ref{poincare}
\begin{align*}
    \int_\cuv r^{-2}(|\aasc|^2-c_2|\aac|^2)&\geq \int_\cuv r^{-2}(|rd_2\aac|^2-c_2|\aac|^2)-\int_\cuv r^{-2}|\Gag^{(1)}\c\O_4^3|^2\gtrsim-\frac{\ep_0^2}{|u|^{s+2}}.
\end{align*}
Next, we have
\begin{align*}
    \int_{\Si_0\cap V} r^{-2}(|\aasc|^2 +|\aamrc|^2) \les \frac{1}{|u|^{s+2}} \int_{\Si_0\cap V} r^{s}(|\aasc|^2 +|\aamrc|^2)\les\frac{\ep_0^2}{|u|^{s+2}}.
\end{align*}
Then, we compute
\begin{align*}
\int_V r^{-3}|\aasc|^2+\ep\int_V r^{-3}|\aamrc||\aac| &\les \int_{|u|}^\ub \frac{d\ub}{r^3}\int_\ucuv |\aasc|^2+\ep\int_{|u|}^\ub r^{-2} \left(\int_\ucuv r^{-2}|\aamrc|^2 \right)^\frac{1}{2}\left(\int_\ucuv |\aac|^2\right)^\frac{1}{2}d\ub\\
&\les \ep_0^2\int_{|u|}^\ub \frac{d\ub}{r^3|u|^s}+\ep\int_{|u|}^\ub r^{-2}\left(\frac{\ep^2}{|u|^{s+2}}\right)^\frac{1}{2}\left(\frac{\ep_0^2}{|u|^{s}}\right)^\frac{1}{2}d\ub\\
&\les\frac{\ep_0^2}{|u|^{s+2}}.
\end{align*}
We also have
\begin{align*}
    &\int_V r^{-2}|\aamrc,\aasc||\fl[\aa]|\\
    \les&\int_V r^{-2}|\aamrc,\aasc|\left|\O_4^2\c\Gab^{(1)}+\O_2^1\c\aac^{(1)}\right|\\
    \les&\int_V M^2 r^{-6}|\aamrc||\Gab^{(1)}|+M^2 r^{-6}|\aasc||\Gab^{(1)}|+Mr^{-4}|\aamrc||\aac^{(1)}|+Mr^{-4}|\aasc||\aac^{(1)}|\\
    \les&M^2\int_{|u|}^\ub \frac{d\ub}{r^5}\left(\int_\ucuv r^{-2}|\aamrc|^2 \right)^\f12\left(\int_\ucuv |\Gab^{(1)}|^2\right)^\f12+M^2\int_{|u|}^\ub \frac{d\ub}{r^5}\left(\int_\ucuv|\aasc|^2\right)^\f12\left(\int_\ucuv |r^{-1}\Gab^{(1)}|^2\right)^\f12\\
    &+M\int_{|u|}^\ub\frac{d\ub}{r^3}\left(\int_\ucuv r^{-2}|\aamrc|^2\right)^\f12\left(\int_\ucuv|\aac^{(1)}|^2\right)^\f12+M\int_{|u|}^\ub\frac{d\ub}{r^4}\left(\int_\ucuv|\aasc|^2 \right)^\f12\left(\int_\ucuv |\aac^{(1)}|^2\right)^\f12\\
    \les& M^2\int_{|u|}^\ub \frac{d\ub}{r^5}\frac{\ep}{|u|^\frac{s+2}{2}}\left(\int_{-\ub}^u du|\Gab^{(1)}|^2_{2,S}\right)^\f12+M^2\int_{|u|}^\ub\frac{d\ub}{r^5}\frac{\ep}{|u|^\frac{s}{2}}\left(\int_{-\ub}^u du |r^{-1}\Gab^{(1)}|^2_{2,S}\right)^\f12\\
    &+M\int_{|u|}^\ub \frac{d\ub}{r^3}\frac{\ep}{|u|^\frac{s+2}{2}}\frac{\ep}{|u|^\frac{s}{2}}+M\int_{|u|}^\ub\frac{d\ub}{r^4}\frac{\ep}{|u|^\frac{s}{2}}\frac{\ep}{|u|^\frac{s}{2}}\\
    \les&\frac{M^2\ep^2}{r^4|u|^s}+\frac{M\ep^2}{r^2|u|^{s+1}}+\frac{M\ep^2}{r^3|u|^s}\\
    \les&\frac{\ep_0^2}{|u|^{s+2}}.
\end{align*}
We proceed as in Proposition C.7 in \cite{ShenMink} to deduce
\begin{align*}
    \int_V r^{-2}|\aamrc,\aasc||\err[\aa]|\les\frac{\ep_0^2}{|u|^{s+2}}.
\end{align*}
Moreover, we have
\begin{align*}
    \int_V|\aac||\O_7^3\c\Gab^{(1)}|&\les M^3\int_{|u|}^\ub \frac{d\ub}{r^7}\left(\int_\ucuv |\aac|^2\right)^\f12 \left(\int_\ucuv |\Gab^{(1)}|^2\right)^\f12 \\
    &\les M^3 \int_{|u|}^\ub \frac{d\ub}{r^7}\frac{\ep}{|u|^\frac{s}{2}}\left(\int_{-\ub}^u du |\Gab^{(1)}|^2_{2,S}\right)^\f12\\
    &\les M^3\ep\int_{|u|}^\ub \frac{d\ub}{r^7|u|^\frac{s}{2}}\frac{\ep}{|u|^\frac{s-2}{2}}\\
    &\les\frac{\ep_0^2}{|u|^{s+2}}.
\end{align*}
Finally, we have
\begin{align*}
    \int_V r^{-2}(\O_2^1+\Gab)|\aac|^2\les\int_V(\ep+\EE_0)r^{-3}|\aac|^2\les \frac{\ep_0^2}{|u|^{s+2}}.
\end{align*}
Combining the above estimates, we deduce
\begin{align*}
    \int_\ucuv r^{-2}|\aamrc|^2\les \frac{\ep_0^2}{|u|^{s+2}},
\end{align*}
which implies from \eqref{totallye3aa} and \eqref{e3r}
\begin{align}\label{e3aac}
    \int_\ucuv r^{-2}|\nab_3\aac|^2\les\frac{\ep_0^2}{|u|^{s+2}}.
\end{align}
Combining \eqref{e3aac} and \eqref{oldest} and applying Proposition \ref{sobolevkn}, we deduce
\begin{equation*}
    \RRb_0^S[\aac]\les\ep_0.
\end{equation*}
This concludes the proof of Proposition \ref{finalaa}.
\end{proof}
Combining Propositions \ref{finalpeeling} and \ref{finalaa} with \eqref{oldest}, this concludes the proof of Theorem M1 in the case $s>6$. Notice that Theorems M0, M2, M3, M4 also hold true in the case $s>6$ since we only used $s>3$ in their proofs\footnote{See the discussion in Section \ref{ssecc2}.}. Combining with Section \ref{ssecc3}, this concludes the proof of Theorem \ref{maintheorem} for $s>3$ as stated.
\subsection{Peeling decay for curvature components}
As a consequence of Theorem M1 in the case $s>7$, the curvature components satisfy the following decay:
\begin{align}
    \begin{split}\label{subpeeling}
        |r^{5-\frac{2}{p}}|u|^\frac{s-7}{2}\ac|_{p=2,S}&\les\ep_0,\\
        \sup_{p\in[2,4]}|r^{4-\frac{2}{p}}|u|^{\frac{s-5}{2}}\bc|_{p,S}&\les\ep_0,\\
        \sup_{p\in[2,4]}|r^{3-\frac{2}{p}}|u|^{\frac{s-3}{2}}(\rhoc,\sic)|_{p,S}&\les \ep_0,\\
        \sup_{p\in[2,4]}|r^{2-\frac{2}{p}}|u|^{\frac{s-1}{2}}\bbc|_{p,S}&\les\ep_0,\\
        \sup_{p\in[2,4]}|r^{1-\frac{2}{p}} |u|^{\frac{s+1}{2}}\aac|_{p,S}&\les\ep_0.\\
    \end{split}
\end{align}
\begin{remark}
    Assume $s>7$ and that the initial data in $\kk_{(0)}$ has sufficient regularity properties. Then, commuting $r\nab$ with the Bianchi equations in Propositions \ref{Bianchieq} and \ref{Bianchieqdkb} and the Teukolsky equation \eqref{teukolsky}, proceeding as in Sections \ref{curvatureestimates} and \ref{secd3}, we deduce the following strong peeling properties:
    \begin{align}
    \begin{split}\label{strongpeeling}
        |\ac|_{\infty,S(u,\ub)}&\les \frac{\ep_0}{r^{5}|u|^\frac{s-7}{2}},\qquad\quad \,|\bc|_{\infty,S(u,\ub)}\les \frac{\ep_0}{r^{4}|u|^\frac{s-5}{2}},\\
        |\rhoc,\sic|_{\infty,S(u,\ub)}&\les\frac{\ep_0}{r^3|u|^\frac{s-3}{2}},\qquad\quad\; |\bbc|_{\infty,S(u,\ub)}\les \frac{\ep_0}{r^2|u|^\frac{s-1}{2}},\qquad \quad|\aac|_{\infty,S(u,\ub)}\les \frac{\ep_0}{r|u|^\frac{s+1}{2}},
    \end{split}
\end{align}
which recovers the results obtained in \cite{Caciotta} by the vectorfield method. See also Theorem 1 in \cite{knpeeling} and Remark C.8 in \cite{ShenMink} for the particular case of Minkowski stability.
\end{remark}
\section*{Declarations}
\addcontentsline{toc}{section}{Declarations}
\noindent{\bf Acknowledgements.} The author is very grateful to J\'er\'emie Szeftel for his support, discussions, encouragements and patient guidance.\\ \\
{\bf Funding.} No funding was received to assist with the preparation of this manuscript.\\ \\
{\bf Competing Interest statements.} Conflict of interest does not exist in the manuscript.

\end{document}